\documentclass[11pt]{article}
\usepackage{amsthm}
\usepackage{amsmath}
\usepackage{amssymb}
\usepackage{enumitem}
\usepackage{bbold} 
\usepackage{xfakebold} 
\usepackage{bm}
\usepackage{dsfont}
\usepackage{hyperref}
\usepackage{cases}
\usepackage{empheq}
\usepackage{mathrsfs}
\usepackage{todonotes} 
\usepackage{float}
\usepackage{hyperref}
\definecolor{darkgreen}{rgb}{0.1,0.5,0.1}
\hypersetup{colorlinks,linkcolor={red},citecolor={darkgreen},urlcolor={red}}

\usepackage{graphicx}
\usepackage{tikz}
\usetikzlibrary{arrows,positioning,fit,shapes,calc}

\newtheorem{Thm}{Theorem}
\newtheorem{Lem}[Thm]{Lemma}
\newtheorem{Def}[Thm]{Definition}
\newtheorem{Pro}[Thm]{Proposition}
\newtheorem{Cor}[Thm]{Corollary}
\newtheorem{Rem}[Thm]{Remark}

\newcommand{\beqa}{\begin{eqnarray}}
\newcommand{\eeqa}[1]{\label{#1}\end{eqnarray}}
\newcommand{\beq}{\begin{equation}}
\newcommand{\eeq}[1]{\label{#1}\end{equation}}

\newcommand{\Div}{{\bf \mathrm{ div \,}}}

\newcommand{\curlvec}{\operatorname{\bf curl\,}}

\newcommand{\Tr}{\mathop{\rm Tr}\nolimits}

\newcommand{\redA}{\hat{\mathcal{A}}_\perp}

\newcommand{\beps}{{\setBold[0.5]\varepsilon}}
\newcommand{\bmu}{{\setBold[0.5]\mu}}
\newcommand{\bbD}{{\setBold[0.3]\mathfrak{D}}}

\newcommand{\rmd}{{\mathrm{ d}}}

\newcommand{\rmi}{{\mathrm{ i}}}
\newcommand{\R}{{\mathbb{ R}}}
\newcommand{\N}{{\mathbb{ N}}}
\newcommand{\bbR}{{\mathbb{R}}}
\newcommand{\bbC}{{\mathbb{C}}}

\newcommand{\bk}{\mathbf{k}}
\newcommand{\bx}{\mathbf{x}}
\newcommand{\by}{\mathbf{y}}
\newcommand{\bU}{\mathbf{U}}
\newcommand{\bL}{\mathbf{L}}
\newcommand{\bu}{\mathbf{u}}
\newcommand{\bv}{\mathbf{v}}

\newcommand{\bX}{\mathbf{X}}
\newcommand{\bY}{\mathbf{Y}}

\newcommand{\bE}{\mathbf{E}}
\newcommand{\bH}{\mathbf{H}}

\newcommand{\bF}{\mathbf{F}}
\newcommand{\bV}{\mathbf{V}}

\newcommand{\bn}{\mathbf{n}}
\newcommand{\bw}{\mathbf{w}}

\newcommand{\bbE}{\mathbb{E}}

\newcommand{\bbM}{\mathbb{M}}
\newcommand{\bbU}{\mathbb{U}}
\newcommand{\bbG}{\mathbb{G}}
\newcommand{\bbV}{\mathbb{V}}

\newcommand{\calQ}{{\mathcal{Q}}}

\usepackage{authblk}
\author[1]{Philippe Briet}
\author[2]{Maxence Cassier}
\author[3]{Thomas Ourmières-Bonafos}
\author[2]{Michele Zaccaron}
\affil[1]{Université de Toulon, Aix Marseille Univ, CNRS, CPT, Toulon, France}
\affil[2]{Aix Marseille Univ, CNRS, Centrale Med, Institut Fresnel, Marseille, France}
\affil[3]{Aix-Marseille Univ, CNRS, I2M, Marseille, France}

\begin{document}
\vspace{-1in}
\title{Geometric spectral properties of electromagnetic waveguides}
\maketitle


\begin{abstract}
Consider a reference homogeneous and isotropic electromagnetic waveguide with a simply connected cross-section  embedded in a perfect conductor. In this setting, when the waveguide is straight, the spectrum of the associated self-adjoint Maxwell operator with a  \emph{ constant twist } (which may be zero) lies on the real line and is symmetric with respect to zero and exhibits a spectral gap around the origin. Moreover,  the spectrum is purely essential, and contains 0 which is an eigenvalue of infinite multiplicity. \\
In this work, we present new results on the effects of geometric deformations, specifically bending and twisting, on the spectrum of the Maxwell operator. More precisely, we provide, on the one hand, sufficient conditions on the asymptotic behavior of curvature and twist that ensure the preservation of the essential spectrum of the reference waveguide. Our approach relies on a Birman–Schwinger-type principle, which may be of independent interest and applicable in other contexts.  On the other hand,  we give sufficient conditions (involving in particular the geometrical shape of the cross-section of the waveguide) so that the geometrical deformation creates discrete spectrum (namely isolated eigenvalues of finite multiplicity) within the gap of the essential spectrum. In addition, we give some results on the localization of these discrete eigenvalues. The sufficient condition involving the cross-section is then studied both analytically and numerically. Finally, we examine its stability under shape deformations of the cross-section, focusing in particular on the case of a waveguide with a rectangular cross-section.
\end{abstract}

{\noindent \bf Keywords:} Electromagnetic waveguides, Maxwell's equations, geometric spectral theory, Birman-Schwinger principle, trapped modes, shape derivatives.

\tableofcontents

\section{Introduction and main results}

\subsection{Motivations and State of the art}

Due to their geometrical properties, waveguides appear as key guiding structures to propagate waves in physics and engineering.  They have numerous applications in quantum and classical physics. For instance, in mesoscopic physics,  the so-called quantum waveguides model is used to modelize the transport of electrons in quantum wires such as carbon nanotubes (see e.g. \cite{Exner-2015}) and in engineering, waveguides are intensively used  in telecommunication for the propagation of acoustic or electromagnetic waves (see e.g. \cite{Fah-2000,Mah-91}).

In the context of quantum waveguides, one is interested in the spectrum of the Laplace operator posed in the waveguide with adequate  boundary conditions which ensure its self-adjointness. An important question in this setting is the effect of a geometrical perturbation  of the waveguide (such as bending, twisting,...) on the  structure of the spectrum.  In particular, does the geometrical perturbation create discrete spectrum  or prevent its existence? The discrete spectrum corresponds to isolated eigenvalues of finite multiplicity whose associated  eigenfunctions (the so-called bounded states)  are functions ``confined  in space'' within the waveguide.
In the papers \cite{ESS89,ES89,GJ92,Krej04}, the two-dimensional setting is investigated for waveguides {\it straight at infinity} and the quantum particle confined within the waveguide is subject to Dirichlet boundary conditions. In this setting,  the essential spectrum of the Hamiltonian is a half-line whose bottom is given by the first eigenvalue of the Dirichlet Laplacian posed in the cross-section. Moreover, provided the base curve has non-zero curvature, discrete spectrum appears below the threshold of the essential spectrum. In dimension three, because of a more intricate geometry, the situation is more complex: the base curve may have curvature and the full waveguide may be twisted (think for instance about a constantly rotating cross-section $\omega$ around a given base curve). We mention the work \cite{DE95} where non-twisted waveguides are investigated, \cite{EKK08,BKRS09,BHK14,BMP18} where only twisted ({\it i.e.} non-curved) waveguides are studied and \cite{BMPP18} where a combination of both effects are investigated. Roughly speaking, the general idea is that provided the waveguide is {\it straight} at infinity and the twisting is {\it constant at infinity} the essential spectrum is a half-line but there is an interplay between twisting and bending: bending tends to create eigenvalues below the threshold of the essential spectrum while  certain  twisting prevents them to appear. We also mention  \cite{BMT07,BDK20} for observation of the above phenomena with  other type of boundary conditions than the Dirichlet boundary condition.\\

In the papers \cite{BBKOB22,BKOB23,LTROB24}, the authors investigate the relativistic counterpart of the Dirichlet Laplacian, that is the Dirac operator with infinite mass boundary conditions posed in two-dimensional waveguides. Despite the natural technicity implied by the vector-valued Dirac system, it turns out the most significant result is that the essential spectrum has a gap around the origin whose size is given by a transverse operator and that eigenvalues (and their associated bound states) can appear in the gap provided the base curve is curved.\\

In the context of classical waves, the bound states (or the so-called trapped modes in this domain, see e.g. \cite{BM07,Pagn-2013,CHE-20}) have been observed experimentally for instance for electromagnetic \cite{Anino-06} or water waves \cite{Cob-11}... 
The interest of trapped modes  associated to discrete eigenvalues lies in the fact that if one excites the system in the vicinity of the associated frequency then it produces field exaltation within the waveguide. From the numerical point of view, it is interesting to know the localization of the discrete spectrum due to the fact that the associated time-harmonic equation are not well-posed at these particular frequencies, causing numerical instability.

Because of potential applications in telecommunication,  in this paper  we focus on  electromagnetic waveguides  embedded in a perfect conductor. Numerous theoretical works have been undertaken in this direction and let us mention \cite{Cess-98,daulio,filonov19,filonov21}. However, most of these works are concerned with straight elctromagnetic waveguides, namely $\R \times \omega$ where $\omega$ denotes the two-dimensional cross-section of the waveguide. Our goal is to  partially fill this gap. In the physics literature we mention  the  work \cite{GJ92} (see section IV.A.) where the authors present a way to construct a trapped mode by bending an electromagnetic waveguide with a rectangular cross-section  $\omega$. Their argument is based on the construction of a bound state for the Dirichlet Laplace operator on a two-dimensional waveguide alongside a geometrical reasoning which holds only for rectangular cross-sections.
 We recall their argument in details in \S \ref{sec-rec-waveguide}. However, the authors mention that they were not able to extend the general result obtained for the scalar case of the quantum waveguide to the vectorial case of Maxwell's equations and  more precisely, in the electromagnetic setting, they ``\emph{believe that in more complicated
geometries the existence and nature of bound states depend
more delicately on the shape and curvature than in
the scalar case}''. One goal of this paper is to answer partially to this affirmation.

In the present work, the base curve is more general and one studies the effect of twisting and bending. However, we investigate the spectrum of the Maxwell operator for homogenous isotropic waveguides with constant permittivity and permeability. More precisely, for waveguides {\it straight} and with  {\it constant twist} at infinity, we are able to identify the essential spectrum which has a gap around zero. Additionally, in the case of non-twisted waveguides, we provide a sufficient geometric condition on the cross-section of the waveguide to ensure the existence of discrete eigenvalues in this gap. We would like to emphasize that it is the first time such questions are addressed and investigated  in the context of electromagnetism. 
In particular, the vectorial nature  and the functional framework of Maxwell's equations complicate significantly the analysis, which is not a just an extension of the scalar case and requires the development of new techniques and reasoning adapted to  this case. 
We hope this work will inspire future studies in this direction.
\subsection{Description of the mathematical setting}

\subsubsection{Maxwell's equations and operator reformulation of the evolution problem}

Let $\Omega$ be a non-empty open set  of $\mathbb{R}^3$.
We assume that the electromagnetic material filling  $\Omega$ is a non-dispersive dielectric  material (possibly anisotropic and inhomogeneous) whose  permittivity  $\varepsilon$ and  permeability   $\mu$ are $3\times3\ $ real symmetric matrix-valued functions satisfying the following  conditions: there exists $0<c_\varepsilon \leq C_\varepsilon$ and $0<c_\mu \leq C_\mu$  positive numbers such that
\begin{equation} 	\label{unif:elliptic:eu}
0<c_\varepsilon |\zeta|^2 \leq \varepsilon(\bx) \zeta \cdot \bar\zeta \leq C_\varepsilon |\zeta|^2, \quad 0<c_\mu |\zeta|^2 \leq \mu(\bx) \zeta \cdot \bar\zeta \leq C_\mu |\zeta|^2 \qquad \forall \bx \in \Omega, \forall \zeta \in \bbC^3.
\end{equation}

We denote respectively by  ${\bf E }$ and ${\bf H}$ the electric and magnetic fields.
We assume that in the presence of a source current density ${\bf J}$, the evolution of  $({\bf E}, {\bf H})$ is governed in $\Omega$ by the macroscopic time-dependent Maxwell's equations
(in the following, the notation $\curlvec$ refers to the usual three-dimensional curl operator)
 \begin{align}\label{eq.Maxwell}
\begin{aligned}
	\varepsilon \partial_t {\bf E}(\bx, t) -\curlvec {\bf H}(\bx, t ) &= -\bf{J}(\bx, t),  \\  \mu \partial_t {\bf H}(\bx, t) +\curlvec {\bf E}(\bx, t) &= 0,
 \end{aligned}
  &&
\begin{aligned}
	&\forall (\bx,t )\in \Omega \times \bbR^{+,*}  \\   &\forall (\bx,t )\in \Omega \times \bbR^{+,*}. 
\end{aligned}
\end{align}


\noindent Then, the equations \eqref{eq.Maxwell} have to be completed by  a  boundary condition at the boundary $\partial \Omega$ of $\Omega$. We use here the physical assumption that $\Omega$ is embedded in a perfect conductor which imposes the following  transmission condition on the   tangential  component of the electrical field:
\begin{equation}\label{eq.bc}
{\bf E}(\bx, t)  \times  \bn =0,\quad   \ \forall (\bx,t )\in \partial \Omega \times \bbR^{+,*}.
\end{equation}
Note that Condition \eqref{eq.bc} is  defined   by \eqref{eq.Hcurl0}.
For a more regular open set $\Omega$ , e.g. when $\Omega$ is bounded and Lipschitz (see Remark \ref{remark:normalcurl}),  the vector $\bn$ (where it is defined) has to be understood has the unit outward normal vector to $\partial \Omega$ and   $(\bE \times \bn)|_{\partial \Omega} = 0$  can be interpreted as a vanishing tangential trace condition  of the electric field  in the Sobolev space $\bH^{-\frac12}(\partial \Omega)$.\\[4pt]
\noindent Finally, the problem has to be supplemented by  initial conditions on the  fields
\begin{equation}\label{eq.Ic}
{\bf E}(\bx, 0)={\bf E}_0(\bx) \ \mbox{ and } \ {\bf H}(\bx, 0)= {\bf H}_0(\bx).
\end{equation}

The evolution system  \eqref{eq.Maxwell}, \eqref{eq.bc} and \eqref{eq.Ic} can be reformulated as a conservative Schr\"odinger type evolution equation
\begin{equation}\label{eq.schro}
       \partial_t \bU + i\mathcal{A}_{\varepsilon, \mu}\bU = \bF, \quad \forall \,   t >0  \quad \mbox{ with } \
            \bU(0) = \bU_0
       \end{equation}
in the Hilbert space  $$\bL^2_{\varepsilon, \mu}(\Omega):=\bL^2_{\varepsilon}(\Omega) \oplus \bL^2_{\mu}(\Omega),$$  where the  weighted  $L^2-$space $\bL^2_\Sigma(\Omega)=L^2(\Omega)^3$ is here endowed with the following  inner product
\begin{equation*}
(\bu,\bv)_{\bL^2_\Sigma(\Omega)}=\int_\Omega \Sigma \bu \cdot \overline{\bv} \, \rmd \bx, \quad \forall u, v\in \bL^2_\Sigma(\Omega),
\end{equation*}
where the  weight $\Sigma= \varepsilon$ or  $\Sigma= \mu$ is a  $3\times3\ $ real symmetric matrix-valued functions  satisfying \eqref{unif:elliptic:eu}.
The Hilbert space $\bL^2_{\varepsilon, \mu}(\Omega)$ is thus equipped with the inner product defined by  
\begin{equation} \label{eq.energy}
\hspace{-0.2cm}(\bU,\bU')_{\bL^2_{\varepsilon, \mu}(\Omega)}=
\int_\Omega \left(  \varepsilon \bE_1 \cdot \overline{\bE}_2  + \mu \bH_1 \cdot\overline{\bH}_2  \right) \rmd \bx,  \,  \forall \bU=(\bE, \bH)^{\top}, \, \bU'=(\bE', \bH')^{\top} \in  \bL^2_{\varepsilon, \mu}(\Omega).
\end{equation}
We point out that the Hilbert-norm associated to the above inner product is indeed (up to a factor $1/2$) the electromagnetic energy in $\Omega$.


\noindent In order to define the \emph{Maxwell} operator $\mathcal{A}_{\varepsilon, \mu}$, which plays the role of the Hamiltonian in \eqref{eq.schro}, we follow the approach used in \cite{birsol}.
For all the precise definition of the Hilbert spaces introduced we refer to Section \ref{notation:sect}.
We first define the \emph{maximal} $\curlvec$ operator acting in $\bL^2(\Omega)$ defined on the domain $H(\curlvec, \Omega)$.
Then we define the \emph{minimal} $\curlvec$ operator acting in $\bL^2(\Omega)$ defined on the domain $H_0(\curlvec, \Omega)$. 
These two operators are densely defined, closed and mutually adjoint in $\bL^2(\Omega)$ (see e.g. the proof of Lemma \ref{lem:55} for more details or \cite{birsol}), therefore the following block operator  $\mathcal{A}_{\varepsilon, \mu}: D(\mathcal{A}_{\varepsilon, \mu}):= H_0(\curlvec,\Omega)\times H(\curlvec,\Omega) \subset \bL_{\varepsilon,\mu}^2(\Omega) \to \bL_{\varepsilon,\mu}^2(\Omega)$ is self-adjoint:
\begin{equation}\label{eqn:defmaxgen}
  \mathcal{A}_{\varepsilon, \mu} :=   \rmi
    \begin{pmatrix}
        0 & \varepsilon^{-1} \curlvec \\
        -\mu^{-1}\curlvec & 0
    \end{pmatrix}.
\end{equation}
\begin{Rem} Here, the vector space $H_0(\curlvec,\Omega)$ is the vector space of vector fields $\bE \in H(\curlvec,\Omega)$ satisfying the boundary condition $\bE\times \bn|_{\partial\Omega} = 0$ which has to be understood as in \cite{birsol}. Namely, we say that $\bE\times \bn|_{\partial\Omega} = 0$ if one has 
\[
   (\curlvec \bE, \bV)_{\bL^2(\Omega)} = (\bE,\curlvec\bV)_{\bL^2(\Omega)}, \, \ \forall \ \bV \in H(\curlvec,\Omega).
\]
This definition of $\bE\times \bn|_{\partial\Omega} = 0$ coincides with the usual definition of null tangential trace if the Lipschitz  domain $\Omega$ is  bounded or is the complement of Lipschitz bounded domain \cite{Monk,Boy-12,Kirsch-15}.
\end{Rem}
Finally the source term $\bF$ in  \eqref{eq.schro} is given by $\bF=(- \varepsilon^{-1} \bf{J}, 0)^{\top}\in \bL^2_{\varepsilon, \mu}(\Omega)$. Let us mention that for initial condition $\bU_0 \in D(\mathcal{A}_{\varepsilon, \mu})$ and source term $\bF\in \mathscr{C}^1(\bbR^+,\bL^2_{\varepsilon, \mu}(\Omega) )$, we know from the Hille-Yosida theorem (see e.g. \cite{Arendt-11} section I.3) that the evolution problem  \eqref{eq.schro} admits a unique solution $\bU \in \mathscr{C}^1\big([0, +\infty), \bL^2_{\varepsilon, \mu}(\Omega)\big)\cap\mathscr{C}^0\big([0, +\infty), D(\mathcal{A}_{\varepsilon, \mu})\big) $. \\[4pt]
\noindent In the next section, we describe the geometrical construction of  a general waveguide $\Omega$.




\subsubsection{Construction of waveguides} \label{section:waveguides}

For $d \in \mathbb{N}$, in the Euclidean space $\R^d$, a waveguide is defined as a tubular neighborhood of a {\it base curve} constructed thanks to a Lipschitz domain $\omega \subset \R^{d-1}$, called the {\it cross-section}, and one of the simplest example is the {\it straight} waveguide $\R \times \omega$.
In this paper, for physical reasons, we focus on the case $d=3$.

Let $\omega$ be a bounded simply connected and Lipschitz domain of $\mathbb{R}^2$. We will refer to $\omega$ as the {\it cross-section} of the waveguide and the coordinates on $\omega$ will be denoted $\mathbf{y} = (y_2,y_3)\in \omega$. By {\it straight waveguide}, we mean the Cartesian product $\Omega_0 := \mathbb{R}\times \omega$ (see Figure \ref{fig:1}).

\begin{figure}[b]
\centering
\includegraphics[width=0.4\textwidth]{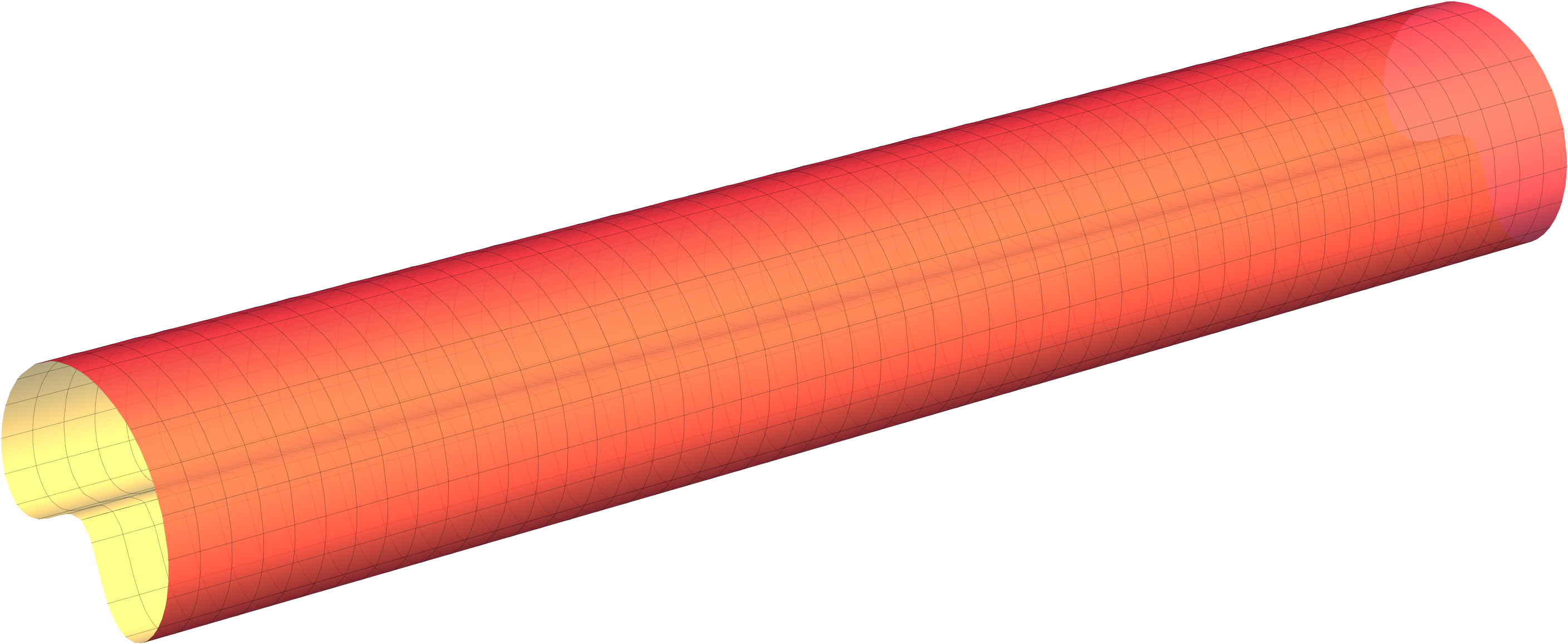}
\caption{An example of a straight waveguide.}
\label{fig:1}
\end{figure}
Let  $\Gamma \subset \mathbb{R}^3$ be a curve with a smooth (at least $\mathscr{C}^2$) injective arc-length parametrization $\gamma: \mathbb{R} \to \mathbb{R}^3$ and whose curvature $\kappa: s\in \mathbb{R} \mapsto |\gamma''(s)|$ is a bounded  function on $\mathbb{R}$.
We need now to define an appropriate moving frame  along $\Gamma$.  The most well-known example is the Frenet frame (see, e.g., \cite{doca}, Section 1.5). However, this frame has two main drawbacks. First, its definition requires that the curvature $\kappa$ does not vanish, which makes it unsuitable for defining waveguides that are straight  or locally straight. Second, it assumes that the arc-length parametrization $\gamma$ is at least $\mathscr{C}^3$. To overcome these limitations, we instead equip the curve $\Gamma$ with a $\mathscr{C}^1$-smooth relatively adapted parallel frame (we refer to \cite{bishop} for the introduction of such frames), which generalizes the Frenet frame and does not require any of these restrictive conditions. Such a frame is  defined by setting, for $s\in \R$, $\mathbf{e}_1(s) := \gamma'(s)$ and choosing two vectors $\mathbf{e}_2(s),\mathbf{e}_3(s)$ which are relatively parallel, that is, verifying for all $s\in \mathbb{R}$:
\begin{equation} \label{frame:evolution}
    \left\{
        \begin{array}{lcl}
            \displaystyle\frac{d \mathbf{e}_1}{ds}(s) & = & k_1(s) \mathbf{e}_2(s) + k_2(s) \mathbf{e}_3(s),  \vspace{3pt} \\
            \vspace{3pt}
            \displaystyle\frac{d \mathbf{e}_2}{ds}(s) & = & - k_1(s) \mathbf{e}_1(s),\\
            \displaystyle\frac{d \mathbf{e}_3}{ds}(s) & = & -k_2(s) \mathbf{e}_1(s).
        \end{array}
    \right.
\end{equation}
Here $k_1, k_2$ are continuous functions on $\R$ satisfying for all $s\in \R$, $k_1(s)^2 + k_2(s)^2 = \kappa(s)^2$. 
More precisely, such a $\mathscr{C}^1$-smooth relatively adapted parallel frame is uniquely defined on $\bbR$ by $\mathbf{e}_1(s) := \gamma'(s)$ and  \eqref{frame:evolution} after fixing initial conditions (for instance at $s=0$): $(\mathbf{e}_2(0),\mathbf{e}_3(0))=(\mathbf{e}_{2,0}, \mathbf{e}_{3,0})$ where the initial condition is chosen such that $(\gamma'(0),\mathbf{e}_{2,0}, \mathbf{e}_{3,0})$  is a positively oriented orthonormal basis of $\R^3$. We refer to  Proposition \ref{appendix:relatparallel} and its proof for more details and also for the construction of the  continuous functions $k_1$ and $k_2$.

\noindent Let $\theta: \mathbb{R}\to \mathbb{R}$ be a smooth map (at least $\mathscr{C}^1$) and introduce the following vectors
\begin{equation}\label{eqn:defetheta}
\mathbf{e}_1^\theta(s):=\mathbf{e}_1(s) \quad \mbox { and } \quad  \begin{pmatrix}\mathbf{e}_2^\theta(s) \\ \mathbf{e}_3^\theta(s) \end{pmatrix}:=
\begin{pmatrix}
\cos(\theta(s)) \, \mathbf{e}_2(s)-\sin\big(\theta(s)\big) \, \mathbf{e}_3(s)\\
\sin\big(\theta(s)\big) \, \mathbf{e}_2(s)  +\cos\big(\theta(s)\big)\,  \mathbf{e}_3(s)
\end{pmatrix}.
\end{equation}
Remark that $(\mathbf{e}_1^\theta,\mathbf{e}_2^\theta,\mathbf{e}_3^\theta)$ is nothing but a rotation in the transverse plane of angle $\theta$ of the relatively adapted parallel frame and will encode the twist of the waveguide.
In particular we have that
\begin{equation} \label{frame:evolution:theta}
    \left\{
        \begin{array}{lcl}
            \displaystyle\frac{d \mathbf{e}_1^\theta}{ds}(s) & = & k_1^\theta(s) \mathbf{e}_2^\theta(s) + k_2^\theta(s)\mathbf{e}_3^\theta(s),  \vspace{3pt} \\
            \vspace{3pt}
            \displaystyle\frac{d \mathbf{e}_2^\theta}{ds}(s) & = & - k_1^\theta(s)\mathbf{e}_1^\theta(s) -\theta'(s) \mathbf{e}_3^\theta(s),\\
            \displaystyle\frac{d \mathbf{e}_3^\theta}{ds}(s) & = & -k_2^\theta(s) \mathbf{e}_1^\theta(s)+\theta'(s) \mathbf{e}_2^\theta(s).
        \end{array}
    \right.
\end{equation}
Here, we used the notation
\begin{equation}\label{eqn:defktheta}
    \mathbf{k^\theta}(s) := \begin{pmatrix}k_1^\theta(s)\\k_2^\theta(s)\end{pmatrix} = \begin{pmatrix}
\cos(\theta(s)) & -\sin(\theta(s))\\
\sin(\theta(s)) & \cos(\theta(s))
\end{pmatrix}\begin{pmatrix}k_1(s)\\k_2(s)\end{pmatrix}.
\end{equation}
Thus, one notices that $| \mathbf{k^\theta}(s)|=\kappa(s)$.\\

\begin{Rem}
For the particular case (not assumed in this paper) where $\gamma$ is $\mathscr{C}^3$ and the curvature $\kappa$ is positive function, the Frenet frame $(\mathbf{T}(s), \mathbf{N}(s), \mathbf{B}(s))$ is defined  by
$$
\mathbf{T}(s) := \gamma'(s), \, \mathbf{N}(s) := \frac{\gamma''(s)}{|\gamma''(s)|}, \quad
\mathbf{B}(s) := \mathbf{T}(s) \times \mathbf{N}(s), \, \forall s \in \mathbb{R}.
$$
This frame is of class $\mathscr{C}^1$ and satisfies the following evolution system
\begin{equation} \label{framefrenet:evolution}
    \left\{
        \begin{array}{lcl}
            \displaystyle\frac{d \mathbf{T}_1}{ds}(s) & = & \kappa(s) \, \mathbf{N}(s) ,  \vspace{3pt} \\
            \vspace{3pt}
            \displaystyle\frac{d \mathbf{N}}{ds}(s) & = & - \kappa(s) \, \mathbf{N}(s) + \tau(s) \, \mathbf{B}(s),\\
            \displaystyle\frac{d \mathbf{B}}{ds}(s) & = &  -\tau(s) \, \mathbf{N}(s).
        \end{array}
    \right.
\end{equation}
where $\tau$ is the torsion defined by $\tau  : s\in \R \mapsto \det(\gamma'(s),\gamma''(s),\gamma^{(3)}(s)) \, \big(\kappa(s)\big)^{-2}$. 
Using \eqref{frame:evolution}, \eqref{eqn:defetheta}, \eqref{frame:evolution:theta}, \eqref{eqn:defktheta}, and \eqref{framefrenet:evolution}, the reader may verify that, starting from the initial condition $(\mathbf{e}_2(0), \mathbf{e}_3(0)) = (\mathbf{N}(0), \mathbf{B}(0))$, the unique $\mathscr{C}^1$ relative adapted frame associated with these initial conditions (see Proposition \ref{ex:un:RAPF}) is obtained with the functions $k_1$ and $k_2$ given for all real $s$  by $k_1(s) = \kappa(s) \cos(\theta(s))$, $k_2(s) = -\kappa(s) \sin(\theta(s))$,
with $\theta'(s) = -\tau(s)$ and $\theta(0)=0$.
Then, one recovers the Frenet frame if one choose the same angle $\theta$  in  \eqref{eqn:defktheta}. Namely, with these functions $k_1$, $k_2$ and $\theta$, one has:
$(\mathbf{e}_1^\theta(s), \mathbf{e}_2^\theta(s), \mathbf{e}_3^\theta(s)) = (\mathbf{T}(s), \mathbf{N}(s), \mathbf{B}(s)),$ for all real $s$.

\end{Rem}

In order to define curved and twisted waveguides, we consider the $\mathscr{C}^1$-smooth map
\begin{equation} \label{diffeo:Phi}
\Phi: \Omega_0 \to \mathbb{R}^3, \quad (s,y_2,y_3) \mapsto \gamma(s) + y_2\mathbf{e}_2^{\theta}(s) + y_3 \mathbf{e}_3^{\theta}(s).
\end{equation}
We now give sufficient conditions for $\Phi$ to be a $\mathscr{C}^1$-diffeomorphism. To this aim, set ${b:= \sup_{\mathbf{y} \in \omega} |\mathbf{y} |}$ and assume that
\begin{equation}\label{assumpt:a:kappa}
b \|\kappa\|_{L^\infty(\R)} <1.
\end{equation}
Note also that $J_\Phi(s,\by)$, the Jacobian matrix of $\Phi$ at $(s,\mathbf{y})\in \Omega_0$ is obtained thanks to the relations
\begin{equation}\label{eqn:defjacobian!}
\left\{
\begin{array}{lcl}\partial_s \Phi(s,\by) &=& \left[ 1-\mathbf{k^\theta} \cdot \mathbf{y} \right] \mathbf{e}_1^{\theta}(s)+ y_3  \, \theta'(s)\, \mathbf{e}_2^{\theta}(s)- y_2\, \theta'(s)\, \mathbf{e}_3^{\theta}(s), \\
\partial_{y_2} \Phi(s,\by) &=&\mathbf{e}_2^{\theta}(s), \\
\partial_{y_3} \Phi(s,\by) &=& \mathbf{e}_3^{\theta}(s).
\end{array}
\right.
\end{equation}
In particular, one has
\[
    \det(J_\Phi(s,\by)) = 1-\mathbf{k^\theta}\cdot \mathbf{y}.
\]
It yields, the following uniform bounds
\begin{equation} \label{unif:bounds:ak}
1-b \|\kappa\|_{L^\infty(\R)} \leq \det(J_\Phi(s,\by))  \leq 1+b \|\kappa\|_{L^\infty(\R)} \quad \text{ for all } (s,\by) \in \R \times \omega.
\end{equation}
With assumption \eqref{assumpt:a:kappa} it implies that $\Phi$ is a local $\mathscr{C}^1$-diffeomorphism and we further assume that $\Phi$ is injective (see Remark  \ref{Rem-Thm-Hadamard-Cacciopoli} for equivalent geomerical assumption for this condition). Therefore $\Phi$ is a $\mathscr{C}^1$-diffeomorphism from $\Omega_0$ to $\Omega := \Phi(\Omega_0)$ and note that there holds $\partial\Omega=\partial\Phi(\Omega_0)= \Phi(\partial \Omega_0)$.

\begin{Rem}\label{Rem-Thm-Hadamard-Cacciopoli} 
Assuming \eqref{assumpt:a:kappa} implies that  $\Phi$, defined by \eqref{diffeo:Phi}, is  a  local  $\mathscr{C}^1$-diffeomorphism from   $\overline{\Omega}_0=\bbR \times \overline{\omega}$ (the closure of  of $\Omega_0$) to $ \Phi(\overline{\Omega_0})$. Under this hypothesis on the local invertibility of $\Phi$,  the  injectivity  of $\Phi$ ensures that $\Phi$ is a  $\mathscr{C}^1$-diffeomorphism from $\overline{\Omega_0}$ to $ \Phi(\overline{\Omega_0})$. We point here out that one can replace in an equivalent way the injectivity assumption by  two   geometrical  conditions. The first one is the simple connectedness of the waveguide $\Phi(\overline{\Omega_0})$ and the second one is that the curve $\gamma$ satisfies $|\gamma(s)|\to +\infty$ when $|s|\to \infty$. More precisely, 
$\overline{\Omega_0}$ is simply connected (since $\omega$ is simply connected) and simple connectedness is   preserved by homeomorphism. Thus,  a necessary condition to obtain a $\mathscr{C}^1$-diffeomorphism is that $ \Phi(\overline{\Omega_0})$ is  simply connected.  Indeed, by the  Hadamard-Cacciopoli Thoerem (see e.g. Theorem 0.1 of \cite{De-Marco-94}),  the simple connectedness of $ \Phi(\overline{\Omega_0})$ and the  fact that $\Phi $ is proper (in other words that  the inverse image of any compact subset of $\mathcal{Y}= \Phi(\overline{\Omega_0})$ is a compact subset of $\mathcal{X}=\overline{\Omega_0}$) is equivalent (under the local invertibility condition) to the  property that $\Phi$ is $\mathscr{C}^1$-diffeomorphism  from $\mathcal{X}$ to $\mathcal{Y}$.
The fact that  $\Phi$ is proper from $\mathcal{X}$ to $\mathcal{Y}$ is equivalent to the fact that  if $|(s,{\bf y})|\to +\infty$ in $\mathcal{X}$, then $|\Phi(s,{\bf y})|\to +\infty$. Thus, as the waveguides $\mathcal{X}$ and $\mathcal{Y}$ are bounded in two directions, this last property is  also equivalent   that the curve $\gamma$ satisfies $|\gamma(s)|\to +\infty$ when $|s|\to \infty$.
 \end{Rem}

We call $\Omega$ a {\it waveguide} (see the left panel of Figure \ref{fig:2}). 
\begin{figure}
\centering
\includegraphics[width=0.9\textwidth]{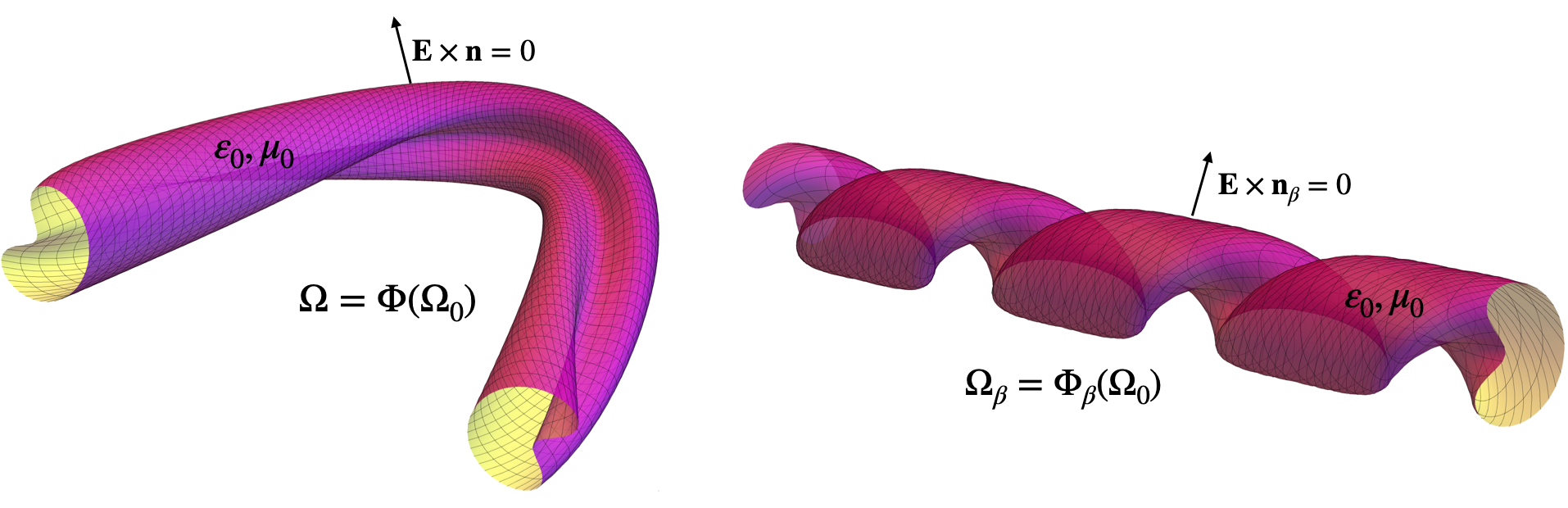}
\caption{On the left, a general perturbed waveguide with bending and twisting; on the right a constantly twisted waveguide.}
\label{fig:2}
\end{figure}
A particular case is that of a constantly twisted waveguide $\Omega_\beta$, constructed from a generating curve which is a straight line and with a constant twist $\beta \in \mathbb{R}$ (see the right panel of Figure \ref{fig:2}). Rigorously, this waveguide is the image of $\Omega_0$ through the following diffeomorphism:
\begin{equation} \label{phiB:def}
\Phi_\beta: \Omega_0 \to \R^3, \quad \Phi_\beta(s,y_2,y_3) = \begin{pmatrix}
s\\
y_2 \cos(\beta s) + y_3 \sin(\beta s) \\
-y_2 \sin(\beta s) + y_3 \cos(\beta s)
\end{pmatrix}.
\end{equation}
In other words, one has $\Omega_{\beta}=\Phi_{\beta}(\Omega_0)$ and one denotes by $\mathcal{A}_{\beta}$ the self-adjoint Maxwell operator defined by (\ref{eqn:defmaxgen}) on the  simply connected Lipschitz domain $\Omega_{\beta}$.

\begin{Rem}
We observe that the definition of $\Phi_\beta$ is of the form \eqref{diffeo:Phi}. Indeed, in the case of a \emph{constantly twisted} waveguide, the unit  tangent vector 
$\gamma'$ is constant and up to a translation (without loss of generality) one can assume that $\gamma(0)=0$ so that $\gamma(s)=s\, \gamma'(0)$.
As the curvature  $\kappa=|\gamma''|=0$,  the coefficient $k_1=k_2=0$  in the ODE system \eqref{eqn:defetheta}.  As a result, a relatively adapted parallel frame remains constant along the curve and takes the form $(\gamma'(0),{\bf e}_2,{\bf e}_3)$  where the vectors ${\bf e}_2$ and ${\bf e}_3$   are determined by the initial condition, subject to the requirement that $(\gamma'(0),{\bf e}_2,{\bf e}_3)$ forms a positively oriented orthonormal basis of  $\bbR^3$. Furthermore, as the twist has a constant value $\beta\geq 0$, $\theta'=\beta$ and choosing  $\theta(s)=\beta \, s$ in \eqref{eqn:defetheta} allows to rewrite the expression    \eqref{diffeo:Phi}  (rewritten  in the basis  $(\gamma'(s),{\bf e}_2,{\bf e}_3)$) as \eqref{phiB:def}.

\end{Rem}

Throughout the paper we assume that any waveguide is asymptotically converging to a constantly twisted one, meaning that the curvature goes to zero at infinity and that the twist converges to a constant $\beta\in \R$:
\begin{equation} \label{relaxed:hyp}
\lim_{|s| \to \infty} \kappa(s)=0 \quad \text{and} \quad \lim_{|s| \to \infty} \left( \theta'(s) - \beta\right)= 0.
\end{equation}

\begin{Rem} \label{remark:arclength}
We remark that, from a practical point of view, while we built the waveguide by means of a curve parametrized by arc-length,  the vast majority of parametrized curves do not possess an explicit expression for the arc-length parameter, as it often involves inverting a function whose inverse cannot be written in closed form via known functions.
However, this is not a problem, since any (injective) parametrization can also be used because it does not change the geometry of the curve (and other geometrical invariant quantities such as the curvature). 
We can still build the unit tangent to the curve, by normalizing the first-order derivative of the parametrization, and from that a relative adapted parallel frame. See e.g. \S \ref{sec:proofcor9}.
\end{Rem}

\subsection{Main results}
Our main results concern the spectrum of the Maxwell operator in a perfectly conducting waveguide filled by a homogeneous and isotropic dielectric medium (e.g the vacuum). Namely, for $\varepsilon_0,\mu_0 > 0$, we consider the operator \begin{equation}\mathcal{A} := \mathcal{A}_{\varepsilon_0 \mathds{1}_3,\mu_0\mathds{1}_3}
\label{eqn:defAMaxwell}
\end{equation}
posed in the waveguide $\Omega$ defined in \eqref{eqn:defmaxgen}. Note its spectrum is symmetric with respect to the origin (see Proposition \ref{prop:symspecmaxwell} in Appendix \ref{appendix:specMaxwell}).\\

Our first result is about the structure of the essential spectrum of the operator $\mathcal{A}$.
\begin{Thm} \label{Main1}
Under assumptions \eqref{relaxed:hyp}, there exists $a_\beta>0$ such that
\begin{equation*}
\sigma_{\mathrm{ess}}(\mathcal{A})=\sigma_{\mathrm{ess}}(\mathcal{A}_{\beta})=(-\infty,-a_\beta] \cup \{0\} \cup [a_\beta,+\infty).
\end{equation*}
Moreover, in the particular case where  $\beta=0$ in  \eqref{relaxed:hyp}, one has 
\[
    a_0 = \sqrt{\lambda_2^N(\omega)}\, c, \  \mbox{ where  }c:=(\varepsilon_0 \mu_0)^{-\frac{1}{2}} \mbox{ is the speed of light in the medium}
\]
and $\lambda_2^N(\omega)>0$ is the first non-trivial eigenvalue of the Neumann Laplacian on $\omega$. 
\end{Thm}
\begin{Rem}\label{rem:defa_betanotwist} 
 Theorem  \ref{Main1} states that under the assumption \eqref{relaxed:hyp}, the Maxwell operators $\mathcal{A}$ and $\mathcal{A}_{\beta}$ (defined on the constantly twisted $\Omega_{\beta}$) have the same essential spectrum.
The number $a_\beta$ is indeed defined by applying \cite[Theorem 9.8.]{filonov19} on  $\mathcal{A}_{\beta}$ after a geometrical unitary  transformation (the so-called Piola transform, see Section \ref{subsec:piola}) which straigthens the waveguide. For more details, we refer to  Section \ref{sec-const-twist}. In the particular case  $\beta =0$, namely, when the waveguide $\Omega_{\beta}$ is straight, the value $a_0$ is given  (see \cite[Theorem 9.8.]{filonov19}) by
$
    a_0 =\sqrt{\lambda_2^N(\omega)}\, c, 
$
where $\lambda_2^N(\omega)$ is the first non-trivial eigenvalue of the Neumann Laplacian $L^{N}_{\omega}$ on $\omega$. More precisely,   the   operator $L^{N}_{\omega}: {D}(L^{N}_{\omega}) \subset L^2(\omega)  \mapsto L^2(\omega)  $ is  defined  by
$$
L^{N}_{\omega} :=-\Delta ,\mbox{ where } D({L}^{N}_{\omega})=\big\{u \in H^1(\omega): \Delta u \in L^2(\omega) \mbox{ and }\partial_{\bn} u=0   \big\}, 
$$ 
(where $\bn$ stands here for the outward normal to $\omega$). $L^{N}_{\omega}$ is  positive semidefinite, self-adjoint and has a compact resolvent.  Hence, its spectrum consist of a sequence of  positive  eigenvalues $(\lambda_n^N)_{n\geq 1}$ of finite multiplicity, tending to $+\infty$ and listed  (with their  multiplicity) as follows 
$0=\lambda_1^{N}(\omega) < \lambda_2^{N}(\omega) \leq  \lambda_3^{N}(\omega) \leq  \ldots  $ 
(The eigenvalue $\lambda_1^{N}(\omega)=0$ is  here simple since $\omega$ is  simply connected).

 In this particular case, the value of $a_0$ can also be obtained by applying  the  so-called transverse electric and transverse magnetic  decomposition to the Maxwell operator $\mathcal{A}_0$ (which holds since the waveguide $\Omega_0$ is straight and the associated medium is homogeneous). Roughly speaking, this decomposition  allows ``to reduce" the analysis of the spectrum $\mathcal{A}_0$ in $\operatorname{ker}(\mathcal{A}_0)^{\perp}$  to the analysis of the direct sum of two operators:  the so-called  transverse electric and transverse magnetic Maxwell operators whose spectra are purely essential and given  respectively in terms of the (purely essential) spectrum  of two ``scalar operators" : the Neumann and Dirichlet Laplacians defined in the separable geometry $\Omega_0=\bbR\times \omega$,
see e.g. \cite[Chapter IX, section 4]{daulio}.  
\end{Rem}

The proof of Theorem \ref{Main1} is inspired by the approach developed by Weder in \cite[Ch.3 \S 3]{weder} in the context of  scattering theory in electromagnetic stratified media. In this setting,  the resolvent of the  perturbed stratified medium is rewritten  through a factorization involving the resolvent of the unperturbed medium and a Fredholm operator.  
In other words, he establishes a Birman-Schwinger principle for stratified electromagnetic media.

We adapt this approach, with many technical adjustments, in the context of electromagnetic waveguides subjected to geometrical perturbation (twisting, bending, \dots) and thus we are able to see the Maxwell operator $\mathcal{A}$ posed in the waveguide $\Omega$ as a perturbation of the Maxwell operator $\mathcal{A}_{\beta}$ posed in a constantly twisted waveguide $\Omega_{\beta}$ (see the right panel of Figure \ref{fig:2}) by proving a Birman-Schwinger type principle (see Proposition \ref{lemmaK}) between their two resolvents. In particular, as a byproduct of this analysis, we show that they have the same essential spectra via a Meromorphic Fredholm  theorem.\\

Another corollary of this Birman-Schwinger principle is a localisation of possible discrete eigenvalues in the spectral gap $(-a_\beta,a_\beta)$ (see Proposition \ref{prop:loceivalue}). Moreover, we provide a geometric condition for which the existence of discrete eigenvalues is ensured in the particular case of non-twisted waveguides.

To state it, we must introduce the following vector quantities. Let $\psi$ be a normalized  eigenfunction associated with $\lambda_2^N(\omega)$. 
Recall the definition of $\mathbf{k^\theta}$ from \eqref{eqn:defktheta}, and  we  set $\mathbf{k} := (k_1,k_2)^\top$.
We define the vectors $\bX,\bY^\theta, \bY \in \R^2$ as
\begin{equation}\label{eqn:defX}
    \bX := \int_{\partial\omega}\bn(\by)|\psi(\by)|^2 \rmd \sigma(\by),\quad \bY^\theta := \int_{\R}\mathbf{k^\theta}(s) \, \mathrm{d}s, \quad \bY:= \int_{\R} \mathbf{k}(s)\, \mathrm{d}s,
\end{equation}
where $\bn$ is the unit outward normal vector to $\partial\omega$, and $\rmd \sigma$ is the $1$-dimensional Hausdorff measure on $\partial\omega$. 
We point out that in the definition \eqref{eqn:defX},  $ \bY$ coincides with $\bY^{\theta}$ for  constant functions $\theta$ of the form $\theta=2m \pi$ with $m\in \mathbb{Z}$. 

Now, we are in a good position to state the main theorem regarding the existence of eigenvalues.
\begin{Thm}\label{thm:discspec}
Suppose $\theta'\equiv 0$, $\displaystyle \lim_{|s|\to \infty} \kappa(s) =0$ and that $\kappa \in L^1(\mathbb{R})$. If there holds
\begin{equation}
 {\bX \cdot \bY^\theta} >2 \lambda_2^N(\omega) \ \frac{ b^2 \|\kappa\|_{L^\infty(\R)}}{1-b\|\kappa\|_{L^\infty(\R)}}\|\kappa\|_{L^1(\R)}
    \label{eqn:ineqdiscspec}
\end{equation}
then 
\begin{equation*}
\sigma_{\mathrm{dis}}(\mathcal{A}) \neq \emptyset
\end{equation*}
and $\sigma_{\mathrm{dis}}(\mathcal{A})$ is symmetric with respect to zero.
\end{Thm}

\begin{Rem} Let us comment the geometric hypothesis of Theorem \ref{thm:discspec}. The conditions $\bX \neq 0$ and $\bY^\theta\neq 0$ are necessary in order to satisfy inequality \eqref{eqn:ineqdiscspec}. Note that these conditions concern the cross-section and the base curve, respectively. It turns out that domains $\omega$ with too many symmetries  verify $\bX = 0$ (see Proposition \ref{prop:symm-X0}). This last point can also  be  checked   with the analytic expressions of the eigenelements of the Neumann Laplacian in simple   geometries  such as disks, rectangles or the equilateral triangle. In \S \ref{subsecisos}, we exhibit domains $\omega$ (namely isosceles right triangles) for which $\bX$ is explicit and $\bX \neq 0$. By perturbation theory, in \S \ref{subsecrectangle} we also prove that there exists infinitely many  perturbations of a rectangular domain such that $\bX\neq 0$.

The hypothesis $\theta' = 0$ means that $\theta$ is constant and that the waveguide is not twisted.
Thus, for planar curves (see \S \ref{subsecisoY}), the condition $\bY^\theta \neq 0$ is  satisfied if and only if the two-dimensional (signed) curvature associated with the planar curve has a non-zero integral over $\mathbb{R}$:
this is clearly the case, for instance, when the curvature has a constant sign (either non-negative or non-positive) and is not identically zero.
\medskip

In Theorem \ref{thm:discspec}, as $\theta$ is constant, the choice of $\theta$ plays a significant role in inequality \eqref{eqn:ineqdiscspec}: it gives a degree of freedom on the choice of the frame $(\mathbf{e}_1^\theta,\mathbf{e}_2^\theta,\mathbf{e}_3^\theta)$ in order to maximize the quantity on the left-hand side of \eqref{eqn:ineqdiscspec}.
Roughly speaking, as the Maxwell's equations are vectorial, it selects ``the best orientation for bending".
Indeed, note that in this case, thanks to \eqref{eqn:defktheta}, one has $\mathbf{k^\theta} = \mathfrak{R}_\theta \mathbf{k}$ for some rotation matrix $\mathfrak{R}_\theta$ of constant angle $\theta$. In particular, as $\bY^{\theta}=\mathfrak{R}_\theta \bY$, there holds
\begin{equation}\label{eqn:defvectorY}
	\bX \cdot \bY^\theta = \bX \cdot (\mathfrak{R}_\theta \bY).
\end{equation}
We emphasize that the geometric conditions $\bX \neq 0$ and $\bY \neq 0$ are independent from each other. They rely  on the choice of the cross-section $\omega$, and the base curve $\Gamma$ along with its relatively adapted parallel frame used to construct the waveguide $\Omega$, respectively.
Thus, one can always choose $\theta$ in such a way that the quantity \eqref{eqn:defvectorY} is maximised, namely there exists $\theta_\star \in [0,2\pi)$ such that  $\bX \cdot (\mathfrak{R}_{\theta_\star} \bY) = |\bX| |\bY|$. 
This specific choice of $\theta_\star$ can be interpreted as an optimal way to construct the waveguide, given a cross-section $\omega$ such that $\bX \neq 0$ and a (relatively adapted parallel) framed base curve $\Gamma$ for which $\bY \neq 0$ by  prescribing how to attach $\omega$ along $\Gamma$ in order to best satisfy inequality~\eqref{eqn:ineqdiscspec}.
\end{Rem}
An important consequence of Theorem \ref{thm:discspec} is the possibility of building a waveguide for which bound states exist (see \S \ref{subsubsub:delta} for such a construction). It yields the  following corollary.

\begin{Cor} \label{cor:vrai:exist}
Let $\Gamma$ be  a base curve with an associated relatively adapted parallel frame for which $\bY \neq 0$. 
Let $\omega \subset \R^2$ be a Lipschitz bounded cross-section such that there exists a normalized eigenfunction associated to $\lambda_2^N(\omega)$ for which $\bX \neq 0$. Then there exists a unique $\theta_\star \in [0,2\pi)$ such that $\bX \cdot (\mathfrak{R}_{\theta_\star} \bY) = |\bX| |\bY|$.
Suppose then that the map $\Phi$ constructed as in \eqref{diffeo:Phi} starting from this framed curve $\Gamma$, this cross-section $\omega$ and this constant angle $\theta=\theta_\star$ is a global $\mathscr{C}^1$-diffeomorphism, and all the geometric and regularity assumptions stated in §\ref{section:waveguides} are satisfied.\\
\noindent In addition, suppose that the curvature  $\kappa \in L^1(\mathbb{R})$ satisfies $\kappa(s) \to 0$ as $|s| \to \infty$. \\ 
\noindent Then, there exists $\delta_\star>0$ such that for all $\delta\in (0, \delta_\star)$, the slightly curved waveguides $\Omega_\delta$ (defined by \eqref{eq.defOmegadelta}) all have non-empty discrete spectrum.
\end{Cor}
In view of Corollary \ref{cor:vrai:exist}, in order to obtain a non-empty discrete spectrum for the Maxwell operator in an admissible  (in the sense of \S  \ref{section:waveguides}) waveguide $\Omega$, it is sufficient to have $\bX,\bY \neq 0$. This last requirement is addressed by the following corollary. 
\begin{Cor} \label{cor:infiniteWG}
There exists an infinite number of admissible waveguides $\Omega$ having $\bX,\bY \neq 0$ and a non-empty discrete spectrum.
\end{Cor}

\subsection{The very specific case of waveguides generated by a planar curve and a rectangular cross-section}\label{sec-rec-waveguide}

In  this section, we present some results of the article \cite{GJ92} on the existence of eigenvalues in the spectrum of the Maxwell operator $\mathcal{A}$  for particular  waveguides generated by a planar curve $\gamma$ and  a rectangular cross-section $\omega$ (see Figure \ref{fig:rec}). The construction of such eigenvalues is based on the existence of discrete spectrum for the Dirichlet Laplacian $L^D$ in a two-dimensional bent waveguide \cite{GJ92}.   We further develop here their arguments. 
\begin{figure}[h!]
\centering
\includegraphics[width=0.8\textwidth]{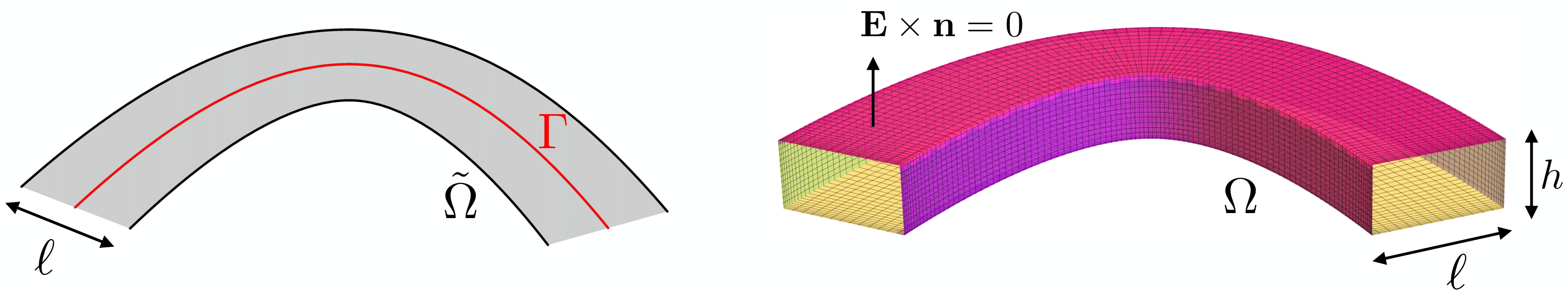}
\caption{Waveguide $\Omega$ generated via a planar curve $\gamma$ and a rectangular cross-section $\omega=(-\ell/2, \ell/2 )\times (-h/2,h/2)$.}
\label{fig:rec}
\end{figure}

Let $({\bf e}_{c,1},{\bf e}_{c,2},{\bf e}_{c,3})$ stand for the canonical basis of $\mathbb{R}^3$. The waveguides considered in   \cite{GJ92} fit the geometrical framework of \S\ref{section:waveguides}.  More precisely, in addition to the assumptions made in this section, the base curve $\Gamma$ is assumed to be planar and  contained (without a loss of generality) in the plane $y_3=0$ (see Figure \ref{fig:rec}). 
Therefore, one can choose the following $\mathscr{C}^1$-smooth, relatively adapted parallel frame along the curve $\Gamma$:
$$
s \in \bbR \mapsto \left(\mathbf{e}_1(s), \mathbf{e}_2(s), \mathbf{e}_3(s)\right) = \left(\gamma'(s), \mathbf{n}(s), \mathbf{e}_{c,3}\right),
$$
where $\mathbf{n}(s)$ denotes the unit  normal vector to the curve at $\gamma(s)$, obtained by rotating the unit tangent vector $\gamma'(s)$ in the plane $y_3=0$ by  $\pi/2$ in the counterclockwise direction. Note that the obtained frame is a positively oriented orthonormal basis of $\bbR^3$ for each $s \in \R$. 
For this particular relatively adapted parallel frame, the function $k_2$  in \eqref{frame:evolution} is zero, and the function $k_1$ is precisely the signed curvature of $\Gamma$ in two dimensions, which satisfies $|k_1|=\kappa$.  
Then, one assumes that there is no twist, i.e. that $\theta'=0$ and   even  that $\theta=0$ so that the frame  $(\mathbf{e}_1^\theta,\mathbf{e}_2^\theta,\mathbf{e}_3^\theta)$ coincides with $(\mathbf{e}_1,\mathbf{e}_2,\mathbf{e}_3)$. The rectangular cross-section is
 defined by $\omega=(-\ell/2, \ell/2 )\times (-h/2,h/2)$ and the waveguide $\Omega$ is given   via \eqref{diffeo:Phi} by $\Omega=\Phi(\bbR\times \omega)$.\medskip
 
In this setting, the Dirichlet Laplacian  $L^D_{\widetilde{\Omega}}$  defined in the two-dimensional waveguide $\widetilde{\Omega}=\Omega\cap \{y_3=0\}$   has an essential spectrum given by $\sigma_{\mathrm{ess}}(L^D_{\widetilde{\Omega}})=[ \pi^2/\ell^2, +\infty)$. Moreover,  as soon as the curvature $\kappa=|\gamma''|\neq 0$, its discrete spectrum  $\sigma_{\mathrm{dis}}(L^D_{\widetilde{\Omega}})$ is non-empty (cf. \cite{GJ92,DE95, ES89,Krej04}), and therefore
\begin{equation}\label{eq.Dirichlet}
\exists  \lambda_{\mathrm{2D}} \in \bbR \ \mbox{ with }  \ 0<\lambda_{\mathrm{2D}}<\frac{\pi^2}{\ell^2} \ \mbox{ and }  \ u\in H^1_0(\widetilde{\Omega})\setminus \{0\}\ \text{ such that } \  L_{\widetilde{\Omega}}^Du= -\Delta u=\lambda_{\mathrm{2D}} u.
\end{equation}
Then,  one  defines  an electric field $\bE$ and a magnetic field $\bH$  in terms of  the eigenpair $(\lambda_{\mathrm{2D}},u)$ as follows: for any $(s,y_2,y_3)\in \bbR\times \omega$,
$$
\bE(\Phi(s,y_2,y_3))= \rmi  \sqrt{\lambda_{\mathrm{2D}}} \, c \  u(\Phi(s,y_2,0))  \, {\bf e}_{c,3} \quad  \mbox{ and } \quad  \bH(\Phi(s,y_2,y_3))=   \mu_0^{-1} \nabla u(\Phi(s,y_2,0)) \times  \, {\bf e}_{c,3}.
$$
Thus, using \eqref{eq.Dirichlet}, a simple computation yields 
\begin{equation}\label{eq.eigenmaxwellrec}
- \rmi \mu_0^{-1} \curlvec  \bE= \sqrt{\lambda_{\mathrm{2D}} } \, c \ \bH \ \mbox{ and }  \ \rmi \, \varepsilon_0^{-1} \curlvec  \bH=\sqrt{\lambda_{\mathrm{2D}} } \, c \ \bE\   \mbox{ in } \ \Omega.
\end{equation}
Hence, in particular, one has $\bU=(\bE, \, \bH)\in  H(\curlvec,\Omega)^2$. Moreover, one notices that  $ {\bf e}_{c,3} \times {\bf n}=0 $ on the top and bottom part of the boundary $\partial \Omega \cap ( \{ y_3=-h/2\} \cup \{ y_3=h/2\})$,   and since $  u\in H^1_0(\widetilde{\Omega})$, the electric field $\bE$ vanishes  on the other part of the boundary.  Thus,  by construction, one has   $\bE \times {\bf n}=0 $ on $\partial \Omega$.  This can be rigorously proved  in the sense of  the definition \eqref{eq.Hcurl0} of $H_0(\curlvec,\Omega)$ by integrating by parts in $\widetilde{\Omega}$  using on the  one hand that   $\Omega=\widetilde{\Omega} \times (-h/2,h/2)$ (i.e the geometrical separability of the domain which holds only for this particular type of waveguides) and  on the other hand that $u\in H^1_0(\widetilde{\Omega})$ (where $H^1_0(\widetilde{\Omega})$ is defined via the two-dimensional analogue of \eqref{eqn:def_H10}).
Thus,  $\bE\in  H_0(\curlvec,\Omega)$ and $\bU\in D(\mathcal{A})=  H_0(\curlvec,\Omega) \times H(\curlvec,\Omega)$.  This implies, together with \eqref{eq.eigenmaxwellrec}, that $\bU$ is an eigenfunction of $\mathcal{A}$ for the eigenvalue $\sqrt{\lambda_{\mathrm{2D}}} \, c$. Hence, one deduces by Proposition \ref{prop.eigensym} that $\pm \sqrt{ \lambda_{\mathrm{2D}}}\, c\in \sigma_p(\mathcal{A})$. However, these eigenvalues could be discrete or embedded in the essential spectrum. More precisely, using Theorem  \ref{Main1}, one has 
$$\sigma_{\mathrm{ess}}(\mathcal{A})=\big(-\infty,- \sqrt{\lambda_2^N(\omega)}\, c \big] \cup \{0\} \cup \big[ \sqrt{\lambda_2^N(\omega)}\, c,+\infty\big)  \ \mbox{ with }  \ \lambda_2^N(\omega)=\frac{\pi^2}{\max(h, \ell)^2} . $$ 
Therefore, for a fixed  width $\ell>0$ of the two-dimensional waveguide $\widetilde{\Omega}$, one gets the following scenarios:
\begin{enumerate}
 \item if $ h \leq  \ell$, one has by \eqref{eq.Dirichlet}
$$0<  \sqrt{\lambda_{\mathrm{2D}}} \, c < \frac{ c \, \pi}{\ell } =\sqrt{\lambda_2^N(\omega)}\, c \quad  \mbox{ and therefore }  \pm  \sqrt{\lambda_{\mathrm{2D}}} \, c  \in \sigma_\mathrm{dis}(\mathcal{A});$$
\item if $\ell  <h$, then $  \displaystyle \sqrt{\lambda_2^N(\omega)} \, c=\frac{\pi \,c }{h }$ there are two sub-cases:
\begin{itemize}
\item  if $\displaystyle h<\frac{\pi}{\sqrt{\lambda_{\mathrm{2D}}}}$ then $\displaystyle   \pm \sqrt{ \lambda_{\mathrm{2D}}} \, c \in \sigma_\mathrm{dis}(\mathcal{A})$;
\item If $ \displaystyle \frac{\pi}{\sqrt{\lambda_{\operatorname{2D}}}}\leq h $ then  $\displaystyle \pm \sqrt{ \lambda_{\mathrm{2D}}} \, c \mbox{ is  an eigenvalue  embedded in $\sigma_{\mathrm{ess}}(\mathcal{A}).$ } $
\end{itemize}
\end{enumerate}

Thus,  this example  illustrates that the conditions of  Theorem \ref{thm:discspec} are just sufficient  conditions  and not necessary conditions for the existence of discrete spectrum of $\mathcal{A}$. Indeed, the vector $\bX$ given by \eqref{eqn:defX} vanishes for a rectangular cross-section (as it is  rotationally symmetric  by a rotation of $\pi$, see  Proposition \ref{prop:symm-X0}). However,  as we discussed above, one can construct discrete spectrum  by bending some rectangular waveguides.  Nevertheless, while  the rectangular waveguide represents a very specific case of separable geometry, Theorem \ref{thm:discspec} has the advantage of being applicable to a broad class of geometrical cross-sections (see Section \ref{secnumerics}). 


\subsection{Structure of the paper}

The paper is organized as follows. In Section \ref{notation:sect}, we introduce rigorously the functional framework of the problem. For further uses, we also introduce the so-called {\it Piola transform}. Section \ref{secessspec} is devoted to the proof of Theorem \ref{Main1} about the structure of the essential spectrum of the Maxwell operator $\mathcal{A}$ defined in \eqref{eqn:defAMaxwell}. Section \ref{secdiscrete} is about localization of possible discrete eigenvalues and the proof of Theorem \ref{thm:discspec}. In this section, we also provide sufficient  conditions for the existence of a Poincar\'e type inequality for the Maxwell operator defined on the full waveguide.  Finally, in Section \ref{secnumerics} we investigate the geometric conditions on the base curve $\Gamma$ : $\bY\neq 0$  and on the cross-section : $\bX \neq 0$ (where $\bX$ is defined in \eqref{eqn:defX}) involved in Theorem \ref{thm:discspec}, Corollary \ref{cor:vrai:exist} and Corollary \ref{cor:infiniteWG} for the existence of discrete spectrum. 
More precisely, we first show that planar curves $\Gamma$ with non-zero curvature satisfy $\bY \neq 0$. We then provide both analytical and numerical examples of domains for which $\bX = 0$ or $\bX \neq 0$. In addition, we investigate the stability of the condition $\mathbf{X} = 0$ with respect to variations in the shape of the cross-section $\omega$. To this end, we perform a detailed analysis of the case where $\omega$ is a rectangular cross-section, a reference domain for which $\mathbf{X} = 0$, by perturbing it  via shape derivative techniques.

\section{Notations, functional framework} \label{notation:sect}
In \S \ref{subsec:funcrame}, we introduce the functional spaces used all along the paper. In \S \ref{subsec:piola}, we discuss the so-called \emph{Piola transform} (see e.g  \cite[Section 3]{lamzac20} \cite[Section 3.9]{Monk} or \cite{Nic-07}) which allows to rewrite the problem initially posed in the waveguide $\Omega$ into the straight waveguide $\Omega_0$ 
(see also \cite[Theorem 2]{lamzac23} for stability properties of such transform).

\subsection{Functional framework}\label{subsec:funcrame}

In this paragraph we  recall  some notations and definitions of standard functional spaces involved in the study of Maxwell's equations.
Let $\mathscr{O}$ be a  (non-empty) open subset  of $\mathbb{R}^3$. Let $\Sigma:\mathscr{O} \to M_3(\mathbb{R})$ be a (measurable) matrix-field for which (as in \eqref{unif:elliptic:eu}) it exists two constants $\Sigma_0, \, \Sigma_1$ with $0<\Sigma_0\leq \Sigma_1$ such that
\begin{equation} \label{hyp:sigma}
\Sigma(\bx)^\top= \Sigma(\bx) \qquad \text{ and } \qquad 0<\Sigma_0 |\zeta|^2 \leq \Sigma(\bx) \, \zeta \cdot \overline{\zeta} \leq \Sigma_1 |\zeta|^2 \quad \forall \bx \in \mathscr{O}, \forall \zeta \in \mathbb{C}^3.
\end{equation}
This in particular implies the symmetric matrix-field $\Sigma$ is bounded in the $L^\infty(\mathscr{O})$-norm and it satisfies a uniform elliptic condition. Moreover, it yields that the matrix $\Sigma(\bx)$ is invertible  for any $\bx \in \mathscr{O}$, and the inverse matrix-field $(\bx \mapsto \Sigma(\bx)^{-1})$ also satisfies assumptions \eqref{hyp:sigma} with $(\Sigma^{-1})_0=(\Sigma_1)^{-1}$ and $(\Sigma^{-1})_1=(\Sigma_0)^{-1}$.

Let $L^2(\mathscr{O}):=L^2(\mathscr{O},\mathbb{C})$ be  the Hilbert space of square-integrable functions with the associated standard inner product defined for all $f,g \in L^2(\mathscr{O})$,  $(f,g)_{L^2(\mathscr{O})}:=\int_{\mathscr{O}} f \,\bar{g} \, \rmd \bx$. Similarly, the Sobolev space $H^1(\mathscr{O})$ is the subspace of $L^2(\mathscr{O})$ consisting of functions whose first-order derivatives are square-integrable. Throughout the paper we use a bold notation when dealing with vector fields, instead of scalar functions. Namely, we set $\bL^2(\mathscr{O}):=L^2(\mathscr{O})^3$ and $\bH^{1}(\mathscr{O}):=H^{1}(\mathscr{O})^3$.

Then we introduce the following objects.
\begin{itemize}
\item The linear space $\mathfrak{D}(\mathscr{O})$ of smooth (infinitely differentiable) functions $f: \mathscr{O} \to \bbC$ with compact support contained in $\mathscr{O}$, and $\mathfrak{D}(\overline{\mathscr{O}})=\{\phi\rvert_{\mathscr{O}} : \phi \in \mathfrak{D}(\mathbb{R}^3)\}$. Then we define $\bbD(\mathscr{O}):=\mathfrak{D}(\mathscr{O})^3$ and $\bbD(\overline{\mathscr{O}}):=\mathfrak{D}(\overline{\mathscr{O}})^3$.
\item The  Hilbert space $\bL_\Sigma^2(\mathscr{O})$ of vector fields endowed with the inner product
$$
(\bu,\bv)_{\bL_\Sigma^2(\mathscr{O})}=(\Sigma \bu,\bv)_{\bL^2(\mathscr{O})} = \int_{\mathscr{O}} \left(\Sigma (\bx) \, \bu(\bx)\right)\cdot  \overline{\bv(\bx)} \, \rmd \bx.
$$
Observe that due to the hypotheses on $\Sigma$, the norm associated with this inner product is equivalent to the norm of the standard one, thus 
\begin{equation} \label{equiv:topspaces}
    \bL_\Sigma^2(\mathscr{O}) \text{ and } \bL^2(\mathscr{O})  \text{ are equal as topological spaces.}
\end{equation} 
\item The maximal domain of the $\curlvec$ operator is defined as the Hilbert space 
$$H(\curlvec,\mathscr{O})=\left\{ \bu\in \bL^2(\mathscr{O}) \,: \,\curlvec \bu \in \bL^2(\mathscr{O})\right\}$$
endowed with the inner product
\begin{equation*}
(\bu,\bv)_{H(\curlvec,\mathscr{O})}= (\bu,\bv)_{\bL^2(\mathscr{O})} + (\curlvec \bu,\curlvec \bv)_{\bL^2(\mathscr{O})}.
\end{equation*}
Here, for $\bu\in \bL^2(\mathscr{O})$, $\curlvec \bu \in \bL^2(\mathscr{O})$ has to be understood in the distributional sense.
\item A vector field $\bu \in H(\curlvec, \mathscr{O})$ has \emph{zero tangential trace}, which we denote by $(\bu \times \bn)\rvert_{\partial \mathscr{O}} =0$ if for all $\bv \in H(\curlvec, \mathscr{O})$
\begin{equation}\label{eq.Hcurl0}
    (\bu, \curlvec \bv)_{\bL^2(\mathscr{O})} = (\curlvec \bu, \bv)_{\bL^2(\mathscr{O})}.
\end{equation}
Note that when defined $\bn$ has to be understood as the outward unit pointing normal to $\mathscr{O}$ (see Remark \ref{remark:normalcurl}).
We now define the space
\begin{equation*}
    H_0(\curlvec, \mathscr{O}) := \{ u \in H(\curlvec, \mathscr{O}) : (\bu \times \bn)\rvert_{\partial \mathscr{O}} =0 \}.
\end{equation*}
This space can also be characterized as the closure of $\bbD(\mathscr{O})$  for the operator norm $\left\|\cdot \right\|_{H(\curlvec,\Omega)}$ (see Lemma \ref{lem:55} in Appendix \ref{appendix}). 
\begin{Rem} \label{remark:normalcurl}
Observe that when $\mathscr{O}$ is a bounded Lipschitz domain, the unit outward normal $\bn$ to $\partial \mathscr{O}$ is well-defined as a function of $L^{\infty }(\partial \Omega)$. Moreover, by \cite[Theorem 2.11]{gira}, the condition $(\bu \times \bn)|_{\partial \mathscr{O}} = 0$  can be understood as a trace equality in the Sobolev space $\bH^{-\frac12}(\partial \mathscr{O})$. 
An interested reader can also have a look at \cite[Chapter 4, \S 2]{daulio} where the case of trace mappings for an infinite cylinder of the form $\mathscr{O}= \R \times \omega$ (with $\omega$ bounded Lipschitz domain of $\R^2$) is treated.
\end{Rem}
\item The maximal domain of the $\Div (\Sigma \, \cdot)$ operator is defined as the Hilbert space 
$$H(\Div \Sigma,\mathscr{O})=\left\{ \bu\in \bL_\Sigma^2(\mathscr{O}) \,: \,\Div (\Sigma \bu) \in L^2(\mathscr{O})\right\}$$
endowed with the inner product
\begin{equation*}
(\bu,\bv)_{H(\Div \Sigma,\mathscr{O})}= (\bu,\bv)_{\bL_\Sigma^2(\mathscr{O})} + (\Div (\Sigma\bu),\Div (\Sigma\bv))_{L^2(\mathscr{O})}.
\end{equation*}
Here, for $\bu\in \bL^2(\mathscr{O})$, $\Div(\Sigma \bu) \in L^2(\mathscr{O})$ has to be understood in the distributional sense. When $\Sigma$ is the identity matrix we write $H(\Div,\mathscr{O}) := H(\Div \mathds{1}_3,\mathscr{O})$.
Note that due to the hypotheses on $\Sigma$ we have that $H(\Div \Sigma,\mathscr{O})=\{\Sigma^{-1} \bu : \bu \in H(\Div,\mathscr{O})\}$.
\item A vector field $\bu \in H(\Div \Sigma, \mathscr{O})$ has \emph{zero normal trace}, which we denote by $((\Sigma\bu) \cdot \bn)\rvert_{\partial \mathscr{O}} =0$ if, for all $f  \in H^1(\mathscr{O})$ there holds
\begin{equation*}
    (\bu, \nabla f)_{\bL^2_\Sigma(\mathscr{O})} = -(\Div (\Sigma\bu), f)_{L^2(\mathscr{O})} 
\end{equation*}
We now define
\begin{equation} \label{DEF:H0div}
    H_0(\Div \Sigma, \mathscr{O}) := \{ u \in H(\Div \Sigma,\mathscr{O}) : ((\Sigma \bu) \cdot \bn)\rvert_{\partial \mathscr{O}} =0 \}.
\end{equation}
This space  can also be characterized  as the closure of $\bbD(\mathscr{O})$ for the operator norm $\left\|\cdot \right\|_{H(\Div \Sigma,\mathscr{O})}$ in the case the matrix-field $\left(\bx \mapsto \Sigma(\bx) \right)$ is $\mathscr{C}^\infty$ (see Lemma  \ref{prop:defHdiv:density} and Remark \ref{remark:sigmasmooth} in Appendix \ref{appendix}).
\begin{Rem} \label{remark:standardHdiv}
Observe that when $\mathscr{O}$ is a bounded Lipschitz domain, by \cite[Theorem 2.5.]{gira}, the notation $((\Sigma\bu)\cdot \bn)|_{\partial \mathscr{O}} = 0$ can be understood as a trace equality in the Sobolev space $H^{-\frac12}(\partial \mathscr{O})$. 
\end{Rem}
\item The Sobolev space $H_0^1(\mathscr{O})$ is defined as the space of functions $f \in H^1(\mathscr{O})$ for which for all $\bu \in H(\Div,\mathscr{O})$, one has
\begin{equation}
    (\nabla f,\bu)_{\bL^2(\mathscr{O})} = -(f,\Div\bu)_{L^2(\mathscr{O})}.
    \label{eqn:def_H10}
\end{equation}
\begin{Rem}
Defined in this way, $H_0^1(\mathscr{O})$ coincides with the closure of $\mathfrak{D}(\mathscr{O})$ with respect to the $H^1(\mathscr{O})$-norm (see Lemma \ref{prop:defH1coincide} in Appendix \ref{appendix}).
\end{Rem}
\end{itemize}

\bigskip

For further uses, we introduce the subspace of gradients
\begin{equation}\label{eqn:defgradientfields}
    \mathcal{G}(\mathscr{O}) := \{\nabla f : f \in H_0^1(\mathscr{O})\}\times\{\nabla g : g \in H^1(\mathscr{O})\}
\end{equation}
and its orthogonal in $\bL^2_{\varepsilon,\mu}(\mathscr{O})$
\begin{equation} \label{solenoidal:set}
\mathcal{J}_{\varepsilon,\mu}(\mathscr{O}) := \bigg\{ (
\bE,\bH)^{\top}\in  \bL^2_{\varepsilon,\mu}(\mathscr{O}) : \Div(\varepsilon \bE) = \Div(\mu \bH)=0,      (\mu \bH) \cdot \bn \rvert_{\partial \mathscr{O}} =0 \bigg\}
\end{equation}
Under mild hypotheses on the waveguide $\Omega$ one can show that $\mathrm{ker}(\mathcal{A}_{\varepsilon,\mu})= \overline{\mathcal{G}(\Omega)}$ (see Proposition \ref{prop:kerneldescription} in the Appendix).

\subsection{A unitarily equivalent operator {\it via} the Piola transform}\label{subsec:piola}
An important tool all along this paper is the so-called Piola transform which allows to rewrite the problem in any waveguide of the form $\Omega$, introduced in \eqref{diffeo:Phi} and below, as a problem recast in the straight waveguide $\Omega_0$.
In this section we prove the following result.
\begin{Pro}\label{prop:unitequivpiola}
The Maxwell operator $\mathcal{A}$ defined in \eqref{eqn:defAMaxwell} is unitarily equivalent to the Maxwell operator acting in $\bL_{\beps,\bmu}^2(\Omega_0)$ as
\[
	\hat{\mathcal{A}} := \rmi \begin{pmatrix}0 &  \beps^{-1}\curlvec\\-\bmu^{-1}\curlvec & 0\end{pmatrix},
\]
where
\begin{equation}\label{eq.defanisotropeps-mu}
    \beps := \varepsilon_0 \left(\frac{J_\Phi^\top J_\Phi}{\det J_\Phi}\right)^{-1},\quad \bmu := \mu_0 \left(\frac{J_\Phi^\top J_\Phi}{\det J_\Phi}\right)^{-1}
\end{equation}
and
\[
    D(\hat{\mathcal{A}}) := H_0(\curlvec,\Omega_0)\times H(\curlvec,\Omega_0).
\]
\end{Pro}
\begin{Rem}
The main feature of Proposition \ref{prop:unitequivpiola} is that one can see the Maxwell operator $\mathcal{A}_{\varepsilon_0,\mu_0}$ associated to an isotropic and homogeneous medium with constant scalar permittivity $\varepsilon_0$ and permeability  $\mu_0$  in the waveguide $\Omega$ as a Maxwell operator $\hat{\mathcal{A}}$ in the straight waveguide $\Omega_0$ associated to an anisotropic medium. The geometry of $\Omega$ is now encoded in the  permittivity and permeability tensors $ \beps$ and $\bmu$ of the  anisotropic medium.
\end{Rem}
Before going through the proof of Proposition \ref{prop:unitequivpiola} we introduce the Piola transform (see e.g  \cite[Section 3]{lamzac20},  \cite[Section 3.9]{Monk} or \cite{Nic-07}) as
\begin{equation}
    \mathscr{P}:\bL^2_\Sigma(\Omega) \to \bL^2_{\widetilde{\Sigma}}(\Omega_0), \quad \bv \mapsto J_\Phi^\top (\bv \circ \Phi),
\end{equation}
where 
\begin{equation}
    \widetilde{\Sigma} :=(\operatorname{det}J_\Phi) \  J_\Phi^{-1} (\Sigma \circ \Phi) (J_\Phi^{-1})^\top.
    \label{def:sigmatilde}
\end{equation}
This map has many properties.
\begin{Pro}
 We have that
\begin{enumerate}[label=\normalfont(\roman*)]
\label{prop:piolaproperties}
\item\label{itm:1} $\mathscr{P}$ is invertible and its inverse $\mathscr{P}^{-1}: \bL_{\widetilde{\Sigma}}^2(\Omega_0) \to \bL_\Sigma^2(\Omega)$ verifies:
\[
    \mathscr{P}^{-1}\bv = ((J_\Phi^{-1})^\top \bv)\circ \Phi^{-1};
\]
\item \label{itm:2} $\mathscr{P}$ is unitary map from $\bL^2_\Sigma(\Omega)$ to $\bL^2_{\widetilde{\Sigma}}(\Omega_0)$;
\item \label{itm:3}$\bv \in H(\curlvec, \Omega)$ if and only if $\mathscr{P}\bv \in H(\curlvec, \Omega_0)$, and in this case
\begin{equation}
    \curlvec (\mathscr{P}\bv)  = (\operatorname{det}J_\Phi) \ J_\Phi^{-1} (\curlvec \bv \circ \Phi); \label{eq.curlpiola}
\end{equation}
\item \label{itm:4}$\mathscr{P}\left(H_0(\curlvec,\Omega) \right) = H_0(\curlvec,\Omega_0)$;
\item \label{itm:5}$\bv \in H(\Div \Sigma,\Omega)$ if and only if $\mathscr{P}\bv \in H(\Div \widetilde{\Sigma},\Omega_0)$, and in this case
\begin{equation*}
    \operatorname{div}\left(\widetilde{\Sigma} \mathscr{P}\bv \right) = (\operatorname{det}J_\Phi) \left(\operatorname{div}(\Sigma \bv) \circ \Phi \right);
\end{equation*}
\item \label{itm:6}$\mathscr{P}(H_0(\Div \Sigma,\Omega)) = H_0(\Div\widetilde{\Sigma},\Omega_0)$.
\end{enumerate}
\end{Pro}
The proof of Proposition \ref{prop:piolaproperties} is reminiscent of \cite[Theorem 2.5.]{lapazaI}. One of the major advantage of the Piola transform $\mathscr{P}$ is that it transforms adequately the natural functional spaces involved in the Maxwell system in the domain $\Omega$ to the same ones in the domain $\Omega_0$. 
\begin{proof}[Proof of Proposition \ref{prop:piolaproperties}]
The proof of Point \ref{itm:1} is straighforward. Point \ref{itm:2} is nothing but rewriting the change of variables. Indeed, let $\bu,\bv \in \bL^2_\Sigma(\Omega)$, one has
\begin{align*}
    \int_\Omega (\Sigma \bu) \cdot \overline{\bv} \, \rmd\bx &= \int_{\Omega_0} \left((\Sigma \circ \Phi)(\bu \circ \Phi)\right) \cdot \overline{(\bv \circ \Phi)} (\det J_\Phi) \,\rmd \bx \\&= \int_{\Omega_0} \left(J_\Phi^{-1}(\Sigma \circ \Phi)(J_\Phi^{-1})^\top J_\Phi^\top(\bu \circ \Phi)\right) \cdot \overline{\left( J_\Phi^\top(\bv\circ \Phi)\right)} (\det J_\Phi) \,\rmd \bx\\
    & = \int_{\Omega_0} \left(\widetilde{\Sigma}\mathscr{P}\bu\right)\cdot \overline{\mathscr{P}\bv}\, \rmd \bx
\end{align*}
which gives Point \ref{itm:2}. Points \ref{itm:3} and \ref{itm:5} can be found in \cite[Theorem 2.5.]{lapazaI}.

We are left with the proof of Points \ref{itm:4} and \ref{itm:6}. Let us deal with Point \ref{itm:4} and take $\bu \in H_0(\curlvec,\Omega)$, by definition for any $\bv\in H(\curlvec,\Omega)$ one has
\[
    \int_\Omega \bu \cdot \overline{(\curlvec \bv)} \,\rmd \bx =  \int_\Omega (\curlvec\bu) \cdot \overline{\bv} \, \rmd \bx
\]
which, after a change of variables and thanks to Point \ref{itm:3}, rewrites
\[
    \int_{\Omega_0} \mathscr{P}\bu \cdot \overline{(\curlvec \mathscr{P}\bv)} \,\rmd \bx =  \int_{\Omega_0} (\curlvec\mathscr{P}\bu) \cdot \overline{\mathscr{P}\bv}\, \rmd \bx.
\]
Again, by Point \ref{itm:3}, one has $\mathscr{P}H(\curlvec,\Omega) = H(\curlvec,\Omega_0)$ which yields (by \eqref{eq.Hcurl0}) that $\mathscr{P}\bu \in H_0(\curlvec,\Omega_0)$ and $\mathscr{P}H_0(\curlvec,\Omega)\subset H_0(\curlvec,\Omega_0)$. The converse inclusion is proved using the application $\mathscr{P}^{-1}$.

Finally, to prove Point \ref{itm:6}, take $\bu \in H_0(\Div \Sigma,\Omega)$, by definition for all $f \in H^1(\Omega)$ there holds 
\[
    \int_\Omega (\Sigma \bu) \cdot \overline{\nabla f} \,\rmd \bx = - \int_\Omega \Div(\Sigma \bu) \overline{f} \,\rmd \bx.
\]
Performing a change of variables and taking Point \ref{itm:5} into account one gets
\begin{equation}
    \int_{\Omega_0}((\Sigma\circ \Phi)(\bu \circ \Phi))\cdot \overline{((\nabla f)\circ \Phi)}(\det J_\Phi)\,\rmd \bx = - \int_{\Omega_0} (\Div(\Sigma \bu)\circ\Phi) \overline{(f\circ \Phi)} (\det J_\Phi) \, \rmd \bx.
    \label{eqn:hdiv0piola}
\end{equation}
Note that by Point \ref{itm:5} the term on the right-hand side of \eqref{eqn:hdiv0piola} rewrites
\[
    - \int_{\Omega_0} (\Div(\Sigma \bu)\circ\Phi) \overline{(f\circ \Phi)} (\det J_\Phi)\,\rmd \bx = - \int_{\Omega_0}\Div(\widetilde{\Sigma}\mathscr{P}\bu)\overline{(f\circ \varphi)} \,\rmd \bx.
\]
Recall that there holds $\nabla (f\circ \Phi) = J_\Phi^\top \left((\nabla f)\circ \Phi\right)$, hence using the definition of $\widetilde{\Sigma}$, one gets that the term on the left-hand side of \eqref{eqn:hdiv0piola} rewrites
\[
    \int_{\Omega_0}\left((\Sigma \circ \Phi)(\bu\circ\Phi)\right)\cdot \overline{((\nabla f)\circ\Phi)}(\det J_\Phi)\,\rmd \bx = \int_{\Omega_0} (\widetilde{\Sigma}\mathscr{P}\bu)\cdot\overline{(\nabla(f\circ\Phi))} \,\rmd \bx.
\]
It gives
\[
\int_{\Omega_0} (\widetilde{\Sigma}\mathscr{P}\bu)\cdot\overline{(\nabla(f\circ\Phi))}\, \rmd \bx = - \int_{\Omega_0}\Div(\widetilde{\Sigma}\mathscr{P}\bU)\overline{(f\circ \Phi)} \,\rmd \bx
\]
and note that $f \in H^1(\Omega)$ if and only $f\circ \Phi \in H^1(\Omega_0)$ thus, $\mathscr{P}\bu \in H_0(\Div\widetilde{\Sigma},\Omega_0)$ and one gets the inclusion $\mathscr{P}H_0(\Div \Sigma,\Omega)\subset H_0(\Div \widetilde{\Sigma},\Omega)$. The converse inclusion is proved in the same way using $\mathscr{P}^{-1}$.
\end{proof}
Note that using the notation \eqref{def:sigmatilde}, choosing $\Sigma = \varepsilon_0 \mathds{1}_3$ or $\Sigma = \mu_0 \mathds{1}_3$, by \eqref{eq.defanisotropeps-mu} one gets $\widetilde{\Sigma} = \beps$ or  $ \widetilde{\Sigma}= \bmu$, respectively.
Therefore, by  \ref{itm:2}, one can  introduce a unitary map $\bbU$ as follows
\begin{equation} \label{bbU:def}
\bbU: \bL_0^2(\Omega) \to \bL_\bbG^2(\Omega_0), \quad 
\begin{pmatrix}
\bE\\
\bH
\end{pmatrix}
\mapsto
\begin{pmatrix}
\mathscr{P}\bE \\
\mathscr{P}\bH
\end{pmatrix},
\end{equation}
where  the matrix-field $\bbG$ is given (in $\Omega_0$)  by
\begin{equation} \label{Gmatrix}
\bbG:=\frac{ J_\Phi^\top  J_\Phi}{\operatorname{det}J_\Phi}.
\end{equation}
and $\bL_0^2(\Omega)$ and $\bL_\bbG^2(\Omega_0)$ are  defined by  $\bL_0^2(\Omega):=\bL_{\varepsilon_0\mathds{1}_3,\mu_0\mathds{1}_3}^2(\Omega)$ and $\bL_\bbG^2(\Omega_0) := \bL_{\beps,\bmu}^2(\Omega_0)$.
\begin{Rem}\label{rem:Ginvertibleunifo}
We recall that condition \eqref{assumpt:a:kappa} implies   \eqref{unif:bounds:ak}. Thus  $\bbG(x)$, defined by \eqref{Gmatrix}, is well-defined  and invertible and one notices that $\operatorname{det}\bbG(\bx)=(\operatorname{det}J_{\Phi}(\bx))^{-1}$. Using \eqref{unif:bounds:ak} and \eqref{relaxed:hyp}, one deduces  that  $\bbG(\bx)$ and $\bbG(\bx)^{-1}= \operatorname{det}J_{\Phi}(\bx)\, \operatorname{adj}\bbG(\bx)$ are uniformly bounded in $\Omega_0$. Thus,  the matrices $\bbG(\bx)$ and  $\bbG(\bx)^{-1}$ are positive real symmetric matrix-fields which are uniformly positive definite in $\Omega_0$. Therefore,  $\beps$ and $ \bmu$ (defined by \eqref{eq.defanisotropeps-mu}) verify the condition \eqref{hyp:sigma}.
\end{Rem}
Now, we are in a good position to prove Proposition \ref{prop:unitequivpiola}.
\begin{proof}[Proof of Proposition \ref{prop:unitequivpiola}] 
Note that $\bbU$, defined in \eqref{bbU:def}, is a unitary map thanks to Point \ref{itm:1} of Proposition \ref{prop:piolaproperties} and we introduce the  operator 
\begin{equation} \label{def:Ahat}
\hat{\mathcal{A}}:=\bbU \mathcal{A}  \bbU^*: D(\hat{\mathcal{A}}) \subset
\bL^2_\bbG(\Omega_0) \to \bL^2_\bbG(\Omega_0)
\end{equation} 
where  $D(\hat{\mathcal{A}}) = \bbU(D(\mathcal{A}))$.
Observe that by Point \ref{itm:2} of Proposition \ref{prop:piolaproperties}, $\hat{\mathcal{A}}$ acts in $\bL_{\bbG}^2(\Omega_0)$ and that Points \ref{itm:3}-\ref{itm:4} yields
\begin{equation*}
D(\hat{\mathcal{A}})= H_0(\curlvec,\Omega_0)\times H(\curlvec,\Omega_0).
\end{equation*}
As  the unitary transform $\bbU^*$ involves the inverse of the Piola transform, by Points \ref{itm:1} and \ref{itm:3} of Proposition \ref{prop:piolaproperties}, one obtains an expression of $   \curlvec ( \mathscr{P}^{-1} \bv)$ (by replacing $\Phi$ by $\Phi^{-1}$ and $ J_\Phi$ by  $J_\Phi^{-1}$ in \eqref{eq.curlpiola}).  This leads to the following  expression of the operator $\hat{\mathcal{A}}$:
\begin{equation*}
\hat{\mathcal{A}}=
\rmi \begin{pmatrix}
0 &\beps^{-1} \curlvec\\
- \bmu^{-1} \curlvec & 0
\end{pmatrix},
\end{equation*}
which ends the proof of Proposition \ref{prop:unitequivpiola}.
\end{proof}


\subsection{Another unitary operator acting in a fixed Hilbert space} \label{sect:weder}
For further uses, we need to be able to compare the operator $\hat{\mathcal{A}}$ of Proposition \ref{prop:unitequivpiola} with the operator $\mathcal{A}_0$ defined as the Maxwell operator on the straight waveguide, with constant permittivity $\varepsilon_0$ and permeability $\mu_0$. Hence, $\mathcal{A}_0$ is the operator which acts in $\bL_0^2(\Omega_0):= \bL^2_{\varepsilon_0 \mathds{1}_3,\mu_0 \mathds{1}_3}(\Omega_0)$ as
\begin{equation} \label{A0:deffi}
\mathcal{A}_0=
\rmi \begin{pmatrix}
0 & \varepsilon_0^{-1} \curlvec\\
- \mu_0^{-1} \curlvec & 0
\end{pmatrix},\quad
D(\mathcal{A}_0)= H_0(\curlvec,\Omega_0)\times H(\curlvec,\Omega).
\end{equation}
Note that $\hat{\mathcal{A}}$ and $\mathcal{A}_0$ do not act in the same Hilbert spaces. However, the spaces $\bL_\bbG^2(\Omega_0)$ and $\bL_0^2(\Omega_0)$ coincides as sets due to Remark \ref{rem:Ginvertibleunifo}. Thus, we are lead to introduce the identification map
\begin{equation} \label{IbbG:def}
  \mathcal{I}_\bbG:\bL_\bbG^2(\Omega_0) \to \bL_0^2(\Omega_0),\quad \mathcal{I}_{\bbG}(\bU) = \bU.
\end{equation}
Moreover, there holds $$\mathcal{I}_\bbG^* 
\begin{pmatrix}
\bE\\
\bH
\end{pmatrix}
= \begin{pmatrix}
\bbG \bE \\
\bbG \bH 
\end{pmatrix}$$
and thus this map is not unitary. However, as $\bbG$ satisfies the condition \eqref{hyp:sigma}  (see Remark \ref{rem:Ginvertibleunifo}), it is continuous with continuous inverse since the norms induced by the inner  products of $\bL_\bbG^2(\Omega_0)$ and $\bL_0^2(\Omega_0)$ are equivalent.  By virtue of the Proposition \ref{prop:unitequivpiola} and the relation \eqref{Gmatrix}, which defines 
$\bbG$, we can express  the operator $\hat{\mathcal{A}}$ as follows:
\begin{equation}\label{hatA:Azero}
\hat{\mathcal{A}} = \begin{pmatrix}
 \bbG & 0 \\
0 &  \bbG
\end{pmatrix}
\mathcal{I}_\bbG^{-1}\mathcal{A}_0 \mathcal{I}_\bbG.
\end{equation}

Now, since for any $x \in \Omega_0$ the real matrix $\bbG(x)$ is symmetric and positive definite (cf. Remark \ref{rem:Ginvertibleunifo}) then for any $x \in \Omega_0$ we can uniquely define its (symmetric) positive definite square root $\bbG^{1/2}$ (see \cite[Thm. 7.2.6]{hojho}). Following an idea of \cite[Ch.3 \S 3]{weder}, we introduce the unitary map $\bbV$ defined as
\begin{equation}\label{eqn:defbbv}
\bbV: \bL^2_\bbG(\Omega_0) \to \bL^2_0(\Omega_0), \quad \begin{pmatrix}
\bE\\
\bH
\end{pmatrix}
\mapsto
\begin{pmatrix}
\bbG^{-1/2} \bE \\
\bbG^{-1/2} \bH
\end{pmatrix}.
\end{equation}
The following holds.

\begin{Pro}\label{prop:unitnotmaxwell}
The operator $\mathcal{A}$ is unitarily equivalent to the operator $\tilde{\mathcal{A}}:= \bbV \hat{\mathcal{A}}\bbV^*$ acting in $\bL_0^2(\Omega_0)$. Moreover, one can express $ \tilde{\mathcal{A}}$ in term of $\mathcal{A}_0$ via the following decomposition
\[
    \tilde{\mathcal{A}} = \mathcal{I}_\bbG \bbV^*\mathcal{A}_0\mathcal{I}_\bbG\bbV^*,\ \mbox{ where } D(\tilde{\mathcal{A}}) = \bbV \mathcal{I}_\bbG^{-1} D(\mathcal{A}_0) = \bbV \left(H_0(\curlvec,\Omega_0)\times H(\curlvec,\Omega_0)\right).
\]
\end{Pro}
The advantage of Proposition \ref{prop:unitnotmaxwell} is that the operator $\tilde{\mathcal{A}}$ is unitarily equivalent to the Maxwell operator $\mathcal{A}$ but acts in the Hilbert space $\bL_0^2(\Omega_0)$ which does not depend on the initial geometry of the waveguide $\Omega$. However, its major drawback is that its differential expression is more involved and we are loosing the Maxwell structure of the operators $\mathcal{A}$ and $\hat{\mathcal{A}}$, defined in \eqref{eqn:defAMaxwell}
 and Proposition \ref{prop:unitequivpiola} respectively.
 
 \begin{proof}[Proof of Proposition \ref{prop:unitnotmaxwell}] Since  $\tilde{\mathcal{A}} = \bbV \hat{\mathcal{A}}\bbV^*$, one has by virtue of \eqref{def:Ahat}, $\tilde{\mathcal{A}}=  \bbV \bbU \mathcal{A}( \bbV \bbU)^* $, thus $\tilde{\mathcal{A}}$ and $\mathcal{A}$ are unitarily equivalent . Moreover, from the factorization\eqref{hatA:Azero}, we deduce that
\begin{equation} \label{AA0}
\begin{split}
\tilde{\mathcal{A}}&= \bbV \hat{\mathcal{A}} \bbV^* = 
\begin{pmatrix}
\bbG^{-1/2} & 0 \\
0 & \bbG^{-1/2}
\end{pmatrix}
\begin{pmatrix}
\bbG & 0\\
0 & \bbG
\end{pmatrix}
\mathcal{I}_\bbG^{-1} \mathcal{A}_0 \mathcal{I}_\bbG \bbV^*\\
&=
\begin{pmatrix}
\bbG^{1/2} & 0 \\
0 & \bbG^{1/2}
\end{pmatrix}
\mathcal{I}_\bbG^{-1} \mathcal{A}_0 \mathcal{I}_\bbG \bbV^*
= \mathcal{I}_\bbG \bbV^* \mathcal{A}_0 \mathcal{I}_\bbG \bbV^*.
\end{split}
\end{equation}
From the above decomposition,  we remark also that 
\begin{equation*}
D(\tilde{\mathcal{A}}) = \bbV \mathcal{I}_\bbG^{-1} D(\mathcal{A}_0) = \bbV \left(H_0(\curlvec,\Omega_0)\times H(\curlvec,\Omega_0)\right).
\end{equation*}
\end{proof}
\section{Essential spectrum} \label{secessspec}
In this section we prove Theorem \ref{Main1}.

\subsection{The case of the constantly twisted waveguide}\label{sec-const-twist}

In this paragraph we focus on the specific case of the constantly twisted waveguide $\Omega_\beta$ defined as the image of the diffeomorphism $\Phi_\beta$ of Equation \eqref{phiB:def}. The Maxwell operator $\mathcal{A}_\beta$ is defined as $\mathcal{A}$ introduced in \eqref{eqn:defAMaxwell} in the specific case $\Omega = \Omega_\beta$ (see right panel of Figure \ref{fig:2}).

The main result of this paragraph is the following proposition.

\begin{Pro} \label{ess:spectr:beta}
There exists $a_\beta > 0$ such that there holds
\[
	\sigma(\mathcal{A}_\beta) = \sigma_{\mathrm{ess}}(\mathcal{A}_\beta) = (-\infty,-a_\beta]\cup\{0\}\cup[a_\beta,+\infty).
\]
Moreover, when   $\beta=0$ (i.e. when the waveguide is straight), one has  $
    a_0 = \sqrt{\lambda_2^N(\omega)}\, c$.
\end{Pro}
The proof of Proposition \ref{ess:spectr:beta} is by applying Proposition \ref{prop:unitequivpiola} to the operator $\mathcal{A}_\beta$. It gives a unitarily equivalent operator $\hat{\mathcal{A}}_\beta$ with a new permittivity $\beps$ and permeability $\bmu$. It turns out that these matrix fields depend only on the transverse variable $\by \in \omega$, thus we can apply the result of \cite[Theorem 1.3.]{filonov19}. 

\begin{proof}[Proof of Proposition \ref{ess:spectr:beta}]
Let $\hat{\mathcal{A}}_\beta$ be the operator unitarily equivalent to $\mathcal{A}_\beta$ obtained {\it via}  Proposition \ref{prop:unitequivpiola}. To be able to apply \cite[Theorem 1.3.]{filonov19} we need to  compute the matrix fields $\beps$ and $\bmu$. To this aim, remark that  using \eqref{phiB:def} there holds
\begin{equation*}
    J_{\Phi_\beta}(s,y_2,y_3) = 
\begin{pmatrix}
1 & 0 & 0\\
\big(-y_2 \sin(\beta \, s)+ y_3 \cos(\beta \, s) \big)\beta & \cos (\beta \, s) & \sin (\beta \, s) \\[4pt]
-\big(y_2 \cos (\beta \, s)+ y_3 \sin (\beta \, s) \big)\beta& -\sin (\beta \, s)& \cos (\beta \, s)
\end{pmatrix},
\end{equation*}
therefore as $ \det(J_{\Phi_\beta})=1$,  $\bbG =\frac{J_{\Phi_\beta}^\top J_{\Phi_\beta}}{\det J_{\Phi_\beta}}$ verifies 
\begin{equation} \label{def:Gbeta}
\bbG(s,y_2,y_3)=\bbG(y_2,y_3) =
\begin{pmatrix}
1 + \beta^2 |\by|^2 & y_3 \beta & -y_2 \beta \\
y_3 \beta & 1 & 0 \\
-y_2 \beta & 0 & 1
\end{pmatrix}.
\end{equation}
Note that $\bbG$ does not depend on the longitudinal variable $s$. Thus neither do $\beps$ and $\bmu$ because they are simply defined as\footnote{Note that in Figure \ref{fig:diag} such a choice of $\beps$ and $\bmu$ are called $\beps_\beta$ and $\bmu_\beta$, respectively.}
$$\beps=\varepsilon_0 \, \bbG^{-1} \mbox{ and }  \bmu=\mu_0 \,\bbG^{-1}.$$
By \cite[Theorem 1.3.]{filonov19}, there exists an $a_\beta > 0$ such that the spectrum is as stated in Proposition \ref{ess:spectr:beta}. The case $\beta=0$ which corresponds to a homogeneous  straight waveguide of permittivity $\varepsilon_0>0$ and  $\mu_0$ is also treated in \cite[Theorem 1.3.]{filonov19} where they prove that $a_0=\sqrt{\lambda_2^N(\omega)}\, c$.
\end{proof}

To prove that the operator $\mathcal{A}$ defined in \eqref{eqn:defAMaxwell}, that is the Maxwell operator in the general waveguide $\Omega$, has the same essential spectrum as the operator $\mathcal{A}_\beta$ we need to compare them. To this aim, we need to rewrite them using the operators $\tilde{A}$ and $\tilde{A}_\beta$ from Proposition \ref{prop:unitnotmaxwell}, respectively. Thus, for further uses, we introduce the following notations.

Let $\tilde{\mathcal{A}}_\beta$ be the unitarily equivalent operator to $\mathcal{A}_\beta$ {\it via} Proposition \ref{prop:unitnotmaxwell} in the specific case of $\Omega = \Omega_\beta$. In this context, we set
\begin{equation} \label{allBeta}
\bbG_\beta:=\bbG, \quad \bL_\beta^2(\Omega_0) := \bL_{\bbG_\beta}^2(\Omega_0),
\end{equation}
where $\bbG$ the matrix defined in \eqref{Gmatrix}.

Note that $\tilde{\mathcal{A}_\beta}$ acts in $\bL^2_0(\Omega_0)$ and that there holds:
\begin{equation}\label{Unitarilyeq-Abeta}
\tilde{\mathcal{A}}_\beta := \bbV_\beta \bbU_\beta \mathcal{A}_\beta \bbU_\beta^* \bbV_\beta^*: D(\tilde{\mathcal{A}}_\beta) \subset \bL^2_0(\Omega_0) \to \bL^2_0(\Omega_0),
\end{equation}
where we have used the subscript $\beta$ for the map $\bbV$ defined in \eqref{eqn:defbbv} and the unitary map $\bbU$ defined in \eqref{bbU:def} to emphasize that we are working in the specific case $\Omega = \Omega_\beta$.

Thanks to Proposition \ref{prop:unitnotmaxwell}, we know that
\begin{equation}  \label{AbA0}
\tilde{\mathcal{A}}_\beta = \mathcal{I}_\beta \mathbb{V}_\beta^* \mathcal{A}_0 \mathcal{I}_\beta \mathbb{V}_\beta^*
\end{equation}
where
\begin{equation*} 
\mathcal{I}_\beta :\bL^2_\beta(\Omega_0) \to \bL^2_0(\Omega_0)
\end{equation*}
is the identity map. This formalism becomes useful in the following paragraph, in which we  compare the resolvent of $\mathcal{A}$ with the one of $\mathcal{A}_\beta$.

\medskip
\begin{figure}[h!!]
\begin{center}
\includegraphics[width=\textwidth]{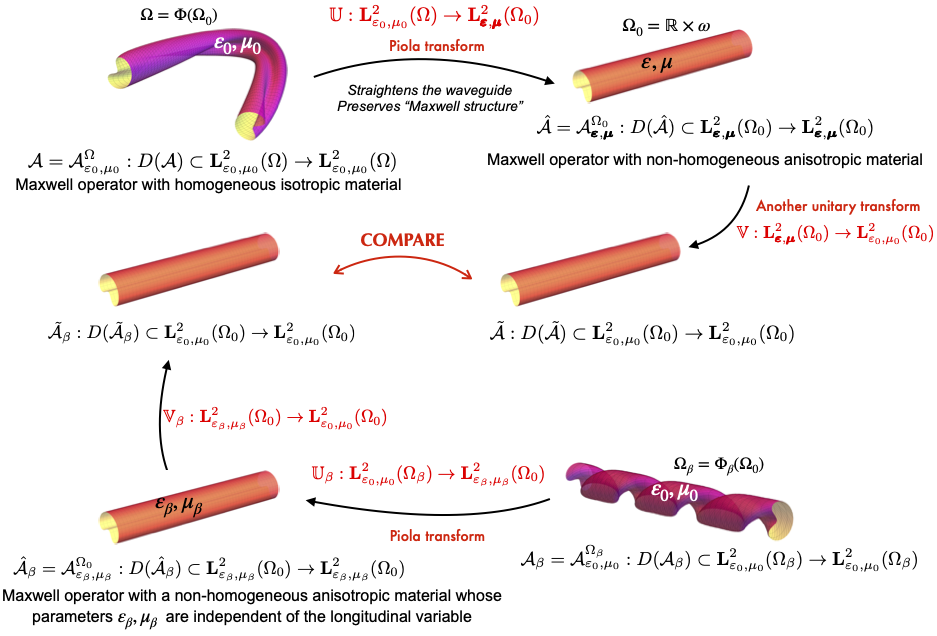}
\end{center}
\caption{Operators and transforms involved in our analysis. As the transforms written in red are unitary, comparing the essential spectra of the operators $\mathcal{A}$ and $\mathcal{A}_\beta$ amounts to compare the essential spectra of $\tilde{\mathcal{A}}$ and $\tilde{\mathcal{A}_\beta}$.}
\label{fig:diag}
\end{figure}

\subsection{Structure of the essential spectrum in the general case}
The objective of this paragraph is to prove Theorem \ref{Main1}. To this aim, in \S \ref{subsub:BS} we introduce a resolvent identity and a Birman-Schwinger type principle. In \S \ref{subsub:proofmain1} we prove Theorem \ref{Main1} using Meromorphic Fredholm Theorem. \S \ref{subsub:compactness} is devoted to the proof of compactness of an operator involved when using the Meromorphic Fredholm Theorem.

\subsubsection{A resolvent type identity and its consequence} \label{subsub:BS}
In this paragraph we prove two lemmas but first we need to introduce a few notations.
For $\lambda \in \rho(\tilde{\mathcal{A}}) \cap \rho(\tilde{\mathcal{A}}_\beta)$. We introduce the resolvents
\begin{equation*}
\tilde{R}(\lambda):= (\tilde{\mathcal{A}}-\lambda I )^{-1},\quad \tilde{R}_\beta(\lambda):= (\tilde{\mathcal{A}}_\beta -\lambda\,  I )^{-1},
\end{equation*}
as well as the reduced resolvents
\begin{equation} \label{red:res:b}
 \check{\tilde{R}}(\lambda):=\tilde{R}(\lambda)+\frac{1}{\lambda} \tilde{P}_{0},\quad \check{\tilde{R}}_\beta(\lambda):=\tilde{R}_\beta(\lambda)+\frac{1}{\lambda} \tilde{P}_{0,\beta}
\end{equation}
where $\tilde{P}_{0}$ and $\tilde{P}_{0,\beta}$ are the orthogonal projections (in $\bL_0^2(\Omega_0)$) onto the kernel of $\tilde{\mathcal{A}}$ and $\tilde{\mathcal{A}}_\beta$, respectively.
The first Lemma is a resolvent identity.
\begin{Lem} \label{lemma:resid1}
Let $\lambda \in \rho(\tilde{\mathcal{A}})\cap \rho(\tilde{\mathcal{A}}_\beta)$, there holds
\begin{equation} \label{1st:res:eq}
\tilde{R}(\lambda)= \bbV \mathcal{I}_\bbG^{-1}  \mathcal{I}_\beta \bbV_\beta^*  \tilde{R}_\beta(\lambda) \calQ_\beta(\lambda)^{-1} \mathcal{I}_\beta \bbV_\beta^* \bbV \mathcal{I}_\bbG^{-1},
\end{equation}
where  $\mathcal{Q}_\beta(\lambda)\in \mathcal{B}\big( \bL^2_0(\Omega_0)\big)$ is invertible and defined on by 
\begin{equation} \label{eqn:def:Qbeta}
    \mathcal{Q}_\beta(\lambda) := I + \lambda \left( I- \mathcal{I}_\beta \mathbb{V}_\beta^* (\bbV \mathcal{I}_\bbG^{-1})^2 \mathcal{I}_\beta \mathbb{V}_\beta^* \right)  \tilde{R}_\beta(\lambda).
\end{equation}
Similarly, one has
\begin{equation} \label{2nd:res:eq}
\tilde{R}_\beta(\lambda)= \bbV_\beta \mathcal{I}_\beta^{-1} \mathcal{I}_\bbG \bbV^* \tilde{R}(\lambda) {\mathcal{Q}}(\lambda)^{-1} \mathcal{I}_\bbG \bbV^* \bbV_\beta \mathcal{I}_\beta^{-1}
\end{equation}
where ${\calQ}(\lambda)\in \mathcal{B}\big( \bL^2_0(\Omega_0)\big)$ is invertible and  defined by
\begin{equation}\label{eqn:def:Q}
{\calQ}(\lambda) := I + \lambda \left( I - \mathcal{I}_\bbG \bbV^* (\bbV_\beta \mathcal{I}_\beta^{-1})^2 \mathcal{I}_\bbG \bbV^* \right) \tilde{R}(\lambda).
\end{equation}
Furthemore,  the definitions \eqref{eqn:def:Qbeta} and  \eqref{eqn:def:Q} of  the operators  $\mathcal{Q}_\beta(\lambda)\in \mathcal{B}\big( \bL^2_0(\Omega_0)\big))$ and  ${\calQ}(\lambda)\in \mathcal{B}\big( \bL^2_0(\Omega_0)\big)$   holds  for  $\lambda\in  \rho(\tilde{\mathcal{A}}_\beta)$ and $\lambda\in  \rho(\tilde{\mathcal{A}})$ respectively.
\end{Lem}

The following proposition describes a Birman-Schwinger type principle.
\begin{Pro} \label{lemmaK}
The following statements hold.
\begin{enumerate}[label=\normalfont(\roman*)]
\item \label{itm:BS1} Fix $\lambda\in \rho(\tilde{\mathcal{A}}_\beta)$ and  $\lambda_0 \in \rho(\tilde{\mathcal{A}}_\beta) \cap \rho(\tilde{\mathcal{A}})$.  Then
$$\calQ_\beta(\lambda)  = \calQ_\beta(\lambda_0) \left( I + \mathcal{K}_\beta(\lambda) \right)$$
where $\mathcal{K}_\beta(\lambda)\in \mathcal{B}\big(\bL_0^2(\Omega_0)\big)$ is defined as follows
\begin{equation} \label{operK}
\mathcal{K}_\beta(\lambda) := (\lambda-\lambda_0) \calQ_\beta(\lambda_0)^{-1} \left[ I - \mathcal{I}_\beta \bbV_\beta^* (\bbV \mathcal{I}_\bbG^{-1})^2 \mathcal{I}_\beta \bbV_\beta^*  \right] \check{\tilde{R}}_\beta(\lambda) \tilde{\mathcal{A}}_\beta \tilde{R}_\beta(\lambda_0).
\end{equation}
Moreover, $\calQ_\beta(\lambda)$ is invertible if and only if $I + \mathcal{K}_\beta(\lambda)$ is invertible.
\item \label{itm:BS2} Fix $\lambda\in \rho(\tilde{\mathcal{A}})$ and   $\lambda_0 \in \rho(\tilde{\mathcal{A}}_\beta) \cap \rho(\tilde{\mathcal{A}})$.  Then
$${\calQ}(\lambda)  = {\calQ}(\lambda_0) \left( I + {\mathcal{K}}(\lambda) \right)$$
where ${\mathcal{K}}(\lambda)\in  \mathcal{B}\big(\bL^2_0(\Omega_0)\big)$ is defined as follows
\begin{equation} \label{operK'}
{\mathcal{K}}(\lambda) := (\lambda-\lambda_0) {\calQ}(\lambda_0)^{-1} \left[ I - \mathcal{I}_\bbG \bbV^* (\bbV_\beta \mathcal{I}_\beta^{-1})^2 \mathcal{I}_\bbG \bbV^* \right] \check{\tilde{R}}(\lambda) \tilde{\mathcal{A}} \tilde{R}(\lambda_0).
\end{equation}
Moreover, ${\calQ}(\lambda)$ is invertible if and only if $I + {\mathcal{K}}(\lambda)$ is invertible.
\end{enumerate}
\end{Pro}
Now, we go through the proofs of Lemma \ref{lemma:resid1} and Proposition \ref{lemmaK}.
\begin{proof}[Proof of Lemma \ref{lemma:resid1}]
We only prove \eqref{1st:res:eq}, as the proof of \eqref{2nd:res:eq} is similar.\\
\noindent Let $\lambda \in \rho(\tilde{\mathcal{A}}_\beta)$. Note that by equations \eqref{AA0},\eqref{AbA0}, there holds
\[
   \tilde{\mathcal{A}} =  \mathcal{I}_\bbG \bbV^*   \bbV_\beta \mathcal{I}_\beta^{-1}  \tilde{\mathcal{A}}_\beta  \bbV_\beta \mathcal{I}_\beta^{-1} \mathcal{I}_\bbG \bbV^*.
\]
Moreover, one has (as $\lambda\in \rho(\tilde{\mathcal{A}}_\beta)$):
\begin{eqnarray} \label{alternative:Qb}
&&\mathcal{I}_\beta \bbV_\beta^* \bbV \mathcal{I}_\bbG^{-1} \left[ \tilde{\mathcal{A}} - \lambda \, I \right] \bbV \mathcal{I}_\bbG^{-1} \mathcal{I}_\beta \bbV_\beta^*  \tilde{R}_\beta(\lambda) \nonumber \\
&&=\mathcal{I}_\beta \bbV_\beta^* \bbV \mathcal{I}_\bbG^{-1} \left[ \mathcal{I}_\bbG \bbV^*   \bbV_\beta \mathcal{I}_\beta^{-1}  \tilde{\mathcal{A}}_\beta  \bbV_\beta \mathcal{I}_\beta^{-1} \mathcal{I}_\bbG \bbV^*  - \lambda \, I \right] \bbV \mathcal{I}_\bbG^{-1} \mathcal{I}_\beta \bbV_\beta^*  \tilde{R}_\beta(\lambda)  \nonumber \\
&&=\mathcal{I}_\beta \bbV_\beta^* \bbV \mathcal{I}_\bbG^{-1} \Big[ \mathcal{I}_\bbG \bbV^*   \bbV_\beta \mathcal{I}_\beta^{-1}  (\tilde{\mathcal{A}}_\beta - \lambda \, I) \bbV_\beta \mathcal{I}_\beta^{-1} \mathcal{I}_\bbG \bbV^* + \lambda \mathcal{I}_\bbG \bbV^*   \bbV_\beta \mathcal{I}_\beta^{-1}  \bbV_\beta \mathcal{I}_\beta^{-1} \mathcal{I}_\bbG \bbV^*   - \lambda \, I \Big] \bbV \mathcal{I}_\bbG^{-1} \mathcal{I}_\beta \bbV_\beta^*  \tilde{R}_\beta(\lambda) \nonumber  \\
&&=I + \lambda \left( I- \mathcal{I}_\beta \mathbb{V}_\beta^* (\bbV \mathcal{I}_\bbG^{-1})^2 \mathcal{I}_\beta \mathbb{V}_\beta^* \right)  \tilde{R}_\beta(\lambda) \nonumber \\
&& = \mathcal{Q}_\beta(\lambda).
\end{eqnarray}
Note that the last formula clearly  implies that  $\mathcal{Q}_\beta(\lambda)\in \mathcal{B}(\bL^2(\Omega_0))$. In addition, if $\lambda\in \rho(\tilde{\mathcal{A}})\cap  \rho(\tilde{\mathcal{A}}_\beta)$, $\mathcal{Q}_\beta(\lambda)$  is invertible and its inverse is given by
\begin{equation}\label{eq.inverseQBeta}
\mathcal{Q}_\beta(\lambda)^{-1}= (\tilde{\mathcal{A}}_\beta- \lambda\, I) (\bbV \mathcal{I}_\bbG^{-1} \mathcal{I}_\beta \bbV_\beta^*)^{-1}  \tilde{R}(\lambda) (\mathcal{I}_\beta \bbV_\beta^* \bbV \mathcal{I}_\bbG^{-1})^{-1}.
\end{equation}
Thus, $\mathcal{Q}_\beta(\lambda)^{-1}$ is bounded and belongs to $\mathcal{B}\big(\bL_0^2(\Omega_0)\big)$ (since $(\mathcal{I}_\beta \bbV_\beta^* \bbV \mathcal{I}_\bbG^{-1})^{-1}\in \mathcal{B}(\bL^2(\Omega_0))$, $ \tilde{R}(\lambda)\in \mathcal{B}\big(\bL_0^2(\Omega_0), D(\tilde{\mathcal{A}})\big)$, $(\bbV \mathcal{I}_\bbG^{-1} \mathcal{I}_\beta \bbV_\beta^*)^{-1}\in \mathcal{B}\big(D(\tilde{\mathcal{A}}),D(\tilde{\mathcal{A}}_{\beta}) \big)$ and $\tilde{\mathcal{A}}_\beta- \lambda\, I \in \mathcal{B}(D(\tilde{\mathcal{A}}_{\beta}), \bL^2(\Omega_0))$ where the Hilbert spaces $D(\tilde{\mathcal{A}})$ and  $D(\tilde{\mathcal{A}}_{\beta})$ are endowed  here with their respective graph norms).  Moreover, \eqref{1st:res:eq}  follows immediately from \eqref{eq.inverseQBeta}. This concludes the proof.
\end{proof}

\begin{proof}[Proof of Proposition \ref{lemmaK}]
Let $\lambda\in \rho(\tilde{\mathcal{A}}_\beta)$ and  $\lambda_0 \in \rho(\tilde{\mathcal{A}}_\beta) \cap \rho(\tilde{\mathcal{A}})$.
By the first resolvent equality
\begin{equation*}
\tilde{R}_\beta(\lambda) -\tilde{R}_\beta(\lambda_0) = (\lambda-\lambda_0) \tilde{R}_\beta(\lambda) \tilde{R}_\beta(\lambda_0).
\end{equation*}
Therefore, by Lemma \ref{lemmaK}, one has
\begin{equation*}
\begin{split}
\calQ_\beta(\lambda) &= I + \lambda \left[ I- \mathcal{I}_\beta \mathbb{V}_\beta^* (\bbV \mathcal{I}_\bbG^{-1})^2 \mathcal{I}_\beta \mathbb{V}_\beta^* \right]  \tilde{R}_\beta(\lambda)\\
&= I + (\lambda_0 - \lambda_0 + \lambda) \left[ I- \mathcal{I}_\beta \mathbb{V}_\beta^* (\bbV \mathcal{I}_\bbG^{-1})^2 \mathcal{I}_\beta \mathbb{V}_\beta^* \right]  (\tilde{R}_\beta(\lambda_0) -\tilde{R}_\beta(\lambda_0) +\tilde{R}_\beta(\lambda))\\
&= \calQ_\beta(\lambda_0) + (\lambda-\lambda_0) \left[ I- \mathcal{I}_\beta \mathbb{V}_\beta^* (\bbV \mathcal{I}_\bbG^{-1})^2 \mathcal{I}_\beta \mathbb{V}_\beta^* \right]  \tilde{R}_\beta(\lambda_0) \\
& \qquad \qquad + \lambda \left[ I- \mathcal{I}_\beta \mathbb{V}_\beta^* (\bbV \mathcal{I}_\bbG^{-1})^2 \mathcal{I}_\beta \mathbb{V}_\beta^* \right]  (\tilde{R}_\beta(\lambda) -\tilde{R}_\beta(\lambda_0))\\
&=\calQ_\beta(\lambda_0)  \left( I + (\lambda-\lambda_0) \calQ_\beta(\lambda_0)^{-1} \left[ I- \mathcal{I}_\beta \mathbb{V}_\beta^* (\bbV \mathcal{I}_\bbG^{-1})^2 \mathcal{I}_\beta \mathbb{V}_\beta^* \right] \left[I + \lambda \tilde{R}_\beta(\lambda) \right] \tilde{R}_\beta(\lambda_0) \right).
\end{split}
\end{equation*}
Moreover, on $D(\tilde{\mathcal{A}}_\beta)$  we have that
\begin{equation*}
I + \lambda \tilde{R}_\beta(\lambda) = (\tilde{\mathcal{A}}_\beta -\lambda \, I)^{-1} (\tilde{\mathcal{A}}_\beta -\lambda \, I)  + \lambda(\tilde{\mathcal{A}}_\beta -\lambda \, I )^{-1} = (\tilde{\mathcal{A}}_\beta -\lambda \, I)^{-1} (\tilde{\mathcal{A}}_\beta - \lambda \, I + \lambda \, I)=\tilde{R}_\beta(\lambda) \tilde{\mathcal{A}}_\beta.
\end{equation*}
Finally, to conclude the proof of Point \ref{itm:BS1}, we note that since the operator $\tilde{\mathcal{A}}_\beta$ commutes with its spectral projectors, then $\tilde{P}_{0,\beta} \tilde{\mathcal{A}}_\beta= \tilde{\mathcal{A}}_\beta \tilde{P}_{0,\beta}= 0$, and thus  $\tilde{R}_\beta(\lambda) \tilde{\mathcal{A}}_\beta = \check{\tilde{R}}_\beta(\lambda) \tilde{\mathcal{A}}_\beta$.\\[6pt]
\noindent The proof of Point \ref{itm:BS2} follows the same lines.
\end{proof}


\subsubsection{Structure of the essential spectrum} \label{subsub:proofmain1} 

In this paragraph we prove Theorem \ref{Main1}. To this aim we will need the Meromophic Fredholm Theorem, which we now state (for more details, see \cite[Theorem 6.1]{borth}).

We start by defining what is a finitely meromorphic family of bounded operators.
\begin{Def} \label{def:fmf}
Let $D$ be an open connected subset of $\mathbb{C}$.
A family of bounded operators $T(s)$ on a Hilbert space $X$, parametrized by $s \in D \subset \mathbb{C}$,
is finitely meromorphic if for each point $a \in D$ we have a Laurent series representation
\begin{equation} \label{Fred:laurent:series}
T(s)= \sum_{k=-m}^\infty (s-a)^k T_k
\end{equation} 
converging in the operator topology in some neighborhood of $a$, where $m \in \mathbb{N}\cup \{0\}$ and the coefficients $T_k$ are finite-rank operators for $k<0$.
\end{Def}

\begin{Rem} \label{rem:discretepoles}
The definition above in particular implies that the poles of the family $T(s)$, i.e. the points $s \in D$ for which the Laurent series \eqref{Fred:laurent:series} cannot be chosen with $m=0$, form a discrete subset of $D$.
\end{Rem}

The Meromorphic Fredholm Theorem reads as follows.
\begin{Thm}[Meromorphic Fredholm Theorem] \label{mer:fredholm}
Let $X$ be a separable Hilbert space.
Suppose $T(s)$ is a finitely meromorphic family of compact operators on $X$, for $s \in D$. If $I - T(s)$ is invertible for at least one $s \in D$, then $(I-T(s))^{-1}$ exists as a finitely meromorphic family on $D$. 
\end{Thm}

In order to prove Theorem \ref{Main1} we also need the following lemma whose proof is  postponed to \S \ref{subsub:compactness}.
\begin{Lem} \label{lemma:KK'compact}
Let assumptions \eqref{relaxed:hyp} hold. 
\begin{enumerate}[label=(\roman*)]
\item \label{itm:KK'compact1} Fix $\lambda\in \rho(\tilde{\mathcal{A}}_\beta)$ and  $\lambda_0 \in \rho(\tilde{\mathcal{A}}_\beta) \cap \rho(\tilde{\mathcal{A}})$.  
Then the operator $\mathcal{K}_\beta:\bL^2_0(\Omega_0) \to \bL^2_0(\Omega_0)$ defined in \eqref{operK}  is compact.
\item \label{itm:KK'compact2}Fix $\lambda\in \rho(\tilde{\mathcal{A}})$ and  $\lambda_0 \in \rho(\tilde{\mathcal{A}}_\beta) \cap \rho(\tilde{\mathcal{A}})$.  
Then the operator $\mathcal{K}:\bL^2_0(\Omega_0) \to \bL^2_0(\Omega_0)$ defined in \eqref{operK'}  is compact.
\end{enumerate}
\end{Lem}
\begin{proof}[Proof of Theorem \ref{Main1}]
Let  \eqref{relaxed:hyp} holds.
By Lemma \ref{lemma:KK'compact}  the operators $\mathcal{K}_\beta(\lambda)$ and ${\mathcal{K}}(\lambda)$ are compact for any $\lambda \in \mathbb{C}$ for which they are  defined. 

Recall that by Proposition \ref{ess:spectr:beta} the  spectrum of $\mathcal{A}_\beta$ is given by
\begin{equation}\label{eq.spec}
\sigma(\mathcal{A}_\beta) = \sigma_{\rm ess}(\mathcal{A}_\beta)=(-\infty, -a_\beta] \cup 	\{0\} \cup [a_\beta,+\infty) .
\end{equation}
Therefore what the theorem asserts is that the essential spectra of $\mathcal{A}$ and $A_\beta$ coincide.
In order to prove that, since $\mathcal{A}, \mathcal{A}_\beta$ are unitarily equivalent to $\tilde{\mathcal{A}}, \tilde{\mathcal{A}}_\beta$,  respectively, it is equivalent to prove that 
\begin{equation*}
\sigma_{\rm ess} (\tilde{\mathcal{A}}) = \sigma_{\rm ess}(\mathcal{A}_\beta).
\end{equation*}

\bigskip

\noindent
{\bf  First inclusion:} we will first show that $\sigma_{\rm ess} (\tilde{\mathcal{A}}) \subset \sigma_{\rm ess}(\mathcal{A}_\beta)$.

\begin{figure}[h!!]
\centering 
\includegraphics[width=0.75\textwidth]{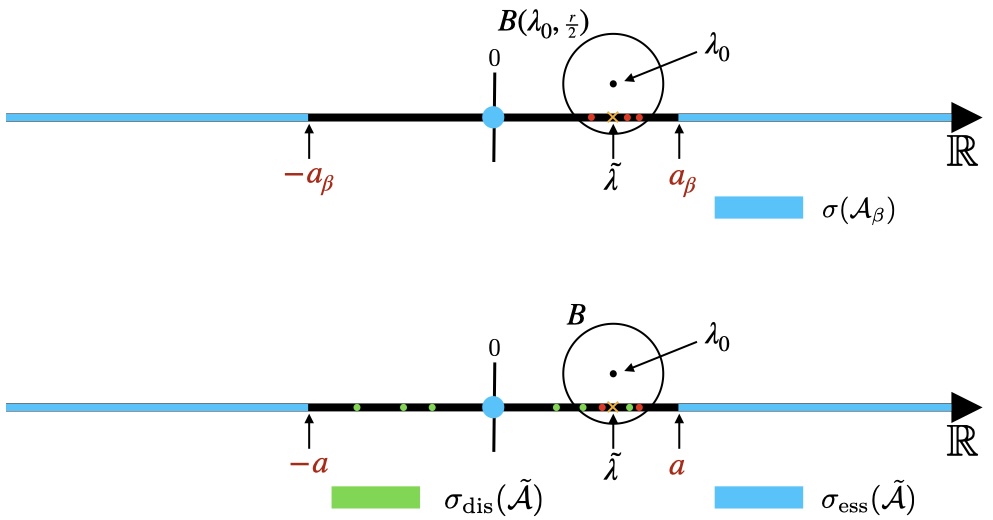}
\caption{In the first panel, the red dots indicate possible poles of the meromorphic function $\lambda \mapsto \left(I+\mathcal{K}_\beta(\cdot)\right)^{-1}$, i.e. points of the discrete set $S$ (cf. Point \ref{itm:listitmmerom1} below). Here the Analytic Fredholm Theorem is sufficient since $\lambda \mapsto \check{\tilde{R}}_\beta(\lambda)$ is analytic in $B(\lambda_0,\frac{r}{2})$. 
In the second panel, besides the red dots, the green dots indicate poles of the resolvent of $\tilde{\mathcal{A}}$ corresponding to potential discrete eigenvalues. In this case we need to use the Meromorphic Fredholm Theorem since $\lambda \mapsto \check{\tilde{R}}(\lambda)$ (and thus $\lambda \mapsto I+\mathcal{K}(\lambda)$) is not necessarily  analytic  in the ball $B$
.}
\label{fig-proof-fredholm}
\end{figure}
Since $\sigma_{\rm ess}(\mathcal{A}_\beta) = \sigma(\mathcal{A}_\beta)$, proving that $\sigma_{\rm ess} (\tilde{\mathcal{A}}) \subset \sigma_{\rm ess}(\mathcal{A}_\beta)$ is equivalent to prove that $\sigma_{\rm ess} (\tilde{\mathcal{A}}) \cap \rho(\mathcal{A}_\beta)= \emptyset$.
Moreover, since $\tilde{\mathcal{A}}$ is self-adjoint and thus $\sigma(\tilde{\mathcal{A}}) \cap \rho(\mathcal{A}_\beta) \subset \bbR \cap \rho(\mathcal{A}_\beta)$ and  by \eqref{eq.spec}, one has $ \bbR \cap \rho(\mathcal{A}_\beta)=(-a_\beta, a_\beta) \setminus \{0\} $. Thus, it is equivalent to prove that
\begin{equation} \label{empty:condition}
\sigma_{\rm ess}(\tilde{\mathcal{A}}) \cap \left( (-a_\beta, a_\beta) \setminus \{0\} \right) =\emptyset.
\end{equation}
Let $\tilde{\lambda}\in \left( (-a_\beta, a_\beta) \setminus \{0\} \right)$. Consider $r = \mathrm{dist}(\tilde{\lambda},\sigma(\mathcal{A}_\beta)) = \min \bigg\{|\tilde{\lambda}|, |\tilde{\lambda} - a_\beta|,|\tilde{\lambda}+ a_\beta|\bigg\}$.
Let $\lambda_0= \tilde{\lambda}+i \frac{r}{4}$ and note that  $\lambda_0 \in \mathbb{C}\setminus \bbR \subset \rho(\mathcal{A}_\beta) \cap \rho(\tilde{\mathcal{A}})$.
Let $B(\lambda_0,\frac{r}{2})$ be the open ball centered in $\lambda_0$ of radius $r/2$. In particular one has $\tilde{\lambda}\in B(\lambda_0,\frac{r}{2})$ and  $B(\lambda_0,\frac{r}{2}) \subset \rho(\mathcal{A}_\beta)$ (see the first  panel of Figure \ref{fig-proof-fredholm}).
\\[4pt]

Consider the map 
\begin{equation*}
\left\{ 
\begin{array}{lcl}
\rho(\mathcal{A}_\beta) & \to & \mathcal{B}(\bL^2_0(\Omega_0))\\
\lambda &\mapsto & \mathcal{K}_\beta(\lambda).
\end{array}
 \right.
\end{equation*}
It is clear  from \eqref{operK} that this  map is analytic since $\lambda \mapsto \check{\tilde{R}}_\beta(\lambda)$ is analytic on $\rho(\mathcal{A}_\beta)$.  Moreover, by Lemma \ref{lemma:KK'compact}, it takes values in the set of compact operators in $\mathcal{B}(\bL^2_0(\Omega_0))$.
Furthermore, observe that $\mathcal{K}_\beta(\lambda_0)=0$, thus $I+\mathcal{K}_\beta(\lambda_0)=I$ is invertible.
Since $B(\lambda_0, r/2)\subset \rho(\mathcal{A}_\beta)$ is an open connected subset of $\mathbb{C}$, we can apply Theorem \ref{mer:fredholm} to the finitely meromorphic family $\left(-\mathcal{K}_\beta(\lambda)\right)_{\lambda \in  B(\lambda_0, r/2)}$ (which here is even analytic in this set), yielding that $(I+\mathcal{K}_\beta(\lambda))^{-1}$ exists as a finitely meromorphic family on $B(\lambda_0, r/2)$.
In particular by Remark \ref{rem:discretepoles} there exists a discrete set  $S$  in $B(\lambda_0, r/2)$ such that
\begin{enumerate}[label=(\alph*)]
\item\label{itm:listitmmerom1} $( I+\mathcal{K}_\beta(\lambda))^{-1}$ exists as a meromorphic function on $B(\lambda_0, r/2)$, whose poles coincide with the set $S$.
\item\label{itm:listitmmerom2} If $\lambda \in S$,  the residue of  $(I+\mathcal{K}_\beta(\lambda))^{-1}$ at $\lambda$ is a finite-rank operator.
\end{enumerate}

Therefore, by Proposition \ref{lemmaK}, $\lambda \mapsto \calQ_\beta(\lambda)^{-1}$ exists as a meromorphic function on  $B(\lambda_0, r/2)$ with same poles in $S$,  and in particular $\lambda \mapsto  \calQ_\beta(\lambda)^{-1}$ is analytic on $B(\lambda_0, r/2) \setminus S$. 
As $\lambda \mapsto \tilde{R}_{\beta}(\lambda)$ is analytic on $B(\lambda_0, r/2) \setminus S\subset  \rho(\mathcal{A}_\beta)$, the function  
$$
\lambda \mapsto \bbV \mathcal{I}_\bbG^{-1}  \mathcal{I}_\beta \bbV_\beta^*  \tilde{R}_\beta(\lambda) \calQ_\beta(\lambda)^{-1} \mathcal{I}_\beta \bbV_\beta^* \bbV \mathcal{I}_\bbG^{-1} 
$$
is analytic on $B(\lambda_0, r/2) \setminus S$. Moreover,  by \eqref{1st:res:eq}, it coincides with  the analytic function $\lambda \mapsto \tilde{R}(\lambda)$ on the open set $\bbC^+\cap (B(\lambda_0, r/2) \setminus S)$ (where $\bbC^+=\{ z\in \bbC : \operatorname{Im}z>0 \}$ denotes here the upper-half plane). Thus, one has constructed an analytic continuation of the resolvent $\tilde{R}(\cdot)$ on $B(\lambda_0, r/2) \setminus S$. As for $\lambda\in \rho(\tilde{\mathcal{A}})$,  $\|\tilde{R}(\lambda)\|= \mathrm{dist}(\lambda, \sigma(\tilde{\mathcal{A}}))^{-1}$, $\tilde{R}(\lambda)$ blows up when   $\lambda$ approaches $\sigma(\tilde{\mathcal{A}})$. Thus the resolvent $\tilde{R}(\cdot)$ cannot be analytically extended on any open set which intersects $\sigma(\tilde{\mathcal{A}})$. Hence, it follows that 
$B(\lambda_0, r/2) \setminus S\subset \rho(\tilde{\mathcal{A}})$ and that
\begin{equation} \label{res:eq:S}
\tilde{R}(\lambda)= \bbV \mathcal{I}_\bbG^{-1}  \mathcal{I}_\beta \bbV_\beta^*  \tilde{R}_\beta(\lambda) \calQ_\beta(\lambda)^{-1} \mathcal{I}_\beta \bbV_\beta^* \bbV \mathcal{I}_\bbG^{-1} \quad \text{for all }\lambda \in B(\lambda_0, r/2)\setminus S.
\end{equation}

Note that by \eqref{res:eq:S}, the poles of the function  $\lambda \mapsto \tilde{R}(\lambda) $ on  $B(\lambda_0,\frac{r}2)$ are exactly the same  as the one  of the function $\lambda  \mapsto \calQ_\beta(\lambda)^{-1}$ on  $B(\lambda_0,\frac{r}2)$ (with the same order). These poles are precisely the elements of the set $S$ because of Point  \ref{itm:listitmmerom1} and the Birman-Schwinger principle in Point \ref{itm:BS1}~Proposition \ref{lemmaK}. Now, as $\tilde{\mathcal{A}}$ is self-adjoint, these poles should lie in $\mathbb{R}$ and thus $S \subset \R$.

It means that if $\tilde{\lambda}\notin S$, then $\tilde{\lambda}\in \rho(\tilde{\mathcal{A}})$, and thus $\tilde{\lambda}\notin \sigma_{\mathrm{ess}}(\tilde{\mathcal{A}})$ and condition \eqref{empty:condition} is satisfied. We are left with dealing with the case $\tilde{\lambda} \in S$, in this case, $\tilde{\lambda}\in \sigma(\tilde{\mathcal{A}})$ and is isolated.
Therefore $\tilde{\lambda}$ is necessarily an eigenvalue of $\tilde{\mathcal{A}}$ since $\tilde{\mathcal{A}}$ is self-adjoint.
Moreover, as it is isolated in $\sigma(\tilde{\mathcal{A}})$, its spectral projector is given by
$$
\bbE(\{\tilde{\lambda}\})=-\frac{1}{2  \pi\rmi}\int_{\mathcal{C}} \tilde{R}(\zeta) \mathrm{d}\zeta=\operatorname{Res}(\tilde{R}(\cdot) \ ,\tilde{\lambda})
$$
where $\mathcal{C}$ is small circle centered at $\tilde{\lambda}$ contained in $B(\lambda_0, r/2)$ positively oriented that does not contain or enclose any  point of $S$ other than $\tilde{\lambda}$. Thus, we are left with understanding what is the residue $\operatorname{Res}(\tilde{R}(\cdot) \ ,\tilde{\lambda})$. Note that for $\xi$ sufficiently close to $\tilde{\lambda}$, there holds
\[
	\tilde{R}_\beta(\xi) = \tilde{R}_\beta(\tilde{\lambda}) + (\xi - \tilde{\lambda)}F(\xi),\quad \calQ_\beta(\xi)^{-1} = \frac{\calQ_{\beta,-1}}{\xi - \tilde{\lambda}} + G(\xi),
\]
where $F,G$ are analytic functions in a neighbourhood of $\tilde{\lambda}$ and with bounded operator values. The fact that the pole of $(\xi \mapsto \calQ_\beta(\xi)^{-1})$ is of order one is a consequence of \eqref{res:eq:S} and the fact that the poles of $\tilde{R}(\lambda)$ are necessarily of order $1$ because $\tilde{\mathcal{A}}$ is a self-adjoint operator (see for instance \cite[Exercise 4.(b) \S 3.3]{Si15}). Here $\calQ_{\beta,-1}$ is the residue of $(\xi \mapsto -\calQ_\beta(\xi)^{-1})$ and is of finite rank by Point \ref{itm:listitmmerom2} and the Birman-Schwinger principle in Point \ref{itm:BS1}~Proposition \ref{lemmaK}. Therefore, for $\xi$ sufficiently close to $\tilde{\lambda}$, from \eqref{res:eq:S} one has
\begin{align*}
	\tilde{R}(\xi) &=\bbV \mathcal{I}_\bbG^{-1}  \mathcal{I}_\beta \bbV_\beta^* \left(\tilde{R}_\beta(\tilde{\lambda}) + (\xi - \tilde{\lambda)}F(\xi)\right)\left(\frac{\calQ_{\beta,-1}}{\xi - \tilde{\lambda}} + G(\xi)\right)\mathcal{I}_\beta \bbV_\beta^* \bbV \mathcal{I}_\bbG^{-1}\\
	& = \frac{1}{\xi - \tilde{\lambda}}\bbV \mathcal{I}_\bbG^{-1}  \mathcal{I}_\beta \bbV_\beta^*\tilde{R}_\beta(\tilde{\lambda})\calQ_{\beta,-1}\mathcal{I}_\beta \bbV_\beta^* \bbV \mathcal{I}_\bbG^{-1} + H(\xi),
\end{align*}
for some analytic map $H$. Hence, one obtains
\[
	\bbE(\{\tilde{\lambda}\}) = \operatorname{Res}(\tilde{R}(\cdot) \ ,\tilde{\lambda}) = \bbV \mathcal{I}_\bbG^{-1}  \mathcal{I}_\beta \bbV_\beta^*\tilde{R}_\beta(\tilde{\lambda})\calQ_{\beta,-1}\mathcal{I}_\beta \bbV_\beta^* \bbV \mathcal{I}_\bbG^{-1}.
\]
Moreover  $\bbE(\{\tilde{\lambda}\})$ is of finite rank, since  $\calQ_{\beta,-1}$ is finite rank and $\tilde{R}_\beta(\tilde{\lambda}) \in \mathcal{B}(\bL_0^2(\Omega_0))$ because $\tilde{\lambda}\in \rho(\mathcal{A}_\beta)$. 
Hence $\tilde{\lambda}$ is an isolated eigenvalue of finite multiplicity, and belongs to the discrete spectrum of $\tilde{\mathcal{A}}$. Therefore condition \eqref{empty:condition} is satisfied, and one concludes that $ \sigma_{\mathrm{ess}}(\tilde{\mathcal{A}}) \cap \rho(\mathcal{A}_\beta)=\emptyset$.\\[4pt]
 
 \noindent
{\bf  Second inclusion:} we will now show that $\sigma_{\rm ess}(\mathcal{A}_\beta) \subset \sigma_{\rm ess} (\tilde{\mathcal{A}})$.
Equivalently, we show that $\sigma_{\rm ess}(\mathcal{A}_\beta) \cap \left( \rho(\tilde{\mathcal{A}}) \cup \sigma_{\rm disc}(\tilde{\mathcal{A}})  \right)= \sigma(\mathcal{A}_\beta) \cap \left( \rho(\tilde{\mathcal{A}}) \cup \sigma_{\rm disc}(\tilde{\mathcal{A}})  \right)= \emptyset$.

Take $\tilde{\lambda} \in \rho(\tilde{\mathcal{A}}) \cup \sigma_{\rm disc} (\tilde{\mathcal{A}})$ and observe that $\tilde{\lambda} \neq 0$ because $0 \in \sigma_{\rm ess}(\tilde{\mathcal{A}})$. 
Indeed $\tilde{\mathcal{A}}= \bbV \hat{\mathcal{A}} \bbV^*$ and $\hat{\mathcal{A}}$ are unitarily equivalent (see Figure \ref{fig:diag}). Thus $\ker(\tilde{\mathcal{A}})= \bbV \ker(\hat{\mathcal{A}})$ where the infinite dimensional space $\ker(\hat{\mathcal{A}})$ is given by Proposition \ref{prop:kerneldescription}.
We can always assume that $\tilde{\lambda}\in \mathbb{R}$, otherwise $\tilde{\lambda}\in \rho(\mathcal{A}_\beta)$.  Moreover, there exists a ball $B\subset \mathbb{C}$ containing $\tilde{\lambda}$ such that $B \setminus \{\tilde{\lambda}\} \subset \rho(\tilde{\mathcal{A}})$ (see the second  panel of Figure \ref{fig-proof-fredholm}). Let $\lambda_0 \in \left(B \setminus \{\tilde{\lambda}\}\right) \cap\left(\rho(\tilde{\mathcal{A}}) \cap \rho(\mathcal{A}_\beta)\right)$ (such a $\lambda_0$ always exists because it can be taken complex valued). We check that $\left(\lambda \in B \mapsto -\mathcal{K}(\lambda) \in \mathcal{B}(\bL_0^2(\Omega_0))\right)$ is a finitely meromorphic family of compact operators.

By construction, the only possible pole in $B$ of $\left(\xi\in B \mapsto \check{\tilde{R}}(\xi)\right)$ is $\tilde{\lambda}$ and as $\mathcal{\tilde{A}}$ is a self-adjoint operator it is of order one. In particular, for $\xi \in B \setminus \{\tilde{\lambda}\}$, one has
\[
	\check{\tilde{R}}(\xi) = -\frac{\operatorname{Res}(\check{\tilde{R}}(\cdot) \ ,\tilde{\lambda})}{\xi- \tilde{\lambda}} + F(\xi),
\]
where $F$ is analytic on $B$ with values in bounded operators. By \eqref{operK'}, one gets
\begin{align*}
	\mathcal{K}(\xi) &= (\xi-\lambda_0) {\calQ}(\lambda_0)^{-1} \left[ I - \mathcal{I}_\bbG \bbV^* (\bbV_\beta \mathcal{I}_\beta^{-1})^2 \mathcal{I}_\bbG \bbV^* \right] \left(-\frac{\operatorname{Res}(\check{\tilde{R}}(\cdot) \ ,\tilde{\lambda})}{\xi- \tilde{\lambda}} + F(\xi)\right) \tilde{\mathcal{A}} \tilde{R}(\lambda_0)\\
	& = -\frac{\xi-\lambda_0}{\xi - \tilde{\lambda}} {\calQ}(\lambda_0)^{-1} \left[ I - \mathcal{I}_\bbG \bbV^* (\bbV_\beta \mathcal{I}_\beta^{-1})^2 \mathcal{I}_\bbG \bbV^* \right]\operatorname{Res}(\check{\tilde{R}}(\cdot) \ ,\tilde{\lambda})\tilde{\mathcal{A}} \tilde{R}(\lambda_0) + H(\xi)\\
	& = -\frac{\tilde{\lambda} - \lambda_0}{\xi - \tilde{\lambda}} {\calQ}(\lambda_0)^{-1} \left[ I - \mathcal{I}_\bbG \bbV^* (\bbV_\beta \mathcal{I}_\beta^{-1})^2 \mathcal{I}_\bbG \bbV^* \right]\operatorname{Res}(\check{\tilde{R}}(\cdot) \ ,\tilde{\lambda})\tilde{\mathcal{A}} \tilde{R}(\lambda_0)   \\& \qquad- {\calQ}(\lambda_0)^{-1} \left[ I - \mathcal{I}_\bbG \bbV^* (\bbV_\beta \mathcal{I}_\beta^{-1})^2 \mathcal{I}_\bbG \bbV^* \right]\operatorname{Res}(\check{\tilde{R}}(\cdot) \ ,\tilde{\lambda})\tilde{\mathcal{A}} \tilde{R}(\lambda_0) + H(\xi)
\end{align*}
where $H$ is an analytic map. Thus, if $\tilde{\lambda} \in \rho(\tilde{\mathcal{A}})$, then  $\operatorname{Res}(\check{\tilde{R}}(\cdot) \ ,\tilde{\lambda}) = 0$. If $\tilde{\lambda}\in \sigma_{\rm dis}(\tilde{\mathcal{A}})$ then $\operatorname{Res}(\check{\tilde{R}}(\cdot) \ ,\tilde{\lambda})$ is the spectral projector on the finite dimensional eigenspace associated with $\tilde{\lambda}$ thus, in both cases, $(\mathcal{K}(\lambda))_{\lambda \in B}$ is a finitely meromorphic family of compact operators. Now, remark that by construction $I+{\mathcal{K}}(\lambda_0)=I$ is invertible and thus one can apply Theorem \ref{mer:fredholm} to this finitely meromorphic family. This implies that  $I+{\mathcal{K}}(\lambda)$ is invertible for $\lambda \in B\setminus (S\cup\{\tilde{\lambda}\})$, where $S$ is a discrete subset of $B$. Therefore, by Point \ref{itm:BS2}~Proposition \ref{lemmaK} ${\calQ}(\lambda)$ is invertible for $\lambda \in B\setminus (S\cup\{\tilde{\lambda}\})$. 
Consequently, using \eqref{2nd:res:eq} and the same analytic continuation argument used in the proof of the first inclusion, one can show that $B \setminus (S\cup\{\tilde{\lambda}\})\subset \rho(\mathcal{A}_\beta)$. By Proposition \ref{ess:spectr:beta} we have $\sigma(\mathcal{A}_\beta) = (-\infty,-a_\beta]\cup\{0\}\cup[a_\beta,+\infty)$ and since $S\cup\{\tilde{\lambda}\}$ is a discrete set which does not contain $0$ then $B \subset \rho(\mathcal{A}_\beta)$. In particular $\tilde{\lambda} \in \rho(\mathcal{A_\beta})$ which concludes the proof.
\end{proof}

\subsubsection{Compactness} \label{subsub:compactness}

The aim of this section is to prove Lemma \ref{lemma:KK'compact}. The proof is divided into two steps. In the first one, we assume that the geometrical perturbation is compactly supported (see Lemma \ref{lemma:Kcompact}). The second step is by considering the general case as a limit of compact geometric perturbations.\\

For further uses, remark that in full generality, the metric tensor $\bbG$ defined in \eqref{Gmatrix} writes for $(s,y_2,y_3) \in \Omega_0$ as
\begin{equation} \label{matr:G}
\bbG (s,y_2,y_3)= \frac{1}{1-\mathbf{k^\theta} \cdot \by}
\begin{pmatrix}
(1-\mathbf{k^\theta}(s) \cdot \by)^2 + \theta'(s)^2 |\by|^2\ \  & y_3 \theta'(s)\  & -y_2 \theta'(s) \\
y_3 \theta'(s) & 1 & 0 \\
-y_2 \theta'(s) & 0 & 1
\end{pmatrix}.
\end{equation}

\begin{Lem} \label{lemma:Kcompact}
Assume $\kappa$ and $\theta'-\beta$ have compact support in $\R$. Then Lemma  \ref{lemma:KK'compact} holds.
\end{Lem}

\begin{proof}[Proof of Lemma \ref{lemma:Kcompact}]
We prove only Point \ref{itm:KK'compact1}~Lemma \ref{lemma:KK'compact} under these stronger assumptions. The proof of Point \ref{itm:KK'compact2}~Lemma \ref{lemma:KK'compact} is the same exchanging the roles $\bbG_\beta$ and $\bbG$.

Start by noting that by assumptions there exists $s_0>0$ such that for all $|s| \geq s_0$ and all $\by \in \omega$ one has $\bbG(s,\by) =\bbG_\beta(s,\by)$ where $\bbG_\beta$ is defined in \eqref{allBeta}.

The operators $\check{\tilde{R}}_\beta(\lambda)$ is  regularizing operator whereas  the operator $ I - \mathcal{I}_\beta \bbV_\beta^* (\bbV \mathcal{I}_\bbG^{-1})^2 \mathcal{I}_\beta \bbV_\beta^*$ is a  localizing operator. Moreover,
we  prove  here that their composition  $[ I - \mathcal{I}_\beta \bbV_\beta^* (\bbV \mathcal{I}_\bbG^{-1})^2 \mathcal{I}_\beta \bbV_\beta^*] \check{\tilde{R}}_\beta(\lambda)$ is compact. 
Then, the operator $\mathcal{K}_\beta(\lambda)$  (defined by \eqref{operK}) will be compact due to the fact that $\calQ_\beta(\lambda_0)^{-1}$ and $\tilde{\mathcal{A}}_\beta \tilde{R}_\beta(\lambda_0)$ are bounded operators.

Denote with $\bx = (s,y_2,y_3) \in \Omega_0$ the generic point of the straight waveguide $\Omega_0$. For $s>0$ we define the finite cylinder $C_{s}:= (-s,s) \times \omega$. Observe that for any $\bU \in \bL^2_0(\Omega_0)$ 
\begin{equation} \label{compsupportB}
 \left[ I - \mathcal{I}_\beta \bbV_\beta^* (\bbV \mathcal{I}_\bbG^{-1})^2 \mathcal{I}_\beta \bbV_\beta^*  \right] \bU (\bx) = 0 \text{ for }\bx \in \Omega_0 \setminus C_{s_0}.
\end{equation}
Consider $\chi:\mathbb{R} \to [0,1]$ a smooth cut-off function such that $\chi \equiv 1$ on $[-s_0,s_0]$ and $\chi \equiv 0$ outside of $[-2s_0,2s_0]$. Then, thanks to \eqref{compsupportB}, we have that
\begin{equation} \label{ident:oper}
\left[ I - \mathcal{I}_\beta \bbV_\beta^* (\bbV \mathcal{I}_\bbG^{-1})^2 \mathcal{I}_\beta \bbV_\beta^*  \right]= \left[ I - \mathcal{I}_\beta \bbV_\beta^* (\bbV \mathcal{I}_\bbG^{-1})^2 \mathcal{I}_\beta \bbV_\beta^*  \right]\mathbb{M}
\end{equation}
where
\begin{equation} \label{def:M}
\mathbb{M}: \bL^2_0(\Omega_0) \to \bL^2_0(\Omega_0), \quad \mathbb{M}\bU = \chi \bU.
\end{equation}
Recall that $\tilde{\mathcal{A}}_\beta = \bbV_\beta \hat{\mathcal{A}}_\beta \bbV_\beta^*$ where $\hat{\mathcal{A}}_\beta$ is the operator unitarily equivalent to $\mathcal{A}_\beta$ obtained {\it via}  Proposition \ref{prop:unitequivpiola}.

Hence  by functional calculus the reduced resolvent of $\tilde{\mathcal{A}}_\beta$ introduced in \eqref{red:res:b} can be written as follows
$$ \check{\tilde{R}}_\beta (\lambda)= \bbV_\beta \check{\hat{R}}_\beta(\lambda)  \bbV_\beta^*.$$
Here $\check{\hat{R}}_\beta(\lambda) := \left[(\hat{\mathcal{A}}_\beta-\lambda)^{-1} + \frac{1}{\lambda} \hat{P}_{0,\beta}\right]: \bL^2_\beta(\Omega_0) \to \bL^2_\beta(\Omega_0)$ denotes the reduced resolvent at $\lambda$ of the operator $\hat{\mathcal{A}}_\beta$, where $\hat{P}_{0,\beta}:\bL^2_\beta(\Omega_0) \to \bL^2_\beta(\Omega_0)$ is the orthogonal projection (in $\bL^2_\beta(\Omega_0)$) onto the kernel of $\hat{\mathcal{A}}_\beta$.

Furthermore, it is immediate to see that $\bbM \bbV_\beta = \bbV_\beta \bbM$. Therefore 
$$\bbM \check{\tilde{R}}_\beta(\lambda)  =  \bbV_\beta  \bbM  \check{\hat{R}}_\beta(\lambda)  \bbV_\beta^*.$$
We will show that $\bbM \check{\hat{R}}_\beta(\lambda)$ is a compact operator from  $\bL^2_\beta(\Omega_0)$ to itself, and the proof of the lemma will follow.

Observe that $\check{\hat{R}}_\beta(\lambda) (\bL^2_\beta(\Omega_0)) = D(\hat{\mathcal{A}}_\beta) \cap \mathrm{ker}(\hat{\mathcal{A}}_\beta)^\perp$, i.e. the range of $\check{\hat{R}}_\beta(\lambda)$ is   the following set (cf. Proposition \ref{prop:kerneldescription} ): 
\begin{eqnarray} \label{compact:supp}
&& \Big\{\bU = (
\bE,
\bH) \in \bL_\beta^2(\Omega_0) : \bE \in H_0(\curlvec, \Omega_0) \cap H(\Div \bbG_\beta^{-1},\Omega_0),  \nonumber  \\
&& \qquad \bH \in H(\curlvec, \Omega_0) \cap H_0(\Div \bbG_\beta^{-1},\Omega_0)  \quad\Div(\bbG_\beta^{-1} \bE) =0= \Div(\bbG_\beta^{-1} \bH)  \Big\}.
\end{eqnarray}
Take $\bV \in \bL_\beta^2(\Omega_0)$. 
As $\chi$ is smooth and its  support in contained in $[-2s_0,2s_0]$, one can check that   $\bbM  \check{\hat{R}}_\beta(\lambda)\bV$  belongs to
\begin{eqnarray} \label{compact:emb:space}
&& \Big\{\bU = (
\bE,
\bH)\in \bL^2_\beta(C_{3s_0}) : \ \bE \in H_0(\curlvec, C_{3s_0}) \cap H(\Div  \bbG_\beta^{-1}, C_{3s_0}), \nonumber \\
& &\qquad  \bH \in H(\curlvec, C_{3s_0}) \cap H_0(\Div  \bbG_\beta^{-1}, C_{3s_0}) \bigg\}.
\end{eqnarray} 
Since $C_{3s_0}$ is bounded and Lipschitz and $\bbG_\beta^{-1}\in L^\infty(C_{3s_0})^{3 \times 3}$ is a symmetric uniformly positive definite matrix-field, the space introduced in \eqref{compact:emb:space} is compactly embeddedd into $\bL^2(C_{3s_0})^2$ (cf. e.g. \cite{weber} and \cite[Thm. 4.7]{BaPaSch} for a more general result). 

Hence $\bbM\ \check{\hat{R}}_\beta(\lambda)$ is compact and, using \eqref{ident:oper}, we conclude that  the operator $\mathcal{K}_\beta(\lambda)$ is compact.
\end{proof}

We now prove Lemma \ref{lemma:KK'compact} in full generality, i.e. under assumptions \eqref{relaxed:hyp}. In order to do so, we introduce and prove some useful lemmas.
\begin{Rem} \label{rem:compsupp}
Recall that the functions $k_1,k_2$ introduced in \eqref{frame:evolution} and the derivative $\theta'$ of $\theta$ (cf. \eqref{eqn:defetheta}) are continuous functions.  Moreover $|k_1(s)|,|k_2(s)| \leq |\kappa(s)|$ for any $s \in \R$, therefore assumptions \eqref{relaxed:hyp} imply that:
\begin{itemize}
\item there exists sequences $(k_{1,n})_{n \in \mathbb{N}}$ and $(k_{2,n})_{n \in \mathbb{N}}$ of continuous compactly supported functions of $\R$ such hat $\displaystyle \lim_{n \to +\infty}  k_{1,n}=k_1$ and $\displaystyle\lim_{n \to +\infty} k_{2,n}=k_2$ in $L^\infty(\R)$, respectively; 
\item there exists a sequence $(f_n)_{n \in \mathbb{N}}$ of continuous compactly supported functions of $\R$ such that $\displaystyle \lim_{n \to +\infty} f_n=\theta'(s) -\beta$ in $L^\infty(\R)$.
\end{itemize}
\end{Rem}

Define
\begin{equation*}
\mathbf{k}_n^\theta (s):= 
\begin{pmatrix}
\cos(\theta(s)) & -\sin(\theta(s))\\
\sin(\theta(s)) & \cos(\theta(s))
\end{pmatrix}\begin{pmatrix}k_{1,n}(s)\\ k_{2,n}(s)\end{pmatrix}
\end{equation*}
and observe that 
\begin{equation} \label{k:thetan:conv}
\lim_{n \to \infty} \| \mathbf{k}_n^\theta - \mathbf{k^\theta} \|_{L^\infty(\R)}=0.
\end{equation}
Define
\begin{equation} \label{bbGn:defi}
\bbG_n (s,\by) := \frac{1}{1-\mathbf{k}_n^\theta \cdot \by}
\begin{pmatrix}
(1-\mathbf{k}_n^\theta \cdot \by)^2+ (f_n+\beta)^2 |\by|^2 & y_3(f_n+\beta) & -y_2 (f_n+\beta)\\
y_3(f_n+\beta) & 1 & 0 \\
-y_2 (f_n+\beta) & 0 & 1
\end{pmatrix}.
\end{equation}

\begin{Lem} \label{GnG:conv:lemma}
We have that
\begin{equation*}
\lim_{n \to \infty}  \| \, \bbG_n - \bbG \, \|^2_{L^\infty(\Omega_0,\R^{3\times 3})} =0,
\end{equation*}
where $\bbG$ is defined in \eqref{Gmatrix}.
\end{Lem}
\begin{proof}[Proof of Lemma \ref{GnG:conv:lemma}]
Recall that $b:= \sup_{\by \in \omega}|\by|$ and by assumption \eqref{assumpt:a:kappa} that $b \|\kappa\|_{L^\infty(\R)} <1$. 
Moreover by \eqref{unif:bounds:ak}
\begin{equation*}
0<1-b \|\kappa\|_{L^\infty(\R)} \leq 1-\mathbf{k^\theta}(s) \cdot \by \leq 1+b \|\kappa\|_{L^\infty(\R)} \qquad \text{for all }(s,\by) \in \Omega_0.
\end{equation*}
Thus, by \eqref{k:thetan:conv} we also have that, for sufficiently large $n \in \mathbb{N}$, there exist $0<c \leq C$, independent of $(s,\by) \in \Omega_0$ such that
\begin{equation} \label{estimate:k:thetan}
0<c \leq 1-\mathbf{k}^\theta_n(s) \cdot \by \leq C \qquad \text{for all }(s,\by) \in \Omega_0.
\end{equation}

Recall that $\bbG$ can be written as in \eqref{matr:G}.
Since all matrix-norms are equivalent, it is convenient to use the max norm in $\R^{3\times 3}$. Thus $\| \bbG_n-\bbG \, \|_{L^\infty(\Omega_0,\R^{3\times 3})}$ can be estimated by the maximum among the following quantities:
\begin{enumerate}[label=(\roman*)]
\item $\| (1-\mathbf{k}^\theta_n \cdot \by) + |\by|^2 (f_n+\beta)^2 (1-\mathbf{k}^\theta_n \cdot \by)^{-1} - (1-\mathbf{k^\theta} \cdot \by) - |\by|^2 (\theta')^2 (1-\mathbf{k^\theta} \cdot \by)^{-1} \|_{L^\infty(\Omega_0)}$;
\item $ \| \by\|_{L^\infty(\omega)} \| (f_n+\beta) (1-\mathbf{k}^\theta_n \cdot \by)^{-1} - \theta'(1-\mathbf{k^\theta} \cdot \by)^{-1} \|_{L^\infty(\Omega_0)}$;
\item $\| (1-\mathbf{k}^\theta_n \cdot \by)^{-1} - (1-\mathbf{k^\theta} \cdot \by)^{-1} \|_{L^\infty(\Omega_0)}$.
\end{enumerate}
These quantities are controlled by  $\| \mathbf{k^\theta}-\mathbf{k}^\theta_n\|_{L^\infty(\R)}$, and $\| f_n -(\theta'-\beta)\|_{L^\infty(\R)}$, at most multiplied by  constants independent of $n \in \mathbb{N}$, and since both quantities tend to zero as $n$ goes to $\infty$,  the proof is concluded.
\end{proof}

For every $n \in \mathbb{N}$ introduce the multiplication operators 
\begin{equation*}
T_n: \bL^2_0(\Omega_0) \to \bL^2_0(\Omega_0), \quad \begin{pmatrix} \bE \\ \bH \end{pmatrix} \mapsto \begin{pmatrix}  \bbG_n\bE \\  \bbG_n\bH \end{pmatrix},
\end{equation*}
where $\bbG_n$ is defined in \eqref{bbGn:defi}. For sufficiently large $n \in \mathbb{N}$, they are invertible and their  inverses are given by
\begin{equation*}
T_n^{-1}: \bL^2_0(\Omega_0) \to \bL^2_0(\Omega_0), \quad \begin{pmatrix} \bE \\ \bH \end{pmatrix} \mapsto \begin{pmatrix}  \bbG_n^{-1}\bE \\  \bbG_n^{-1}\bH \end{pmatrix}.
\end{equation*}

\begin{Rem} \label{rem:bdd:Tn-1}
Observe that $\bbG_n^{-1} = \mathrm{adj}(\bbG_n) (\operatorname{det}\bbG_n)^{-1}$, where $\mathrm{adj}(\bbG_n)$ denotes the adjugate matrix of $\bbG_n$.
On  one hand,  $\|\operatorname{det}\bbG_n \|_{L^\infty(\Omega_0)}  = \|1-\mathbf{k}^\theta_n \cdot \by \|_{L^\infty(\Omega_0)}  \geq c>0$ is bounded away from 0 for sufficiently large $n \in \mathbb{N}$ (cf. \eqref{estimate:k:thetan}).
On the other hand, $\mathrm{adj}(\bbG_n)$ has entries that are minors of the matrix $\bbG_n$. Since all entries of $\bbG_n$ are  bounded in $L^\infty(\Omega_0)$, we can conclude that there exists $C>0$  such that $\|\bbG_n^{-1} \|_{L^\infty(\Omega_0,\R^{3\times 3})} \leq C$ for sufficiently large $n \in \mathbb{N}$. Hence the operator norms $\|T_n^{-1}\|$ are bounded, i.e.
\begin{equation} \label{norm:Tn:inverse}
\|T_n^{-1}\| \leq C
\end{equation}
for sufficiently large $n \in \mathbb{N}$.
\end{Rem}

Without loss of generality, we can assume that $T_n^{-1}$ are defined and satisfy \eqref{norm:Tn:inverse} for all $n \in \mathbb{N}$.

\begin{Lem} \label{lemma:TnT:conv}
The operators $T_n$ converge to $(\mathcal{I}_\bbG \bbV^*)^2$, and their inverses $T_n^{-1}$ converge to $(\bbV \mathcal{I}_\bbG^{-1})^2$.
\end{Lem}
\begin{proof}[Proof of Lemma \ref{lemma:TnT:conv}]
Recall the definitions $\mathcal{I}_\bbG$ and  $\bbV$ in \eqref{IbbG:def} and \eqref{eqn:defbbv} respectively. Denote $T:=(\mathcal{I}_\bbG \bbV^*)^2$.
Then, one has  for  all $\bU=(\bE, \bH)^{\top}\in \bL^2_0(\Omega_0) $:
\begin{equation*}
\begin{split}
\big\| (T_n - T) \bU \big\|^2_{\bL^2_0(\Omega_0)} &= \|(\bbG_n - \bbG) \bE \|^2_{\bL^2_{\varepsilon_0}(\Omega_0)} +\|(\bbG_n - \bbG) \bH \|^2_{\bL^2_{\mu_0}(\Omega_0)} \\
&  \leq \| \bbG_n - \bbG \|^2_{L^\infty(\Omega_0,\R^{3\times 3})} \left( \|\bE\|_{\bL^2_{\varepsilon_0}(\Omega_0)}^2 +\|\bH\|_{\bL^2_{\mu_0}(\Omega_0)}^2 \right).
\end{split}
\end{equation*}
Hence by Lemma \ref{GnG:conv:lemma}
\begin{equation} \label{conv:TnT:eq}
\lim_{n \to \infty}\| T_n - T\| \leq \lim_{n \to \infty}\|\bbG_n - \bbG\|_{L^\infty(\Omega_0,\R^{3\times 3})} =0.
\end{equation}
Observe that
\begin{equation*}
T_n^{-1}-T^{-1} = T_n^{-1} (T-T_n) T^{-1}.
\end{equation*}
By Remark \ref{rem:bdd:Tn-1} the operator norms   $\|T_n^{-1}\|$ are bounded. Therefore from \eqref{conv:TnT:eq} we also have that $T_n^{-1}$ converges in operator norm to $T^{-1}=(\bbV \mathcal{I}_\bbG^{-1})^2$.
\end{proof}

\begin{Lem} \label{lemma:Tn:comp}
For all $n \in \mathbb{N}$, the operator
\begin{equation*}
\left[ I - \mathcal{I}_\beta\bbV_\beta^* T_n^{-1} \mathcal{I}_\beta \bbV_\beta^* \right] \check{\tilde{R}}_\beta: \bL^2_0(\Omega_0) \to \bL^2_0(\Omega_0)
\end{equation*}
is compact.
\end{Lem}
\begin{proof}[Proof of Lemma \ref{lemma:Tn:comp}]
Since by Remark \ref{rem:compsupp} the functions $\mathbf{k}^\theta_n$ and $f_n$ are compactly supported in $\R$, there exists $s_0>0$ such that $\bbG_n(s,\by)=\bbG_\beta(s,\by)$ for all $|s| \geq s_0$ and $\by \in \omega$.
 Reproducing the proof of Lemma \ref{lemma:Kcompact} where $(\bbV \mathcal{I}_\bbG^{-1})^2$ is replaced by $T_n^{-1}$, we conclude.
\end{proof}

\bigskip

We are now ready to prove Lemma \ref{lemma:KK'compact}.
\begin{proof}[Proof of Lemma \ref{lemma:KK'compact}]
By Lemma \ref{lemma:TnT:conv} and Lemma \ref{lemma:Tn:comp} the following operator
\begin{equation*}
\left[ I - \mathcal{I}_\beta\bbV_\beta^* (\bbV \mathcal{I}_\bbG^{-1})^2 \mathcal{I}_\beta \bbV_\beta^* \right] \check{\tilde{R}}_\beta: \bL^2_0(\Omega_0) \to \bL^2_0(\Omega_0),
\end{equation*}
appearing in the definition of $\mathcal{K}_\beta (\lambda)$ (cf. \eqref{operK}), is the limit (in operator norm) of a sequence of compact operators in $\mathcal{B}(\bL^2_0(\Omega_0))$ and it is thus compact.
This proves Point \ref{itm:KK'compact1}~Lemma \ref{lemma:KK'compact}.

The proof of Point \ref{itm:KK'compact2}~Lemma \ref{lemma:KK'compact} is the same once one exchanges the roles $\bbG_\beta$ and $\bbG$ in all the results of this section.
\end{proof}

\begin{Rem}
We would like to mention the works \cite{alberti-19, bogli20023, Fer-24} for general results on the essential spectra of Maxwell operators in dispersive and dissipative unbounded media. More precisely, these studies deal with locally perturbed media (with respect to the vacuum) in domains with cylindrical ends \cite{Fer-24}, or with more general locally perturbed  configurations \cite{alberti-19,bogli20023}. Their results are complementary to those developed in this section: while our focus is on the influence of geometry on the essential spectrum, their work primarily investigates the effects of dissipation and dispersion. In these references, dissipation leads to the loss of self-adjointness of the Maxwell operators, while dispersion is accounted for by allowing the material parameters -- the electric permittivity $\varepsilon$ and magnetic permeability $\mu$ -- to depend on the frequency. These considerations motivate the use of spectral theory for operator pencils and techniques involving the numerical range to  localize and classify  the different types of  essential spectra of the associated non-self-adjoint Maxwell operator.
In the paper \cite{Fer-24}, for instance, they analyse a medium consisting of a locally perturbed region connected to straight half-waveguides made of a standard, homogeneous, non-dissipative material. An abstract decomposition of the essential spectrum is established, separating the real essential spectra contributed by the straight half-waveguides from the complex essential spectrum generated by dispersion and dissipation in the locally perturbed region (see Theorem 1.1 and Figure 2 of \cite{Fer-24}). It would be interesting to explore how their results might be combined with Theorem \ref{Main1} to investigate cases where the half-waveguides are not straight. Finally, we also mention the works \cite{Bro-25,cas-kach-jol-17,Cas-Haz-Jol-22,cas-jol-ros-25} that provide precise descriptions of the spectral structure of dispersive and/or dissipative Maxwell operators in separable unbounded geometries.

\end{Rem}

\section{Discrete eigenvalues} \label{secdiscrete}

In \S \ref{sect:localization} we use the Birman-Schwinger principle proved in Proposition \ref{lemmaK} to both localize possible discrete eigenvalues in the gap and proved a Poincar\'e-type inequality. In \S \ref{sect:exisEV} we prove Theorem \ref{thm:discspec} about the existence of discrete eigenvalues under a geometric condition.

\subsection{Localization of possible eigenvalues in the gap and Poincar\'e inequality} \label{sect:localization}

\begin{figure}[h!!]
\begin{center}
\includegraphics[width=0.7\textwidth]{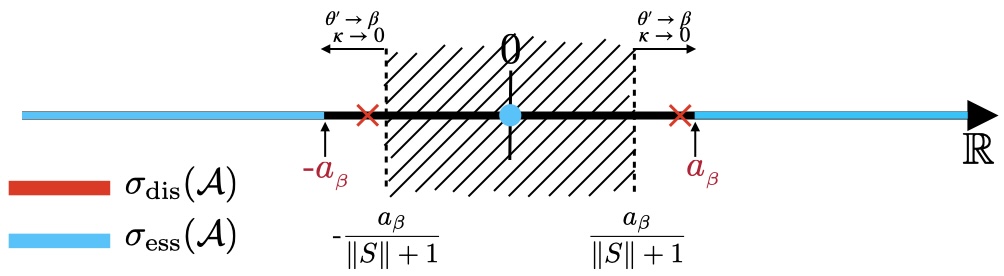}
\end{center}
\caption{When the waveguide $\Omega$ is sufficiently close to the straight one $\Omega_0$ (i.e. when $\|S\| <1$), the possible discrete eigenvalues  appear continuously from the thresholds of the essential spectrum. In particular, this localization result prevents the discrete spectrum to concentrate around the origin, more precisely in the interval $\big(-\frac{a_\beta}{\|S\|+1},\frac{a_\beta}{\|S\|+1} \big)$, leaving 0 an isolated point of the spectrum of $\mathcal{A}$.}
\label{fig:evboud}
\end{figure}

Driven by  \eqref{operK} we are led to introduce the  following operator 
\begin{equation} \label{def:opS}
S:=\left( I- \mathcal{I}_\beta \mathbb{V}_\beta^* (\bbV \mathcal{I}_\bbG^{-1})^2 \mathcal{I}_\beta \mathbb{V}_\beta^* \right):\bL^2_0(\Omega_0) \to \bL^2_0(\Omega_0).
\end{equation}
The following proposition provides information on the localization of the discrete spectrum of $\mathcal{A}$, as illustrated in Figure~\ref{fig:evboud}. Moreover, as we will see below (see Corollary \ref{cor:poincare} and Remark \ref{Rem-Poincare}), this result allows us, under certain conditions, to establish a Poincar\'e-type inequality for the Maxwell operators $\mathcal{\tilde{A}}$ (in the straight waveguide $\Omega_0$) and $\mathcal{A}$ (in the original waveguide $\Omega$).
\begin{Pro}\label{prop:loceivalue}
Let  $\lambda\in \bbR\cap \rho(\mathcal{A}_\beta)$.
Then, there exists a bounded linear isomorphism $$\varphi_\lambda: \ker (\tilde{\mathcal{A}}-\lambda \, I) \to \ker \left(S\tilde{R}_\beta(\lambda) + \frac{1}{\lambda} I\right).$$
\noindent In particular, if $\|S\|<1$ and  $\lambda \in \sigma_{\mathrm{dis}}(\mathcal{A})$, then one has 
\begin{equation}\label{eq.localisation}
|\lambda| \in \left[\frac{a_\beta}{\|S\|+1}, a_\beta\right)  \mbox{ and  $0$ is an isolated eigenvalue  in  $\sigma(\mathcal{A})$.}
\end{equation}
\end{Pro}

\begin{proof}[Proof of Proposition \ref{prop:loceivalue}]
Let $\lambda\in \bbR\cap \rho(\mathcal{A}_\beta)$ then as $\mathcal{A}_\beta$ and $\tilde{\mathcal{A}}_\beta$ are unitarily equivalent (see \eqref{Unitarilyeq-Abeta}),   $\tilde{\mathcal{A}}_\beta -\lambda \, I$ is invertible . Thus, for $\psi \in \ker(\tilde{\mathcal{A}}-\lambda\, I)$, one defines an injective bounded linear map  $\varphi_\lambda$ by $$\varphi_\lambda (\psi) :=(\tilde{\mathcal{A}}_\beta -\lambda) \left( \bbV_\beta \mathcal{I}_\beta^{-1} \mathcal{I}_\bbG \bbV^* \right) \psi.$$ By definition of  $\calQ_\beta(\lambda)$ (cf.\eqref{eqn:def:Qbeta}), if $\psi \in \ker(\tilde{\mathcal{A}}-\lambda \, I)$ and by performing the same computations as in \eqref{alternative:Qb} one has 
\[
0=(\tilde{\mathcal{A}}-\lambda \, I) \psi = \mathcal{I}_\bbG \bbV^* \bbV_\beta \mathcal{I}_\beta^{-1} \calQ_\beta(\lambda) \varphi_\lambda(\psi).
\]
In particular we have $0 = \calQ_\beta(\lambda) \varphi_\lambda(\psi)$ which means that ${\rm ran}(\varphi_\lambda) \subset \ker(\mathcal{Q}_\beta(\lambda))$. Now, remark that actually these two sets are equal. Indeed, take $\bU \in \ker(\mathcal{Q}_\beta(\lambda))$ and set $\psi := \bbV \mathcal{I}_\bbG^{-1} \mathcal{I}_\beta \bbV_\beta^* \tilde{R}_\beta(\lambda) \bU$. Again, by following  the same computations as in \eqref{alternative:Qb} one has $(\tilde{\mathcal{A}}-\lambda) \psi=0$ as well as $\varphi_\lambda(\psi) = \bU$ which proves that $\varphi_\lambda$ is a bijective linear map thus an isomorphism.

Note that $\varphi_\lambda(\ker(\tilde{A}-\lambda \, I )) = \ker(\calQ_\beta(\lambda))$. As $\lambda\neq 0$  (since $0\in \sigma(\mathcal{A}_\beta)$ by Proposition \ref{ess:spectr:beta}), one has also  $\ker(\calQ_\beta(\lambda))= \ker(S\tilde{R}_\beta(\lambda) + \frac1\lambda I)$ which proves the first part of the lemma.

Now, if $\lambda \in \sigma_{\mathrm{dis}}(\mathcal{A})=\sigma_{\mathrm{dis}}(\tilde{\mathcal{A}})$, then,  by Theorem \ref{Main1}, $\lambda\in (-a_{\beta}, a_{\beta}) \setminus \{ 0\}$, and  there exists $\bU \in \ker(S\tilde{R}_\beta(\lambda) + \frac1\lambda I)\setminus \{ 0\}$.  Thus, it implies that $\|\bU\|_{\bL^2_0(\Omega_0)} = |\lambda| \|S \tilde{R}_\beta(\lambda) \bU\|_{\bL^2_0(\Omega_0)}$. In particular, since $\tilde{R}_\beta(\lambda)$ is the resolvent of a self-adjoint operator, we get that
\begin{equation} \label{quotient:S}
1 \leq \frac{|\lambda| \ \|S\|}{\mathrm{dist(\lambda,\sigma(\tilde{\mathcal{A}}_\beta))}},
\end{equation}
where we recall that (by Proposition \ref{ess:spectr:beta}) $\sigma(\tilde{\mathcal{A}}_\beta) = \sigma({\mathcal{A}_\beta})=(-\infty,-a_\beta] \cup \{0\} \cup [a_\beta,+\infty)$.
Assume $\|S\|<1$. Suppose by contradiction  that $|\lambda|\leq a_\beta/2$, then $\mathrm{dist}(\lambda,\sigma(\tilde{\mathcal{A}}_\beta))=|\lambda|$. Then the quotient in the right hand-side of \eqref{quotient:S} is equal to $\|S\|$, hence it is absurd. Therefore necessarily $|\lambda| > a_\beta/2$ in which case  $\mathrm{dist}(\lambda,\sigma(\tilde{\mathcal{A}}_\beta))=a_{\beta}-|\lambda|$ and  the right hand-side in \eqref{quotient:S} reads as follows:
\begin{equation*}
1 \leq \frac{|\lambda| \ \|S\|}{a_\beta - |\lambda|},
\end{equation*}
which yields  $\displaystyle | \lambda|\geq \frac{a_{\beta}}{\|S\|+1}$ and thus the second statement of the lemma.
\end{proof} 

Observe that for $\bU \in \bL_0^2(\Omega_0)$, the operator $S$ defined in \eqref{def:opS} acts as a multiplication operator as follows: 
$$S \bU = \left( \mathds{1}_3 - \bbG_\beta^{1/2} \bbG^{-1}\bbG_\beta^{1/2}\right) \bU.$$
Here $\mathds{1}_3$ is the unit $3\times 3$ matrix. 
The purpose of the following proposition is to show that the operator norm of $S$ is controlled by $\|\kappa\|_{L^\infty(\R)}$ and $\|\theta' - \beta\|_{L^\infty(\R)}$.
In particular if $\kappa$  and $\theta'-\beta$ are small enough in $L^\infty(\R)$-norm we have that $\|S\|<1$ and we can thus apply Proposition \ref{prop:loceivalue}.
Therefore, taking a small  deformation $\Omega$ of the constantly twisted waveguide $\Omega_\beta$ creates possible discrete spectrum in a  continuous way from the thresholds $\pm a_\beta$.
\begin{Pro} \label{cont:creation}
There exists $C > 0$ such that
\begin{equation} \label{estimate:Snormop}
\|S\| \leq C \left( \|\kappa\|_{L^\infty(\R)} + \|\theta'-\beta\|_{L^\infty(\R)} \right).
\end{equation}
Moreover,  in the particular case where $\theta'=\beta$ (in other words when the twist  $\beta\geq 0$ is constant), if one assumes that $ b\|\kappa\|_{L^\infty(\R)} < \frac{1}{2}$, then one has $\|S\|<1$ which implies that \eqref{eq.localisation}  holds and  that $0$ is an isolated eigenvalue in the spectrum  of $\mathcal{A}$.
\end{Pro}

\begin{proof}[Proof of Proposition \ref{cont:creation}]
Recall the expressions \eqref{def:Gbeta} for $\bbG_\beta$  and \eqref{matr:G} for $\bbG$, as well as the definition  \eqref{eqn:defktheta} of $\mathbf{k^\theta}$.
Set $M:=\mathds{1}_3-\bbG_\beta^{1/2}\bbG^{-1}\bbG_\beta^{1/2}$. Then $M$ is real symmetric matrix field  and observe that its  eigenvalues $\{\lambda_j\}_{j=1,2,3}$ (not in any particular order) are given by
\begin{eqnarray}\label{eq.eigenvalues}
    &&  \lambda_1(s,\by) = \mathbf{k^\theta} \cdot \by;  \\[5pt]
      &&  \lambda_2(s,\by)  = -\frac{1}{2(1-\mathbf{k^\theta} \cdot \by)} \bigg( (\mathbf{k^\theta} \cdot \by)^2 + (\beta-\theta')^2 |\by|^2 \nonumber  \\
        & &\qquad \qquad \qquad  \qquad \qquad   + \sqrt{\left( (\mathbf{k^\theta} \cdot \by)^2 + (\beta-\theta')^2 |y|^2 \right) \left( (2- \mathbf{k^\theta} \cdot \by )^2 + (\beta-\theta')^2 |\by|^2 \right)}  \bigg);  \nonumber \\
     &&   \lambda_3(s,\by)  = -\frac{1}{2(1-\mathbf{k^\theta} \cdot \by)} \bigg( (\mathbf{k^\theta} \cdot \by)^2 + (\beta-\theta')^2 |\by|^2 \nonumber \\
        && \qquad \qquad \qquad  - \sqrt{\left( (\mathbf{k^\theta} \cdot \by)^2 + (\beta-\theta')^2 |\by|^2 \right) \left( (2-\mathbf{k^\theta} \cdot \by )^2 + (\beta-\theta')^2 |\by|^2 \right)}  \bigg) \nonumber.
\end{eqnarray}
Recalling assumption \eqref{assumpt:a:kappa} and that $|\mathbf{k^\theta}|=\kappa$, the curvature of the base curve $\Gamma$, it is immediate to see that for all $j=1,2,3$
\begin{equation*}
    |\lambda_j(s, \by)| \leq C \left( \|\kappa\|_{L^\infty(\R)} + \|\theta'-\beta\|_{L^\infty(\R)} \right), \quad \forall (s,\by)\in \Omega_0,
\end{equation*}
where $C>0$ depends uniformly on $\sup_{\by \in \omega} |\by|^2$, $\sup_{s \in \R} \kappa(s)$ and dimensional constants.
Therefore, one has 
\begin{equation}\label{eq.bornSeigen}
\|S \|\leq \sup_{(s,\by)\in \Omega_0}\|M(s,\by)\| = \sup_{(s,\by)\in \Omega_0} \max(|\lambda_1(s,\by) |, |\lambda_2(s,\by) ) |\leq C \left( \|\kappa\|_{L^\infty(\R)} + \|\theta'-\beta\|_{L^\infty(\R)} \right)
\end{equation}
where $\|M(s,\by)\|$ stands here for the spectral norm of the  $3\times 3$  matrix $M(s,\by)$.\\ 

In the particular case where $\theta'=\beta$, the expression  \eqref{eq.eigenvalues}  of the eignevalues $\lambda_1(s,\by)$ and $\lambda_2(s,\by)$  simplifies to 
\begin{equation}\label{eq.eigenvaluesconsttwist}
  \lambda_1 (s, \by)= \mathbf{k^\theta} \cdot \by \quad \mbox{  and }  \quad     \lambda_2(s, \by)=-\frac{1}{2(1-\mathbf{k^\theta} \cdot \by)} \big((\mathbf{k^\theta} \cdot \by)^2+ |\mathbf{k^\theta} \cdot \by| \, (2-\mathbf{k^\theta} \cdot \by ) \big)
\end{equation}
(where we have used   that  $2-\mathbf{k^\theta} \cdot \by >0$ by virtue of   \eqref{assumpt:a:kappa}). Moreover, one has  $\mathbf{k^\theta} \cdot \by\in  {\rm I}:=(- b\|\kappa\|_{L^\infty(\R)} ,  b\|\kappa\|_{L^\infty(\R)})$ for all $(s,\by)\in \Omega_0$.   Thus, as one  assumes  here  that $ b\|\kappa\|_{L^\infty(\R)} < \frac{1}{2}$,  the segment $\overline{\rm I}$ is included in $ (-1/2,1/2) $. Thus, one obtains from \eqref{eq.bornSeigen} and \eqref{eq.eigenvaluesconsttwist}  that
$$
\|S \| \leq \sup_{x\in I} f(x)< f\Big(\frac{1}{2}\Big)=1 \mbox{ where } f(x):=\max\left(|x|,  \frac{1}{2(1-x)} \big(x^2+  |x| \, (2-x) \big) \right) \mbox{ on } \Big(-\frac{1}{2}, \frac{1}{2}\Big).
$$
Hence, by applying Proposition \ref{cont:creation}, one concludes   that \eqref{eq.localisation}  holds and  that $0$ is an isolated eigenvalue in $\sigma(\mathcal{A})$.
\end{proof}
\noindent Propositions \ref{prop:loceivalue} and \ref{cont:creation} yield the following Poincar\'e inequality.
\begin{Cor}\label{cor:poincare}{(A Poincar\'e-type inequality)}
If $\|S\|<1$ then there exists $C > 0$ such that for all $\bU = (\bE,\bH)^\top \in D(\hat{\mathcal{A}})\cap \ker(\hat{\mathcal{A}})^\perp=\big(H_0(\curlvec,\Omega_0)\times H(\curlvec,\Omega_0) \big) \cap \mathcal{J}_{\beps,\bmu}(\Omega_0)$ there holds
\begin{equation}\label{eq.Poincare}
    \|\bU\|_{\bL_{\beps,\bmu}^2(\Omega_0)} \leq C \|\hat{\mathcal{A}}\bU\|_{\bL_{\beps,\bmu}^2(\Omega_0)}.
\end{equation}
Thus, in particular, if  $\theta'=\beta$  and   $ b\|\kappa\|_{L^\infty(\R)} < \frac{1}{2}$ then the inequality \eqref{eq.Poincare} holds.
\end{Cor}
\begin{proof}[Proof of Corollary \ref{cor:poincare}]
We recall that $\hat{\mathcal{A}}$ is the operator defined in Proposition \ref{prop:unitequivpiola} which is unitarily equivalent to $\mathcal{A}$. Moreover, by Proposition \ref{prop:kerneldescription}, 
$\hat{\mathcal{A}}$ has an infinite dimensional kernel $\ker(\hat{\mathcal{A}}) = \overline{\mathcal{G}(\Omega_0)}$  and  $\ker(\hat{\mathcal{A}})^\perp = \mathcal{J}_{\beps,\bmu}(\Omega_0)$.
Thus, by Proposition \ref{prop:loceivalue},  if $\|S\|<1$ then $0$ is an isolated eigenvalue of $\sigma(\hat{\mathcal{A}})=\sigma(\mathcal{A})$. Therefore, the self-adjoint  operator $\hat{\mathcal{A}}_\perp := \hat{\mathcal{A}}|_{\ker(\hat{\mathcal{A}})^\perp}$ defined  on  $D(\hat{\mathcal{A}}_\perp) = D(\hat{\mathcal{A}})\cap \ker(\hat{\mathcal{A}})^\perp$ is invertible  and   $(\hat{\mathcal{A}}_\perp)^{-1}\in \mathcal{B}\big(\bL^2_{\beps,\bmu}(\Omega_0) \big)$. Hence, there exists $C>0$ such that 
$$
\|(\hat{\mathcal{A}}_\perp)^{-1}\bV \|_{\bL^2_{\beps,\bmu}(\Omega_0)} \leq C \| \bV \|_{\bL^2_{\beps,\bmu}(\Omega_0)}  \quad \forall \bV\in \bL^2_{\beps,\bmu}(\Omega_0)
$$
Therefore, setting $\bU:=(\hat{\mathcal{A}}_\perp)^{-1}\bV$, it follows,  by the invertibility of $\hat{\mathcal{A}}_\perp$, that 
\begin{equation*}
    \|\bU\|_{\bL_{\beps,\bmu}^2(\Omega_0)} \leq C \|\hat{\mathcal{A}}_{\perp }\bU\|_{\bL_{\beps,\bmu}^2(\Omega_0)}= \|\hat{\mathcal{A}}\bU\|_{\bL_{\beps,\bmu}^2(\Omega_0)}  \quad  \forall \bU\in  D(\hat{\mathcal{A}}_\perp)=D(\hat{\mathcal{A}})\cap \ker(\hat{\mathcal{A}})^\perp .
\end{equation*}
which is precisely the  Poincar\'e inequality \eqref{eq.Poincare} since $D(\hat{\mathcal{A}})\cap \ker(\hat{\mathcal{A}})^\perp=\big(H_0(\curlvec,\Omega)\times H(\curlvec,\Omega) \big) \cap \mathcal{J}_{\beps,\bmu}(\Omega_0) $. \\[4pt]
Furthermore, by Proposition \ref{cont:creation}, one has $\|S\|<1$ if  $\theta'=\beta$ and $ b\|\kappa\|_{L^\infty(\R)} < \frac{1}{2}$. This concludes the proof.
\end{proof}
\begin{Rem}\label{Rem-Poincare}
From the Corollary  \ref{cor:poincare}, the  unitary equivalence \eqref{def:Ahat} between  the operators $\mathcal{A}$ and $\hat{\mathcal{A}}$ and the Proposition \ref{prop:piolaproperties}, it clearly follows (under the same assumptions as in   Corollary  \ref{cor:poincare}) a  Poincar\'e-type inequality in the waveguide  $\Omega$:  there exists $C>0$ such that
\begin{equation*}
    \|\bU\|_{\bL_0^2(\Omega)} \leq C \|\mathcal{A}\bU\|_{\bL_0^2(\Omega)}, \, \forall \bU\in\big(H_0(\curlvec,\Omega)\times H(\curlvec,\Omega) \big) \cap \mathcal{J}_{\varepsilon_0  \mathds{1}_3, \mu_0  \mathds{1}_3}(\Omega) .
\end{equation*}

\end{Rem}

\subsection{Existence of eigenvalues} \label{sect:exisEV}

\subsubsection{Outlines of the proof and preliminaries}
In this section, we prove Theorem \ref{thm:discspec} about the existence of discrete spectrum for the self-adjoint operator $\mathcal{A}$ under adequate geometric conditions. The fact that its discrete spectrum is symmetric with respect  to the origin is a classical property of Maxwell operators  (see Proposition \ref{prop.eigensym} in the appendix section \ref{appendix:specMaxwell}.)

Recall that $0\in \sigma_{\mathrm{ess}}(\mathcal{A})=(-\infty,-a_\beta] \cup \{0\} \cup [a_\beta,+\infty)$  with $a_{\beta}>0$ (see Theorem \ref{Main1}) and is an eigenvalue of infinite multiplicity of $\mathcal{A}$ (see Proposition \ref{prop:unitequivpiola} and Proposition \ref{prop:kerneldescription}). Thus, if $0$ is  not an isolated eigenvalue of $\mathcal{A}$, there exists a sequence $(\lambda_n)_{n\in \mathbb{N}}$ of discrete eigenvalues of $\mathcal{A}$ converging to $0$. Hence, in this case $\sigma_{\operatorname{dis}}(\mathcal{A})$ is not empty and has even  an infinite cardinality. Thus, Theorem \ref{thm:discspec} obviously holds.
Therefore, for the remainder of the proof, we focus only on the case where
$$
0 \mbox{ is an isolated eigenvalue in }\sigma(\mathcal{A}) .
$$
We note that, if one assumes that $\theta'=0$ (i.e. that there is no twist) as in Theorem \ref{thm:discspec},  this latter property  is always satisfied for instance when  $b \|\kappa\|_{\|\kappa\|_{L^\infty(\R)}}<1/2$ (see Proposition \ref{cont:creation}).

The strategy of the proof is to work with the quadratic form associated with the operator $\hat{\mathcal{A}}^2$, where $\hat{\mathcal{A}}$ is the operator defined in Proposition \ref{prop:unitequivpiola} (which is unitarily  equivalent to $\mathcal{A}$).  However, as the kernel of $\hat{\mathcal{A}}$ is infinite dimensional and given by the closure of the space of gradient fields $\mathcal{G}(\Omega_0)$ defined in \eqref{eqn:defgradientfields}), one must instead work with the reduced self-adjoint  operator $\hat{\mathcal{A}}_\perp := \hat{\mathcal{A}}|_{\ker(\hat{\mathcal{A}})^\perp}$ defined on the Hilbert space $ \ker(\hat{\mathcal{A}})^\perp=\mathcal{J}_{\beps,\bmu}(\Omega_0)$ (a closed subspace of $\bL_0^2(\Omega_0)$ endowed with the $\bL_0^2(\Omega_0)$-inner product). The domain of  $\hat{\mathcal{A}}_\perp$ is given (see Propositions \ref{prop:unitequivpiola}  and  \ref{prop:kerneldescription}) by $$D(\hat{\mathcal{A}}_\perp) = D(\hat{\mathcal{A}})\cap \ker(\hat{\mathcal{A}})^\perp=\big(H_0(\curlvec,\Omega_0)\times H(\curlvec,\Omega_0) \big) \cap \mathcal{J}_{\beps,\bmu}(\Omega_0).$$
\S \ref{subsubsec:proofthmdiscspec} is about the proof of Theorem \ref{thm:discspec} and uses the standard min-max technique with the quadratic form of the self-adjoint operator $(\hat{\mathcal{A}}_\perp)^2$. The goal is to build an adequate trial function which has energy strictly below the positive  threshold of the essential spectrum of $(\hat{\mathcal{A}}_\perp)^2$, yielding the existence of an eigenvalue of finite multiplicity for $\hat{\mathcal{A}}$ (and thus for $\mathcal{A}$). Note that the main difficulty is to obtain a trial function belonging to $\ker(\hat{\mathcal{A}})^\perp$. Finally in \S \ref{subsubsub:delta} we construct a family of waveguides for which there is always discrete Maxwell eigenvalues.

Let us consider the sesquilinear form acting in the Hilbert space $\mathcal{J}_{\beps,\bmu}(\Omega_0)$ endowed with the scalar product of $\bL_{\bbG}^2(\Omega_0)$ and defined as

\begin{equation*}
\begin{split}
q \left[(\bE_1,\bH_1)^\top,(\bE_2,\bH_2)^\top\right] &:= \left( \redA  (\bE_1,\bH_1)^\top, \redA (\bE_2,\bH_2)^\top  \right)_{\bL^2_\bbG(\Omega_0)} \\
&= \int_{\Omega_0}\bmu^{-1} \curlvec \bE_1 \cdot \curlvec \overline{\bE}_2 +  \int_{\Omega_0} \beps^{-1} \curlvec \bH_1 \cdot \curlvec \overline{\bH}_2
\end{split}
\end{equation*}
with
\[
(\bE_1,\bH_1)^\top,(\bE_2,\bH_2)^\top \in D(q) := D(\redA).
\]
The following lemma holds.
\begin{Lem}\label{lem:fqofsquare} We have:
\begin{enumerate}[label=\normalfont(\roman*)]
	\item\label{itm:1fqofsquare} $q$ is a closed, densely defined, symmetric and bounded below form.
	\item\label{itm:2fqofsquare} $q$ is the quadratic form associated with the operator $(\redA)^2$.
\end{enumerate}
\end{Lem}
\begin{proof}[Proof of Lemma \ref{lem:fqofsquare}]
Point \ref{itm:1fqofsquare} is immediate once noticing that $\redA$ is a self-adjoint  injective operator acting in the Hilbert space $\mathcal{J}_{\beps,\bmu}(\Omega_0)$. Concerning Point \ref{itm:2fqofsquare}, it is a consequence of the general fact that if $\mathscr{L}$ is a self-adjoint operator on a separable Hilbert space $X$, the unique sesquilinear form associated with $\mathscr{L}^2$ by \cite[Thm. 2.1 \& Thm. 2.6 Ch. 6 \S 2]{kato} is given by
\[
	\mathscr{Q}[x,y] = \langle\mathscr{L}x,\mathscr{L}y\rangle_X,\quad x,y \in D(\mathscr{Q}) = D(\mathscr{L}).
\]
\end{proof}
Finally, we end this paragraph by the following lemma about the structure of the spectrum of $(\redA)^2$.
\begin{Lem}\label{lem:specsquarereduced} Assume that $0$ is an isolated eigenvalue of $\hat{\mathcal{A}}$. The following holds.
\begin{enumerate}[label=\normalfont(\roman*)]
	\item\label{itm:1specsquarereduced} Let $\lambda \in (0,+\infty)$. Then one has
\begin{equation*}
\lambda \text{ is an eigenvalue of } (\redA)^2  \iff \pm \sqrt{\lambda} \text{ are eigenvalues of } \redA.
\end{equation*}
	\item\label{itm:2specsquarereduced} One has $\sigma_{\mathrm{ess}}((\redA)^2) = [a_\beta^2, +\infty)$.
\end{enumerate}
\end{Lem}
\begin{proof}[Proof of Lemma \ref{lem:specsquarereduced}] Let us start by proving Point \ref{itm:1specsquarereduced}.
Let $(\bE,\bH) \in D((\redA)^2)$ be an eigenvector of $(\redA)^2$ associated with $\lambda > 0$. Then $(\bE, \mp \frac{\rmi}{\sqrt{\lambda}} \bmu^{-1} \curlvec \bE)$ belongs to $D(\redA)$ and it is an eigenvector of $\redA$ (and $\hat{\mathcal{A}}$) associated to $\pm \sqrt{\lambda}$.
Note that we can also choose $(\pm \frac{\rmi}{\sqrt{\lambda}} \beps^{-1} \curlvec \bH,\bH)$. The other implication is obvious.

Now let us deal with Point \ref{itm:2specsquarereduced}. By definition of the operator $\hat{\mathcal{A}}$ (see Proposition \ref{prop:unitequivpiola}) and  Theorem \ref{Main1}, one has $\sigma_{\mathrm{ess}}(\hat{\mathcal{A}}) = (-\infty,a_\beta] \cup  \{0 \}\cup (-\infty,a_\beta]$. Moreover, on the one hand $0$ is an isolated eigenvalue of $\hat{\mathcal{A}}$ and on the other hand,  the reduced self-adjoint operator $\redA= \hat{\mathcal{A}}|_{\ker(\hat{\mathcal{A}})^\perp}$ is injective by construction. Thus, the essential spectrum of  $\redA$ is deduced from the one of $\hat{\mathcal{A}}$ by removing $0$. In other words, one has  $\sigma_{\mathrm{ess}}(\redA) = \sigma_{\mathrm{ess}}(\hat{\mathcal{A}})\setminus \{ 0\} = (-\infty,-a_\beta]  \cup [a_\beta,+\infty)$.
 Then, Point \ref{itm:2specsquarereduced}  is a consequence of the standard fact that for self-adjoint operators one has $\sigma_{\mathrm{ess}}((\redA)^2) = \sigma_{\mathrm{ess}}((\redA))^2$.
\end{proof}
\subsubsection{Proof of Theorem \ref{thm:discspec}}\label{subsubsec:proofthmdiscspec}
The goal of this paragraph is to prove Theorem  \ref{thm:discspec}. In order to prove it, we start by recalling the min-max principle. To this aim, we need the following definition.
\begin{Def} Let $\mathcal{Q}$ be a closed and bounded-below quadratic form with dense domain $D(\mathcal{Q})$ in a complex Hilbert space $\mathcal{H}$. For $n \in \mathbb{N}$, the $n$-th min-max value of $\mathcal{Q}$ is defined as
\begin{equation}\label{eqn:defminmaxlevel}
	\alpha_n(\mathcal{Q}) := \inf_{W \subset D(\mathcal{Q})\atop \dim W = n} \sup_{u \in W \setminus\{0\}} \frac{\mathcal{Q}[u]}{\|u\|_\mathcal{H}^2}.
\end{equation}
We also denote by $\mathcal{Q}$ the associated sesquilinear form and if $\mathcal{L}$ is the unique self-adjoint operator associated with the sesquilinear form $\mathcal{Q}$ {\it via} Kato's first representation theorem (see \cite[Chapter IV, Theorem 2.1]{kato}), we shall refer to \eqref{eqn:defminmaxlevel} as the $n$-th min-max value of $\mathcal{L}$ and set $\alpha_n(\mathcal{L}) := \alpha_n(\mathcal{Q})$.
\end{Def}
The min-max principle reads as follows.
\begin{Pro}\label{prop:min-max} Let $\mathcal{Q}$ be a closed and bounded-below quadratic form with dense domain $D(\mathcal{Q})$ in a complex Hilbert space $\mathcal{H}$ and let $\mathcal{L}$ be the unique self-adjoint operator associated with $\mathcal{Q}$. Then, for $n\in \mathbb{N}$, we have the following alternative:
\begin{enumerate}[label=\normalfont(\roman*)]
\item if $\alpha_n(\mathcal{L}) < \inf \sigma_{\mathrm{ess}}(\mathcal{L})$ then $\alpha_n(\mathcal{L})$ is the $n$-th discrete eigenvalue of $\mathcal{L}$ (counted with multiplicity and in increasing order),
\item if $\alpha_n(\mathcal{L}) = \inf \sigma_{\mathrm{ess}}(\mathcal{L})$ then for all $k \geq n$ there holds $\alpha_k(\mathcal{L}) = \inf \sigma_{\mathrm{ess}}(\mathcal{L})$.
\end{enumerate}
\end{Pro}

Theorem \ref{thm:discspec} is a consequence of the following proposition.

\begin{Pro}\label{prop:existtrialfunc} Let $\bX$ and $\bY^\theta$ be the vectors of $\bbR^2$ defined in \eqref{eqn:defX}. Assume that $0$ is an isolated eigenvalue of  $\hat{\mathcal{A}}$. Suppose $\theta'\equiv 0$, $\displaystyle \lim_{|s|\to \infty} \kappa(s) =0$ and that $\kappa \in L^1(\mathbb{R})$.
\[
    \bX \cdot \ \bY^\theta > \frac{2\lambda_2^N(\omega)}{\varepsilon_0} \ \frac{b^2 \|\kappa\|_{L^\infty(\R)}}{1-b\|\kappa\|_{L^\infty(\R)}}\|\kappa\|_{L^1(\R)}
\]
then there exists $(\bE,\bH)^\top\neq(0,0)^\top \in D(\redA)$ such that
\begin{equation}
	q[(\bE,\bH)^\top] - a_\beta^2 \|(\bE,\bH)^\top\|^2_{\bL_\bbG^2(\Omega_0)} < 0.
	\label{eqn:evexistence}
\end{equation}
\end{Pro}

We postpone the proof of Proposition \ref{prop:existtrialfunc} to the end of this section. We are now in a good position to prove Theorem \ref{thm:discspec}.
\begin{proof}[Proof of Theorem \ref{thm:discspec}] Assume \eqref{eqn:ineqdiscspec} holds. By Proposition \ref{prop:existtrialfunc} there exists $(\bE,\bH)^\top\neq(0,0)^\top \in D(\redA)$ such that \eqref{eqn:evexistence} holds. It yields
\[
	\alpha_1((\redA)^2) = \alpha_1(q) \leq \frac{q((\bE,\bH)^\top)}{\|(\bE,\bH)^\top\|_{\bL_{\bbG}^2(\Omega_0)}} < a_\beta^2.
\]
Thus, by Point \ref{itm:2specsquarereduced}~Lemma \ref{lem:specsquarereduced} we have $\alpha_1((\redA)^2) < \inf \sigma_{\mathrm{ess}}((\redA)^2)=a_{\beta}^2$  and by Proposition \ref{prop:min-max} we know that $\alpha_1((\redA)^2)$ is a discrete eigenvalue of $(\redA)^2$. Note that as $\redA$ is injective, so is $(\redA)^2$ and necessarily $\alpha_1((\redA)^2) > 0$. Now, by Point \ref{itm:1specsquarereduced}~Lemma \ref{lem:specsquarereduced}, we get that $\pm \sqrt{\alpha_1((\redA)^2)}$ are eigenvalues of $\redA$ of finite multiplicity  lying in $(-a_{\beta},a_{\beta})\setminus \{0\}$. Thus they are  discrete eigenvalues of $\hat{\mathcal{A}}$ (and thus of $\mathcal{A}$).
\end{proof}

The rest of this paragraph is devoted to the proof of Proposition \ref{prop:existtrialfunc}.
\begin{proof}[Proof of Proposition \ref{prop:existtrialfunc}]
The proof is divided into three steps. In the first step, we construct a sequence of trial functions of the form $(\bF_n,0) \in D(\hat{\mathcal{A}})$, for $
n\in\mathbb{N}$, and in the second step we project it onto the solenoidal space $\mathcal{J}_{\beps,\bmu}(\Omega_0)$ to obtain a trial function of the form $(\bE_n,0)\in D(\redA)$. The last step, is devoted to the computation of the energy $q[(\bE_n,0)^\top] - a_\beta^2 \|(\bE_n,0)^\top\|_{\bL_\bbG^2(\Omega_0)}^2$.

Before going through all steps of the proofs, note that as the twist $\theta'\equiv 0$, then the Jacobian matrix $J_\Phi$ defined in \eqref{eqn:defjacobian!} is diagonal and the coefficients $\beps$ and $\bmu^{-1}$ of Proposition \ref{prop:unitequivpiola} read as follows:
\begin{align*}
\beps &= \varepsilon_0
\begin{pmatrix}
(1- \mathbf{k^\theta} \cdot \by)^{-1} & 0 & 0\\
0 & 1- \mathbf{k^\theta} \cdot \by & 0\\
0 & 0 & 1- \mathbf{k^\theta} \cdot \by
\end{pmatrix},\\[6pt]
\bmu^{-1} &= \mu_0^{-1}
\begin{pmatrix}
1- \mathbf{k^\theta} \cdot \by & 0 & 0\\
0 & (1- \mathbf{k^\theta} \cdot \by)^{-1} & 0 \\
0 & 0 & (1- \mathbf{k^\theta} \cdot \by)^{-1}
\end{pmatrix}.
\end{align*}
For $(\bE,0) \in D(\redA)$, this yields the following expression
\begin{equation}\label{eqn:exprexplicitq}
q[(\bE,0)^\top]= \frac{1}{\mu_0} \int_{\Omega_0} \Big((1- \mathbf{\bk^\theta} \cdot \by) |(\curlvec \bE)_1|^2 + \frac{1}{1- \mathbf{\bk^\theta} \cdot \by} \left( |(\curlvec \bE)_2|^2 + |(\curlvec \bE)_3|^2 \right) \Big).
\end{equation}
\paragraph{Step 1: } For $n \in \mathbb{N}$, we define the sequence of functions $(\varphi_n: \R \to [0,1])$ of cut-off functions defined as follows:
\begin{equation} \label{def:phin}
\varphi_n(s) := 
\begin{cases}
\displaystyle \frac{s}{n}+2, & s \in (-2n,-n)\\[4pt]
1, & s \in (-n,n)\\
\displaystyle -\frac{s}{n}+2, & s \in (n,2n)\\[4pt]
0, & \text{otherwise}.
\end{cases}
\end{equation}
Observe that as defined $\varphi_n \in H^1(\R)$ and
\begin{equation} \label{cutoff:der:estimate}
\|\varphi'_n\|_{L^2(\R)} =\sqrt{ \frac{2}{n}}.
\end{equation}
Define $(\bF_n)_{n\in \mathbb{N}}$ to be the sequence of vector fields with separated variables defined for $(s,\by)\in \Omega_0$
\begin{equation} \label{defin:Fn}
\bF_n(s,\by):= \varphi_n(s) \begin{pmatrix}
0\\
\partial_{y_3} \psi(\by) \\
-\partial_{y_2} \psi(\by)
\end{pmatrix}, \text{for all } n\in \mathbb{N}.
\end{equation}
Here, recall that $\psi$ is an eigenfunction associated with the first non-trivial Neumann eigenvalue  $\lambda_2^N(\omega)$ in the cross-section $\omega$ whose eigenvalue is related to $a_\beta^2$, that is the bottom of the essential spectrum of $(\redA)^2$. Note that in our case, since $\theta'\equiv 0$, one has  $a_\beta^2 := \sqrt{\lambda_2^N(\omega)}\, c$ (see Remark \ref{rem:defa_betanotwist}). Note that as defined $\bF_n \in H(\curlvec,\Omega_0)$, but actually it is in $H_0(\curlvec,\Omega_0)$. To prove it, remark that there holds
\begin{equation}\label{eqn:curfn}
\curlvec \bF_n = \begin{pmatrix}
\lambda_2^N(\omega) \varphi_n \psi \\
\varphi'_n \nabla_\by \psi
\end{pmatrix}
\end{equation}
and that for $\bv=(v_1,v_2,v_3)^\top \in{\bbD}(\overline{\Omega_0})$ one has
\[
	\int_{\Omega_0} \curlvec \bF_n \cdot \overline{\bv} \rmd s\, \rmd \by= \lambda_2^N \int_{\Omega_0}\varphi_n(s)\psi(\by)\,  \overline{v_1}(s,\by) \rmd s\, \rmd \by + \int_{\Omega_0}\varphi_n'(s) \nabla_\by \psi(\by)\cdot\overline{ \begin{pmatrix}v_2(s,\by)\\ v_3(s,\by)\end{pmatrix}} \rmd s  \,\rmd \by.
\]
Similarly, one has
\begin{align*}
	\int_{\Omega_0} \bF_n \cdot \overline{\curlvec \bv} \rmd s \rmd \by& = \int_{\Omega_0}\varphi_n(s)\left(\partial_{y_3}\psi(\by)\, \overline{(\partial_{y_3} v_1(s,\by) - \partial_s v_3(s,\by))}\right.\\
	&\quad  \quad \left.- \partial_{y_2}\psi(\by)\, \overline{(\partial_s v_2(s,\by) - \partial_{y_2} v_1(s,\by))}\right)\rmd s\, \rmd \by\\
	& = \int_{\Omega_0} \varphi_n(s) \nabla_\by\psi(\by) \cdot \overline{\nabla_\by v_1(s,\by)} \,\rmd s \, \rmd \by\\
	& \quad - \int_{\Omega_0} \varphi_n(s) \left( \partial_{y_3}\psi (\by)\, \overline{\partial_s v_3(s,\by)} + \partial_{y_2}\psi(\by)\overline{\partial_s v_2(s,\by)}\right) \rmd s\, \rmd \by.
\end{align*}
Using the variational definition of $\psi$ being a mode of the Neumann Laplacian in $\omega$ one obtains
\[
	\int_{\Omega_0} \varphi_n(s) \nabla_\by\psi(\by) \cdot \overline{\nabla_\by v_1(s,\by)} \rmd s \, \rmd \by = \lambda_2^N(\omega) \int_{\Omega_0}\varphi_n(s) \psi(\by) \overline{v_1(s,\by)} \rmd s \,\rmd \by
\]
and integrating by parts with respect to $s$, one has
\begin{eqnarray*}
	- \int_{\Omega_0} \!\varphi_n(s) \Big(\partial_{y_3}\psi(\by) \overline{\partial_s v_3(s,\by)} + \partial_{y_2}\psi(\by)\overline{\partial_s v_2(s,\by)}\Big) \rmd s \rmd \by  = \int_{\Omega_0}  \!\varphi_n'(s) \nabla_\by\psi(\by) \cdot \overline{\begin{pmatrix}{v_2(s,\by)}\\[4pt]{v_3(s,\by)}\end{pmatrix}}\rmd s \rmd \by.
\end{eqnarray*}
It yields that for all $\bv \in \bbD(\overline{\Omega_0})$ one has
\[
	\int_{\Omega_0} \curlvec \bF_n \cdot \overline{\bv} \rmd s \rmd \by = \int_{\Omega_0} \bF_n \cdot \overline{\curlvec \bv} \rmd s \rmd \by.
\]
As $\bbD(\overline{\Omega_0})$ is dense in $H(\curlvec,\Omega_0)$ (see e.g. \cite[Theorem 2.10.]{gira} or Theorem 1 page 279 of \cite{daulio}), one can conclude that $\bF_n \in H_0(\curlvec,\Omega_0)$.
\paragraph{Step 2.} Note that the elements of the sequence of vector fields $((\bF_n,0)^\top)_{n\in\mathbb{N}}$ do not necessarily belong to the solenoidal space $\mathcal{J}_{\beps,\bmu}(\Omega_0)$. To overcome this difficulty, for all $n\in \mathbb{N}$, we project $(\bF_n,0)^\top$ orthogonally onto $\mathcal{J}_{\beps,\bmu}(\Omega_0)$. This is performed by applying, for all $n\in \mathbb{N}$, Lemma \ref{lem:projsolenoidaleortho} to $\bF_n$ (see Appendix \ref{appendix:PDE}). It yields the existence of $u_n \in H_0^1(\Omega_0)$ weak solution of 
\begin{equation*}
\begin{cases}
\Div \beps \nabla u_n = \Div \beps \bF_n, & \text{in }\Omega_0,\\
u_n=0 &  \text{on } \partial\Omega_0.
\end{cases}
\end{equation*}
Set $\bE_n := \bF_n - \nabla u_n$. Remark that since $u_n \in H_0^1(\Omega_0)$, $\nabla u_n \in H_0(\curlvec,\Omega_0)$ and so is $\bE_n$. Remark that as constructed, $\Div(\beps\bE_n) = 0$ and thus $(\bE_n,0)^\top \in \mathcal{J}_{\beps,\bmu}(\Omega_0)\cap H_0(\curlvec,\Omega_0)\times H(\curlvec,\Omega_0) = D(\redA)$.

\paragraph{Step 3.} In this step we compute the quantity
\[
	q[(\bE_n,0)^\top] - a_\beta^2 \|(\bE_n,0)\|_{\bL_{\bbG}(\Omega_0)}^2 = q[(\bE_n,0)^\top] -\lambda_2^N \, c^2 \|(\bE_n,0)\|_{\bL_{\bbG}(\Omega_0)}^2.
\]

Note that as defined $\bE_n$ is in the orthogonal of the space of gradients $\nabla (H_0^1(\Omega_0))$. Hence, by the Pythagorean theorem, there holds  
\begin{equation} \label{bepsEE:term}
\begin{split}
\|\bE_n\|_{\bL_\beps^2(\Omega_0)}^2 &= \|\bF_n\|_{\bL_\beps^2(\Omega_0)}^2 - \|\nabla u_n\|_{\bL_\beps^2(\Omega_0)}^2 \\
&= \varepsilon_0 \int_{\Omega_0}  \left((1-\mathbf{k^\theta}(s) \cdot \by) |\varphi_n(s)|^2 |\nabla_{\by} \psi(\by)|^2\right) \rmd s \rmd \by - \|\nabla u_n\|_{\bL^2_\beps(\Omega_0)}^2.
\end{split}
\end{equation}
Moreover, using \eqref{eqn:curfn} and \eqref{eqn:exprexplicitq}, one has
\begin{eqnarray} 
\displaystyle \nonumber q[(\bE_n,0)^\top]  = q[(\bF_n,0)^\top] =&  \displaystyle   \frac{(\lambda_2^N(\omega))^2}{\mu_0} \int_{\Omega_0} \left((1- \mathbf{k^\theta}(s) \cdot \by) |\varphi_n(s)|^2 |\psi(s,\by)|^2 \right)\rmd s\, \rmd \by \nonumber \\
\displaystyle  &+  \displaystyle   \frac{1}{\mu_0} \int_{\Omega_0} \left((1-\mathbf{k^\theta}(s) \cdot \by)^{-1} |\varphi'_n(s)|^2 |(\nabla\psi)(\by)|^2 \right)\rmd s\, \rmd \by.\label{qE:term}
\end{eqnarray}
Moreover, by definition of $\psi$, one has
\begin{equation} \label{term:zero:ibp}
\frac{(\lambda_2^N(\omega))^2}{\mu_0} \int_{\Omega_0} \left(|\varphi_n(s)|^2 |\psi(\by)|^2\right)\rmd s\, \rmd \by = \frac{\lambda_2^N(\omega)}{\mu_0} \int_{\Omega_0} \left(|\varphi_n(s)|^2 |\nabla_{\by} \psi(\by)|^2\right)\rmd s \rmd \by.
\end{equation}
Hence, combining \eqref{bepsEE:term}, \eqref{qE:term}, and \eqref{term:zero:ibp} and recalling that $c^2=(\varepsilon_0 \mu_0)^{-1}$,  we obtain:
\begin{equation}
\begin{split}
q[(\bE_n,0)^\top] - \lambda_2^N(\omega) \,c^2 \|\bE_n\|_{\bL_{\beps}^2(\Omega_0)}^2 &=\frac{\lambda_2^N(\omega)}{\mu_0} \int_{\Omega_0} (\mathbf{k^\theta}(s) \cdot \by)  \,  |\varphi_n(s)|^2 \left( |\nabla_{\by} \psi(\by)|^2 -  \lambda_2^N(\omega) |\psi(\by)|^2 \right)\rmd s\, \rmd \by \\& \quad\quad + \frac{1}{\mu_0} \int_{\Omega_0} \left((1-\mathbf{k^\theta}(s) \cdot \by)^{-1} |\varphi'_n(s)|^2 |\nabla_{\by} \psi(\by)|^2\right)\rmd s\, \rmd \by\\
&\quad\quad+ \lambda_2^N(\omega)\, c^2 \|\nabla u_n\|_{\bL^2_\beps(\Omega_0)}^2.
\end{split}
\label{eqn:diffqneghope}
\end{equation}
Let $j \in \{1,2\}$, as $\psi$ is an eigenfunction of the Neumann Laplacian in $\omega$, it turns out the function $\Xi_j : \by\in \mapsto y_j \nabla_{\by} \psi(\by) \in H_0(\Div,\omega)$ and that $\Div(\Xi_j)(\by) = y_j (\Delta \psi)(\by) + (\partial_j \psi)(\by)$. Indeed, to prove this last formula pick $\varphi \in H^1(\omega)$ and note that
\begin{equation}\label{eq.div}
	y_j \nabla_{\by} \varphi = \nabla_{\by} (y_j \varphi) - \varphi \  (\delta_{1,j},\delta_{2,j})^\top,
\end{equation}
where for $j,k \in \{1,2\}$, $\delta_{j,k}$ is the Kronecker symbol. 
Then, setting $\varphi = \psi$ and taking the divergence of both sides  of equation \eqref{eq.div} yields that  $\Div(\Xi_j)(\by) = y_j (\Delta \psi)(\by) + (\partial_j \psi)(\by) \in L^2(\omega)$ and thus $\Xi_j \in H(\Div,\omega)$.  The fact that  $\Xi_j \in H_0(\Div,\omega) $  follows  from the following integration by parts (which holds by \eqref{eq.div} since $y_j \varphi \in H^1(\omega)$ for $\varphi \in H^1(\omega)$):
\begin{align}
	\int_{\omega} \Xi_j(\by) \cdot \overline{\nabla_{\by} \varphi}(\by) \rmd \by &= \int_{\omega}  \nabla_{\by}\psi(\by) \cdot \overline{\nabla_{\by} (y_j \varphi)(\by)} \rmd \by - \int_\omega \partial_j \psi(\by)\, \overline{\varphi(\by)}\rmd \by \nonumber\\
	& = -\int_\omega \left(y_j(\Delta\psi)(\by) + \partial_j \psi(\by) \right) \overline{\varphi(\by)} \rmd \by \label{2nd:ligne} \\
	& = - \int_\omega \Div(\Xi_j)(\by) \,  \overline{\varphi(\by)}\rmd \by. \nonumber
\end{align}
As a small remark,  observe  that we have used the same (weak) definitions for the space $H(\Div,\omega)$ and $H_0(\Div,\omega)$ as the ones in \S \ref{subsec:funcrame}, which are introduced for open subsets of $\R^3$, while here for the 2-dimensional domain $\omega \subset \R^2$. In this case, since $\omega$ is bounded and Lipschitz, we also have the more standard characterization of such Hilbert spaces via traces (cf. Remark \ref{remark:standardHdiv}).
Hence, using \eqref{2nd:ligne} on each component of the following vector, one gets
\begin{equation*}
\int_\omega \by |\nabla \psi|^2 \rmd \by = - \int_\omega \left( \nabla \psi \, \overline{\psi}  + \by \Delta \psi  \, \overline{\psi} 	\right) \rmd \by = \int_\omega \left( \lambda_2^N(\omega) |\psi|^2 \by - \nabla \psi \, \overline{\psi}  \right) \rmd \by,
\end{equation*}
so that
\begin{eqnarray}
&&\frac{\lambda_2^N(\omega)}{\mu_0} \int_{\Omega_0}\left( (\mathbf{k^\theta}(s) \cdot \by)\, |\varphi_n(s)|^2 \left( |\nabla_{\by} \psi(\by)|^2 -  \lambda_2^N(\omega) |\psi(\by)|^2 \right)\right)\rmd s\, \rmd \by  \nonumber\\ &&=- \frac{\lambda_2^N(\omega)}{\mu_0} \int_{\Omega_0}\left( |\varphi_n(s)|^2\, \mathbf{k^\theta}(s)    \cdot \nabla_\by\ \psi(\by) \, \overline{\psi(\by)}\right)\rmd s\, \rmd \by.
\label{eqn:realpsinablapsi}
\end{eqnarray}
Observe that in the above equation the left-hand side is real and thus the right-hand side necessarily real too. In particular, as $\psi \in H^1(\omega)$, integrating by parts, one has 
\[
	\int_\omega \nabla_{\by}\psi \, \overline{\psi} \rmd \by = \frac12 \int_\omega\left( \nabla_{\by}\psi  \, \overline{\psi}  +  \psi \, \overline{\nabla_{\by} 
	\psi } \right)\rmd \by= \frac12 \int_{\partial\omega} \bn(\by) |\psi|^2 \rmd\sigma(\by)  = \frac{1}{2} \bX,
\]
where $\bn(\by)$ is the outward pointing normal to $\partial\omega$. Combining \eqref{eqn:realpsinablapsi} with \eqref{eqn:diffqneghope} yields
\begin{equation*}
\begin{split}
q[(\bE_n,0)^\top] -  \frac{\lambda_2^N(\omega)}{\varepsilon_0 \mu_0} \|\bE_n\|^2_{\bL_{\beps}^2(\Omega_0)}= & -\frac{\lambda_2^N(\omega)}{2\mu_0} \left(\int_{\R} |\varphi_n(s)|^2 \mathbf{k^\theta}(s)\rmd s\right) \cdot \bX\\
& \ + \frac{1}{\mu_0} \int_{\Omega_0} \left((1-\mathbf{k^\theta}(s) \cdot \by)^{-1} |\varphi'_n(s)|^2 |\nabla_{\by}\psi(\by)|^2\right)\rmd s\, \rmd \by\\& \ + \lambda_2^N(\omega) \, c^2 \|\nabla u_n\|_{\bL^2_\beps(\Omega_0)}^2 .
\end{split}
\end{equation*}
Finally, recalling \eqref{estimate:nablau:eps} to estimate the  term  $ \|\nabla u_n\|_{\bL^2_\beps(\Omega_0)}^2$ and that $c^2=(\varepsilon_0 \mu_0)^{-1}$,  we gets:
\begin{equation*}
\begin{split}
q[(\bE_n,0)^\top] -  \frac{\lambda_2^N(\omega)}{\varepsilon_0 \mu_0} \|\bE_n\|^2_{\bL_{\beps}^2(\Omega_0)} \leq & -\frac{\lambda_2^N(\omega)}{2\mu_0} \left(\int_{\R} |\varphi_n(s)|^2 \mathbf{k^\theta}(s)\rmd s\right) \cdot \bX\\
& \ + \frac{\lambda_2^N(\omega)}{\mu_0 (1-b \|\kappa\|_{L^\infty(\R)})} \|\varphi'_n\|_{L^2(\R)}^2 
\\& \ + \lambda_2^N(\omega)^2 \frac{1}{\mu_0} \frac{b^2 \|\kappa\|_{L^\infty(\R)}}{1-b\|\kappa\|_{L^\infty(\R)}}\|\kappa\|_{L^1(\R)}.
\end{split}
\end{equation*}
By the Dominated Convergence Theorem and \eqref{cutoff:der:estimate} the right-hand side of the above inequality converges to
\begin{equation*}
\frac{\lambda_2^N(\omega)}{2\mu_0}  \left( -\left(\int_{\R}  \mathbf{k^\theta}(s)\rmd s\right) \cdot \bX + 2\lambda_2^N(\omega)\ \frac{b^2 \|\kappa\|_{L^\infty(\R)}}{1-b\|\kappa\|_{L^\infty(\R)}}\|\kappa\|_{L^1(\R)}\right).
\end{equation*}
By hypothesis this quantity is strictly negative, hence for a sufficiently large $n \in \mathbb{N}$ inequality \eqref{eqn:evexistence} holds with $(\bE,\bH)^\top= (\bE_n, 0)^\top$. 
\end{proof}

\subsubsection{Slightly curved waveguides: Proof of Corollary \ref{cor:vrai:exist}} \label{subsubsub:delta}
In this paragraph we prove Corollary \ref{cor:vrai:exist}. 

Consider the base curve $\Gamma$ and the associated relatively adapted parallel frame $(\mathbf{e}_1,\mathbf{e}_2,\mathbf{e}_3)$ stated in the corollary. We recall the definition \eqref{eqn:defX} of $\bY= \int_\R \mathbf{k}(s) \, \rmd s$, where $\mathbf{k}=(k_1,k_2)^\top$ with $k_1,k_2$  continuous functions (cf. \eqref{frame:evolution}). By assumptions, we have that $\bY \neq 0$.

Next, we construct a slightly curved waveguide starting from the base curve $\Gamma$ and the relatively adapted parallel frame $(\mathbf{e}_1,\mathbf{e}_2,\mathbf{e}_3)$, as in \S \ref{section:waveguides}. To do so, for $\delta \in (0,1]$, one introduces
\begin{equation*}
\gamma_\delta: \mathbb{R} \to \mathbb{R}^3, \quad s \mapsto \frac{1}{\delta}\gamma(\delta s).
\end{equation*}
Note that the curvature $\kappa_\delta$ of $\Gamma_\delta:=\gamma_\delta(\R)$ at a point $\gamma(s)$ verifies $\kappa_\delta(s)= \delta \kappa(\delta s)$. In particular $\|\kappa_\delta\|_{L^\infty(\R)}=\delta \|\kappa\|_{L^\infty(\R)}$ and 
\[
    \|\kappa_\delta\|_{L^1(\R)} = \int_\R \delta \kappa(\delta s) \rmd s = \|\kappa\|_{L^1(\R)}.
\]
Furthermore, observe that for $j \in \{1,2,3\}$, setting $\mathbf{e}_j^\delta(s) := \mathbf{e}_j(\delta s)$ for all $\delta \in (0,1]$, we obtain a relatively adapted parallel frame (as defined in the Appendix  \ref{appendix:relatparallel}) to $\Gamma_\delta$, that is
\begin{equation*} 
    \left\{
        \begin{array}{lcl}
            \displaystyle\frac{d \mathbf{e}_1^\delta}{ds}(s) & = & k_1^\delta(s) \mathbf{e}_2^\delta(s) +  k_2^\delta(s) \mathbf{e}_3^\delta(s),  \vspace{3pt} \\
            \vspace{3pt}
            \displaystyle\frac{d \mathbf{e}_2^\delta}{ds}(s) & = & -  k_1^\delta(s) \mathbf{e}_1^\delta(s),\\
            \displaystyle\frac{d \mathbf{e}_3^\delta}{ds}(s) & = & - k_2^\delta(s) \mathbf{e}_1(s),
        \end{array}
    \right.
\end{equation*}
where 
\[
	k_1^\delta(s) := \delta k_1(\delta s),\quad k_2^\delta(s) := \delta k_2(\delta s).
\]
Remark that $k_1^\delta,k_2^\delta:\R \to \R$ are  continuous and satisfy $k_1^\delta(s)^2+k_2^\delta(s)^2=\kappa_\delta(s)^2$ for all $s \in \R$.
Now observe that 
\[
	\bY_\delta := \int_{\R}\begin{pmatrix}k_1^\delta(s)\\ k_2^\delta(s)\end{pmatrix} \rmd s = \delta \int_{\R}\begin{pmatrix}k_1(\delta s)\\ k_2(\delta s)\end{pmatrix}  \rmd s = \int_{\R}\begin{pmatrix}k_1( s)\\ k_2( s)\end{pmatrix}  \rmd s = \bY \neq 0.
\]
Moreover, by assumptions, the cross-section $\omega$ is such that there exists an normalized eigenfunction associated to $\lambda_2^N(\omega)$ for which $\bX \neq 0$.
Then, as mentioned in the statement of the corollary, there exists a unique $\theta_\star \in [0,2\pi)$ such that
\[
\bX \cdot (\mathfrak{R}_{\theta_\star} \bY) = |\bX| \, |\bY|> 0,
\]
where  $\mathfrak{R}_{\theta_\star}$ is the (anticlockwise) rotation in  $\bbR^2$ of angle $\theta_\star$.
Then, as in \eqref{eqn:defetheta}, we define the rotated frame $(\mathbf{e}_1^{\delta, \theta_\star} ,\mathbf{e}_2^{\delta, \theta_\star}, \mathbf{e}_3^{\delta, \theta_\star})$ in term of $(\mathbf{e}_1^\delta(s), \mathbf{e}_2^\delta(s), \mathbf{e}_3^\delta(s))$. 

Define
\begin{equation*}
    \mathbf{k}_\delta^{\theta_\star}(s):= \mathfrak{R}_{\theta_\star}
    \begin{pmatrix}k_1^\delta(s)\\k_2^\delta(s)\end{pmatrix}.
\end{equation*}
as in \eqref{eqn:defktheta}.
Then, using notation \eqref{eqn:defX}, observe that the quantity $\bY_\delta^{\theta_\star}:=  \int_{\R}\mathbf{k}_\delta^{\theta_\star}(s) \, \mathrm{d}s$ associated to the new relatively adapted parallel frame $(\mathbf{e}_1^{\delta, \theta_\star} ,\mathbf{e}_2^{\delta, \theta_\star}, \mathbf{e}_3^{\delta, \theta_\star})$ to $\Gamma_\delta$,
is such that
\begin{equation*}
    \bY^{\theta_\star}_{\delta} =\mathfrak{R}_{\theta_\star} \bY \quad \mbox{ and thus } \quad \bX \cdot  \bY^{\theta_\star}_{\delta} = |\bX| \, |\bY|> 0.
\end{equation*}

Now, for $\delta\in (0,1]$, we consider  the waveguide $\Omega_{\delta}$ defined as in \eqref{diffeo:Phi}, that is
\begin{equation}\label{eq.defOmegadelta}
	\Omega_\delta :=\Phi_{\delta} (\R \times \omega)=\{\Phi_{\delta}(s,y_2,y_3)=\gamma_{\delta}(s) + y_2 \mathbf{e}_2^{\delta, \theta_\star} + y_3 \mathbf{e}_3^{\delta, \theta_\star}  : s\in \R, (y_2,y_3) \in \omega\}.
\end{equation}
On the one hand, note that the condition $b \, \|\kappa\|_{L^\infty(\R)}<1 $ implies that 
$$ \det(J_{{\Phi}_{\delta}}(s,\by)) = 1-\mathbf{k}_{\delta}^{\theta_\star}(s)\cdot \mathbf{y}=1- \delta \, \mathbf{k}^{\theta_\star}(\delta s) \cdot \mathbf{y},
$$ 
and thus Condition \eqref{unif:bounds:ak} holds for $\Phi_{\delta}$ for any $\delta\in (0,1]$.
On the other hand, the  injectivity of $\Phi$ (for $\delta=1$) implies that $\Phi_{\delta}$ is injective for all $\delta \in (0,1]$.
Therefore, $\Phi_{\delta}$ is  $\mathscr{C}^1$-diffeomorphism from $\R \times \omega$ to $\Omega_{\delta}$ for any $\delta \in (0,1]$.
As $\|\kappa_\delta\|_{L^\infty(\R)} = \delta \|\kappa\|_{L^\infty(\R)}$,  one has that: 
\begin{eqnarray*}
&& \displaystyle \mbox{ if } \delta \in (0; \delta_\star) \ \mbox{ with } \ \delta_\star=\min\left( \frac{|\bX| |\bY|}{\left(|\bX| |\bY|+ 2 \, b \, \|\kappa\|_{L^1(\R)} \, \lambda_2^N(\omega)  \right) \, b  \, \|\kappa\|_{L^\infty(\R)} },1\right)
\\[10pt]
&&\displaystyle  \mbox{ then } {\bX \cdot \bY_\delta^{\theta_\star}} >2 \lambda_2^N(\omega) \ \frac{ b^2 \|\kappa_{\delta}\|_{L^\infty(\R)}}{1-b\|\kappa_{\delta}\|_{L^\infty(\R)}}\|\kappa\|_{L^1(\R)}.
\end{eqnarray*}
Thus, Theorem \ref{thm:discspec} combined with the above inequality  yields the Corollary \ref{cor:vrai:exist}.

\section{Discussion on the geometric condition for the existence of discrete spectrum} \label{secnumerics}
In this section we focus on the vectors
$$\bX= \int_{\partial\omega} \bn(\by) |\psi(\by)|^2\rmd \sigma(\by) \ \mbox{ and  } \ \bY= \int_{\R} \mathbf{k}(s) \rmd  s $$
introduced in \eqref{eqn:defX} which play a fundamental role in the sufficient condition \eqref{eqn:ineqdiscspec} of Theorem \ref{thm:discspec} and in the Corollary \ref{cor:vrai:exist} in order to ensure the existence of eigenvalues in the Maxwell gap $(-a_\beta,a_\beta)\setminus\{0\}$. In particular, in both of these results one needs, at least, $\bX \neq 0$ and  $\bY\neq 0$. In this section we investigate  these two conditions.  In \S\ref{subsecisoY}, we show that, for example, if the base curve $\Gamma$ is planar and its two-dimensional curvature has non-zero integral over $\R$, then $\bY \neq 0$. In \S \ref{subsecisos} we provide an example of a cross-section $\omega$ for which the vector $\bX \neq 0$. 
Then in \S \ref{sec:proofcor9} we prove Corollary \ref{cor:infiniteWG} showing that, in fact, there exists an infinite number of waveguides satisfying the hypotheses of Corollary \ref{cor:vrai:exist}, thus having non-empty discrete Maxwell spectrum.
In \S \ref{subsec:crossecbreak} we prove that when the domain has sufficient symmetries, one always have $\bX = 0$. This includes most cross-sections for which $\bX$ can be computed explicitly such as disks or rectangles. In \S \ref{subsecrectangle} we focus on slight perturbations of rectangles and prove that even though $\bX = 0$ when $\omega$ is a rectangle, one can find infinitely many small perturbations of the rectangle for which $\bX \neq 0$. This is proved using shape derivative techniques.
Finally, in \S \ref{section:numerics} we provide some numerical examples.

\subsection{Examples of base curves $\Gamma$ satisfying the geometric condition $\bY\neq 0$} \label{subsecisoY}
Let $({\bf e}_{c,1},{\bf e}_{c,2},{\bf e}_{c,3})$ stand for the canonical basis of $\mathbb{R}^3$.
As in  \S \ref{sec-rec-waveguide}, assume the that the base  curve $\Gamma$ is planar and without of loss of generality that it is contained in the plane $y_3=0$. 
Since the geometry is simpler, it is easy to describe any relatively adapted parallel frame along this type of curve. For example, one can choose the following $\mathscr{C}^1$-smooth frame:
$$
s \in \bbR \mapsto \left(\mathbf{e}_1(s), \mathbf{e}_2(s), \mathbf{e}_3(s)\right) = \left(\gamma'(s), \mathbf{n}(s), \mathbf{e}_{c,3}\right),
$$
where $\mathbf{n}(s)$ denotes the unit  normal vector to the curve at $\gamma(s)$, obtained by rotating the unit tangent vector $\gamma'(s)$ in the plane $y_3=0$ by  $\pi/2$ in the counterclockwise direction.
Note that the obtained frame is a positively oriented orthonormal basis of $\bbR^3$ for each $s \in \R$. 
For this frame, the function $k_2=0$ and the function $k_1$ coincides with the two-dimensional definition of \emph{signed} curvature and satisfies $|k_1|=\kappa$ (cf. \eqref{frame:evolution}).

Moreover, since a relatively adapted parallel frame to a curve is unique up to the initial condition on the transverse plane (cf. Proposition \ref{ex:un:RAPF}), any other relatively adapted parallel frame to $\Gamma$ is of the form $(\mathbf{e}_1^{\bar\theta}(s),\mathbf{e}_2^{\bar\theta}(s),\mathbf{e}_3^{\bar\theta}(s))$ with $\mathbf{e}_1^{\bar\theta}(s)=\mathbf{e}_1(s)=\gamma'(s)$ and $\mathbf{e}_2^{\bar\theta},\mathbf{e}_3^{\bar\theta}$ as in \eqref{eqn:defetheta}, for some constant angle ${\bar\theta}\in [0,2\pi)$. 
That is, the relatively adapted parallel frame is defined up to the choice of a constant (clockwise) rotation in the transverse plane.
The same waveguide, built as in \S \ref{section:waveguides} with this rotated frame $(\mathbf{e}_1^{\bar\theta}(s),\mathbf{e}_2^{\bar\theta}(s),\mathbf{e}_3^{\bar\theta}(s))$, can be built equivalently keeping the initial frame $\left(\gamma'(s), \mathbf{n}(s), \mathbf{e}_{c,3}\right)$ while choosing  as new cross-section the (counterclockwise) rotation $\mathfrak{R}_{\bar\theta} \omega$ of the initial cross-section $\omega$.

Using e.g.  \eqref{frame:evolution:theta} after remarking that $\bar\theta'=0$, one can see that  the frame  satisfies the required differential evolution condition to be a relatively adapted parallel frame:
\begin{equation*}
    \left\{
        \begin{array}{lcl}
            \displaystyle\frac{d \mathbf{e}_1^{\bar\theta}}{ds}(s) & = & k_1^{\bar\theta}(s) \mathbf{e}_2^{\bar\theta}(s) + k_2^{\bar\theta}(s)\mathbf{e}_3^{\bar\theta}(s),  \vspace{3pt} \\
            \vspace{3pt}
            \displaystyle\frac{d \mathbf{e}_2^{\bar\theta}}{ds}(s) & = & - k_1^{\bar\theta}(s)\mathbf{e}_1^{\bar\theta}(s),\\
            \displaystyle\frac{d \mathbf{e}_3^{\bar\theta}}{ds}(s) & = & -k_2^{\bar\theta}(s) \mathbf{e}_1^{\bar\theta}(s)
        \end{array}
    \right.
\end{equation*}
with $k_1^{\bar\theta}(s),k_2^{\bar\theta}(s)$ continuous functions on $\R$ defined as follows
\begin{equation*}
\begin{pmatrix}k_1^{\bar\theta}(s)\\k_2^{\bar\theta}(s)\end{pmatrix} := k_1(s)\begin{pmatrix}
\cos({\bar\theta}) \\
\sin({\bar\theta})
\end{pmatrix}.
\end{equation*}
Finally, observe that
\begin{equation} \label{Yneq0:formula}
 \bY^{\bar\theta} =\int_{\R} k_1(s)\begin{pmatrix}
\cos({\bar\theta}) \\
\sin({\bar\theta})
\end{pmatrix} \, \mathrm{d}s = \begin{pmatrix}
\cos({\bar\theta}) \\
\sin({\bar\theta})
\end{pmatrix}
\int_{\R} k_1(s)
\mathrm{d}s \neq 0  \quad  \mbox{ if only if }
  \quad  \int_{\R} k_1(s) \, \mathrm{d}s \neq 0.
\end{equation}
As  $k_1$ is continuous, one has for instance  $\int_{\R} k_1(s) \, \mathrm{d}s \neq 0$ (and thus $\bY^{\bar\theta} \neq 0$) if $k_1$ has a constant sign (either non-negative or non-positive) and is not identically zero.  This is the case of a planar concave (or convex) curve.
Conversely, if $k_1$ is for instance odd, one has $\int_{\R} k_1(s) \, \mathrm{d}s=0$ and thus $\bY^{\bar\theta}=0$. This is the case of a planar curve that presents symmetries.

\subsection{A geometry satisfying  the  condition $\bX\neq 0$} \label{subsecisos}
In this paragraph, we exhibit a cross-section $\omega\subset \R^2$ for which $\bX \neq 0$. This is the purpose of the following proposition.

\begin{Pro} \label{prop:righttriangleX}
When $\omega$ is an isosceles right triangle then one has $\bX \neq 0$.
\label{prop:isoceslesprop}
\end{Pro}
\begin{proof}[Proof of Proposition \ref{prop:isoceslesprop}]
For the sake of simplicity we assume that $\omega$ is the right triangle of vertices $(0,0), (1,0)$ and $(0,1)$. Any other isosceles right triangle being, up to a translation and a rotation, homothetic to this one.

The modes of the Neumann Laplacian  in $\omega$ are obtained thanks to the one of the Neumann Laplacian $L_{\mathscr{C}}^N$ in the square $\mathscr{C} := [0,1]\times [0,1]$. Indeed, by separation of variables, one easily obtains that    $\sigma(L_\mathscr{C}^N) = \{\pi^2(n^2+m^2)\}_{m,n\in \mathbb{N}\cup\{0\}}$.
Then, to  construct the eigenelements of $L_{\omega}^N$ from those of $L_\mathscr{C}^N$, we introduce the following orthogonal affine symmetry $S$ defined by
$$
S(y_1, y_2) = (1 - y_2, 1 - y_1), \ \forall (y_1,y_2)\in \bbR^2
$$
whose axis of symmetry contain the diagonal of the square $\{(y_1, 1 - y_1) : y_1 \in (0,1)\}$.
First, note that one has obviously $S(\mathscr{C}) = \mathscr{C}$, $
S(\overline{\omega}) \cup \overline{\omega }= \mathscr{C}$ and $S(\omega)\cap \omega=\emptyset$.
Next, observe that any eigenfunction $\phi$ of $L_\mathscr{C}^N$  satisfying the symmetry condition
$
\phi(S(y_1, y_2)) = \phi(y_1, y_2)$ for all $(y_1,y_2)\in \omega$
(in other words the eigenfunction ``$\phi$ is even with respect to $D$'') gives rise to an eigenfunction of $L_\omega^N$ for the same eigenvalue. Conversely, given an eigenfunction of the Neumann Laplacian on $\omega$, one can recover an eigenfunction of $L_\mathscr{C}^N$ (for the same eigenvalue) by extending it symmetrically via the reflection $S$ (across the diagonal 
$D$).\\
In particular, one deduces that $\lambda_2^N(\omega) = \pi^2$  is simple and  associated with the normalized eigenfunction $\psi$ defined by
\[
	\psi(y_1,y_2) := \sqrt{2}\cos(\pi y_2) - \sqrt{2}\cos(\pi y_1), \  \forall \by=(y_1,y_2) \in \overline{\omega}.
\]
Remark that by the divergence theorem there holds
\begin{equation}\label{eqn:sevexprX}
	\bX = \int_{\partial\omega} \bn(\by) |\psi(\by)|^2 \rmd \sigma(\by) = \int_\omega \nabla (\psi^2)(\by) \rmd \by = 2 \int_\omega \psi(\by) \, \nabla \psi(\by) \, \rmd \by.
\end{equation}
Hence, we obtain
\begin{equation*}
\bX = 4 \int_0^1 \int_0^{1-y_1}  \left( \cos(\pi y_2)-\cos(\pi y_1) \right)
\begin{pmatrix}
 \pi \sin(\pi y_1) \\
-\pi \sin(\pi y_2)
\end{pmatrix} \rmd y_2 \, \rmd y_1
= \begin{pmatrix}
1 \\
1
\end{pmatrix},
\end{equation*}
proving that $\bX \neq 0$.
\end{proof}

\subsection{Proof of Corollary \ref{cor:infiniteWG}} \label{sec:proofcor9}

In this paragraph we prove Corollary \ref{cor:infiniteWG}.
We will build an infinite class of admissible waveguides (as in  \S \ref{section:waveguides}) for which $\bX,\bY \neq 0$, and satisfying condition \eqref{eqn:ineqdiscspec} of Theorem \ref{thm:discspec}, thus having discrete eigenvalues in the Maxwell gap. In order to do so, we will make use of the results discussed in the Sections \ref{subsecisoY} and \ref{subsecisos}, and we will apply Corollary \ref{cor:vrai:exist}.\\[4pt]

We  take as cross-section $\omega$ the right isosceles right triangle of vertices $(0,0), (1,0)$ and $(0,1)$. By Proposition \ref{prop:righttriangleX} we have that $\bX \neq 0$.

Then we consider a planar curve $\Gamma$ having a non-negative two-dimensional signed curvature which is not identically zero. For example we could consider a planar parabola parametrized as follows:
\begin{equation*}
\gamma:\R \to \R^3, \quad t \mapsto (\gamma_1(t), \gamma_2(t), \gamma_3(t))^\top = (t,t^2,0)^\top.
\end{equation*}
In this case we have that the following vectors
\begin{equation} \label{planar:RAPF}
\mathbf{e}_1(t) =\frac{\gamma'(t)}{\|\gamma'(t)\|}= \frac{1}{\sqrt{1+4t^2}} \begin{pmatrix}
1\\2t\\0
\end{pmatrix}, \quad
\mathbf{n}(t)=  \frac{1}{\sqrt{1+4t^2}} \begin{pmatrix}
-2t\\1\\0
\end{pmatrix},
\quad  \mathbf{e}_{c,3}= \begin{pmatrix}
0\\ 0\\ 1
\end{pmatrix}
\end{equation}
form a relatively adapted parallel frame along $\Gamma$.
Even if the parametrization $\gamma$ is not arc-length (cf. Remark \ref{remark:arclength}), we can compute the two-dimensional signed curvature at each point  $\gamma(t)$ of the curve for any $t \in \R$:
$$k(t) = \frac{\gamma_1'(t) \gamma_2''(t) - \gamma_2'(t) \gamma_1''(t)}{\left( \gamma_1'(t)^2 + \gamma_2'(t)^2 \right)^{3/2}}=\frac{2}{(1 + 4 t^2)^{\frac32}}>0.$$
 Observe that  $\|k\|_{L^\infty(\R)}=\|k\|_{L^1(\R)}=2$.
Therefore $\bY^{\bar \theta} \neq 0$ for any $\bar \theta \in \R$ (cf. \eqref{Yneq0:formula}).
We can already choose $\bar\theta=\theta_\star \in [0,2\pi)$, the angle maximizing the scalar product as in Corollary \ref{cor:vrai:exist}.
Moreover, since rotating the frame in the transverse plane by a constant angle $\theta_\star$ can be re-interpreted as just choosing as new cross-section $\omega_\star := \mathfrak{R}_{\theta_\star} \omega$ (the counterclockwise rotation by  angle $\theta_\star$ of the initial cross-section $\omega$), as already noted in \S \ref{subsecisoY}, we can assume $\theta_\star=0$.
Therefore, we now only need to show that the  map
\begin{equation*}
\Phi:\R \times \omega \to \R^3, \quad (t,\by) \mapsto \gamma(t) + y_2 \mathbf{n}(t) +y_3 \mathbf{e}_{c,3},
\end{equation*}
which defines the waveguide $\Omega:=\Phi(\R \times \omega)$ as in \S\ref{section:waveguides}, has the property of being a global $\mathscr{C}^1$-diffeomorphism. 
A sufficient condition for that, for this particular geometry, is  $b:=\sup_{\by \in \omega} |\by|<1/2$ (that is Condition \eqref{assumpt:a:kappa} is satisfied).
This  can be seen for example  using Remark \ref{Rem-Thm-Hadamard-Cacciopoli}. Evidently $|\gamma(t)| \to \infty$ as $t \to \infty$. Moreover the waveguide $\Omega$ is clearly simply connected if it verifies the local non-self-intersecting condition \eqref{assumpt:a:kappa}, i.e. if the cross-section's width $b$ is small enough compared to the curvature.

\medskip

Thus we can apply Corollary \ref{cor:vrai:exist} to the waveguide $\Omega$ constructed as such, since it satisfies the assumptions in its statement. This gives rise to an infinite class of (slightly curved) waveguides $\{\Omega_\delta\}_{\delta\in (0, \delta_\star)}$ (defined by \eqref{eq.defOmegadelta}), each of which has a non-empty discrete spectrum.

\subsection{Cross-sections breaking the geometric condition $\bX\neq0$}\label{subsec:crossecbreak}
We say that a domain $\omega \subset \R^2$ has rotational symmetry if there exists a point $\mathfrak{c} \in \omega$ and  $\varphi \in (0, 2\pi)$ such that $\mathfrak{R}_\varphi \omega = \omega$, where $\mathfrak{R}_\varphi$ is the rotation in $\R^2$ of angle $\varphi$ around the point $\mathfrak{c}$.
Without loss of generality, we can assume that $\mathfrak{c}$ coincides with the origin of $\R^2$.

Note that in all the examples provided in \S \ref{subsecisos}, the cross-section $\omega$ has no rotational symmetry. In all these examples, this is actually necessary for $\bX$ to be non-zero, this is the purpose of the following proposition.
\begin{Pro}\label{prop:symm-X0}
Suppose that the first non-trivial Neumann Laplacian eigenvalue $\lambda_2^N(\omega)$ in the cross-section $\omega$ is simple. If $\omega$ presents at least one rotational symmetry then $\bX=0$.
\end{Pro}
\begin{Rem} Note that the hypothesis of $\lambda_2^N(\omega)$ being simple is necessary for Proposition \ref{prop:symm-X0} to hold. Indeed, if one choses $\omega$ to be the equilateral triangle of side $1$ then $\lambda_2^N(\omega)$ has multiplicity 2 (see for instance \cite[\S 3.5]{GN13}). One can compute that there exists a normalized eigenfunction associated with $\lambda_2^N(\omega)$ such that $\bX = \int_{\partial\omega} \bn(\by) |\psi(\by)|^2 \rmd \sigma(\by) \neq 0$. One obstruction for it to hold is that $\bX$ is not linear with respect to $\psi$.
\end{Rem}

\begin{proof}[Proof of Proposition \ref{prop:symm-X0}]
We will identify $\mathfrak{R}_\varphi$ with its $2 \times 2$ matrix representation and observe that if $\mathfrak{R}_\varphi \omega= \omega$, then $\mathfrak{R}_{- \varphi} \, \omega = \omega$ and recall that $\mathfrak{R}_\varphi \mathfrak{R}_\varphi^\top=\mathds{1}_2$.

Since $\lambda_2^N(\omega)$ is simple, there exists a unique normalized eigenfunction $\psi \in H^1(\omega)$  up to a sign. Define $g:=\psi\circ\mathfrak{R}_\varphi$ and observe that $g \in L^2(\omega)$, has $L^2$-norm equal to 1, belongs to $H^1(\omega)$ and
\begin{equation} \label{nabla:g}
\nabla g = \mathfrak{R}_\varphi^\top (\nabla \psi)\circ \mathfrak{R}_\varphi. 
\end{equation}
Therefore, for all $f \in H^1(\omega)$, by several changes of variables there, holds
\begin{equation*}
\begin{split}
\int_\omega \nabla g (\by) \cdot \overline{\nabla f(\by)} \, \rmd \by &= \int_\omega   \left(\mathfrak{R}_\varphi^\top(\nabla \psi)  (\mathfrak{R}_\varphi \by)\right)  \cdot \overline{\nabla f (\by)}  \,\rmd \by \\
&= \int_ {\mathfrak{R}_\varphi \omega} \left(\mathfrak{R}_\varphi^\top(\nabla \psi)  (\by)\right)  \cdot \overline{(\nabla f) (\mathfrak{R}_{-\varphi}\by)} \, \rmd \by \\
&= \int_ {\mathfrak{R}_\varphi\omega} \nabla \psi  (\by)  \cdot \overline{(\nabla( f \circ \mathfrak{R}_{-\varphi})) (\by)} \,\rmd \by \\
&= \lambda_2^N(\omega) \int_{\mathfrak{R}_\varphi\omega} \psi(\by) \, \overline{(f\circ \mathfrak{R}_{-\varphi})(\by)}   \, \rmd \by = \lambda_2^N(\omega) \int_\omega g(\by) \overline{f(\by)} \,\rmd \by.
\end{split}
\end{equation*}
Hence, $g$ is an eigenfunction of the Neumann Laplacian on $\omega$ associated with $\lambda_2^N(\omega)$. Thus, as it is a simple eigenvalue and from \eqref{nabla:g} we deduce that there exists $\eta \in \{-1,1\}$ such that 
\begin{equation}
	\psi\circ\mathfrak{R}_\varphi = \eta \psi,\quad \mathfrak{R}_\varphi \nabla \psi = \eta (\nabla \psi)\circ \mathfrak{R}_\varphi.
	\label{eqn:etasignphir}
\end{equation}
Using the same expression for $\bX$ as the one given in \eqref{eqn:sevexprX}, one observes using \eqref{eqn:etasignphir} that 
\begin{equation*}
\begin{split}
    \mathfrak{R}_\varphi \bX = 2 \int_\omega \psi(\by) \mathfrak{R}_\varphi  (\nabla \psi) (\by) \rmd \by &=  2 \eta^2 \int_\omega \psi(\mathfrak{R}_\varphi \by) \, (\nabla \psi) (\mathfrak{R}_\varphi \by) \rmd \by\\& =   2 \int_{\mathfrak{R}_{-\varphi}\,\omega} \psi(\by) \, \nabla \psi (\by) \rmd \by =  2 \int_{\omega} \psi(\by)  \, \nabla \psi (\by) \rmd \by =  \bX.
\end{split}
\end{equation*}
Thus $\bX \in \ker(\mathfrak{R}_\varphi - \mathds{1}_2)$ which is not possible unless $\varphi \in \{2k\pi\}_{k\in \mathbb{Z}}$. Since $\varphi \in (0,2\pi)$, necessarily $\bX=0$.
\end{proof}

Note that the hypothesis that $\lambda_2^N(\omega)$ is simple is fundamental. Indeed the equilateral triangle presents obvious symmetries. Yet its first non-trivial Neumann Laplacian eigenvalue is double and one can find two orthonormal eigenfunctions both having non-zero  $\bX$.

\subsection{A pertubative approach for rectangular cross-sections} \label{subsecrectangle}
In this paragraph, we use a perturbative approach by means of shape derivatives to investigate the variation of the quantity $\bX$ with respect to the domain $\omega$.

Namely, we start with a domain $\omega$ for which $\bX = 0$ and whose first non-trivial eigenvalue $\lambda_2^N(\omega)$ is simple and we perturb it {\it via} a small parameter $t \in \R$ and a given vector field to obtain a family of domains $(\omega_t)_t$ for which we compute the shape derivative of $(t \mapsto \bX_t)$ at $t = 0$. In \S \ref{subsubsec:adjoint problem} we introduce the so-called adjoint problem for domains verifying $\bX = 0$. In \S \ref{subsubsec:regularityeigenpair} we prove a regularity results on the dependence of the eigenvalue and the eigenfunctions with respect to the parameter $t$. In \S \ref{subsubsec:shapederivativetheoretical} we compute explicitly the shape derivative of $(t \mapsto \bX_t)$ at $t =0$ using a Lagrangian method. Finally, in \S \ref{subsubsec:shapederivativerectangle}, we apply these theoretical results to the specific case of $\omega$ being the rectangle $\omega := (0,\ell)\times (0,L)$ for $L > 0$ with $\ell > L$.

\subsubsection{The adjoint problem for domains verifying $\bX = 0$}\label{subsubsec:adjoint problem}
Let $\omega$ be a Lipschitz domain for which $\bX = 0$ and whose first non-trivial eigenvalue $\lambda_2^N(\omega)$ is simple and associated with a normalized eigenfunction $\psi_\omega$. For a given $\bw$ in the unit sphere $\mathbb{S}^1$ of $\R^2$,  we are interested in finding a solution $q^*$ of the following adjoint problem
\begin{equation}\label{eqn:ajointpb1}
\begin{cases}
-\Delta q^* = \lambda_2^N(\omega) q^*, &\text{in }\omega,\\
\partial_\bn q^* = -2 \psi_\omega  \, (\bn \cdot \bw), & \text{on }\partial\omega.
\end{cases}
\end{equation}
Actually, the following lemma, proves that there always exists a weak solution to \eqref{eqn:ajointpb1} in the following sense: $q^* \in H^1(\omega)$ is a weak solution of \eqref{eqn:ajointpb1} if for all $v \in H^1(\omega)$ there holds
\begin{equation*}
    \int_\omega \nabla q^* \cdot \overline{\nabla v} \, \rmd \by - \lambda_2^N(\omega) \int_\omega q^* \overline{v} \, \rmd \by = - 2\int_{\partial\omega} \psi_\omega   \overline{v} \, ( \bn \cdot \bw) \, \rmd \sigma(\by).
\end{equation*}
\begin{Lem}\label{lem:solweakadjoint}
Under the assumption that $\bX = 0$ and $\lambda_2^N(\omega)$ is simple, there exists a unique weak solution $q^*$ to the adjoint problem \eqref{eqn:ajointpb1} orthogonal in $L^2(\omega)$ to $\psi_\omega$.
\end{Lem}
The proof of Lemma \ref{lem:solweakadjoint} is by using Lax-Milgram theorem in adequate functional spaces. Note that the condition $\bX = 0$ could be re-interpreted as a necessary condition to use Fredholm's alternative.
\begin{proof}[Proof of Lemma \ref{lem:solweakadjoint}]
Let $\mathscr{H}$ be the orthogonal in $L^2(\omega)$ of ${\rm span}(1,\psi_\omega)$ and remark that any $w \in L^2(\omega)$ decomposes as
\begin{equation}
        w = \underline{w} + ( w, \psi_\omega)_{L^2(\omega)}\psi_\omega + w^\perp,\quad {\text{where} }\quad \underline{w} := \int_{\omega} w \, \rmd \by \quad\text{and}\quad w^\perp \in \mathscr{H}.
        \label{eqn:decomporthogood}
\end{equation}
By definition, if $q^* \in H^1(\omega)$ is a weak solution  to \eqref{eqn:ajointpb1} it verifies for all $v \in H^1(\omega)$
\[
a[q^*,v] = \ell[v]
\]
where for $u,v \in H^1(\omega)$, we have set
\[
	a[u,v] := \int_\omega (\nabla u)\cdot \overline{(\nabla v)} \, \rmd \by - \lambda_2^N(\omega)\int_\omega u \overline{v}\rmd \by,\quad \ell[v] := -2 \int_{\partial\omega}\psi_\omega \,  \overline{v} \, (\bn\cdot \bw) \,\rmd \sigma(\by).
\]
Now, decompose any $q^*,v \in H^1(\omega)$ as in \eqref{eqn:decomporthogood}  by assuming that $q_*$ is orthogonal in $L^2(\omega)$ to $\psi_{\omega}$. Because of orthogonality relations, the weak formulation rewrites
\begin{align*}
        -\lambda_2^N(\omega) \, |\omega| \, \underline{q^*}\ \underline{v} +a[(q^*)^\perp,v^\perp] & = -2 \overline{\underline{v}}\int_{\partial\omega} \psi_\omega \, (\bn \cdot \bw) \, \rmd  \sigma(\by) -2 \overline{( v, \psi_\omega)}_{L^2(\omega)} (\bX\cdot \bw) \\& \quad- 2 \int_{\partial\omega} \psi_\omega  \, \overline{v^\perp} \, (\bn\cdot\bw) \, \rmd  \sigma(\by)\\
        & = -2 \overline{\underline{v}}\int_{\partial\omega} \psi_\omega (\bn \cdot \bw)\, \rmd  \sigma(\by)- 2 \int_{\partial\omega} \psi_\omega \, \overline{v^\perp}\,  (\bn\cdot\bw)\,  \rmd  \sigma(\by)
\end{align*}
where the last equality holds because $\bX = 0$. Choosing $v \equiv 1$, we get
\[
    \underline{q^*} = \frac{2}{\lambda_2^N(\omega)|\omega|}\int_{\partial\omega} \psi_\omega (\bw \cdot \bn) \rmd \sigma(\by).
\]
As  $q_*$ is orthogonal in $L^2(\omega)$, we are left with finding $(q^*)^\perp$. To this aim, one notices that the decomposition \eqref{eqn:decomporthogood} obviously implies that  $\mathscr{H} \cap H^1(\omega)$ is a Hilbert space as a closed subspace of $H^1(\omega)$ for the $H^1$-norm. Then,  one can  consider the sesquilinear form $a^\perp := a|_{(\mathscr{H} \cap H^1(\omega))^2}$ and the antilinear form $b^\perp := l|_{\mathscr{H} \cap H^1(\omega)}$.  Note that $a^\perp$ is continuous on $(\mathscr{H} \cap H^1(\omega))^2$ (because $a$ is  continuous on $H^1(\omega)^2$). Moreover, it is also coercive. Indeed, if one takes $u \in \mathscr{H}\cap H^1(\omega)$ and $\delta \in (0,1)$, then letting $\lambda_3^N(\omega)$ to be the third Neumann eigenvalue in $\omega$,  the min-max principle yields
\[
	a[u,u] =\int_\omega |\nabla u|^2\rmd \by - \lambda_2^N(\omega)|u|^2\rmd \by \geq \delta \|\nabla u\|_{\bL^2(\omega)}^2 + \left((1-\delta) \lambda_3^N(\omega) - \lambda_2^N(\omega)\right)\|u\|_{L^2(\omega)}^2.
\]
Thus, as $\lambda_2^N(\omega)$ is simple one has $\lambda_3^N(\omega) > \lambda_2^N(\omega)$ and choosing $\delta$ sufficiently small, each term on the right-hand side is positive yielding the existence of $c_\delta > 0$ such that $a[u,u]\geq c_\delta \|u\|_{H^1(\omega)}^2$. 
The antilinear form $\ell$ is continuous on $\mathscr{H}\cap H^1(\omega)$ thanks to the continuity of the Dirichlet trace from $H^1(\omega)\to L^2(\partial\omega)$. 
Thus, one has the existence and uniqueness  of $(q^*)^\perp$ by the Lax-Milgram theorem.  Finally, one concludes, by tracing back the chain of equivalences, that 
$q^* = \underline{q^*} + (q^*)^\perp$
is the unique weak solution to the adjoint problem \eqref{eqn:ajointpb1}, which is  $L^2(\omega)$-orthogonal to $\psi_{\omega}$.
 
\end{proof}
\subsubsection{Eigenpair regularity with respect to uni-parametric domain perturbations}\label{subsubsec:regularityeigenpair}
Consider a  bounded Lipschitz domain $\omega\subset \R^2$ and  the Sobolev space
$$
W^{1,\infty}(\R^2,\R^2):=\Big\{ G=(G_1,G_2)\in L^{\infty}(\bbR^2)^2 \mid \partial_{y_j} G_i \in L^{\infty}(\bbR^2) \mbox{ for }i,j=1,2 \Big\}.
$$
The space $W^{1,\infty}(\R^2,\R^2)$ is a Banach space endowed with the norm $\| \cdot\|_{W^{1,\infty}(\R^2,\R^2)}$ defined by
$$
\| G\|_{W^{1,\infty}(\R^2,\R^2)}:=\sum_{i=1}^2\|G_i\|_{L^{\infty}(\R^2)}+ \sum_{i=1}^2 \sum_{j=1}^2\|\partial_{y_j} G_i\|_{L^{\infty}(\R^2)},\ \forall G\in W^{1,\infty}(\R^2,\R^2).
$$

For $T>0$ and a fixed generic vector field $V: \R^2 \to \R^2$ of class $W^{1,\infty}(\R^2,\R^2)$, we define the map
\begin{equation} \label{def:Thethat}
\Theta: \left\{
			\begin{array}{ccl}
				[-T,T] &\to& W^{1,\infty}(\R^2, \R^2)\\
				t & \mapsto & I+t V
			\end{array}.
		\right.
\end{equation}
Note that $\Theta(0)$ is the identity and that $\Theta'(0) = V$. Let $J_{\Theta(t)}$ denote the Jacobian matrix of $\Theta(t)$. For sufficiently small $T > 0$, the map $t \mapsto J_{\Theta(t)}$ belongs to the class $\mathscr{C}^{1}([-T, T], L^\infty(\mathbb{R}^2, \mathbb{R}^{2 \times 2}))$, and the determinant map $t \mapsto \det(J_{\Theta(t)})\in \mathscr{C}^{1}([-T, T], L^\infty(\mathbb{R}^2, \mathbb{R}))$ is a positive function, uniformly bounded from below with respect to $t\in [-T,T]$ for the $L^\infty(\mathbb{R}^2, \mathbb{R}))$-norm. As a result, the inverse Jacobian $t \mapsto J_{\Theta(t)}^{-1}$ is well-defined and also belongs to $\mathscr{C}^{1}([-T, T], L^\infty(\mathbb{R}^2, \mathbb{R}^{2 \times 2}))$, and the reciprocal of the determinant, $t \mapsto \det(J_{\Theta(t)})^{-1}$, lies in $\mathscr{C}^{1}([-T, T], L^\infty(\mathbb{R}^2, \mathbb{R}))$. Therefore, for all $t \in [-T, T]$, the map $\Theta(t)$ is clearly injective and defines a $W^{1,\infty}$-diffeomorphism from $\omega$ onto $\omega_t := \Theta(t)(\omega)$ for all $t \in [-T,T]$. \medskip 

For $t$ fixed in $[-T,T]$, the spectrum of the Neumann Laplacian $L_{\omega_t}^N$  in $\omega_t$ consist of a sequence of  positive  eigenvalues $(\lambda_n^N)_{n\geq 1}$ of finite multiplicity, tending to $+\infty$ and listed  (with their  multiplicity) as follows 
$0=\lambda_1^{N}(\omega_t) \leq \lambda_2^{N}(\omega_t) \leq  \lambda_3^{N}(\omega_t) \leq  \ldots  $. One first establishes the following Lemma on the regularity of theses eigenvalues as function of $t$. 
\begin{Lem}\label{Lem-Lipschitz}
For $j\in \mathbb{N}$, the function $t\mapsto \lambda_2^{N}(\omega_t)$ are  Lipschitz continuous on $[-T, T]$.
\end{Lem}
The Lipschitz continuity of Neumann eigenvalues with respect to domain deformations induced by locally Lipschitz continuous homeomorphisms has been already established in \cite{Lamb-05}. Since we focus here on the much simpler case of a one-parameter family of domain perturbations, we provide a shorter and more elementary proof in the Appendix \ref{Sec-reg-Lips-eigen-Neumann}.

Assuming that $\lambda_2^N(\omega)$ is simple, the aim of this second lemma is to show that, provided $T$ is sufficiently small, the second Neumann eigenvalue $\lambda_2^N(\omega_t)$ and a corresponding normalized eigenfunction $\psi_{t}$ can be smoothly followed with respect to $t$.

\begin{Lem}\label{lem-implcitefunction}
Let $\omega$ be a bounded Lipschitz domain in $\R^2$ for which $\lambda_2^N(\omega)$ is simple and associated with a normalized eigenfunction $\psi_\omega \in H^1(\omega)$. There exists $t_\star > 0$ and a $\mathscr{C}^1$-smooth map
\[
\Psi: \left\{\begin{array}{ccc}(-t_\star,t_\star) &\to &  \R \times H^1(\omega)\\
t & \mapsto & (\lambda^N(t),\varphi_t))
\end{array}\right.
\]
such that there holds
\begin{enumerate}[label=\normalfont(\roman*)]
    \item $(\lambda^N(0),\varphi_0) = (\lambda_2^N(\omega),\psi_\omega)$,
    \item for all $t\in(-t_\star,t_\star)$, the function $\psi_t := \varphi_t \circ (\Theta(t))^{-1} \in H^1(\omega_t)$ is an eigenfunction of the Neumann Laplacian in $\omega_t$ associated with the simple eigenvalue $\lambda^N(t) = \lambda_2^N(\omega_t)$.
\end{enumerate}
\label{lem:implicit_modes}
\end{Lem}
We first note that such a regularity result is stated without proof in \cite[Thm.~2.5.7]{Henrot}, under the assumption that the domain is of class~$\mathscr{C}^3$ for the regularity of the eigenfunction. The proof given here closely follows the approach used to prove \cite[Theorem~5.5.1]{HP18}, which considers the regularity of the solution to a Neumann-type  problem for the Laplacian with respect to the parameter $t$ in domains $\omega_t$ (under the assumption that  the boundary of $ \omega$ is $\mathscr{C}^1$). However, it is adapted to the present setting, which is slightly different, as we consider here not only an eigenvalue problem, but also a weaker regularity for the domain $\omega$, namely Lipschitz regularity. The argument relies on an application of the implicit function theorem within a suitable functional framework. The proof is divided into three steps.

\begin{proof}[Proof of Lemma \ref{lem:implicit_modes}]
{\bf Step 1: pull back of the eigenvalue problem to $\omega$.} \\[4pt]
Let us assume that $t \in (-T,T)$ and let $(\lambda,\tilde{\varphi}) \in \R \times H^1(\omega_t)$ be an eigenpair of the Neumann Laplacian in $\omega_t$.
Thus, for all $\tilde{f} \in H^1(\omega_t)$, one as
\[
    \int_{\omega_t} \nabla \tilde{\varphi}\cdot\overline{\nabla \tilde{f}} \, \rmd \by = \lambda \int_{\omega_t} \tilde{\varphi} \overline{\tilde{f}} \, \rmd \by.
\]
Pulling back the problem to the fixed domain $\omega$, setting $\varphi = \tilde{\varphi}\circ \Theta(t)$, for all $f \in H^1(\omega_t)$ the variational formulation reads for all $f \in H^1(\omega)$:
\[
    \int_\omega (J_{\Theta(t)}^{-\top} \nabla \varphi) \cdot (J_{\Theta(t)}^{-\top} \overline{\nabla f}) \det(J_\Theta(t))\, \rmd  \by = \lambda \int_\omega\varphi \overline{f} \det(J_{\Theta(t)}) \,\rmd \by.
\]
with $J_{\Theta(t)}$ the Jacobian  matrix associated to the $W^{1,\infty}$-diffeomorphism $\Theta(t)$ introduced in \eqref{def:Thethat}.
Driven by this formulation, for any $u \in H^1(\omega)$, we introduce the continuous antilinear form 
\[
    F(t,\lambda,u) : \left\{
    \begin{array}{ccl}
        H^1(\omega) & \to & \R   \\
         f & \mapsto &  \displaystyle\int_\omega (J_{\Theta(t)}^{-\top} \nabla u) \cdot (J_{\Theta(t)}^{-\top} \overline{\nabla f}) \det(J_\Theta(t))\, \rmd \by - \lambda \int_\omega u \overline{f} \det(J_{\Theta(t)}) \, \rmd \by.
    \end{array}
    \right.
\]
By definition $F(t,\lambda,\varphi) \in (H^1(\omega))'$ and we consider
\begin{equation}\label{eq.defmathcalF}
    \mathscr{F} : \left\{
    \begin{array}{ccl}
        (-T,T)\times \R \times H^1(\omega) & \to & (H^1(\omega))'\times \R \\
         (t,\lambda,u) & \mapsto &   \displaystyle \Big(F(t,\lambda,u),\Big(\int_\omega |u|^2 \det(J_{\Theta(t)}) \, \rmd \by\Big) - 1\Big).
    \end{array}
    \right.
\end{equation}
{\bf Step 2: Application of the implicit function theorem to $\mathcal{F}$.} \\[4pt]
Note that as constructed, there holds $\mathscr{F}(0,\lambda_2^N(\omega),\psi_\omega) = (0,0)$. Our aim is to apply the implicit function theorem to $\mathscr{F}$ at the point $(0,\lambda_2^N(\omega),\psi_\omega) \in (-T,T)\times \R \times H^1(\omega)$. First, let us prove that $F$ is of class $\mathscr{C}^1$. Consider $(t,\lambda,u) \in (-T,T) \times \R \times H^1(\omega)$ and note that there holds $ \Theta(t+h) = \Theta(t) + h V$ yielding
\begin{equation}\label{eq.taylorinverse}
    J_{\Theta(t+h)} = J_{\Theta(t)} + h J_{V}, \quad J_{\Theta(t+h)}^{-\top} = J_{\Theta(t)}^{-\top} - h J_{\Theta(t)}^{-\top}J_V^{\top}J_{\Theta(t)}^{-\top} + h^2 R_h.
\end{equation}
where $R_h \in L^\infty(\R^2,\R^{2\times2})$ uniformly in $h$ for sufficiently small $h$. Similarly, one gets
\begin{equation}\label{eq.taylordet}
    \det(J_{\Theta(t+h)}) = \det(J_{\Theta(t)}) + h \det(J_{\Theta(t)})\Tr(J_{\Theta(t)}^{-1}J_V) + h^2 r_h
\end{equation}
where $r_h \in L^\infty(\R^2,\R)$ uniformly in $h$ for sufficiently small $h$. A rather lengthy but straightforward computation yields that $F$ is Fréchet differentiable at $(t,\lambda,u)$ and that its Fréchet derivative in a direction $(h,\mu,v)\in \R\times\R\times H^1(\omega)$ verifies for all $f \in H^1(\omega)$ 
\begin{align*}
 \big(D F_{(t,\lambda,u)}(h,\mu,v)\big)(f) &:= \int_{\omega} (J_{\Theta(t)}^{-\top}\nabla v) \cdot (J_{\Theta(t)}^{-\top}\overline{\nabla f}) \det(J_{\Theta(t)})\,  \rmd \by \\&\quad-  h \int_\omega (J_{\Theta(t)}^{-\top}J_V^\top J_{\Theta(t)}^{-\top} \nabla u) (J_{\Theta(t)}^{-\top} \overline{  \nabla f}) \det(J_{\Theta(t)}) \, \rmd \by \\&\quad- h \int_\omega (J_{\Theta(t)}^{-\top} \nabla u) (J_{\Theta(t)}^{-\top}J_V^\top J_{\Theta(t)}^{-\top} \overline{ \nabla f}) \det(J_{\Theta(t)}) \, \rmd \by\\
    & \quad+h \int_\omega (J_{\Theta(t)}^{-\top} \nabla u) (J_{\Theta(t)}^{-\top} \overline{\nabla f}) \det(J_{\Theta(t)})\Tr(J_{\Theta(t)}^{-1}J_V) \, \rmd \by\\
    & \quad- h \lambda \int_\omega u \overline{ f} \det(J_{\Theta(t)})\Tr(J_{\Theta(t)}^{-1}J_V)\, \rmd \by\\
    & \quad- \mu \int_\omega u \overline{f} \det(J_{\Theta(t)}) \, \rmd \by - \lambda \int_\omega v \overline{f} \det(J_{\Theta(t)})\, \rmd \by.
\end{align*}
One can also check from its expression that $DF$ is a continuous map from $(-T,T)\times \R \times H^1(\omega)$ to the space of continuous linear maps $\mathscr{L}(\R \times \R \times H^1(\omega),(H^1(\omega))')$. Thus $F$ is of class $\mathcal{C}^1$ on $(-T,T) \times \R \times H^1(\omega)$ . Note that for $(\mu,v) \in \R \times H^1(\omega)$ and $f \in H^1(\omega)$, one has
\begin{equation}\label{eq.DefDFaupointsinteressants}
    \Big(D F_{(0,\lambda_2^N(\omega),\psi_\omega)}(0,\mu,v)\Big)(f) = \int_\omega \nabla v\cdot \overline{\nabla f}\, \rmd \by - \mu \int_\omega \psi_\omega \overline{f} \, \rmd \by - \lambda_2^N(\omega) \int_\omega v \overline{f} \, \rmd \by.
\end{equation}
Set $\mathscr{\widetilde{F}}(t,\lambda,u) := \int_\omega u^2 \det(J_{\Theta(t)})\rmd \by - 1$ and remark that it is also $\mathcal{C}^1$ on  $(-T,T) \times \R \times H^1(\omega)$ with Fréchet derivative at a point $(h,\mu,v) \in \R \times \R \times H^1(\omega)$ given by
\begin{equation}\label{eq.DefDwidetildeFaupointsinteressants}
    (D\mathscr{\widetilde{F}}_{(t,\lambda,u)})(h,\mu,v) = 2 \int_\omega u \overline{v} \det(J_{\Theta(t)})\, \rmd \by + h \int_{\omega}|u|^2 \det(J_{\Theta(t)})\Tr(J_{\Theta(t)}^{-1}J_V)\, \rmd \by.
\end{equation}
Thus,  $\mathscr{F}$,  defined by \eqref{eq.defmathcalF} is of class $\mathscr{C}^1$ and  $D\mathscr{F}$ is a $\mathscr{C}^0$-map from $(-T,T)\times \R \times H^1(\omega)$ to $(\R\times\R\times H^1(\omega))'$ and  one has (by  \eqref{eq.DefDFaupointsinteressants} and \eqref{eq.DefDwidetildeFaupointsinteressants} evaluated at $(t, \lambda,u)=(0,\lambda_2^N(\omega),\psi_\omega)$):
\begin{equation*}
    \big(D \mathscr{F}_{(0,\lambda_2^N(\omega),\psi_\omega)}(0,\mu,v)\big)(f) = \Big(\int_\omega \nabla v\cdot \overline{\nabla f} \rmd \by - \mu \int_\omega \psi_\omega \overline{f} \rmd \by - \lambda_2^N(\omega) \int_\omega v \overline{f} \, \rmd \by, 2 \int_\omega \psi_\omega \overline{v} \,\rmd \by\Big).
\end{equation*}
We prove now that  $(\mu,v) \mapsto D \mathscr{F}_{(0,\lambda_2^N(\omega),\psi_\omega)}(0,\mu,v)$ is an isomorphism. To this aim, take $(\Lambda,x) \in (H^1(\omega))' \times \R$  and note that by Riesz representation theorem, there exists a unique $g \in H^1(\omega)$ such that for all $f \in H^1(\omega)$ there holds
\begin{equation}
    \Lambda(f) = \int_{\omega} \nabla g \cdot \overline{\nabla f}\, \rmd \by + \int_\omega g \overline{f}\, \rmd \by.
    \label{eqn:RieszLambda}
\end{equation}
We are looking for $(\mu,v) \in \R \times H^1(\omega)$ such that for all $f \in H^1(\omega)$
\begin{equation}\label{eqn:surjectivityconditions}
    \int_\omega \nabla v\cdot \overline{\nabla f} \, \rmd \by - \mu \int_\omega \psi_\omega \overline{f} \rmd \by - \lambda_2^N(\omega) \int_\omega v \overline{f} \, \rmd \by = \Lambda(f),\quad x = 2 \int_\omega \psi_\omega \overline{v} \, \rmd \by.
\end{equation}
Taking $f = \psi_\omega$ gives $\mu = - \Lambda(\psi_\omega)$. Remark that the first equation in \eqref{eqn:surjectivityconditions} can be rewritten as follows:  for all $f\in H^1(\omega)$ there holds $a[v,f] = \ell[f]$, where
\[
    a: (w,f) \in H^1(\omega)^2 \mapsto \int_\omega \nabla w\cdot \overline{\nabla f} \, \rmd \by  - \lambda_2^N(\omega) \int_\omega w \overline{f}\, \rmd \by
\]
and
\[
    \ell: f \in H^1(\omega) \mapsto \Lambda(f) + \mu \int_{\omega}\psi_\omega \overline{f}\, \rmd \by.
\]
We look for $v \in H^1(\omega)$ decomposed as
\[
    v =  c_1+ c_2 \psi_\omega + v^\perp
\]
for some $c_1,c_2 \in \R$ and $v^\perp \in H^1(\omega)\cap \mathscr{H}$, where $\mathscr{H}$ is the orthogonal in $L^2(\omega)$ of ${\rm span}(1,\psi_\omega)$. Note that the second equation in \eqref{eqn:surjectivityconditions} immediately yields $c_2 = \frac{x}2$. Let us look for $v^\perp$ verifying for all $f^\perp \in H^1(\omega)\cap \mathscr{H}$ the equation $a^\perp[v^\perp,f^\perp] = \ell^\perp[f^\perp]$ where we have set $a^\perp := a|_{(H^1(\omega)\cap \mathscr{H}^\perp)^2}$ and $\ell^\perp := \ell|_{H^1(\omega)\cap \mathscr{H}^\perp}$.

As defined $a^\perp$ is a continuous sesquilinear form of $(H^1(\omega)\cap \mathscr{H})^2$ and $\ell^\perp$ is a antilinear form on $H^1(\omega)\cap \mathscr{H}$. Note that $a^\perp$ is also coercive because if one takes $w \in H^1(\omega) \cap \mathscr{H}$ and $\delta \in (0,1)$, then letting $\lambda_3^N(\omega)$ to be the third Neumann eigenvalue in $\omega$,  the min-max principle yields
\[
	a^\perp[w,w] =\int_\omega |\nabla w|^2\rmd \by - \lambda_2^N(\omega)\int_\omega|w|^2\rmd \by \geq \delta \|\nabla w\|_{\bL^2(\omega)}^2 + \left((1-\delta) \lambda_3^N(\omega) - \lambda_2^N(\omega)\right)\|w\|_{L^2(\omega)}^2.
\]
Thus, as $\lambda_2^N(\omega)$ is simple one has $\lambda_3^N(\omega) > \lambda_2^N(\omega)$ and choosing $\delta$ sufficiently small, each term on the right-hand side is positive yielding the existence of $c_\delta > 0$ such that $a^\perp[w,w]\geq c_\delta \|w\|_{H^1(\omega)}^2$. Hence by the Lax-Milgram theorem, there exists a unique $v^\perp \in H^1(\omega) \cap \mathscr{H}$ such that for all $f^\perp \in H^1(\omega)\cap \mathscr{H}$ the equation $a^\perp[v^\perp,f^\perp] = \ell^\perp[f^\perp]$ holds.

Now, decompose any function $f \in H^1(\omega)$ as
\[
    f = \underline{f} +\left(\int_\omega f \, \overline{\psi_\omega} \, \rmd \by\right)\psi_\omega + f^\perp
\]
where $\underline{f} := \int_{\omega}f \rmd \by$ and $f^\perp \in H^1(\omega)\cap \mathscr{H}$ and remark that there holds
\[
    a[v,f] = a[c_1,\underline{f}] + a^\perp[v^\perp,f^\perp],\quad \ell[f] = \underline{f} \int_\omega g \, \rmd \by+ \Lambda(f^\perp) = \underline{f} \Lambda(1) + \Lambda(f^\perp),
\]
where we have used expression \eqref{eqn:RieszLambda} and the fact that $\mu=-\Lambda(\psi_{\omega})$ to simplify the term $\Lambda(f)$.
Thus, using that $a^\perp[v^\perp,f^\perp] = b^\perp[f^\perp] = \Lambda(f^\perp)$, the expression $a[v,f] = b[f]$ rewrites
\[
    a[c_1,\underline{f}] = \underline{f} \Lambda(1).
\]
It reads
\[
    -\lambda_2^N(\omega)\, c_1 \, \underline{f}\, |\omega| = \underline{f} \Lambda(1).
\]
As this should be true for all $f \in H^1(\omega)$ we obtain
\[
    c_1 = -\frac{\Lambda(1)}{\lambda_2^N(\omega) |\omega|}.
\]
Hence we have proven that necessarily
\[
    \mu = - \Lambda(\psi_\omega), \quad v = -\frac{\Lambda(1)}{\lambda_2^N(\omega) |\omega|} + \frac{x}2\psi_\omega + v^\perp. 
\]
In particular the sought expression of $(\mu,v)$ is unique and by construction the obtained expressions verify $D\mathscr{F}(0,\lambda_2^N(\omega),\psi_\omega)(0,\mu,v) = (\Lambda,x)$.

Now, by the implicit function theorem, there exists $t_\star > 0$ and a map $\Psi$ of class $\mathscr{C}^1$ defined as
\[
\Psi: \left\{\begin{array}{ccc}(-t_\star,t_\star) &\to &  \R \times H^1(\omega)\\
t & \mapsto & (\lambda^N(t),\varphi_t)
\end{array}\right.
\]
such that $\Psi(0) = (\lambda_2^N(\omega), \psi_\omega)$ and verifying  $\mathscr{F}(t, \Psi(t))=\mathscr{F}(t,\lambda^N(t),\varphi_t) = (0,0)$ for all $t \in (-t_\star,t_\star)$. In particular, for all $f \in H^1(\omega)$ there holds
\[
   \int_\omega (J_{\Theta(t)}^{-\top} \nabla \varphi_t) \cdot (J_{\Theta(t)}^{-\top} \overline{\nabla f}) \det(J_\Theta(t))\rmd \by- \lambda^N(t) \int_\omega \varphi_t \overline{f} \det(J_{\Theta(t)})\rmd \by= 0
\]
and moreover one has
\[
    \int_\omega |\varphi_t|^2 \det(J_{\Theta(t)}) \rmd \by= 1.
\]
Setting $\psi_t := \varphi_t \circ (\Theta(t))^{-1}$ pushing forward the equation to $\omega_t$ this reads, for all $\tilde{f} \in H^1(\omega_t)$
\[
    \int_{\omega_t}\nabla \psi_t\cdot \overline{\nabla \tilde{f}} \rmd \by= \lambda_2^N(t) \int_{\omega_t} \psi_t \overline{\tilde{f}} \rmd \by,\quad \int_{\omega_t}|\psi_t|^2 \rmd \by= 1
\]
and thus $\psi_t$ is a normalized eigenmode of the Neumann Laplacian associated with $\lambda^N(t)$.\\[6pt]
{\bf Step 3: Continuity argument.}\\[4pt]
In fact, for sufficiently small values of $t$, we necessarily have $\lambda^N(t) = \lambda_2^N(\omega_t)$, and this eigenvalue remains simple. This follows from the simplicity of $\lambda^N(0) = \lambda_2^N(\omega)$, together with the fact that, for each $j \in \mathbb{N}$, the functions $t \mapsto \lambda_j^N(\omega_t)$ are, by Lemma \ref{Lem-Lipschitz}, continuous and indeed, even  Lipschitz continuous on $[-T, T]$. Therefore, possibly after choosing a smaller value $t_*> 0$ than the one provided by the implicit function theorem in Step 2, $\lambda^N(t) = \lambda_2^N(\omega_t)$ and is simple on the interval $(-t_*, t_*)$.


\end{proof}
\begin{Rem}
We point out that in the proof of the Lemma \ref{lem-implcitefunction}, one can obtain an explicit expression for the eigenpair derivative at $t = 0$, namely
$\big(\left.\frac{\rmd \lambda_2^N(\omega_t)}{\rmd t}\right|_{t=0}, \left.\frac{\rmd \varphi_t}{\rmd t}\right|_{t = 0}\big)$,
via the implicit function theorem, by differentiating with respect to $t$ the equation
$\mathscr{F}(t, \lambda_2^N(t), \varphi_t) = (0,0)$ and evaluating the resulting expression at $t = 0$.
Using the invertibility of 
$D \mathscr{F}_{(0, \lambda_2^N(\omega), \psi_\omega)}$, this yields
\[
\Big(\left.\frac{\rmd \lambda_2^N(\omega_t)}{\rmd t}\right|_{t=0}, \left.\frac{\rmd \varphi_t}{\rmd t}\right|_{t = 0}\Big)
= - \Big[ D \mathscr{F}_{(0, \lambda_2^N(\omega), \psi_\omega)} \Big]^{-1}
\Big(0, D \mathscr{F}_{(0, \lambda_2^N(\omega), \psi_\omega)}(1, 0, 0) \Big).
\]
Furthermore, since $\psi_t := \varphi_t \circ (\Theta(t))^{-1}$, one obtains that
$\left.\frac{\rmd \psi_t}{\rmd t}\right|_{t=0}
= \left.\frac{\rmd \varphi_t}{\rmd t}\right|_{t=0}+ \psi_\omega \circ (\Theta^{-1})'(0),$
with $(\Theta^{-1})'(0) = -\Theta^{-1}(0) \Theta'(0) \Theta^{-1}(0)=-V$.
We will not make explicit use of these formulas in the current proof, but they may be useful for other applications. Therefore, we leave the simplification of these derivatives to the interested reader.

\end{Rem}

\subsubsection{Shape derivative of $(t \mapsto \bX_t)$}\label{subsubsec:shapederivativetheoretical}
Let $\omega$ be a  bounded Lipschitz domain for which $\bX = 0$ and assume that $\lambda_2^N(\omega)$. By Lemma \ref{lem:implicit_modes}, there exists $t_\star > 0$ such that for all $t\in(-t_\star,t_\star)$, $ (\lambda^N(t),\psi_t)$ is a eigenpair of the Neumann Laplacian in $\omega_t$ $\psi_t$ a normalized eigenfunction.
Moreover, $t\mapsto (\lambda^N(t),\psi_t)$ is of class $\mathscr{C}^1$ on $(-t_\star, t_\star)$, and $\lambda^N(t)$ remains simple and coincides with the second Neumann eigenvalue $\lambda_2^N(\omega_t)$. Therefore, in what follows, we write $\lambda_2^N(t) := \lambda^N(t)$.\\
We consider the quantity
\begin{equation}\label{eqn:defXt}
\bX_t:=  \int_{\partial\omega_t} \bn_t(\by) |\psi_t(\by)|^2 \rmd \sigma(\by), \quad \forall t\in (-t_*,t_*)
\end{equation}
Remark that $\bX_t$ depends on the choice of the perturbation field $V$  (cf. \eqref{def:Thethat}). For $t =0$ one has  $\omega_0=\omega$ and for the sake of simplicity, we may also denote $\psi_\omega$ by $\psi_0$ and $\bX_0$ by $\bX$.\\

We emphasize that, throughout this section, we work within the framework of shape optimization. Consequently, we consider real-valued eigenfunctions $\psi_t$. Indeed, since the operator $L^N_{\omega_t}$ commutes with complex conjugation, it is legitimate to study $L^N_{\omega_t}$ on the real-valued Hilbert space $L^2(\omega_t)$, replacing all associated Sobolev spaces involves in the definition of the domain of the operator with their real-valued counterparts. This modification does not affect the sequence of eigenvalues, which retains the same multiplicities. Moreover, all results from \S \ref{subsubsec:adjoint problem} and \S \ref{subsubsec:regularityeigenpair} (specifically Lemmas \ref{lem:solweakadjoint}, \ref{Lem-Lipschitz}, and \ref{lem:implicit_modes}) remain valid in the real setting. Their proofs carry over directly by simply omitting the complex conjugation in the integrals. In particular, this implies that if one starts with a real-valued eigenfunction $\psi_{\omega}$ of $L^N_{\omega}$ in this real framework setting, then the corresponding eigenfunctions $\psi_t$ constructed in Lemma \ref{lem:implicit_modes} are also real-valued for all $t \in (-t_{\star}, t_{\star})$.\\ Finally, observe that it is not restrictive to consider real-valued eigenfunctions. Indeed, our goal is to prove that the coefficient $\bX_t$ is non-zero for sufficiently small $ t $, assuming that $\bX_0 = 0$  and that the eigenvalue $ \lambda_2^N(\omega)$ is simple. By continuity (see Lemma \ref{Lem-Lipschitz}), $\lambda_2^N(t) $ remains simple for small values of $ t$ . Consequently, the coefficient $\bX_t$, as defined in \eqref{eqn:defXt}, is independent of the choice of normalized eigenfunctions associated with this simple eigenvalue.

\begin{Lem} \label{prop:perturb}
Let $\omega$ be a bounded Lipschitz domain in $\R^2$ for which $\bX = 0$ and consider the second Neumann eigenvalue $\lambda_2^N(\omega)$. Assume it is simple and denote by $\psi_\omega$ an associated normalized real-valued eigenfunction. Let   $\bX_t$ be defined by \eqref{eqn:defXt} on $(-t_*,t_*)$ for $t_*>0$  and consider $q^*\in H^1(\omega)$ the unique weak solution of the adjoint problem given by Lemma \ref{lem:solweakadjoint}.\medskip

\noindent If $\psi_\omega, q^*, \psi_\omega^2 \in H^2(\omega)$, then the map $t \mapsto \bX_t $  is differentiable at $t=0$ and thus
\begin{equation}\label{eq.derivee}
    \bX_t  =   \Big(\frac{\rmd \bX_t}{\rmd t}   \Big)\rvert_{t = 0} \ t + o(t), \quad \text{when } t\to 0.
\end{equation}
Moreover in any fixed direction $\bw$ of the two-dimensional unit sphere $\mathbb{S}^1$, the derivative $\frac{\rmd}{\rmd t} (\bX_t \cdot \bw)\rvert_{t=0}$  is given by 
\begin{equation}\label{eqn:defexpressionderivXt}
\begin{split}
\Big(\frac{\rmd}{\rmd t} \bX_t \cdot \bw \Big)\rvert_{t = 0} &=  \int_{\partial \omega} \nabla (\psi_\omega^2)\cdot \bw \ (V \cdot \bn) \, \rmd \sigma(\by) - \lambda_2^N(\omega) \int_{\partial\omega} q^* \psi_\omega \ (V\cdot \bn) \, \rmd \sigma(\by)\\
&\qquad + \int_{\partial\omega} \nabla q^* \cdot \nabla \psi_\omega \ (V \cdot \bn) \, \rmd \sigma(\by).
\end{split}
\end{equation}
\end{Lem}
The proof of Lemma \ref{prop:perturb} relies on  the so-called Céa’s Method \cite{Cea-86} in order to have for any fixed $\bw\in \mathbb{S}^1$ an  expression of the derivative 
$
\left((\frac{\rmd}{\rmd t} \bX_t \cdot \bw \right)\rvert_{t=0}
$
which does not involve  explicitly  $\left(\frac{\rmd}{\rmd t} \psi_t\right) \rvert_{t=0}$, but rather the solution of the adjoint problem \eqref{eqn:ajointpb1}.
The main advantage of this method is that the obtained expression for $\left(\frac{\rmd}{\rmd t} \bX_t \cdot \bw \right)\rvert_{t=0}$ exhibits an explicit linear dependence on the deformation field $V$ as can be seen in Lemma \ref{prop:perturb}.
\begin{proof}[Proof of Lemma \ref{prop:perturb}]
Assume $\bX_0 = 0$ and $\bw\in \mathbb{S}^1$ is fixed. The proof is divided in four steps.\\[6pt]
{\bf Step 1 : Introduction of the Lagrangian $\mathcal{L}_\bw$ and its Frechet derivative.}\\[6pt]
We define the functional
\[
	J_\bw : (t,u)\in (-t_\star,t_\star) \times H^1(\omega) \mapsto \int_{\omega} J_{\Theta(t)}^{-\top}(\nabla (u^2)  \cdot \bw) \det(J_{\Theta(t)}) \, \rmd \by\in \R.
\]
Note that for $(t,u) \in (-t_\star,t_\star)\times H^1(\omega)$, there holds
\[
	J_\bw (t,u)= 2 \int_\omega u \big((J_{\Theta(t)}^{-\top}\nabla u) \cdot \bw\big) \det(J_{\Theta(t)}) \, \rmd \by.
\]
Introduce the Lagrangian $\mathcal{L}_\bw$ defined for $(t,u,q,\lambda) \in (-t_\star,t_\star) \times H^1(\omega)\times H^1(\omega)\times \R \to \R$ by
\begin{equation*}
 \mathcal{L}_\bw(t,u,q,\lambda) := J_\bw(t,u)+ \int_{\omega} (J_{\Theta(t)}^{-1}J_{\Theta(t)}^{-\top} \nabla q) \cdot \nabla u \det(J_{\Theta(t)}) \rmd \by- \lambda \int_{\omega} q u \det(J_{\Theta(t)}) \rmd \by.
\end{equation*}

Moreover, using formulas \eqref{eq.taylorinverse} and \eqref{eq.taylordet} evaluated at $t=0$, along with the identity $\Tr(J_V) = \Div V$, one can show the following. On the one hand, for any fixed $q \in H^1(\omega)$, the map
$
(t, u, \lambda) \in (-t_1, t_1) \times H^1(\omega) \times \mathbb{R} \mapsto \mathcal{L}_\bw(t, u, q, \lambda)
$
is Fréchet differentiable at every point $(0, u, \lambda) \in \{0\} \times H^1(\omega) \times \mathbb{R}$. On the other hand, its Fréchet derivative is given by the following bounded linear operator:

\begin{equation}
	\left\{
		\begin{array}{ccl}
			\R \times H^1(\omega) \times \R & \longrightarrow & \R\\[6pt]
			(h,v,\mu) & \mapsto & 2 h \displaystyle \int_\omega \left( u  \nabla u \ \Div V - u J_V^\top  \nabla u\right)\cdot \bw \, \rmd \by+ 2\int_\omega\left( v \nabla u  + u \nabla v \right)\cdot \bw \, \rmd \by\\[6pt]
			&& \displaystyle  +h \int_\omega\Big(\nabla q \cdot \nabla u\ \Div V - (J_V^\top + J_V) \nabla q \cdot \nabla u\Big) \,\rmd \by+ \int_\omega \nabla q \cdot \nabla v\, \rmd \by\\[6pt]
			&& \displaystyle  - \lambda h \int_\omega q\, u\  \Div V\, \rmd \by- \mu \int_\omega q \, u \,\rmd \by- \lambda \int_{\omega} q\, v \,\rmd \by.
		\end{array}
	\right.
	\label{eqn:Frechetderiv}
\end{equation}
Observe that pulling-back the eingenvalue problem in $\omega_t$ to $\omega$, one gets for all $q\in H^1(\omega)$ that
\begin{equation}\label{eqn:lagrangiantoJ}
\mathcal{L}_\bw(t,\psi_t\circ \Theta(t),q,\lambda_2^N(t))=J_\bw(t,\psi_t\circ \Theta(t)).
\end{equation}

\noindent {\bf Step 2: Expression of the partial derivatives of $\mathcal{L}_\bw$.}\\[6pt]
Thanks to \eqref{eqn:Frechetderiv}, one obtains that the Lagrangian $\mathcal{L}_\bw$ admits a partial derivative with respect to its second variable. This partial derivative is a continuous linear map from $H^1(\omega)$ to $\R$ and defined for all $v \in H^1(\omega)$ as 

\begin{equation}\label{eqn:partialderivativesecondvariablelagrangian}
\frac{\partial \mathcal{L}_\bw}{\partial u}(0, \psi_\omega,q,\lambda_2^N(\omega))[v]=2\int_\omega ( v\, \nabla \psi_\omega +\psi_\omega \, \nabla v)\cdot\bw \,\rmd \by +\int_{\omega}  \nabla q \cdot \nabla v  \, \rmd \by- \lambda_2^N(\omega) \int_{\omega}  q \, v \, \rmd \by.
\end{equation}
In addition, performing an integration by parts, one gets that for all $v \in H^1(\omega)$ that there holds
\[
	\int_\omega v \nabla \psi_\omega \cdot \bw \, \rmd \by= - \int_\omega \nabla v\cdot \bw \psi_\omega\, \rmd \by+ \int_{\partial\omega}v \, \psi_\omega \bw \cdot \bn \, \rmd \sigma
\]
and one can check that
\begin{equation}\label{eqn:partialderadjointzero}
\frac{\partial \mathcal{L}_\bw}{\partial u}(0,\psi_\omega,q,\lambda_2^N(\omega))=0
\end{equation}
if and only if $q$ is a weak solution of the  adjoint problem \eqref{eqn:ajointpb1}, namely
\begin{equation*}
\begin{cases}
-\Delta q = \lambda_2^N(\omega) q, &\text{in }\omega,\\
\partial_\bn q = -2 \psi_\omega\,  \bn \cdot \bw, & \text{on }\partial\omega.
\end{cases}
\end{equation*}
Again, thanks to \eqref{eqn:Frechetderiv}, $\mathcal{L}_{\bw}$ admits a partial derivative with respect to its fourth entry which is given for any $\mu \in \R$
\begin{equation}\label{eq.derivpartielle-lambda}
\frac{\partial \mathcal{L}_{\bw}}{\partial \lambda}(0, \psi_\omega,q,\lambda_2^N(\omega))[\mu]= -\mu \int_\omega q \, u  \,  \rmd\by.
\end{equation}
and  for all $q \in H^1(\omega)$
\begin{eqnarray}
\frac{\partial \mathcal{L}_\bw}{\partial t}(0,\psi_\omega,q,\lambda_2^N(\omega))[1]  &=&  2 \int_\omega\Big(\psi_\omega \, \nabla \psi_\omega \, \Div V - \psi_\omega \, J_V^\top \, \nabla \psi_\omega \Big)\cdot \bw \rmd \, \by  \nonumber \\[4pt]& &+ \int_\omega\Big(\nabla q \cdot \nabla \psi_\omega \Div V - (J_V^\top + J_V) \nabla q \cdot \nabla \psi_\omega \Big) \, \rmd \by \nonumber \\[4pt]& &- \lambda_2^N(\omega) \int_\omega q \, \psi_\omega \, \Div V \, \rmd \by.\label{eqn:exprfirstderlagrangianok}
\end{eqnarray}
{\bf  Step 3: Expression of the shape derivative $\left(\frac{d}{dt} \bX_t\right) \rvert_{t=0}$}\\[6pt]
\noindent Before calculating the shape derivative $\left(\frac{d}{dt} \bX_t\right) \rvert_{t=0}$ we start by noticing that by pulling back the expression of $X_t$ to the fixed domain $\omega$, one has
\begin{align*}
	\bX_t = \int_{\partial \omega_t}\bn_t |\psi_t|^2\,  \rmd \sigma(\by) = \int_{\omega_t} \nabla(|\psi_t|^2) \, \rmd \by&= 2 \int_{\omega_t} \psi_t \nabla \psi_t \, \rmd \by\\& = 2\int_{\omega} (\psi_t \circ \Theta(t))\,  (J_{\Theta(t)}^{-\top}\nabla(\psi_t \circ \Theta(t)) \det (J_{\Theta(t)})\, \rmd \by\\& = {J}_\bw(t, \psi_t \circ \Theta(t)).
\end{align*}
Thus by \eqref{eqn:lagrangiantoJ}, for all $q \in H^1(\omega)$ there holds
\[
\bX_t \cdot \bw  = \mathcal{L}_\bw(t, \psi_t \circ \Theta(t),q,\lambda_2^N(t)).
\]
By Lemma \ref{lem:implicit_modes}, the derivatives $\frac{d}{dt}(\psi_t\circ\Theta(t))|_{t=0}$ and $\frac{d}{dt} \lambda_2^N(t) \rvert_{t=0}$ exists. Thus, using that the Frechet derivatives \label{eqn:Frechetderive}  exists at $(0,\psi_{\omega},q, \lambda_2^N(\omega))$ for all $q\in H^1(\omega)$,  one can apply the chain rule to differentiate $(t\mapsto \bX_t \cdot \bw)$ at $t = 0$. This yields:
\begin{eqnarray}
\left(\frac{\rmd}{\rmd  t} \bX_t \cdot \bw \right)\rvert_{t = 0} &= & \displaystyle\frac{\partial \mathcal{L}_\bw}{\partial t}(0, \psi_\omega,q,\lambda_2^N(\omega))[1] + \frac{\partial \mathcal{L}_\bw}{\partial u}(0,\psi_\omega,q,\lambda_2^N(\omega))\left[\frac{d}{dt} (\psi_t\circ \Theta(t)) \rvert_{t=0}\right] \nonumber\\[4pt]
&& \displaystyle + \frac{\partial \mathcal{L}_\bw}{\partial \lambda}(0,\psi_\omega,q,\lambda_2^N(\omega))\left[ \frac{d}{dt} \lambda_2^N(t) \rvert_{t=0} \right].\label{eqn:chainrule}
\end{eqnarray}
Now we specify our choice for the function $q$. Indeed, using Lemma \ref{lem:solweakadjoint}, we take it to be the  unique weak solution of the adjoint problem \eqref{eqn:ajointpb1} orthogonal to $\psi_{\omega}$ and denote it $q^*$. Note that by \eqref{eqn:partialderadjointzero} there holds  $\frac{\partial \mathcal{L}_\bw}{\partial u}(0,\psi_\omega,q^*,\lambda_2^N(\omega))=0$. \medskip
\noindent
By virtue of \eqref{eq.derivpartielle-lambda}, as $q_*$ is orthogonal to $\psi_{\omega}$, one has
$$\frac{\partial \mathcal{L}_\bw}{\partial \lambda}(0,\psi_{\omega},q_*,\lambda_2^N(\omega))\left[ \frac{d}{dt} \lambda_2^N(t) \rvert_{t=0} \right]=- \left[ \frac{d}{dt} \lambda_2^N(t) \rvert_{t=0} \right] \int_\omega q_* \psi_{\omega}=0.$$
By \eqref{eqn:chainrule} and \eqref{eqn:exprfirstderlagrangianok} we get
\begin{eqnarray} 
	\left(\frac{\rmd}{\rmd t} \bX_t \cdot \bw \right)\rvert_{t = 0} &=& \displaystyle\frac{\partial \mathcal{L}_\bw}{\partial t}(0, \psi_\omega,q,\lambda_2^N(\omega))[1]  \nonumber \\
	&=& \displaystyle 2 \int_\omega\Big(\psi_\omega \nabla \psi_\omega \Div V - \psi_\omega J_V^\top\, \nabla \psi_\omega \Big)\cdot \bw \, \rmd \by  \label{der:Xtw0} \\& &+ \int_\omega\Big(\nabla q^* \cdot \nabla \psi_\omega \, \Div V - (J_V^\top + J_V) \nabla q^*\cdot \nabla \psi_\omega\Big) \, \rmd \by- \lambda_2^N(\omega) \int_\omega q^* \psi_\omega \Div V\,  \rmd \by.\nonumber
\end{eqnarray}
{\bf  Step 4: Derive an expression of  $\left(\frac{d}{dt} \bX_t\right) \rvert_{t=0}$ that involves boundary integrals.}\\[6pt]
Let us look first at the first line of the right-hand side of \eqref{der:Xtw0}:
\begin{equation*}
\begin{split}
-2\int_\omega \psi_\omega J_V^\top  \nabla \psi_\omega \cdot \bw \, \rmd \by=-\int_\omega J_V^\top \nabla(\psi_\omega^2) \cdot \bw \, \rmd \by= - \sum_{j,k=1}^2 \int_\omega \partial_j V_k \, \partial_k(\psi_\omega^2) \, w_j \, \rmd \by.
\end{split}
\end{equation*}
After integrating by parts we get
\begin{equation*}
-2\int_\omega \psi_\omega J_V^\top \nabla \psi_\omega \cdot \bw \, \rmd \by =  \Big(\sum_{j,k=1}^2 \int_\omega V_k \, \partial_{jk}^2 (\psi_\omega^2) \, w_j \, \rmd \by\Big) -  \int_{\partial \omega} V \cdot \nabla(\psi_\omega^2) \, \bw\cdot \bn \, \rmd \sigma(\by).
\end{equation*}
As we have assumed $\psi_\omega^2 \in H^2(\omega)$, exchanging the second-order derivatives and performing once again an integration by parts we get:
\begin{equation*}
\begin{split}
-2\int_\omega \psi_\omega\, J_V^\top \nabla \psi_\omega \cdot \bw \, \rmd \by &= \Big(\sum_{j,k=1}^2 \int_\omega V_k \ \partial_{kj}^2 (\psi_\omega^2) w_j \, \rmd \by\Big) -  \int_{\partial \omega} V \cdot \nabla(\psi_\omega^2) \bw\cdot \bn \, \rmd \sigma(\by)\\
 & = -\int_\omega \Div V  \ \nabla (\psi_\omega^2) \cdot \bw \, \rmd \by+ \int_{\partial \omega} \nabla (\psi_\omega^2)\cdot \bw \ (V \cdot \bn) \rmd \sigma(\by) \\
 & \, \quad - \int_{\partial \omega} V \cdot \nabla(\psi_\omega^2) \ \bw\cdot \bn \, \rmd \sigma(\by)
\end{split}
\end{equation*}
Thus, the first line of the right-hand side of \eqref{der:Xtw0} can be rewritten as follows:
\begin{equation}\label{eqn:firstlinecrazyL}
\begin{split}
2 \int_\omega\Big( \psi_\omega \, \nabla \psi_\omega \, \Div V - \psi_\omega \, J_V^\top \nabla  \psi_\omega\Big)\cdot \bw\,  \rmd \by=&  \int_{\partial \omega} \left(\nabla (\psi_\omega^2)\cdot \bw\right) (V \cdot \bn) \, \rmd \sigma(\by) \\& - \int_{\partial \omega} \left(V \cdot \nabla(\psi_\omega^2)\right) (\bw\cdot \bn) \, \rmd \sigma(\by).
\end{split}
\end{equation}
Now we look at the second line of the right-hand side of \eqref{der:Xtw0}. Using that $\partial_\bn \psi_\omega =0$ on $\partial \omega$, an integration by parts yields
\begin{equation*}
\begin{split}
&-\int_\omega (J_V^\top + J_V) \nabla q^* \cdot \nabla \psi_\omega \, \rmd \by= -\Big(\sum_{j,k=1}^2 \int_\omega \partial_k V_j\, \partial_k q^*\,  \partial_j \psi_\omega \, \rmd \by\Big) -\Big(\sum_{j,k=1}^2 \int_\omega \partial_j V_k \, \partial_k q^*\,  \partial_j \psi_\omega \, \rmd \by\Big) \\
&=\Big(\sum_{j,k=1}^2\int_\omega V_j \, \partial_{kk}^2 q^*\, \partial_j \psi_\omega \, \rmd \by\Big)+\Big(\sum_{j,k=1}^2\int_\omega V_j \, \partial_k q^* \partial_{kj}^2 \psi_\omega \, \rmd \by\Big) -\int_{\partial \omega} V \cdot \nabla \psi_\omega \ \partial_\bn q^* \, \rmd \sigma(\by) \\
&  \quad +\Big(\sum_{j,k=1}^2\int_\omega V_k\,  \partial_{jk}^2 q^* \partial_j \psi_\omega \, \rmd \by\Big)+\Big(\sum_{j,k=1}^2\int_\omega V_k \, \partial_k q^*\, \partial_{j}^2 \psi_\omega \, \rmd \by\Big) -\int_{\partial \omega} V \cdot \nabla q^* \ \partial_\bn \psi_\omega \, \rmd \sigma(\by) \\
& =  \int_\omega  V \cdot \nabla \psi_\omega \ \Delta q^* \, \rmd \by+\Big(\sum_{j,k=1}^2 \int_\omega  V_j\, \partial_{k} q^* \partial_{kj}^2 \psi_\omega \, \rmd \by\Big) -\int_{\partial\omega} V \cdot \nabla\psi_\omega \ \partial_\bn q^* \, \rmd \sigma(\by) \\
& \quad +\Big(\sum_{j,k=1}^2\int_\omega V_k \partial_{jk}^2 q^* \partial_j \psi_\omega \, \rmd \by\Big)+\int_\omega V \cdot \nabla q^* \ \Delta \psi_\omega \, \rmd \by.
\end{split}
\end{equation*}
As we assumed $\psi_\omega \in H^2(\omega)$, exchanging the second-order derivatives and performing once again an integration by parts we get:
\begin{equation*}
\begin{split}
&\sum_{j,k=1}^2 \int_\omega  V_j\, \partial_{k} q^* \, \partial_{kj}^2 \psi_\omega \, \rmd \by = \sum_{j,k=1}^2 \int_\omega  V_j \, \partial_{k} q^*\,  \partial_{jk}^2 \psi_\omega \, \rmd \by\\
& =-\int_\omega \Div V  (\nabla q^* \cdot \nabla \psi_\omega) \, \rmd \by- \Big(\sum_{j,k=1}^2  \int_\omega V_j \partial_{jk}^2 q^* \partial_k \psi_\omega \, \rmd \by\Big) + \int_{\partial\omega} \left(\nabla q^* \cdot \nabla \psi_\omega \right) (V \cdot \bn) \, \rmd \sigma(\by).
\end{split}
\end{equation*}
Therefore, since we assumed $q^* \in H^2(\omega)$ (so that $\partial_{jk}^2 q^* = \partial_{kj}^2 q^*$ in $L^2(\omega)$), and recalling that $q^*$ solves \eqref{eqn:ajointpb1} and $\psi_\omega$ is an eigenfunction of the Neumann Laplacian in $\omega$ we get
\begin{equation*}
\begin{split}
&-\int_\omega (J_V^\top + J_V) \nabla q^* \cdot \nabla \psi_\omega \, \rmd \by\\& =  \int_\omega  V \cdot \nabla \psi_\omega \ \Delta q^* \, \rmd \by+\int_\omega V \cdot \nabla q^* \ \Delta \psi_\omega \, \rmd \by-\int_\omega \Div V \ \nabla q^* \cdot \nabla \psi_\omega \, \rmd \by\\
& \quad   + \int_{\partial\omega} \nabla q^* \cdot \nabla \psi_\omega \ V \cdot \bn \, \rmd \sigma(\by)  -\int_{\partial\omega} V \cdot \nabla\psi_\omega \ \partial_\bn q^* \, \rmd \sigma(\by)  \\
&=  - \lambda_2^N(\omega) \int_\omega  V \cdot \nabla \psi_\omega \  q^* \, \rmd \by-\lambda_2^N(\omega) \int_\omega V \cdot \nabla q^* \  \psi_\omega \, \rmd \by -\int_\omega \Div V \ \nabla q^* \cdot \nabla \psi_\omega \, \rmd \by \\
& \quad + \int_{\partial\omega} \nabla q^* \cdot \nabla \psi_\omega \ V \cdot \bn \, \rmd \sigma(\by)  +2\int_{\partial\omega} V \cdot \nabla\psi_\omega \ \psi_\omega \   \bw \cdot \bn \, \rmd \sigma(\by).
\end{split}
\end{equation*}
Now, integrating by parts, one gets
\begin{equation*}
\begin{split}
-\lambda_2^N(\omega) \int_\omega V \cdot \nabla q^* \  \psi_\omega \, \rmd \by &= -\lambda_2^N(\omega) \left(\sum_{j=1}^2 \int_\omega V_j \partial_j q^* \, \psi_\omega \, \rmd \by\right) \\ &= \lambda_2^N(\omega)  \int_\omega q^* \psi_\omega \ \Div V  \, \rmd \by+\lambda_2^N(\omega) \int_\omega V\cdot \nabla \psi_\omega  \ q^* \,\rmd \by\\
&\quad - \lambda_2^N(\omega) \int_{\partial\omega} q^* \psi_\omega \ V\cdot \bn \, \rmd \sigma(\by).
\end{split}
\end{equation*}
Thus, one obtains
\begin{equation*}
\begin{split}
&-\int_\omega (J_V^\top + J_V) (\nabla q^*) \cdot (\nabla \psi_\omega) \, \rmd \by \\  &=\lambda_2^N(\omega)  \int_\omega q^* \psi_\omega \ \Div V  \, \rmd \by- \lambda_2^N(\omega) \int_{\partial\omega} q^* \psi_\omega \ V\cdot \bn \, \rmd \sigma(\by) -\int_\omega \Div V \ \nabla q^* \cdot \nabla \psi_\omega \, \rmd \by \\
& \quad  + \int_{\partial\omega} (\nabla q^* \cdot \nabla \psi_\omega) \ (V \cdot \bn) \, \rmd \sigma(\by)  +2\int_{\partial\omega} V \cdot \nabla\psi_\omega \ \psi_\omega  \bw \cdot \bn \, \rmd \sigma(\by).
\end{split}
\end{equation*}
Finally we get that the second line in the right-hand side term of \eqref{der:Xtw0} can be rewritten as follows:
\begin{equation} \label{eqn:secondlinecrazyL}
\begin{split}
&\int_\omega\left(\nabla q^*\cdot \nabla \psi_\omega\, \Div V - (J_V^\top + J_V)\nabla q^*\cdot \nabla \psi_\omega\right) \, \rmd \by- \lambda_2^N(\omega) \int_\omega q^* \psi_\omega \, \Div V \, \rmd \by\\
& = - \lambda_2^N(\omega) \int_{\partial\omega} q^* \psi_\omega \ (V\cdot \bn) \ \rmd \sigma(\by)+ \int_{\partial\omega} (\nabla q^* \cdot \nabla \psi_\omega) \ (V \cdot \bn) \, \rmd \sigma(\by)  +2\int_{\partial\omega} V \cdot \nabla\psi_\omega \ \psi_\omega  \bw \cdot \bn \, \rmd \sigma(\by).
\end{split}
\end{equation}
Gathering equations \eqref{eqn:firstlinecrazyL} and \eqref{eqn:secondlinecrazyL}, we arrive at
\begin{equation*}
\begin{split}
\left(\frac{\rmd}{\rmd t} \bX_t \cdot \bw \right)\rvert_{t = 0} &=  \int_{\partial \omega} \nabla (\psi_\omega^2)\cdot \bw \ (V \cdot \bn) \rmd \sigma(\by) - \lambda_2^N(\omega) \int_{\partial\omega} q^* \psi_\omega  (V\cdot \bn) \, \rmd \sigma(\by)\\
&\qquad + \int_{\partial\omega} \left(\nabla q^* \cdot \nabla \psi_\omega \right) (V \cdot \bn) \, \rmd \sigma(\by).
\end{split}
\end{equation*}
Thus, as $(t\mapsto \bX_t\cdot\bw)$ admits a derivative a $t = 0$, by definition we get that for $t$ sufficiently small
\begin{equation}\label{eq.derivecscal}
	\bX_t\cdot\bw = t \left(\frac{\rmd}{ \rmd t} \bX_t \cdot \bw \right)\rvert_{t = 0} + o(t)
\end{equation}
and the expression of $\left(\frac{\rmd}{\rmd t} \bX_t \cdot \bw \right)\rvert_{t = 0}$ is given by \eqref{eqn:defexpressionderivXt}. Since equation \eqref{eq.derivecscal} holds for any fixed $\bw \in \mathbb{S}^1$, it holds in particular for the two vectors of the canonical orthonormal basis of $\R^2$. This, in turn, implies equation \eqref{eq.derivee}, thereby proving Lemma \ref{prop:perturb}.
\end{proof}
\subsubsection{Instability of the geometric condition $\bX = 0$ for rectangles}\label{subsubsec:shapederivativerectangle}

In this paragraph, we focus on the specific case of $\omega$ being the rectangle $\omega = (0,\ell)\times(0,L)$ with $L > 0$ with $\ell > L$. Due to the separation of variables, it is easy to verify that $\lambda_2^N(\omega)$ is simple and satisfies $\lambda_2^N(\omega) = \pi^2/\ell^2$.  
Since the rectangle is invariant under a rotation by an angle $\pi$, Proposition~\ref{prop:symm-X0} implies that $\bX = 0$.  
Moreover, in this case, using separation of variables, one can also compute  $\bX$ explicitly  and show directly that $\bX=0$.

For a given vector field $V \in W^{1,\infty}(\R^2,\R^2)$ and a sufficiently small $t$, one can consider the domains $\omega_t$ defined as $\omega_t = \Theta(t)\omega$ where $\Theta$ is introduced in \eqref{def:Thethat}. It turns out, one can find a $t$ and a vector field $V$ for which the quantity $\bX_t$ introduced in \eqref{eqn:defXt} is non-zero. This is the purpose of the following Proposition.
\begin{Pro}
Let $\ell,L > 0$ with $\ell > L$ and consider the rectangle $\omega = (0,\ell)\times(0,L)$.  There exists $V \in W^{1,\infty}(\R^2,\R^2)$ and $t_0 > 0$ such that there holds $\bX_{t} \neq 0$ for all $t \in (0,t_0)$.
\label{prop:smalldefrect}
\end{Pro}
Before proving Proposition \ref{prop:smalldefrect}, we need a useful lemma that justifies a density argument.
\begin{Lem}\label{lem-density}
Let $\omega=(0,\ell)\times (0,L)$ be a rectangle and $f \in L^2(\partial\omega)$. Let $\mathcal{U}\subset \partial\omega$ be an open subset of the boundary.
If $\int_{\partial\omega}  f (V\cdot n) =0$ for all $V \in W^{1,\infty}(\R^2,\R^2)$  having a trace  $\partial \omega$ supported on $\mathcal{U}$,  then $f =0$ a.e. on $\mathcal{U}$.
\end{Lem}

\begin{proof}[Proof of Lemma \ref{lem-density}]
\noindent {\bf Step 1: Construction of  a vector field $\tilde{\mathscr{V}} \in \mathscr{C}^{1}(\R^2,\R^2)$ by extension via an adapted function defined  on $\partial \omega$.
\\[6pt]}
We denote by $\partial \omega_k$ for $k=1, \ldots, 4$ the four sides of $\partial \omega$, more precisely
\begin{equation*}
    \partial\omega_1 = [0,\ell] \times \{0\}, \quad \partial\omega_2 = \{\ell\} \times [0,L] , \quad \partial\omega_3 = [0,\ell] \times \{L\}, \quad \partial\omega_4 = \{0\} \times [0,L].
\end{equation*}
Let $\bv_1,\bv_3 \in \mathscr{C}^1(0,\ell)^2$ with compact support in $(0,\ell)$, and $\bv_2,\bv_4 \in \mathscr{C}^1(0,L)^2$ with compact support in $(0,L)$.
Then define 
$$\mathscr{V}(t,0)= \bv_1(t), \mathscr{V}(t,L)=\bv_3(t), \forall \in (0,\ell) \mbox{ and } \mathscr{V}(0,t)=\bv_4(t),\ \mathscr{V}(\ell,t)=\bv_2(t), \forall t\in (0,L).$$ We define $\mathscr{V}$ equal to $(0,0)$ on each vertex of the rectangle $\omega$.\smallskip\\
\noindent Now we extend this vector field $\mathscr{V}:\partial\omega \to \R^2$ to a vector field $\tilde{\mathscr{V}}$ defined on the whole of $\R^2$. To do so, we apply the Whitney extension theorem to each component of $\mathscr{V}$. The reader can easily check that by construction  the vector field $\mathscr{V}$ has the required assumptions of \cite[Theorem 2.3.6]{horm}. Fundamentally, this is due the fact that $\mathscr{V}$  is supported away from the vertex of the rectangle.
Thus $\tilde{\mathscr{V}} \in \mathscr{C}^1(\R^2,\R^2)$ is such that $\tilde{\mathscr{V}}\rvert_{\partial\omega}=\mathscr{V}$.
\\[10pt]
\noindent {\bf Step 2: Construction of a vector field in $W^{1,\infty}(\R^2,\R^2)$.\\[6pt]}
We consider now a smooth cut-off function $\chi \in \mathscr{C}^\infty(\R^2,\R)$ such that $\chi =1$ on $\bar\omega$ and $\chi=0$ outside of a ball strictly containing $\bar \omega$.
Then the  vector field $V:=\chi \tilde{\mathscr{V}} \in \mathscr{C}^1(\R^2,\R^2) \cap W^{1,\infty}(\R^2,\R^2)$  is such that $V\rvert_{\partial \omega}= \mathscr{V}$.\\[10pt]
\noindent {\bf Step 3:  The particular case where $\mathcal{U}=\partial \omega$ \\[6pt]}
For each vector field $\bv_i$ (with $i=1, \ldots,4$), one denotes by $\bv_i=(v_{i,1}, v_{i,2})$. Then, one can rewrite $\int_{\partial\omega} f (V \cdot n) \,\rmd \sigma(\by)$ as 
\begin{equation*}
    \begin{split}
        \int_{\partial\omega} f (V \cdot n) \,\rmd \sigma(\by)&= -\int_0^\ell  f(t,0) \, v_{1,2}(t,0) \, dt + \int_0^L f(\ell,t) \, v_{2,1}(\ell,t) \, dt \\ 
        &- \int_0^\ell f(\ell-t,L) \, v_{3,2}(\ell-t,L) \,\rmd t +\int_0^L f(0,L-t) \, v_{4,1}(0,L-t)\, \rmd t.
    \end{split}
\end{equation*}
If $\mathcal{U}= \partial\omega$, we can arbitrarily choose $v_{1,2}$  as any $\mathscr{C}^1$ function compactly supported in the interval $(0,\ell)$, and choose $v_{2,1}$, $v_{3,2}$  and $v_{4,1}$ to be equal to zero.
As this space is dense in $L^2(0,\ell)$ we conclude that $f=0$ almost everywhere on $\partial\omega_1$.
Analogously we can see that $f=0$ almost everywhere on $\partial\omega_k$ for $k=2,3,4$. Thus $f=0$ almost everywhere on $\partial\omega$.
\\[10pt]
\noindent {\bf Step 4: The general case \\[6pt]}
If $\mathcal{U}\subset \partial\omega$, we have that 
$\int_{\partial \omega} f (V \cdot n) \,\rmd \sigma(\by)= \int_{\mathcal{U}} f (V \cdot n) \,\rmd \sigma(\by)$ since the trace of $V$ on $\partial \omega$ is supported in $\mathcal{U}$. Then
\begin{equation*}
    \begin{split}
        \int_{\mathcal{U}} f (V \cdot n) \,\rmd \sigma(\by)&= -\int_{\mathbb{P}_1 \left(\partial \omega_1 \cap \mathcal{U} \right)}  f(t,0) \, v_{1,2}(t,0) \, dt + \int_{\mathbb{P}_2 \left(\partial \omega_2 \cap \mathcal{U} \right)} f(\ell,t) \, v_{2,1}(\ell,t) \, dt \\ 
        &-\int_{\mathbb{P}_1 \left(\partial \omega_3 \cap \mathcal{U} \right)} f(\ell-t,L) \, v_{3,2}(\ell-t,L) \,\rmd t +\int_{\mathbb{P}_2 \left(\partial \omega_4 \cap \mathcal{U} \right)} \, v_{4,1}(0,L-t)\, \rmd t.
    \end{split}
\end{equation*}
where $\mathbb{P}_1: \R^2 \to \R, (y_1,y_2) \mapsto y_1$ and $\mathbb{P}_2: \R^2 \to \R, (y_1,y_2) \mapsto y_2$ are the projections on the axes.
Reasoning as in Step 3, and using the density of $\mathscr{C}^1$ compactly supported functions in $\mathscr{O}$ in $L^2(\mathscr{O})$ where $\mathscr{O}$ is any open subset of $(0,\ell)$ or $(0,L)$, we conclude.
\end{proof}

\begin{proof}[Proof of Proposition \ref{prop:smalldefrect}]
The proof is organized in three steps. In the first step we compute explicitly the solution of the adjoint problem \eqref{eqn:ajointpb1} for $\omega$ being the rectangle $(0,\ell)\times(0,L)$. In the second step, we make use of Lemma \ref{prop:perturb} to prove that $\left(\frac{d}{dt}\bX_t\right)|_{t = 0}$ can not be zero for all vector fields $V \in W^{1,\infty}(\R^2,\R^2)$. The proof is concluded in the last step using classical tools from analysis.\\[10pt]
\noindent {\bf Step 1: The solution of the adjoint problem.\\[6pt]}
Let us look for a solution to the adjoint problem \eqref{eqn:ajointpb1} when $\omega = (0,\ell)\times(0,L)$. In this setting, the second Neumann eigenvalue in $\omega$ is simple, verifies $\lambda_2^N(\omega) = \frac{\pi^2}{\ell^2}$ and  an associated normalized eigenmode is given by
\[
    \psi_\omega : (y_1,y_2) \in \omega \mapsto \sqrt{\frac{2}{\ell L}} \cos(\frac{\pi}\ell y_1).
\]

As $\bX_0= 0$, for any fixed $\bw \in \mathbb{R}^2$ there exists a unique weak solution $q^*$ to the adjoint problem \eqref{eqn:ajointpb1} orthogonal in $L^2(\omega)$ to $\psi_\omega$. We compute it explicitly looking for a strong real-valued solution verifying

\begin{equation*}
\begin{cases}
-\Delta q^* = \frac{\pi^2}{\ell^2} q^*, &\text{in }\omega,\\
\partial_\bn q^* = -2 \psi_\omega \bn \cdot \bw, & \text{on }\partial\omega.
\end{cases}
\end{equation*}
We will solve the problem above for $\bw=\mathfrak{e}_j$ ($j=1,2$), yielding some $q_j$, where we have set $\mathfrak{e}_1=\begin{pmatrix}1 \\0 \end{pmatrix}$ and $\mathfrak{e}_2=\begin{pmatrix}0\\ 1 \end{pmatrix}$. Then, by linearity, the general solution  for $\bw =\begin{pmatrix}\alpha \\ \beta \end{pmatrix}$ will be given by $q^* = \alpha q_1 + \beta q_2$.

Let us start by solving
\begin{equation*}
\begin{cases}
-\Delta q_1 = \frac{\pi^2}{\ell^2} q_1, &\text{in }\omega,\\
\partial_\bn q_1 = -2 \psi_\omega \bn \cdot \begin{pmatrix} 1 \\ 0 \end{pmatrix} & \text{on }\partial\omega.
\end{cases}
\end{equation*}
We rewrite the boundary condition on each part of the boundary $\partial\omega$. Namely, on $(0,\ell)\times\{0\}$ and $(0,\ell)\times\{L\}$ it rewrites for almost all $y_1 \in (0,\ell)$ as
\[
(\partial_2 q_1) (y_1,0) = (\partial_2 q_1)(y_1,L) = 0 
\]
and on $\{0\}\times(0,L)$ and $\{\ell\}\times(0,L)$ for all $y_2 \in (0,L)$ one has
\[
   ( \partial_1 q_1)(0,y_2) = - 2 \sqrt{\frac{2}{\ell L}},\quad (\partial_1 q_1)(\ell,y_2) = 2 \sqrt{\frac{2}{\ell L}}.
\]
We can look for a solution which has separated variables and assume that $q_1(y_1,y_2) = q_1(y_1)$. In this case, such a solution would verify the following one-dimensional problem:
\begin{equation*}
\begin{cases}
-q_1''(y_1)=\frac{\pi^2}{\ell^2} q_1(y_1), & 0<y_1<\ell,\\
q_1'(0)=-2 \sqrt{\frac{2}{\ell L}},\\
q_1'(\ell)= 2 \sqrt{\frac{2}{\ell L}}.
\end{cases}
\end{equation*}
The general solution to this differential equation is given by
\[
    q_1(y_1) = A \cos\left(\frac{\pi}\ell y_1\right) + B \sin\left(\frac{\pi}\ell y_1\right),\quad \text{for some constants } A,B \in \R.
\]
In order to obtain a solution $q_1$ orthogonal to $\psi_\omega$ in $L^2(\omega)$ we immediately choose $A=0$.
Now, remark that there holds
\[
    q_1'(y_1) = B\frac\pi{\ell}  \cos\left(\frac{\pi}\ell y_1\right).
\]
The boundary conditions yields $B = - \frac{2\sqrt{2}}\pi \sqrt{\frac{\ell}L}$, hence one has  for $y_1 \in (0,\ell)$:
\[
    q_1(y_1) = - \frac{2\sqrt{2}}\pi \sqrt{\frac{\ell}L} \sin\left(\frac{\pi}\ell y_1\right).
\]
Now, we focus on $q_2$ which solves the problem
\begin{equation}
\begin{cases}
-\Delta q_2 = \frac{\pi^2}{\ell^2} q_2, &\text{in }\omega,\\
\partial_\bn q_2 = -2 \psi_\omega \bn \cdot \begin{pmatrix} 0 \\ 1 \end{pmatrix} & \text{on }\partial\omega.
\end{cases}
\label{eqn:solq_2}
\end{equation}
The boundary condition on each part of the boundary $\partial\omega$ rewrites
\[
(\partial_1 q_2) (0,y_2) = (\partial_1 q_2)(\ell,y_2) = 0,\quad \text{for almost all }y_2\in(0,L)
\]
and
\[
   ( \partial_2 q_2)(y_1,0) = ( \partial_2 q_2)(y_1,L)=  - 2 \sqrt{\frac{2}{\ell L}} \cos\left( \frac{\pi}{\ell} y_1 \right),\quad \text{for almost all } y_1 \in (0,\ell).
\]
We look for a solution of the form $q_2 (y_1,y_2) = f(y_1) g(y_2)$ with $f$ and $g$ verifying
\[
    \left\{\begin{array}{lcl}
         -f''(y_1)& = & \frac{\pi^2}{\ell^2} f(y_1), \quad 0 < y_1 < \ell, \\
         f'(0) &=&0,\\
          f'(\ell) &=&0
    \end{array}\right. \text{ and }\quad\quad \left\{\begin{array}{lcl}
         -g''(y_2)& = & 0, \quad 0 < y_2 < L,\\
         g'(0) &=&- 2 \sqrt{\frac{2}{\ell L}},\\
          g'(L) &=&- 2 \sqrt{\frac{2}{\ell L}}.
    \end{array}\right.
\]
In particular, the first equation yields that one can choose $f(y_1) = \cos(\frac\pi\ell y_1)$ and then, $g$ can be chosen of the form
\begin{equation*}
g(y_2) := - 2 \sqrt{\frac{2}{\ell L}} y_2 + K,
\end{equation*}
for some constant $K \in \R$ that will be determined later. Note that as defined, one has
\[
    - \Delta q_2 = -f'' g - f g'' = \frac{\pi^2}{\ell^2} fg = \frac{\pi^2}{\ell^2} q_2
\]
and that on the boundary condition one has for all $(y_1,y_2) \in \omega$
\begin{eqnarray*}
   (\partial_1 q_2)(0,y_2) &=& f'(0)g(y_2) = 0,\quad (\partial_1 q_2)(\ell,y_2) = f'(\ell)g(y_2) = 0\\[4pt]
   \partial_2 q_2(y_1,0) &=& f(y_1) g'(0) = - 2 \sqrt{\frac{2}{\ell L}}f(y_1) = - 2 \sqrt{\frac{2}{\ell L}}\cos(\frac{\pi}\ell y_1),\\[4pt] \partial_2 q_2(y_1,L) &=& f(y_1) g'(L) = - 2 \sqrt{\frac{2}{\ell L}}f(y_1) = - 2 \sqrt{\frac{2}{\ell L}}\cos(\frac{\pi}\ell y_1).
\end{eqnarray*}
Hence, a solution $q_2$ of \eqref{eqn:solq_2}  is given by
\[  
    q_2(y_1,y_2) = f(y_1)g(y_2) = \big(-2 \sqrt{\frac{2}{\ell L}} y_2+K\big) \, \cos(\frac{\pi}\ell y_1) .
\]
Note that there holds
\begin{align*}
    (q_2,\psi_\omega)_{L^2(\omega)}&= \int_0^L\Big(-2\sqrt{\frac{2}{\ell L}}y_2 + K\Big)\Big(\int_0^\ell \cos\left(\frac\pi\ell y_1\right)^2 \rmd y_1 \Big) \rmd y_2\\& = \frac\ell2\int_0^L\Big(-2\sqrt{\frac{2}{\ell L}}y_2 + K\Big) \rmd y_2\\
    & = \frac{\ell L}2\Big(- \sqrt{\frac{2 L}\ell} + K\Big).
\end{align*}
Thus, the unique solution $q_2$ of  \eqref{eqn:solq_2} orthogonal to $\psi_\omega$ is given for $K = \sqrt{\frac{2 L}\ell}$ which reads
\begin{equation} \label{q2:def}
    q_2(y_1,y_2)  = \sqrt{\frac{2} {\ell}}\left(- \frac{2}{\sqrt{L}} y_2 + \sqrt{L}\right) \cos(\frac\pi\ell y_1).
\end{equation}

{\noindent \bf Step 2 : For any  $\bw\in \mathbb{S}^1$,  there exists $\bV\in W^{1,\infty}(\R^2,\R^2)$ such that $(\frac{d}{dt} \bX_t \cdot \bw)|_{t= 0}\neq 0$.}\\[6pt] Let us make use of Lemma \ref{prop:perturb} and assume by contradiction that for all $\bw \in \mathbb{S}^1$ and all vector field $V \in W^{1,\infty}(\R^2,\R^2)$ one has $(\frac{d}{dt} \bX_t \cdot \bw)|_{t= 0} = 0$. As this is true for all vector field $V \in W^{1,\infty}(\R^2,\R^2)$,
by \eqref{eqn:defexpressionderivXt} and Lemma \ref{lem-density}, one necessarily has almost everywhere on $\partial\omega$
\begin{equation}\label{eq.compatibility}
    \nabla (\psi_\omega^2 )\cdot \bw - \lambda_2^N(\omega) q^* \psi_\omega + \nabla q^* \cdot \nabla \psi_\omega = 0.
\end{equation}
Let us compute this quantity for $\bw = \mathfrak{e}_j$ and $q^* = q_j$ (for $j \in \{1,2\}$).
Note that for $(y_1,y_2) \in \overline{\omega}$ there holds
\[
    \nabla(\psi_\omega^2)(y_1,y_2)= - \frac{4 \pi}{\ell^2 L}\begin{pmatrix}
    \cos(\frac\pi\ell y_1)\sin(\frac\pi\ell y_1)\\0
    \end{pmatrix}
\]
yielding
\[
     \nabla(\psi_\omega^2)(y_1,y_2)\cdot \bw_1 = - \frac{4 \pi}{\ell^2 L}\cos(\frac\pi\ell y_1)\sin(\frac\pi\ell y_1).
\]
Then, one obtains
\[
    - \lambda_2^N(\omega) \, q_1(y_1,y_2) \, \psi_\omega(y_1,y_2) =  \frac{4\pi}{\ell^2 L}\cos(\frac\pi\ell y_1)\sin(\frac\pi\ell y_1)
\]
as well as
\[
    \nabla q_1(y_1,y_2) \cdot \nabla \psi_\omega(y_1,y_2) = \frac{4 \pi}{\ell^2 L} \cos(\frac\pi\ell y_1)\sin(\frac\pi\ell y_1).
\]
It yields for  all $(y_1,y_2) \in \partial\omega$
\begin{equation}
     \nabla(\psi_\omega^2) \cdot \bw_1 - \lambda_2^N(\omega) q_1 \psi_\omega + (\nabla q_1) \cdot \nabla \psi_\omega = \frac{4 \pi}{\ell^2 L} \cos(\frac\pi\ell y_1)\sin(\frac\pi\ell y_1).
     \label{eqn:thisshouldbezero1}
\end{equation}
Hence,  $  \nabla(\psi_\omega^2) \cdot \bw_1 - \lambda_2^N(\omega) q_1 \psi_\omega + \nabla q_1 \cdot \nabla \psi_\omega \neq 0$ for   $(y_1,y_2)\in \big((0,\ell)\setminus \{\frac{\ell}{2}\}\big)\times (\{0\} \cup \{L\})$.\\ 

It turns out that the condition \eqref{eq.compatibility} is also not satisfied on $\partial \omega$ for $j=2$. Indeed, one gets for all $(y_1,y_2)\in \overline{\omega}$
\[
    \nabla(\psi_\omega^2)(y_1,y_2)\cdot \bw_2 = 0.
\]
Similarly, one has
\[
     - \lambda_2^N(\omega) q_2(y_1,y_2) \psi_\omega(y_1,y_2) = -\frac{2\pi^2}{\ell^3 }\Big(- \frac{2}{ L}y_2 +1\Big)\cos(\frac\pi\ell y_1)^2
\]
as well as
\[
\nabla q_2(y_1,y_2) \cdot \nabla \psi_\omega(y_1,y_2) = \frac{2\pi^2}{\ell^3} \Big(- \frac{2}{L} y_2 + 1\Big) \sin(\frac\pi\ell y_1)^2.
\]
It yields on $\partial\omega$:
\begin{equation}
     \nabla(\psi_\omega^2) \cdot \bw_2 - \lambda_2^N(\omega) q_2 \psi_\omega + \nabla q_2 \cdot \nabla \psi_\omega = \frac{2\pi^2}{\ell^3} \Big(- \frac{2}{L} y_2 + 1\Big) \left( \sin(\frac{\pi}{\ell} y_1)^2 - \cos(\frac{\pi}{\ell} y_1)^2  \right).
         \label{eqn:thisshouldbezero2}
\end{equation}
Now, take any $\bw = (\alpha,\beta)^\top \in \mathbb{S}^1$ and take $q^*$ the solution of the adjoint problem \eqref{eqn:ajointpb1}. Remark that then, for $(y_1,y_2) \in \partial\omega$ there holds
\[
    (\frac{d}{dt}\bX_t\cdot \bw)|_{t = 0}(y_1,y_2) =  \frac{2\pi}{\ell^2 }\left[ \frac{2\alpha}{L} \cos(\frac\pi\ell y_1)\sin(\frac{\pi}\ell y_1) + \frac{\beta \pi}{\ell} \big(- \frac{2}{ L}y_2 +1\big) \left( \sin(\frac{\pi}{\ell} y_1)^2 - \cos(\frac{\pi}{\ell} y_1)^2  \right)  \right].
\]
If this quantity is zero on $\partial\omega$, we get, selecting the lateral sides, namely choosing $y_1\in \{0,\ell\}$, that for all $y_2 \in (0,L)$:
\begin{equation*}
    0 = -\frac{\beta \pi}{\ell} \big(- \frac{2}{ L}y_2 +1\big).
\end{equation*}
This necessarily yields that $\beta = 0$.  Consequently we have that for  all $y_1 \in (0,\ell)$ there holds
\[
    0 = \alpha \frac{4\pi}{\ell^2 L} \cos(\frac\pi\ell y_1)\sin(\frac{\pi}\ell y_1).
\]
This necessarily implies that $\alpha = 0$. Hence  $\bw = (0,0)^\top$ which  is a contradiction since $\bw \in \mathbb{S}^1$. 
Therefore, for any $\bw$ there exists a vector field $V \in W^1(\R^2,\R^2)$ such that $(\frac{d}{dt}\bX_t \cdot \bw)|_{t = 0} \neq (0,0)^\top$.\medskip

\noindent{\bf Step 3: Existence of a deformation  field $V\in W^{1,\infty}(\bbR^2,\bbR^2)$ and a real $t_0>0$ such that $X_t\neq 0$ on $(0,t_0)$.}\\[6pt]
Now, fix $\bw \in \mathbb{S}^1$ and chose a vector field $V \in W^1(\R^2,\R^2)$ such that for $(\frac{d}{dt} \bX_t \cdot \bw) |_{t = 0} \neq0$. Note that this is always possible thanks to Step 2. By Lemma \ref{prop:perturb}, there holds
$\bX_t \cdot \bw=  t \left( \frac{d}{dt} \bX_t \cdot \bw \right)\rvert_{t = 0} + o(t)$. Thus, in particular, if $t_0 > 0$ is sufficiently small then $\bX_{t_0}\cdot \bw \neq 0$. It concludes the proof of Proposition \ref{prop:smalldefrect}.
\end{proof}

If one follows closely the steps of the proof of Proposition \ref{prop:smalldefrect}, one realizes that we actually have proven a better result. Indeed, the following corollary shows that for infinitely many vector fields $ V \in W^{1,\infty}(\mathbb{R}^2,\R^2)$ there exists $t_0 > 0 $ such that $\bX_t $ is non-zero on the interval $ (0, t_0)$.
More precisely, for any non-empty open subset $\mathcal{U} \subset \partial \omega $, there exists a vector field $V \in W^{1,\infty}(\mathbb{R}^2,\R^2) $ having trace on $\partial\omega$ supported in $\mathcal{U}$
for which the corresponding vector  $\bX_t $ is non-zero for all $t$ belonging to some interval of the form $ (0, t_0) $. Thus the condition $\bX_0=0$ for the rectangle $\omega$ is clearly not stable by shape deformation.

\begin{Cor} \label{cor:local:boundary}
Let $L > 0$ with $\ell > L$ and consider the rectangle $\omega = (0,\ell)\times(0,L)$. Let $\mathcal{U} \subset \partial \omega$ be an open subset of the boundary. The following hold: 
\begin{enumerate}[label=\normalfont(\roman*)]
\item \label{item:Ulateral} Set $\mathcal{U}_{\rm lat} :=\mathcal{U} \cap \left( \{0\} \times (0,L) \cup \{\ell\} \times (0,L) \right)  $ and assume that $\mathcal{U}_{\rm lat} \neq \emptyset$.  Then, there exists $V \in W^{1,\infty}(\R^2,\R^2)$ with trace on $\partial\omega$  supported in  $\mathcal{U}_{\rm lat}$ and $t_0 > 0$ such that the second component of $\bX_{t}$ is non-zero  on $(0,t_0)$; 
\item \label{item:Uflat}  Set $\mathcal{U}_{\rm hor} := \mathcal{U} \cap \left( (0,\ell) \times \{0\} \cup (0,\ell) \times \{L\} \right)  $ and assume that $\mathcal{U}_{\rm hor}\neq \emptyset$. Then, there exists $V \in W^{1,\infty}(\R^2,\R^2)$ with trace on $\partial\omega$ supported in $\mathcal{U}_{\rm hor}$ and $t_0 > 0$ such that the first component   of $\bX_{t}$ is non-zero on $(0,t_0)$.  Moreover, the same property holds also for the second component of $X_{t}$.
\end{enumerate}
\end{Cor}
\begin{proof}[Proof of Corollary \ref{cor:local:boundary}]
Let us prove the first case \ref{item:Ulateral}.
Suppose that $\left( \frac{d}{dt}  \bX_{t} \cdot \begin{pmatrix}0 \\ 1 \end{pmatrix} \right)\rvert_{t = 0}=0$ for all vector fields $V$ having trace on $\partial\omega$  supported in $\mathcal{U}_{\rm lat}$. Then, recalling the expression \eqref{q2:def} of $q_2$, by \eqref{eqn:defexpressionderivXt} and Lemma \ref{lem-density}, one has
\[
    \nabla(\psi_\omega^2) \cdot \begin{pmatrix}0 \\ 1 \end{pmatrix} - \lambda_2^N(\omega) q_2 \psi_\omega + (\nabla q_2) \cdot \nabla \psi_\omega = 0 \quad \text{ on }  \mathcal{U}_{\rm lat}. 
\]
Since $\mathcal{U}$ is open in $\partial \omega$, then the set $\mathcal{U}_{\rm lat}$ has non-zero Hausdorff measure. Thus, reproducing the same calculations as in \eqref{eqn:thisshouldbezero2}, we get that
\begin{equation} \label{condition:U}
\left( - \frac{2}{L} y_2 + 1\right) \left( \sin(\frac{\pi}{\ell} y_1)^2 - \cos(\frac{\pi}{\ell} y_1)^2  \right)= -\left( - \frac{2}{L} y_2 + 1\right) =0,\quad \mbox{ for all }   (y_1,y_2) \in \mathcal{U}_{\rm lat}.
\end{equation}
It is a contradiction, therefore there exists a vector field $V \in W^{1,\infty}(\R^2,\R^2)$ with trace on $\partial\omega$ supported in $\mathcal{U}_{\rm lat}$ for which $\left( \frac{d}{dt}  \bX_{t}\cdot \begin{pmatrix}0 \\ 1 \end{pmatrix} \right)\rvert_{t = 0}\neq 0$.

The proof of the second case \ref{item:Uflat} goes the same way. In particular (for the first component) \eqref{eqn:thisshouldbezero1} gives
\[
 \cos(\frac\pi\ell y_1)\sin(\frac\pi\ell y_1)=0 \qquad  \text{for all }   (y_1,y_2) \in \mathcal{U}_{\rm hor}. 
\]
which again, is not possible. Moreover, as for the  case \ref{item:Ulateral} (for the  second component), one gets

\begin{equation} \label{condition:Ucase1}
\left(- \frac{2}{L} y_2 + 1\right) \left( \sin(\frac{\pi}{\ell} y_1)^2 - \cos(\frac{\pi}{\ell} y_1)^2  \right)=0,\quad \mbox{ for all }   (y_1,y_2) \in \mathcal{U}_{\rm hor}.
\end{equation}
The first factor of the expression on the left never vanishes since the only possible values of $y_2$ are $0$ or $L$. In addition, the second factor vanishes only for $y_1=\ell/4$ or $y_1= 3\ell/4$.
Thus \eqref{condition:Ucase1} is not satisfied on the whole open subset $\mathcal{U}_{\rm hor}$ of the boundary, which leads to a contradiction.
\end{proof}

\subsection{Numerical examples of cross-section $\omega$ satisfying the condition $\bX\neq0$} \label{section:numerics}
In this paragraph we provide numerical evidences that the quantity $\bX$ defined in \eqref{eqn:defX} is non-zero for several domains $\omega$. The computations are implemented with finite elements library FreeFEM++ \cite{hect} and displayed with Matlab. 

\subsubsection{A numerical example for which $\bX \neq 0$}
The first numerical example is an uneven dumbbell domain whose mesh is represented in Figure \ref{dumbellpic}. It is built starting from two disks, the left one of radius 1 and centered in $(-2,0)$, while the right one has radius 2 and center $(2,0)$. The angles at which the connecting strip meets the left circle are $\pm\pi/15$.
Numerically, we obtain
\[
	\bX_{\rm dumbbell} \simeq \begin{pmatrix}-0.1257\\0\end{pmatrix}.
\]
\begin{figure}[H]
\begin{center}
\includegraphics[width=0.6\textwidth]{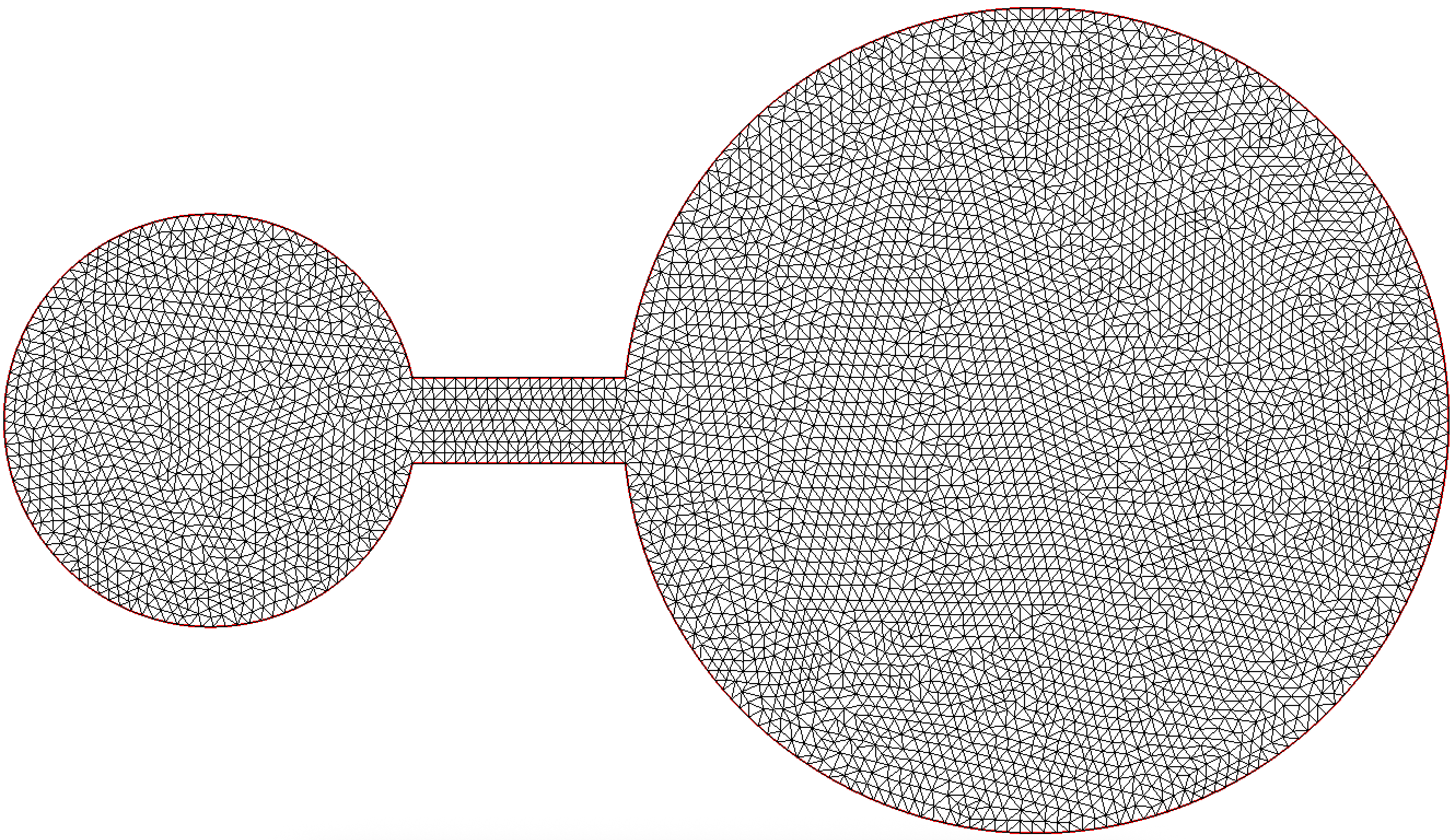}
\end{center}
\caption{Triangular mesh with P2 Lagrange elements. Number of  triangles: 9742; number of vertices: 5042.}
\label{dumbellpic}
\end{figure}

\subsubsection{Bumps in rectangle}
In this second example we are concerned with looking at rectangle deformed with a single bump. In Figure \ref{fig:rectbump} several considered bumps are plotted in different colors. We obtained
\[
\begin{split}
	\bX_{\rm blue} \simeq \begin{pmatrix}0.00315511\\0.00126525\end{pmatrix}&,\quad \bX_{\rm yellow} \simeq \begin{pmatrix}-0.000798985\\0.00301619\end{pmatrix},\\ \bX_{\rm red} \simeq \begin{pmatrix}-0.0152116\\0.00285829\end{pmatrix}&,\quad \bX_{\rm green} \simeq \begin{pmatrix}0.0000322362\\0.000374581\end{pmatrix}.
	\end{split}
\]
Since the bumps represent small perturbations of the initial domain $\omega$, the values of $\bX$ in the perturbed domains are non-zero, although they remain small. Nevertheless, these numerical simulations illustrate the Corollary \ref{cor:local:boundary} of \S \ref{subsubsec:shapederivativerectangle}, namely that different localized perturbations of the boundary of $\omega$ lead to a non-zero $\bX$.
\begin{figure}[H]
\begin{center}
\includegraphics[width=0.6\textwidth]{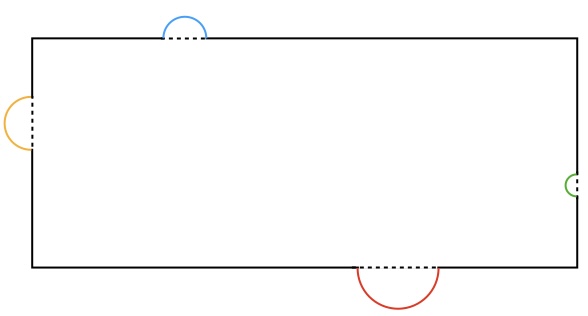}
\end{center}
\caption{Rectangle of base $2\pi$ and height $\pi$ with bumps. The bumps are half-disks of radius given as follows: $r_{\rm blue}= 0.2$, $r_{\rm yellow}= 0.3$, $r_{\rm red}= 0.5$, $r_{\rm green}= 0.1$. The distance of their center to the closest corner of the rectangles are as follows: $d_{\rm blue} = 1.7$, $d_{\rm yellow} = \pi - 2$, $d_{\rm red} = 2\pi - 4$, $d_{\rm green} = 1.1$.}
\label{fig:rectbump}
\end{figure}

\subsubsection{Bump's graph}
In this third example, we focus on bumps on the rectangle keeping  the center of the half-disk fixed (see Figure \ref{fig:rectsevbumpsameradius}).
This example provides a numerical illustration of formula \eqref{eq.derivee}, showing that, in the vicinity of $t = 0$, the norm of $X_t$ increases with respect to $t$.

\begin{figure}[H]
\begin{center}
    \includegraphics[width=0.6\textwidth]{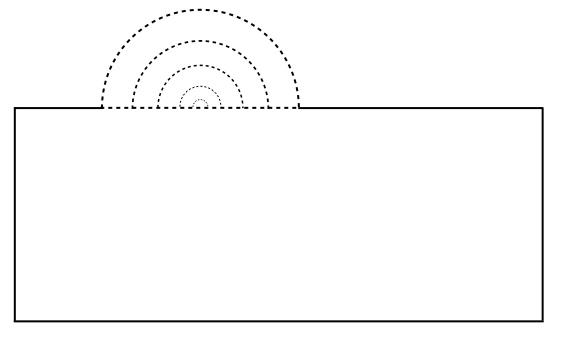}
\end{center}
\caption{Deformation of the rectangle of base $2\pi$ and height $\pi$ for a given vector field $V$ and different values of $t$.}
\label{fig:rectsevbumpsameradius}
\end{figure}

In the graph down below (cf. Figure \ref{fig:rectsevbumpsameradius}) we plotted 89 values of the deformation $\omega_t$ of the rectangle, where the dependence of $\bX_t$ with respect to $t$ can be observed.
The deformed rectangles $\omega_t$ are obtained by adding bumps to the upper side of the rectangle $[0,2\pi]\times[0,\pi]$ consisting of a half-disk with the center lying 1.7 on the right of the upper left corner and with increasing radius $r_t=t$, starting from $t=0.2$ up to $t=0.64$ with increments of 0.05.
\begin{figure}[H]
\begin{center}
    \includegraphics[width=\textwidth]{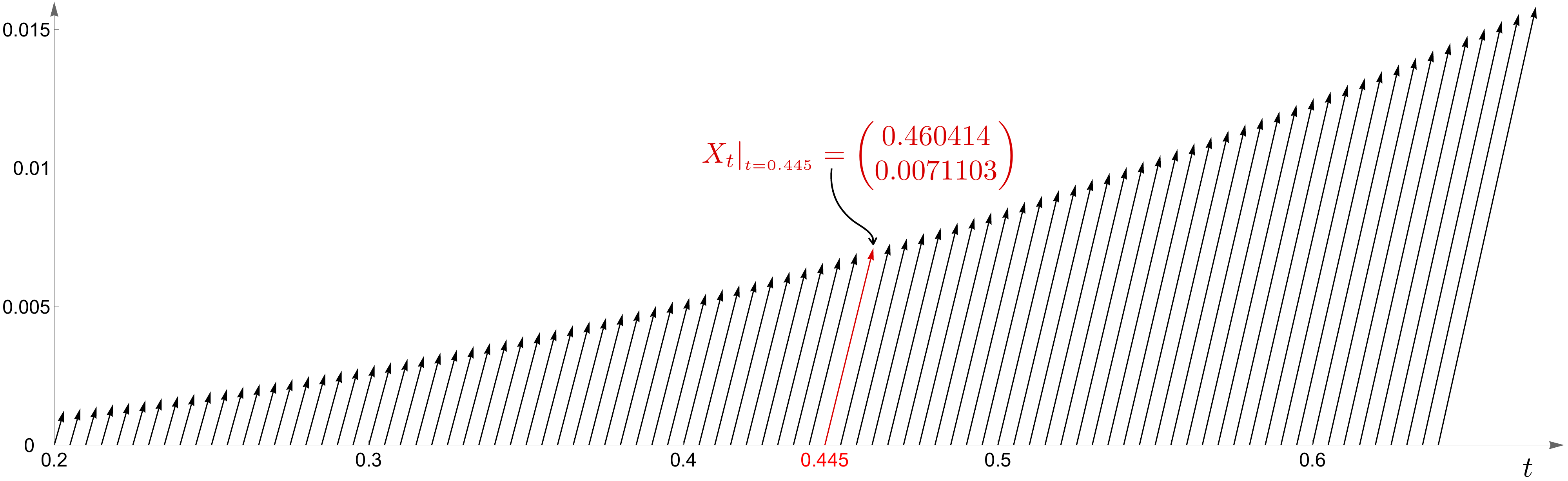}
\end{center}
\caption{The value of $\bX_t$ as the parameter $t$ changes. Note that in this graph, the horizontal axis  denotes both the radius $r_t=t$ of the bump of the perturbed rectangle $\omega_t$, as well as the axis for the first coordinate of $\bX_t$. Each vector $\bX_t$ is represented in the plane with the origin in the point $(t,0)$ for each $t \in [0.2,0.64]$. The vector $\bX_t$ at $t=0.445$ is highlited in red as an example.
As one can observe, the bigger the  deformation, the bigger the vector $\bX_t$ becomes. The direction in which each $\bX_t$ points does not seem to change much.}
\label{fig:graphofX}
\end{figure}

\section*{Acknowledgements \& Funding} 
This work received support from the french government under the France 2030 investment plan, as part of the Initiative d’Excellence d’Aix-Marseille Université – A*MIDEX AMX-21-RID-012.\\[6pt]
\noindent The authors would like to thank their colleague  Anne-Sophie Bonnet-Ben Dhia for the reference \cite{filonov21}. The authors are also grateful to their colleague   David Krej\v{c}i\v{r}\'ik for introducing them to the notion of relatively adapted parallel frame with the references \cite{bishop,Krej-12}.

\appendix

\section{Construction of relatively adapted parallel frame}\label{appendix:relatparallel}
This appendix closely follows the construction explained in \cite{bishop,Krej-12}. For the sake of completeness, we recall it here in the specific case where $\Gamma \subset \R^3$ is a curve with a $\mathscr{C}^2$, injective and arc-length parametrization $\gamma : \R \to \R^3$.\\

By definition, a  frame is an orthonormal basis field of $\R^3$. Moreover a frame is positively (resp. negatively) oriented  if it is a positively oriented (resp.negatively oriented)  orthonormal basis  and we need a few more definitions.  

\begin{Def}
A frame $(\mathbf{e}_1,\mathbf{e}_2,\mathbf{e}_3)$ is said to be \emph{adapted} to the curve $\Gamma$ if for any $s\in \R$ and $j\in\{1,2,3\}$, $\mathbf{e}_j(s)$ is either colinear to $\gamma'(s)$ or perpendicular to it.
\end{Def}

\begin{Def}\label{Def-relatively-parallel}
A vector field $\mathbf{f}$ is said to be \emph{normal} to the curve $\Gamma$ if for any $s \in \R$
\[
\mathbf{f}(s)\cdot \gamma'(s) = 0.
\]
Such a vector field is said to be \emph{relatively parallel} to $\Gamma$ if for all $s\in \R$, $\mathbf{f}'(s)$ is colinear to $\gamma'(s)$.
\end{Def}

\begin{Def}
A frame $(\mathbf{e}_1,\mathbf{e}_2,\mathbf{e}_3)$ is said to be a \emph{relatively adapted parallel frame} to $\Gamma$ if it is an adapted frame to $\Gamma$ with $\mathbf{e}_1 = \gamma'$ and $\mathbf{e}_2$ as well as $\mathbf{e}_3$ are relatively parallel to $\Gamma$.
\end{Def}

\begin{Rem}
All the above definitions can be generalized to local vector fields and frames, i.e. vector fields and frames which are only defined on a bounded interval of $\R$.
\end{Rem}

\begin{Rem} \label{rem:constlength}
Let $\mathbf{f}$ be a (local) vector field normal to $\Gamma$. Assume also that $\mathbf{f}$ is  relatively parallel to $\Gamma$. Then $\bV$ has  constant length, because 
\begin{equation*}
\left(|\mathbf{f}|^2\right)' = (\mathbf{f} \cdot \mathbf{f})' = 2 \mathbf{f} \cdot \mathbf{f}' =0
\end{equation*}
by definition.
\end{Rem}

\begin{Lem} \label{lem:uniqframe}
Let ${\rm I},{\rm J} \subset \R$ be two bounded intervals of $\R$. Let  $(\mathbf{e}_1,\mathbf{e}_2,\mathbf{e}_3)$ be a local relatively adapted parallel frame to $\Gamma$ on ${\rm I}$ and $(\xi_1,\xi_2,\xi_3)$ a local relatively adapted parallel frame to $\Gamma$ on ${\rm J}$.
If  $(\mathbf{e}_1(s_\star),\mathbf{e}_2(s_\star),\mathbf{e}_3(s_\star))= (\xi_1(s_\star),\xi_2(s_\star),\xi_3(s_\star))$ for some $s_\star \in {\rm I} \cap {\rm J}$, then $(\mathbf{e}_1,\mathbf{e}_2,\mathbf{e}_3)\equiv (\xi_1,\xi_2,\xi_3)$ on $I \cap J$.
\end{Lem}
\begin{proof}[Proof of Lemma~\ref{lem:uniqframe}]
Obviously we have $\mathbf{e}_1=\xi_1=\gamma'$  on ${\rm I} \cap {\rm J}$. 
Fix $j \in \{2,3\}$. As the (normal) vectors  $e_j$ and $\xi_j$ are relatively parallel to $\Gamma$, one  observe that their difference  $e_j -\xi_j$ is also relatively parallel to $\Gamma$. Thus, by Remark \ref{rem:constlength}, the vector field $e_j -\xi_j$ has constant length. As $|(e_j -\xi_j)(s_\star)|=0$, it follows that $e_j -\xi_j=0$ on ${\rm I} \cap {\rm J}$. This concludes the proof.
\end{proof}

The main goal of this appendix is to prove the following proposition.
\begin{Pro} \label{ex:un:RAPF}
Let $(\gamma'(0),e_{0,2},e_{0,3})$ be a positively oriented orthonormal basis of $\R^3$. There exists a unique $\mathscr{C}^1$-regular relatively adapted parallel frame $(\mathbf{e}_1,\mathbf{e}_2,\mathbf{e}_3)$ to $\Gamma$ verifying the initial condition $(\mathbf{e}_1(0),\mathbf{e}_2(0),\mathbf{e}_3(0)) = (\gamma'(0),\mathbf{e}_{0,2}, \mathbf{e}_{0,3})$. Moreover, there exists two (unique) continuous functions $k_1,k_2:\R \to \R$  such that:
\begin{equation} \label{diffeq:parframe}
    \left\{
        \begin{array}{lcl}
            \displaystyle\frac{d \mathbf{e}_1}{ds}(s) & = & k_1(s) \mathbf{e}_2(s) + k_2(s) \mathbf{e}_3(s),  \vspace{3pt} \\
            \vspace{3pt}
            \displaystyle\frac{d \mathbf{e}_2}{ds}(s) & = & - k_1(s) \mathbf{e}_1(s),\\
            \displaystyle\frac{d \mathbf{e}_3}{ds}(s) & = & -k_2(s) \mathbf{e}_1(s).
        \end{array}
    \right.
\end{equation}
\label{prop:relatparallelframe}
and the frame $(\mathbf{e}_1(s),\mathbf{e}_2(s),\mathbf{e}_3(s))$  is positively oriented for all $s\in \mathbb{R}$.
\end{Pro}
Before going through the proof of Proposition \ref{prop:relatparallelframe} we need the following lemma whose proof is postponed at the end of this section.
\begin{Lem} \label{lemma:mininter}
Let $\gamma: \R \to \R^d$ be a $\mathscr{C}^2$-regular curve parametrized by arc-length, with $d \geq 2$, and assume that $\gamma'' \in L^\infty(\R)^d$ and denote by $\kappa=|\gamma''|$. 
Let $j \in \{1,\dots,d\}$. Let $s \in \R$ and $C=\max(\|\kappa\|_{L^\infty(\R)},1)>0$. Assume that $|\gamma'_j(s)| \geq 1/\sqrt{d}$, then one has $$
|\gamma'(s)|>0, \quad  \forall s\in \big[s-\frac{1}{2\, C\, \sqrt {d}}, s+ \frac{1}{2\, C\,  \sqrt{d}}\big].
$$
\end{Lem}
\begin{Rem} \label{rem:1sqrtd}
Let $s \in \R$. Since $\gamma$ is parametrized by arc-legth, then $|\gamma'(s)|^2= \sum_{j=1}^d |\gamma'_j(s)|^2=1$.
Hence for at least one $j \in \{1,\dots,d\}$ one has that $|\gamma'_j(s)|^2 \geq 1/d$.
\end{Rem}
\begin{proof}[Proof of Proposition \ref{prop:relatparallelframe}]
We first notice that the global uniqueness of the $\mathscr{C}^1$-smooth relatively adapted parallel frame on the whole real line is proved as Lemma \ref{lem:uniqframe}  (for the local uniqueness) by taking ${\rm I} = {\rm J} =\mathbb{R}$ in the corresponding proof. The proof of the existence of  a  $\mathscr{C}^1$-regular positively oriented relatively parallel frame is divided in two steps. In a first step we construct locally such a frame and then glue each piece together using the fact that $|\gamma'(s)| = 1$ for all $s \in \R$. 
\paragraph{Step 1: local construction}  We consider the origin $s=0$ and we assume without loss of generality that $|\gamma_1'(0)| \geq 1/\sqrt{3}$ (see Remark \ref{rem:1sqrtd}). By Lemma~\ref{lemma:mininter} with $d=3$, for $s_0=1/(\sqrt{2} \, C\, d) >0$ (with $C=\max(\|\kappa\|_{\infty},1)$), one has  $|\gamma_1'(s)| > 0$ for all $s \in [-s_0,s_0]$.
For $s \in [-s_0,s_0]$, we can define the local direct frame:
\[
    \mathbf{f}_1(s) = \gamma'(s),\quad \mathbf{f}_2(s) =\frac1{\sqrt{(\gamma_1'(s))^2 + (\gamma_2'(s))^2}}(-\gamma_2'(s),\gamma_1'(s),0),\quad \mathbf{f}_3(s) := \mathbf{f}_1(s) \times \mathbf{f}_2(s).
\]
Note that $\mathbf{f}_2$ is well defined and non-zero because $\gamma_1'$ is non-zero on $[-s_0,s_0]$.
Moreover, $(\mathbf{f}_1,\mathbf{f}_2,\mathbf{f}_3)$ is an adapted frame and for all $j \in \{1,2,3\}$, $\mathbf{f}_j$ is of class $\mathscr{C}^1$ on $[-s_0,s_0]$. Note, that as defined, for all $s\in [-s_0,s_0]$, one has $|\mathbf{f}_j(s)| = 1$, in particular for $j\in \{1,2,3\}$, there holds
\[
    \mathbf{f}_j'(s) \cdot \mathbf{f}_j(s) = 0.
\]
Thus, for all $j\in \{1,2,3\}$, $\mathbf{f}_j' \in \mathrm{span}\{\mathbf{f}_k\}_{k\in \{1,2,3\}\setminus\{j\}}$ which rewrites
\begin{equation}\label{eqn:derpropframe}
    \mathbf{f}_1' = (\mathbf{f}_1'\cdot \mathbf{f}_2) \mathbf{f}_2 + (\mathbf{f}_1'\cdot \mathbf{f}_3)\mathbf{f}_3,\quad \mathbf{f}_2' = (\mathbf{f}_2'\cdot \mathbf{f}_1)\mathbf{f}_1 + (\mathbf{f}_2'\cdot \mathbf{f}_3)\mathbf{f}_3,\quad \mathbf{f}_3' = (\mathbf{f}_3'\cdot \mathbf{f}_1)\mathbf{f}_1 + (\mathbf{f}_3'\cdot \mathbf{f}_2) \mathbf{f}_2.
\end{equation}
Moreover, since $\mathbf{f}_j\cdot \mathbf{f}_k = \delta_{j,k}$, differentiating with respect to $s$, one gets 
\[
    (\mathbf{f}_1'\cdot \mathbf{f}_2) = - (\mathbf{f}_2'\cdot \mathbf{f}_1),\quad (\mathbf{f}_1'\cdot \mathbf{f}_3) = -(\mathbf{f}_3'\cdot \mathbf{f}_1),\quad (\mathbf{f}_2'\cdot \mathbf{f}_3) = -(\mathbf{f}_3'\cdot \mathbf{f}_2).
\]
Defining the $\mathscr{C}^0$ maps $p_{1,2} := (\mathbf{f}_1'\cdot \mathbf{f}_2)$, $p_{1,3} := (\mathbf{f}_1'\cdot \mathbf{f}_3)$ and $p_{2,3} := (\mathbf{f}_2'\cdot \mathbf{f}_3)$, \eqref{eqn:derpropframe} rewrites
\begin{equation}\label{eq.derivfframe}
        \mathbf{f}_1' = p_{1,2}\ \mathbf{f}_2 + p_{1,3}\ \mathbf{f}_3,\quad \mathbf{f}_2' = -p_{1,2}\ \mathbf{f}_1 + p_{2,3}\ \mathbf{f}_3,\quad \mathbf{f}_3' = -p_{1,3}\ \mathbf{f}_1 -p_{2,3}\ \mathbf{f}_2.
\end{equation}
Now we would like to find a rotation of angle $\varphi \in \mathscr{C}^1([-s_0,s_0])$, so
that on $[-s_0,s_0]$, the direct adapted parallel frame $(\mathbf{e}_1, \mathbf{e}_2, \mathbf{e}_3)$ is defined in term of the direct frame  $(\mathbf{f}_1, \mathbf{f}_2, \mathbf{f}_3)$ by
\begin{equation}\label{eq.defe}
\mathbf{e}_1=\mathbf{f}_1, \quad     \mathbf{e}_2 = \cos(\varphi) \mathbf{f}_2 - \sin(\varphi)\mathbf{f}_3,\quad \mathbf{e}_3 = \sin(\varphi) \mathbf{f}_2 +\cos(\varphi)\mathbf{f}_3.
\end{equation}
Thus, if such  a function $\varphi$ exists, one has
\begin{eqnarray} \label{eq.cauchysystem2}
    \mathbf{e}_2' &=& (-p_{1,2} \cos(\varphi) + p_{1,3}\sin(\varphi)) \mathbf{e}_1 + (-\varphi' + p_{2,3})\mathbf{e}_3,\\[4pt] \mathbf{e}_3'& =& -(p_{1,2}\sin(\varphi) + p_{1,3}\cos(\varphi))\mathbf{e}_1 + (\varphi' - p_{2,3})\mathbf{e}_2.
\end{eqnarray}
To obtain a relatively adapted parallel frame to $\Gamma$ we need to chose $\varphi' = p_{2,3}$. Hence, for $s\in [s_0,s_0]$, one chooses $\varphi = \varphi_0 + \int_0^sp_{2,3}(\tau)\rmd \tau$ for some constant $\varphi_0$. Now, to satisfy the initial condition, one needs to chose the angle $\varphi_0$ such that
\[
 \mathbf{e}_{0,2} = \cos(\varphi_0) \mathbf{f}_2(0) - \sin(\varphi_0)\mathbf{f}_3(0),\quad \mathbf{e}_{0,3} = \sin(\varphi_0) \mathbf{f}_2(0) +\cos(\varphi_0)\mathbf{f}_3(0)
\]
which is always possible because of the chosen orientation. In particular, one defines the maps $k_1,k_2$ on the line segment $[s_0,-s_0]$ as
\begin{equation}\label{def.k1-k2}
    k_1 = p_{1,2}\cos(\varphi) - p_{1,3}\sin(\varphi),\quad k_2 = p_{1,2}\sin(\varphi) + p_{1,3}\cos(\varphi)
\end{equation}
and remark that they are continuous on $[-s_0,s_0]$.\\[4pt]
One can now easily check that constructed frame $(\mathbf{e}_1, \mathbf{e}_2, \mathbf{e}_3)$ defined by \eqref{eq.defe} is a $\mathscr{C}^1-$smooth positively oriented adapted parallel frame on $[-s_0,s_0]$. In addition, combining the two first equations of \eqref{eq.cauchysystem2} and \eqref{def.k1-k2} yields the second and the third equations of \eqref{diffeq:parframe} on $[-s_0,s_0]$. Furthermore, by virtue of \eqref{eq.defe}, \eqref{eq.derivfframe} and \eqref{def.k1-k2}, one deduces that on $[-s_0,s_0]$:
\begin{eqnarray}\label{eq.cauchysystem1}
\mathbf{e}_1'&=&\mathbf{f}_1'= p_{1,2}\ \mathbf{f}_2 + p_{1,3}\ \mathbf{f}_3 \nonumber \\
&=& p_{1,2}\ (\cos(\phi) \mathbf{e}_2+ \sin(\phi) \mathbf{e}_3 )+ p_{1,3}\ (\cos(\phi) \mathbf{e}_3-\sin(\phi) \mathbf{e}_2 ) \nonumber\\
&=& k_1 \mathbf{e}_2+k_2 \, \mathbf{e}_3, 
\end{eqnarray}
which is precisely the first equation of \eqref{diffeq:parframe}. Moreover, one gets directly from \eqref{eq.cauchysystem1} that
\[
    |\mathbf{e}_1'|^2 = \kappa^2=k_1^2 + k_2^2 .
\]
\paragraph{Step 2: ``from local to global".} Considering the change of variables $s \mapsto s+s_0$ we set $\tilde{\gamma}(s) := \gamma(s+s_0)$. Reproducing the same argument as in Step 1 for the new curve $\tilde \gamma$ with the new initial condition $(\mathbf{e}_1(s_0),\mathbf{e}_2(s_0),\mathbf{e}_3(s_0))$ yields a local relatively adapted parallel frame to $\tilde \gamma$  on $[-s_0,s_0]$, denoted by $(\tilde\xi_1,\tilde\xi_2,\tilde\xi_3)$.
Considering $\xi_j(s):= \tilde\xi_j(s-s_0)$ for $j=1,2,3$ and for $s \in [0,2s_0]$, yields a local $\mathscr{C}^1$-smooth frame $(\xi_1,\xi_2,\xi_3)$ on $[0,2s_0]$. It is easy to check that $(\xi_1,\xi_2,\xi_3)$ is a local relatively adapted parallel positively oriented frame to $\gamma$. 
Moreover we have that $(\xi_1(s_0),\xi_2(s_0),\xi_3(s_0)) = (\mathbf{e}_1(s_0),\mathbf{e}_2(s_0),\mathbf{e}_3(s_0))$. Hence by Lemma \ref{lem:uniqframe} we have that $(\xi_1,\xi_2,\xi_3) \equiv (\mathbf{e}_1,\mathbf{e}_2,\mathbf{e}_3)$ on $[0,s_0]$.
Reproducing the same argument using now the change of variables $s \mapsto s-s_0$, we can extend the local relatively adapted parallel frame $(\mathbf{e}_1,\mathbf{e}_2,\mathbf{e}_3)$ from $[-s_0,s_0]$ to $[-2s_0,2s_0]$, and by iteration, to $[-n s_0,n s_0]$ for any $n \in \mathbb{N}$, thus to the whole of $\R$.

Finally, observe that once we are given a relatively  adapted parallel frame and its  evolution along $s \in \R$ is of the  form in \eqref{diffeq:parframe} for some unique  continuous functions $k_1,k_2:\R \to \R$.
\end{proof}

\begin{proof}[Proof of Lemma~\ref{lemma:mininter}]
Observe that since $\gamma'$ is $\mathscr{C}^1$ and $\gamma''$ is bounded, by the mean-value inequality, $\gamma'$ is  uniformly Lipschitz continuous with a Lipschitz constant $\|\kappa\|_{L^\infty(\R)}$. Thus, in particular, one has for   $C=\max(\|\kappa\|_{L^\infty(\R)},1)>0$:
$$
 |\gamma'(\tilde{s})- \gamma'(s)|\leq \|\kappa\|_{L^\infty(\R)} |s-\tilde{s}|\leq  C\, \frac{1}{2 \sqrt{d} \, C}=\frac{1}{2 \sqrt{d}}, \quad \forall \tilde{s}\in {\rm I}_s:=\Big[s-\frac{1}{2 \sqrt{d} \, C}, s+\frac{1}{2 \sqrt{d} \, C} \Big].
$$
Thus, since $|\gamma_j'(s)|\geq 1/\sqrt{d}$, by the reverse triangle inequality, one gets that  for all  $\tilde{s}\in {\rm I}_{\tilde{s}}$:
$$
|\gamma'(\tilde{s})|=|\gamma'(s)-(\gamma'(s)-\gamma'(\tilde{s}))|\geq \big| |\gamma'(s)|- |\gamma'(s)-\gamma'(\tilde{s})|\big| =|\gamma'(s)|- |\gamma'(s)-\gamma'(\tilde{s})|\geq \frac{1}{2 \sqrt{d}}>0. 
$$
This concludes the proof.




\end{proof}

\section{Some results about  functional spaces} \label{appendix}

Let $\mathscr{O}$  be a (non-empty) open subset of $\R^3$.
The first proposition shows that defining the space $H_0^1(\mathscr{O})$ as in \eqref{eqn:def_H10} is equivalent with the standard definition of this space as the closure of $\mathfrak{D}(\mathscr{O})$ with respect to the $H^1(\mathscr{O})$-norm.
\begin{Lem}\label{prop:defH1coincide} 
Let $\mathscr{O}$ be an open subset of $\R^3$. The closure of $\mathfrak{D}(\mathscr{O})$ with respect to the $H^1(\mathscr{O})$-norm is $H_0^1(\mathscr{O})$ as defined in \eqref{eqn:def_H10}.
\end{Lem}
\begin{proof}[Proof of Lemma \ref{prop:defH1coincide}]
Let us start by proving that the operator $\mathscr{L} := \nabla$ defined on $D(\mathscr{L}) = \mathfrak{D}(\mathscr{O})$ is closable. To this aim, consider the operator $\mathscr{T} := \nabla$ defined on $D(\mathscr{T}) := H_0^1(\mathscr{O})$. Let $(f_n,\nabla f_n) \in H_0^1(\mathscr{O}) \times \bL^2(\mathscr{O})$ be a sequence which converges in $L^2(\mathscr{O}) \times \bL^2(\mathscr{O})$ to $(f,\bv) \in  L^2(\mathscr{O}) \times \bL^2(\mathscr{O})$. Let $\bu \in H(\Div,\mathscr{O})$, by definition of the space $H(\Div,\mathscr{O})$, there holds
\[
    (\nabla f_n, \bu)_{\bL^2(\mathscr{O})} = -(f_n,\Div \bu)_{L^2(\mathscr{O})}.
\]
Taking the limit as $n\to +\infty$ there holds
\[
    (\bv, \bu)_{\bL^2(\mathscr{O})} = -(f,\Div \bu)_{L^2(\mathscr{O})}.
\]
Taking $\bu \in \bbD(\mathscr{O})$, in the sense of distributions it implies that $\nabla f = \bv \in \bL^2(\mathscr{O})$. It gives that for all $\bu \in H(\Div,\mathscr{O})$ one has
\[
    (\nabla f, \bu)_{\bL^2(\mathscr{O})} = -(f,\Div \bu)_{L^2(\mathscr{O})}.
\]
which is precisely saying that $f \in H_0^1(\mathscr{O})$. Hence $\mathscr{T}$ is a closed operator and as $\mathscr{T}|_{\mathfrak{D}(\mathscr{O})} = \mathscr{L}$ then $\mathscr{L}$ is a closable operator.

Now, remark that by \eqref{eqn:def_H10}, $H(\Div,\mathscr{O})\subset D(\mathscr{L}^*) $ and  using distribution theory, one proves easily the reverse inclusion. Moreover,  by \eqref{eqn:def_H10}, one has for  $\bu \in  D(\mathscr{L}^*)=H(\Div,\mathscr{O})$, $\mathscr{L}^* u = - \Div \bu$. Similarly, one can show that $(\mathscr{L}^*)^* = \mathscr{T}$. Finally, thanks to \cite[Theorem 2.4.]{kaas}, one gets $\mathscr{T} = (\mathscr{L}^*)^* = \overline{\mathscr{L}}$ which yields the result.
\end{proof}

Similarly one can prove a density results for the space $H(\Div, \mathscr{O})$.
\begin{Lem} \label{prop:defHdiv:density} 
Let $\mathscr{O}$ be an open subset of $\R^3$. The closure of $\bbD(\mathscr{O})$ with respect to the $H(\Div,\mathscr{O})$-norm is $H_0(\Div, \mathscr{O})$ as defined in \eqref{DEF:H0div} with $\Sigma=\mathds{1}_3$.
\end{Lem}
\begin{Rem} \label{remark:sigmasmooth}
It is straightforward to see that $H_0(\Div \Sigma,\mathscr{O})=\{\Sigma^{-1} \bu : \bu \in H_0(\Div,\mathscr{O})\}$ given Definition \eqref{DEF:H0div}. Hence, observing  that $\|\cdot \|_{\bL^2_\Sigma(\mathscr{O})}$ and $\|\Sigma \cdot \|_{\bL^2(\mathscr{O})}$ are equivalent norms (cf. \eqref{equiv:topspaces}), if $\Sigma$ is $\mathscr{C}^\infty$-smooth (and thus so is $\Sigma^{-1}$ by assumptions \eqref{hyp:sigma}),  Lemma \ref{prop:defHdiv:density}  implies that  the closure of $\bbD(\mathscr{O})$ with respect to the $H(\Div\Sigma,\mathscr{O})$-norm is $H_0(\Div \Sigma, \mathscr{O})$.
\end{Rem}
\begin{proof}[Proof of Lemma \ref{prop:defHdiv:density}]
We give a sketch of the proof.

Consider the operators $\mathscr{L}:=\Div$ defined on $D(\mathscr{L})=\bbD(\mathscr{O})$ and $\mathscr{T}:=\Div$ defined on $D(\mathscr{T})=H_0(\Div,\mathscr{O})$. One first verifies that $\mathscr{T}$ is an extension of $\mathscr{L}$, and it is closed (thus $\mathscr{L}$ is closable).
Then one shows that $(\mathscr{L}^*)^* = \mathscr{T}$. Using the fact that $(\mathscr{L}^*)^*=\overline{\mathscr{L}}$ one concludes.
\end{proof}

The following Lemma is standard in the study of Maxwell operators but we recall it here because the waveguide $\Omega$ as defined in \S \ref{section:waveguides} is an unbounded domain.
\begin{Lem} The space $\bbD(\mathscr{O})$ is dense in $H_0(\curlvec,\mathscr{O})$ for the $H(\curlvec,\mathscr{O})$-norm.
\label{lem:55}
\end{Lem}

\begin{proof}[Proof of Lemma \ref{lem:55}]
Let us consider the operator $\mathscr{L}$ acting in $\bL^2(\mathscr{O})$ defined on $D(\mathscr{L}) := \bbD(\mathscr{O})$ and acting for $\bu \in D(\mathscr{L})$ as
\[
    \mathscr{L}\bu := \curlvec \bu.
\]
We remark also that $\mathscr{L}$ is a symmetric and densely defined operator. Hence, it is closable. By definition, the adjoint $\mathscr{L}^*$ of $\mathscr{L}$ has the following domain
\begin{equation*}
    D(\mathscr{L}^*) = \{\bv \in \bL^2(\mathscr{O}) : \exists\ \bw \in \bL^2(\mathscr{O}) \mbox{ such that } (\mathscr{L}\bu, \bv)_{\bL^2(\mathscr{O})} = (\bu, \bw)_{\bL^2(\mathscr{O})}, \forall  \bu \in \bbD(\mathscr{O})\}
\end{equation*}
and for $\bv \in D(\mathscr{L}^*)$, we set $\mathscr{L}^* \bv = \bw$. Note that for such a $\bv$, for all $\bu \in \bbD(\mathscr{O})$, there holds
\begin{align*}
    (\overline{\bw},\bu)_{\bbD'(\mathscr{O}),\bbD(\mathscr{O})} = (\bu,\bw)_{\bL^2(\mathscr{O})} = (\mathscr{L}\bu, \bv)_{\bL^2(\mathscr{O})} = (\overline{\bv},\mathscr{L}\bu)_{\bbD'(\mathscr{O}),\bbD(\mathscr{O})}& = (\overline{\bv},\curlvec \bu)_{\bbD'(\mathscr{O}),\bbD(\mathscr{O})}\\ &=(\overline{\curlvec\bv}, \bu)_{\bbD'(\mathscr{O}),\bbD(\mathscr{O})}.
\end{align*}
Hence, in $\bbD'(\mathscr{O})$, one has $\curlvec \bv = \bw \in \bL^2(\mathscr{O})$. It follows that $\bv \in H(\curlvec,\mathscr{O})$ proving that $D(\mathscr{L}^*) \subset H(\curlvec,\mathscr{O})$. The reverse inclusion holds by definition of the space $H(\curlvec,\mathscr{O})$. This proves that
\[
    D(\mathscr{L}^*) = H(\curlvec,\mathscr{O}),\ \mbox{ and } \ \mathscr{L}^* \bu = \curlvec \bu \quad  \text{for all } u\in D(\mathscr{L}^*).
\]
$\mathscr{L}^*$ is called the maximal $\curlvec$ operator. Now,  the minimal $\curlvec$ operator is $\overline{\mathscr{L}}$, the smallest closed extension of $\mathscr{L}$. Note that there holds $
    (\mathscr{L}^*)^* = \overline{\mathscr{L}}
$
yielding
\begin{eqnarray*}
    D(\overline{\mathscr{L}})& =&D\left((\mathscr{L}^*)^*\right) \\&=& \{\bu \in \bL^2(\mathscr{O}) :  \exists\ \bw \in \bL^2(\mathscr{O})  \mbox{ such that } (\mathscr{L}^*\bv, \bu)_{\bL^2(\mathscr{O})} = (\bv, \bw)_{\bL^2(\mathscr{O})}, \, \forall \, \bv \in H(\curlvec,\mathscr{O})\},
\end{eqnarray*}
and for $\bu \in D(\overline{\mathscr{L}})$ we set $\overline{\mathscr{L}}\bu = (\mathscr{L}^*)^*\bu := \bw$. Now take $\bu \in D(\overline{\mathscr{L}})$ and let $\bw = \overline{\mathscr{L}} \bu$. Note that for such a $\bu$, for all $\bv \in \bbD(\mathscr{O})\subset H(\curlvec,\mathscr{O})$, as $\mathscr{L}^*\bv = \mathscr{L}\bv$, there holds
\begin{align*}
    (\overline{\bw},\bv)_{\bbD'(\mathscr{O}),\bbD(\mathscr{O})} = (\bv,\bw)_{\bL^2(\mathscr{O})} = (\mathscr{L}^*\bv, \bu)_{\bL^2(\mathscr{O})} = (\overline{\bu},\mathscr{L}\bv)_{\bbD'(\mathscr{O}),\bbD(\mathscr{O})}& = (\overline{\bu},\curlvec \bv)_{\bbD'(\mathscr{O}),\bbD(\mathscr{O})}\\ &=(\overline{\curlvec\bu}, \bv)_{\bbD'(\mathscr{O}),\bbD(\mathscr{O})}.
\end{align*}
It proves that in $\bbD'(\mathscr{O})$ one has $\overline{\mathscr{L}}\bu = \curlvec \bu = \bw \in \bL^2(\mathscr{O}) $. In particular, $\bu \in H(\curlvec,\mathscr{O})$ and for all $\bv \in H(\curlvec,\mathscr{O})$ one has
\[
    (\curlvec \bv , \bu)_{\bL^2(\mathscr{O})} = ( \bv , \curlvec \bu)_{\bL^2(\mathscr{O})}
\]
which by definition means that $\bu \in H_0(\curlvec,\mathscr{O})$. This yields that  $D(\overline{\mathscr{L}})\subset H_0(\curlvec,\mathscr{O})$. The converse inclusion is true by definition of $H_0(\curlvec,\mathscr{O})$. Hence, one gets
\[
    D(\overline{\mathscr{L}}) = H_0(\curlvec,\mathscr{O}),\quad \overline{\mathscr{L}} \bu = \curlvec \bu \quad  \text{for all } \bu\in D(\overline{\mathscr{L}}). 
\]
Now, as $\overline{\mathscr{L}}$ is the closure of $\mathscr{L}$ it means that
\[
    \overline{\{(\bu,\curlvec u) \in \bL^2(\mathscr{O})^2 : \bu \in \bbD(\mathscr{O})\}} = \{(\bu,\curlvec u) \in \bL^2(\mathscr{O})^2 : \bu \in H_0(\curlvec,\mathscr{O})\}.
\]
In particular, the set $\bbD(\mathscr{O})$ is dense in $H_0(\curlvec,\mathscr{O})$ for the $H(\curlvec,\mathscr{O})$-norm.

\end{proof}

\section{Classical spectral properties of the Maxwell operator}\label{appendix:specMaxwell}
The first proposition is well known and characterizes the kernel of the operator $\hat{\mathcal{A}}$ and its orthogonal by means of solenoidal spaces.

\begin{Pro}\label{prop:kerneldescription} There holds
\[
	\ker(\hat{\mathcal{A}}) = \overline{\mathcal{G}(\Omega_0)},\quad \ker(\hat{\mathcal{A}})^\perp = \mathcal{J}_{\beps,\bmu}(\Omega_0).
\]
Here, $\beps$ and $\bmu$ are defined in Proposition \ref{prop:unitequivpiola} and the solenoidal subspace $\mathcal{J}_{\beps,\bmu}(\Omega_0)$ is defined in \eqref{solenoidal:set}.
\end{Pro}
\begin{proof}[Proof of Proposition \ref{prop:kerneldescription}]~

\noindent
{\bf Step 1:} In the first step we work with the operator $\mathcal{A}_0$ defined in \eqref{A0:deffi}.\\[6pt]
\noindent One can check that $\mathcal{G}(\Omega_0) \subset \ker(\mathcal{A} _0)$, and since $\mathcal{A} _0$ is a closed operator then $\overline{\mathcal{G}(\Omega_0)}\subset \mathrm{ker}(\mathcal{A}_0)$. Therefore $\mathrm{ker}(\mathcal{A}_0)^\perp \subset \mathcal{G}(\Omega_0)^\perp$.
Using Lemma \ref{lemma:ortodeco} (see below) with $B=\mathrm{ker}(\mathcal{A}_0)^\perp$ and $Z= \mathcal{G}(\Omega_0)^\perp$, we get that 
\begin{equation*}
    \mathcal{G}(\Omega_0)^\perp = \mathrm{ker}(\mathcal{A}_0)^\perp \overset{\bL^2_0(\Omega_0)}{\oplus} \left( \mathrm{ker}(\mathcal{A}_0) \cap \mathcal{G}(\Omega_0)^\perp \right)
\end{equation*}
since $\mathrm{ker}(\mathcal{A}_0)$ is closed.

Note that the space $\mathrm{ker}(\mathcal{A}_0) \cap \mathcal{G}(\Omega_0)^\perp \subset D(\mathcal{A}_0) \cap \mathcal{G}(\Omega_0)^\perp$ and the latter coincides with  the domain of the Maxwell operator defined in \cite[Definition 1.2]{filonov19}. By \cite[Point 3), Theorem 1.3.]{filonov19}, there holds $\mathrm{ker}(\mathcal{A}_0) \cap \mathcal{G}(\Omega_0)^\perp = \{0\}$ and thus $\mathrm{ker}(\mathcal{A}_0)= \overline{\mathcal{G}(\Omega_0)}$.\\[12pt]
\noindent
{\bf Step 2:} Now we consider the operator $\hat{\mathcal{A}}$.\\[6pt]
Note that there holds  $D(\mathcal{A}_0)=D(\hat{\mathcal{A}})=H_0(\curlvec,\Omega_0)\times H(\curlvec,\Omega_0)$  coincide as sets. 
Hence by Step 1
\begin{equation*}
\begin{split}
        \mathrm{ker}(\hat{\mathcal{A}})&= \bigg\{\begin{pmatrix}
    \bE \\ \bH 
    \end{pmatrix} \in D(\hat{\mathcal{A}}) : \operatorname{curl}\bE=\operatorname{curl}\bH=0\bigg\} \\
    &=
     \bigg\{\begin{pmatrix}
    \bE \\ \bH 
    \end{pmatrix} \in D(\mathcal{A}_0) : \operatorname{curl}\bE=\operatorname{curl}\bH=0\bigg\} 
    = \mathrm{ker}(\mathcal{A}_0)=  \overline{\mathcal{G}(\Omega_0)}.
\end{split}
\end{equation*}
Here, the closure can be taken either in $\bL_0^2(\Omega_0)$ or $\bL^2_{\beps,\bmu}(\Omega_0)$ because by \eqref{relaxed:hyp} and \eqref{eqn:defjacobian!}, we have that $J_\Phi, (J_\Phi)^{-1}\in L^\infty(\Omega_0)^{3 \times 3}$ and therefore the inner product of $\bL^2_{\beps,\bmu}(\Omega_0)$ and the one of $\bL^2_0(\Omega_0)$ are equivalent. Remark that $\ker(\hat{\mathcal{A}})^\perp = (\overline{\mathcal{G}(\Omega_0)})^\perp = {\mathcal{G}(\Omega_0)}^\perp = \mathcal{J}_{\beps,\bmu}(\Omega_0)$.
\end{proof}

\begin{Rem} Note that thanks to the unitary map defined in \eqref{bbU:def}, Proposition \ref{prop:kerneldescription} also yields the structure of $\ker(\mathcal{A})$ and of its orthogonal where $\mathcal{A}$ is defined in \eqref{eqn:defAMaxwell}. Actually, for such a result to hold, one could allow more general waveguides than the one constructed through the map $\Phi$ defined in \eqref{diffeo:Phi} only requiring that $J_\Phi, (J_\Phi)^{-1}\in L^\infty(\Omega_0)^{3 \times 3}$ instead of \eqref{relaxed:hyp}.
\end{Rem}

This second proposition is well known and deals with the symmetry of the spectrum of the Maxwell operator with respect to the origin.
\begin{Pro}\label{prop.eigensym}
The spectrum, the point spectrum  and the discrete spectrum of the Maxwell operator $\mathcal{A}_{\varepsilon,\mu}$ defined in \eqref{eqn:defmaxgen} are symmetric with respect to the origin.
\label{prop:symspecmaxwell}
\end{Pro}
\begin{proof}[Proof of Proposition \ref{prop:symspecmaxwell}]
As $\mathcal{A}_{\varepsilon,\mu}$ is self-adjoint, one has $\sigma(\mathcal{A}_{\varepsilon,\mu}) \subset \R$. Consider  the  map $$\mathbb{T}: \bL^2_{\varepsilon,\mu}(\Omega) \to \bL^2_{\varepsilon,\mu}(\Omega), \quad  \mathbb{T}\begin{pmatrix}
\bE \\ \bH 
\end{pmatrix} 
\mapsto
\begin{pmatrix}
\bE \\ -\bH
\end{pmatrix}.$$ Observe that $\mathbb{T}$ is a symmetry (i.e. $\mathbb{T}^2=I$) and it is unitary.  Thus, one has $\mathbb{T}^* = \mathbb{T}^{-1}= \mathbb{T}$.
Then defining $\breve{\mathcal{A}}_{\varepsilon,\mu} :=\mathbb{T} \mathcal{A}_{\varepsilon,\mu} \mathbb{T}^*$, one notices using \eqref{eqn:defmaxgen} that  $\breve{\mathcal{A}}_{\varepsilon,\mu} :=-\mathcal{A}_{\varepsilon,\mu}$. Hence, $ \mathcal{A}_{\varepsilon,\mu}$ and $- \mathcal{A}_{\varepsilon,\mu}$  are unitarily-equivalent and therefore $$\sigma(\mathcal{A}_{\varepsilon,\mu})= \sigma(\breve{\mathcal{A}}_{\varepsilon,\mu})=\sigma(-\mathcal{A}_{\varepsilon,\mu})=\{ \lambda \in \bbR :- \lambda\in \sigma(\mathcal{A}_{\varepsilon,\mu}) \},$$ which  implies  that $\sigma(\mathcal{A}_{\varepsilon,\mu})$  is symmetric with respect to the origin. Moreover,  from the relations: $- \mathcal{A}_{\varepsilon,\mu} =\mathbb{T} \mathcal{A}_{\varepsilon,\mu} \mathbb{T}^*$  and $\mathbb{T}^* = \mathbb{T}^{-1}= \mathbb{T}$, one deduces that $\mathcal{A}_{\varepsilon,\mu}$ and $\mathbb{T}$ anti-commute, i.e. $\mathcal{A}_{\varepsilon,\mu}\mathbb{T}=-\mathbb{T}\mathcal{A}_{\varepsilon,\mu}$. Thus,
 $\big(\lambda, (\bE, \bH)^\top \big)$ is an eigenpair of $\mathcal{A}_{\varepsilon,\mu}$ if only if $(-\lambda, \mathbb{T}(\bE, \bH)^\top)$ is an eigenpair of $\mathcal{A}_{\varepsilon,\mu}$. Hence,  the point and discrete spectra of  $\mathcal{A}_{\varepsilon,\mu}$ are also symmetric with respect to $0$.

\end{proof}

The following lemma is necessary in order to characterize the kernel of the operator $\hat{\mathcal{A}}$ in Proposition \ref{prop:kerneldescription}.
\begin{Lem} \label{lemma:ortodeco}
Let $\mathcal{H}$ be a Hilbert space and $Z,B \subset \mathcal{H}$ be two closed subspaces.
Suppose $B \subset Z$.
Then
\begin{equation*}
    Z=B \overset{\perp}{\oplus} (B^{\perp_\mathcal{H}} \cap Z).
\end{equation*}
\end{Lem}
\begin{proof}[Proof of Lemma \ref{lemma:ortodeco}]
In what follows, $\overset{\mathcal{H}}{\oplus}$ denotes orthogonality with respect to the inner product of $\mathcal{H}$, while for the expression ${}^{\perp_Z}$ denotes the orthogonal (with respect to the inner product of $\mathcal{H}$) in the subspace $Z$.

Since $Z$ is also a Hilbert space (with the same inner product), and $B \subset Z$, then 
\begin{equation*}
    Z = B \overset{\mathcal{H}}{\oplus} B^{\perp_Z}.
\end{equation*}
Then it is readily seen that $B^{\perp_Z}= \{x \in Z : (x,b)_\mathcal{H}=0 \ \ \forall b \in B\} = B^{\perp_\mathcal{H}} \cap Z$, and the proof is concluded.
\end{proof}

\section{A useful PDE} \label{appendix:PDE}
The proof of Theorem \ref{thm:discspec} is by constructing an adequate trial function. To do so, one needs the following result. We recall that  $\lambda_2^N(\omega)>0$ denotes the first non-trivial eigenvalue of the Neumann Laplacian on $\omega$.
\begin{Lem}\label{lem:projsolenoidaleortho}
Assume the same hypotheses as in  Theorem \ref{thm:discspec}.
Let $\bF$ be a vector field of the form $\bF(s,\by) = \varphi(s) (0, \partial_{y_3}\psi,-\partial_{y_2}\psi)^\top$ where $\psi$ is a normalized eigenfunction associated with $\lambda_2^N(\omega)$ and $\varphi\in L^\infty(\R) $ with $\|\varphi\|_{L^\infty(\R)}\leq 1$. Let $\beps$ be the $3 \times 3$ diagonal matrix given by  
$$\beps =\varepsilon_0 \operatorname{diag}\left((1-\mathbf{k^\theta} \cdot \by)^{-1},1-\mathbf{k^\theta} \cdot \by,1-\mathbf{k^\theta} \cdot \by\right).$$
Then one has $\bF \in H(\Div \beps, \Omega_0)$ and there exists a unique weak solution  $u \in H^1_0(\Omega_0)$  to the following Dirichlet problem:
\begin{equation} \label{dirichlet:prob:eps}
\begin{cases}
\Div \beps \nabla u = \Div \beps \bF, & \text{in }\Omega_0,\\
u=0 & \text{on } \partial\Omega_0,
\end{cases}
\end{equation}
which satisfies the following estimate: 
\begin{equation} \label{estimate:nablau:eps}
\|\nabla u \|_{\bL^2_\beps(\Omega_0)}^2 \leq \varepsilon_0 \lambda_2^N(\omega)\frac{b^2\|\kappa\|_{L^\infty(\R)}}{1 - b \|\kappa\|_{L^\infty(\R)}} \|\kappa\|_{L^1(\R)}.
\end{equation}
\end{Lem}

\begin{Rem}
Note that being a weak solution to \eqref{dirichlet:prob:eps} means that
\begin{equation}
(\nabla u, \nabla v)_{\bL^2_\beps(\Omega_0)} = -(\Div (\beps \bF), v)_{L^2(\Omega_0)}
\end{equation}
for all $v \in H^1_0(\Omega_0)$.
\end{Rem}

\begin{proof}[Proof of Lemma \ref{lem:projsolenoidaleortho}]
As $|\mathbf{k^\theta}| = \kappa$ and by assumptions $\kappa \in L^1(\R) \cap L^\infty(\R)$ one has  that $\varphi \, \mathbf{k^\theta} \in L^1(\R) \cap L^\infty(\R) \subset L^2(\R)$.
Thus, as by construction $\Div \bF=0$, working in the sense of distributions, one notices   that
\begin{equation*}
\Div \beps \bF = -\varepsilon_0 \varphi \mathbf{k^\theta} \cdot 
\begin{pmatrix}
\partial_{y_3} \psi\\
-\partial_{y_2} \psi
\end{pmatrix} \in L^2(\Omega_0),
\end{equation*}
and therefore  $\bF \in H(\Div \beps, \Omega_0)$.
%
Remark that  there holds
\[
    \beps = \varepsilon_0 (\mathds{1}_3 + \mathfrak{M}),\quad\bmu = \mu_0 (\mathds{1}_3 + \mathfrak{M})
\]
where the perturbation matrix $\mathfrak{M}$ (with respect to the vacuum)  is defined as follows
\begin{equation}
    \mathfrak{M} := \begin{pmatrix}\frac{\mathbf{k^\theta} \cdot \by}{1 - \mathbf{k^\theta} \cdot \by} &0 & 0\\ 0 & - \mathbf{k^\theta} \cdot \by & 0\\ 0&0&-\mathbf{k^\theta} \cdot \by\end{pmatrix}.
    \label{def:remainderbeps}
\end{equation}
As $\Div\bF =0$, one obtains also that $\Div(\beps \bF) = \varepsilon_0\Div(\mathfrak{M} \bF) \in L^2(\Omega_0)$ where $\mathfrak{M}$ is defined in \eqref{def:remainderbeps}. Since $u$ is a weak solution to \eqref{dirichlet:prob:eps}, an integration by parts yields 
\[
    \|\nabla u\|_{\bL_{\beps}^2(\Omega_0)}^2 = - \varepsilon_0\int_{\Omega_0}\Div(\mathfrak{M} \bF) \overline{u} \rmd s \rmd \by= \varepsilon_0\int_{\Omega_0}(\mathfrak{M}\bF)\cdot(\overline{\nabla u}) = \varepsilon_0(\beps^{-1}(\mathfrak{M}\bF),\nabla u)_{\bL_{\beps}^2(\Omega_0)}.
\]
By the Cauchy-Schwarz inequality, one gets
\begin{equation}\label{eq.cauch-schw}
    \|\nabla u\|_{\bL_{\beps}^2(\Omega_0)} \leq \varepsilon_0\|\beps^{-1}(\mathfrak{M}\bF)\|_{\bL_{\beps}^2(\Omega_0)}.
\end{equation}
Note that one has
\[
    \varepsilon_0^2\|\beps^{-1}(\mathfrak{M}\bF)\|_{\bL_{\beps}^2(\Omega_0)}^2 = \varepsilon_0^2\int_{\Omega_0}\beps^{-1}(\mathfrak{M} \bF)\cdot (\mathfrak{M} \bF) \rmd s \rmd \by= \varepsilon_0^2\int_{\Omega_0}(\mathfrak{M}\beps^{-1}\mathfrak{M} \bF)\cdot \bF \rmd s \rmd \by.
\]
But one remarks that there holds
\[
    \varepsilon_0^2(\mathfrak{M}\beps^{-1}\mathfrak{M} \bF) =   \varepsilon_0 \frac{(\mathbf{k^\theta} \cdot \by)^2}{1-\mathbf{k^\theta} \cdot \by} \bF.
\]
This implies with \eqref{eq.cauch-schw} that
\[
    \|\nabla u\|_{\bL_{\beps}^2(\Omega_0)}^2 \leq  \varepsilon_0 \frac{b^2\|\kappa\|_{L^\infty(\R)}}{1 - b \|\kappa\|_{L^\infty(\R)}}\int_{\Omega_0}|\kappa| |\varphi|^2|\nabla \psi|^2 \rmd s \rmd \by\leq \varepsilon_0  \lambda_2^N(\omega)\frac{b^2\|\kappa\|_{L^\infty(\R)}}{1 - b \|\kappa\|_{L^\infty(\R)}} \|\kappa\|_{L^1(\R)},
\]
where we have used the pointwise inequality $|\mathbf{\mathbf{k^\theta}} \cdot \by| \leq b |\kappa|$ and that $\|\varphi\|_{L^\infty(\R)} \leq 1$.
It concludes the proof of the lemma.
\end{proof}
\section{Proof of Lemma \ref{Lem-Lipschitz} on the Lipschitz regularity of $t\mapsto \lambda_j^N(\omega_t)$}\label{Sec-reg-Lips-eigen-Neumann}
\begin{proof}[Proof of Lemma \ref{Lem-Lipschitz}]
Let $j \in \N$ and $t\in [-T,T]$.  $\lambda_j^N(\omega_t)$ be the $j$-th eigenvalue of the Neumann Laplacian $L^{N}_{\omega_t}$ in $\omega_t$. Remark that as $L^{N}_{\omega_t}$ is self-adjoint, positive semi-definite and has a compact resolvent, $\lambda_j^N(\omega_t)$ coincides with the $j$-th min max-level of the Rayleigh quotient
\begin{equation*}
\lambda_j^N(\omega_t)=\min_{\widetilde{W} \subset H^1(\omega_t) \atop \dim \widetilde{W} = j} \max_{\tilde{u} \in \widetilde{W} \setminus\{0\}} \tilde{R}_q(\tilde{u};t) \ \mbox{ with } \ \tilde{R}_q(\tilde{u};t):=   \frac{\int_{\omega_t}|\nabla \tilde{u}|^2 \rmd \by}{\int_{\omega_t}|\tilde{u}|^2 \rmd \by},\ \forall \tilde{u} \in H^1(\omega_t)\setminus\{0\}.
\end{equation*}
To Pull-back this expression in the fixed domain $\omega$, one introduce the unitary transforms (indexed by $t$) $$\bbV_{\Theta_t}:  L^2(\omega_t) \to L^2(\omega), \quad  \tilde{u} \mapsto u=\tilde{u} \circ \Theta_t.$$
One checks  that $\bbV_{\Theta_t}(H^1(\omega_t))=H^1(\omega)$ and that actually $\lambda_j^N(\omega_t)$ is  also given by the following $j$-th min-max level of the  Rayleigh quotient $R_q(u,\cdot)$ defined as follows:
\begin{eqnarray}\label{label:minmax}
 &&\lambda_j^N(\omega_t)=\min_{W \subset H^1(\omega) \atop \dim W = j} \max_{u \in W \setminus\{0\}} R_q(u;t) \\[6pt] \mbox{ with } && 
R_q(u;t)=   \frac{\int_\omega|J_{\Theta(t)}^{-\top}\nabla u|^2\det(J_{\Theta(t)})\rmd \by}{\int_{\omega}|u|^2 \det(J_{\Theta(t)})\rmd \by},\ \forall u \in H^1(\omega)\setminus\{0\}.
\label{label:rayleigh}
\end{eqnarray}
\noindent The remainder of the proof is partially inspired by the approach developed in the Appendix A of \cite{FW:14}, which we adapt to our context. \medskip

First,  one has  (see Section \S \ref{subsubsec:regularityeigenpair}) that $t\in[-T,T]\to J_{\Theta(t)}^{-\top}$ is $\mathscr{C}^{1}\left( [-T,T],L^\infty(\R^2,\R^{2\times 2})\right)$ and  $t\in[-T,T]\mapsto \det (J_{\Theta(t)})$ is
$\mathscr{C}^{1}\left([-T,T],L^\infty(\R^2,\R)\right)$  (and they are also  Lipschitz continuous on $[-T,T]$). One has also that $t\in[-T,T]\mapsto \det (J_{\Theta(t)})^{-1}$  
is  bounded in $t$  with respect to the norm of $L^\infty(\R^2,\R)$.  Therefore it implies (by some straightforward computations) that  $R_q(u;\cdot)$  (as a function of $t$)
is  Lipschitz continuous (uniformly with respect to $u\in H^1(\omega)$).
More precisely, there exists $C>0$ such that for all $t,t' \in [-T,T]$:
\begin{equation} \label{lip:cont:Rayleigh}
| R_q(u;t)-R_q(u;t')|\leq C \, |t-t'| \ \frac{\|\nabla u\|_{L^2(\omega)}^2}{\|u\|_{L^2(\omega)}^2}, \ \forall u \in H^1(\omega).
\end{equation}

Let  $t,t'\in [-T,T]$. Combining  \eqref{label:minmax} and \eqref{lip:cont:Rayleigh} gives that for any subspace $W$ of $H^1(\omega)$ of dimension $j$, we have
 \begin{equation*}
\max_{u \in W \setminus\{0\}} R_q(u;t) \leq \max_{u \in W \setminus\{0\}} R_q(u;t')+ C |t-t'| \ \max_{u \in W \setminus\{0\}}  \frac{\|\nabla u\|_{L^2(\omega)}^2}{\|u\|_{L^2(\omega)}^2} . 
 \end{equation*}
Thus, lowering the left-hand side by the minimum among all subspaces of $H^1(\omega)$ of dimension $j$, we get
 \begin{equation}\label{eq.minmaxbis}
 \lambda_j^N(\omega_t) \leq \max_{u \in W \setminus\{0\}} R_q(u;t')+ C |t-t'| \ \max_{u \in W \setminus\{0\}}  \frac{\|\nabla u\|_{L^2(\omega)}^2}{\|u\|_{L^2(\omega)}^2} . 
 \end{equation}
Then, one chooses for $W$  a $j-$dimensional subspace $W_{t',j}$ of $H^1(\omega)$ for which  the min in \eqref{label:minmax}  is reached for $t=t'$.
Thus, together with \eqref{eq.minmaxbis}, it yields that
\begin{equation}\label{eq.estimate-min-max}
\lambda_j^N(\omega_t)\leq \lambda_j^N(\omega_{t'})+C \, |t-t'|  \max_{u \in W_{t',j} \setminus\{0\}}  \frac{\|\nabla u\|_{L^2(\omega)}^2}{\|u\|_{L^2(\omega)}^2}.
\end{equation}
In order to control the last term in the right-hand side of \eqref{eq.estimate-min-max}, one  shows  that  there exists $\tilde{C}>0$ such that  for all $u \in H^1(\omega)\setminus \{0\}$ and for all $s\in[-T,T]$:
\begin{equation}\label{eq.estimate-min-max2}
\frac{\|\nabla u\|^2_{L^2(\omega)}}{\|u\|^2_{L^2(\omega)}} =   \frac{\int_\omega| J_{\Theta(s)}^{\top} J_{\Theta(s)}^{-\top}\nabla u|^2 \det(J_{\Theta(s)})^{-1} \det(J_{\Theta(s)})\rmd \by}{\int_{\omega}|u|^2 \det(J_{\Theta(s)})^{-1}  \det(J_{\Theta(s)})\rmd \by} \leq  \tilde{C} R_q(u;s), 
\end{equation}
where we have  bounded $\|J_{\Theta(s)}^{\top}\|_{L^\infty(\R^2,\R^{2 \times 2})}$, $\|\det(J_{\Theta(s)})^{-1}\|_{L^\infty(\R^2,\R)}$ and $\|\det(J_{\Theta(s)})\|_{L^\infty(\R^2,\R)}$ uniformly for $ s \in [-T,T]$.
Thus, up to changing the constant $C>0$, combining \eqref{eq.estimate-min-max} and \eqref{eq.estimate-min-max2}  for $s=t'$ gives that
\begin{equation}\label{eq.estimatelambdat}
\lambda_j^N(\omega_t)\leq \lambda_j^N(\omega_{t'})+C |t-t'|  \, \max_{u \in W_{t',j}  \setminus\{0\}} R_q(u;t') = \lambda_j^N(\omega_{t'})+C |t-t'|  \, \lambda_j^N(\omega_{t'}).
\end{equation}
In a similar manner, reversing the role of $t$ and $t'$, one gets
\begin{equation}\label{eq.estimatelambdatp}
\lambda_j^N(\omega_{t'})\leq \lambda_j^N(\omega_{t'})+C \, |t-t'|  \, \lambda_j^N(\omega_{t}).
\end{equation}
Thus, by virtue of \eqref{eq.estimatelambdat} and \eqref{eq.estimatelambdatp},  one gets that
\begin{eqnarray*}
|\lambda_j^N(\omega_{t'})-\lambda_j^N(\omega_t)|&\leq & C |t-t'|  \, (\lambda_j^N(\omega_{t'})+ \lambda_j^N(\omega_{t}) ) \\
&= &C |t-t'|  \, ( 2\lambda_j^N(\omega_{t})  +\lambda_j^N(\omega_{t'}) - \lambda_j^N(\omega_{t}))  \\
&\leq & 2C |t-t'|  \, \lambda_j^N(\omega_{t})+C |t-t'|  \, |\lambda_j^N(\omega_{t'})- \lambda_j^N(\omega_{t})|
\end{eqnarray*}
Thus, for $t,t'\in [-T,T]$ satisfying $|t-t'|< 1/(2\, C)$, it  follows that
\begin{equation}\label{eq.Lipsch}
\frac{1}{2} |\lambda_j^N(\omega_{t'})-\lambda_j^N(\omega_t) |\leq 2 \,C  \, \lambda_j^N(\omega_{t})\, |t-t'| .
\end{equation}
This proves that $t\mapsto \lambda_j(\omega_t)$ is Lipschitz pointwise on $[-T,T]$.\medskip

To conclude that $t\mapsto \lambda_j(\omega_t)$ is  Lipschitz continuous on $[-T,T]$, one proves that $t\mapsto
\lambda_j^N(\omega_t)$ is bounded on  $[-T,T]$. Indeed, similarly as in \eqref{eq.estimate-min-max2}, using that $\|J_{\Theta(s)}^{-\top}\|_{L^\infty(\R^2,\R^{2 \times 2})}$, $\|\det(J_{\Theta(s)})^{-1}\|_{L^\infty(\R^2,\R)}$ and $\|\det(J_{\Theta(s)})\|_{L^\infty(\R^2,\R)}$ are  uniformly bounded in $ s$ for $s \in [-T,T]$, one shows that there exists a constant $C_{1}>0$ such that  
$$
R_q(u;s)\leq C_{1} \, \frac{\|\nabla u\|_{L^2(\omega)}^2}{\|u\|_{L^2(\omega)}^2}, \, \forall u \in H^1(\omega)\setminus \{0\} \mbox{ and } \forall s\in[-T,T].
$$
Thus, using the min-max principle, it implies that 
\begin{equation}\label{eq.boundlambdaN}
0 \leq \lambda_j^N(\omega_s) \leq  C_{1} \, \lambda_j^N(\omega).
\end{equation}
Combining \eqref{eq.Lipsch} and \eqref{eq.boundlambdaN} yields that
for $t,t'\in [-T,T]$ satisfying $|t-t'|< 1/(2\, C)$, 
\begin{equation}\label{eq.Lipsch2}
 |\lambda_j^N(\omega_{t'})-\lambda_j^N(\omega_t) |\leq 4 \,C  \, C_{1} \, \lambda_j^N(\omega) \, |t-t'| .
\end{equation}
This last inequality clearly implies that there eixsts a constant $C_{2}>0$ such that 
$$
|\lambda_j^N(\omega_{t'})-\lambda_j^N(\omega_t) |\leq C_{2}\ |t-t'|,\quad \forall t, t'\in [-T,T],
$$
or in other words that $t \mapsto\lambda_j^N(\omega_t) $ is Lipschitz continuous on $[-T,T]$.
\end{proof}

\end{document}